\documentclass[leqno]{article}

\usepackage[frenchb,english]{babel}
\usepackage[utf8]{inputenc}
\usepackage{amsmath}
\usepackage{amssymb}
\usepackage{amsfonts}
\usepackage{enumerate}
\usepackage{vmargin}
\usepackage[all]{xy}
\usepackage{mathrsfs}
\usepackage{mathtools}
\usepackage{lmodern}
\usepackage{slashed}
\usepackage[colorlinks=true,linkcolor=blue,pagebackref=true]{hyperref}%
\setmarginsrb{3cm}{3cm}{3.5cm}{3cm}{0cm}{0cm}{1.5cm}{3cm}
\usepackage{comment}
\usepackage{tikz}
\usetikzlibrary{patterns}

\newcommand{\C}{\mathrm{C}}
\newcommand{\D}{\ensuremath{\mathcal{D}}}

\newcommand{\B}{\mathrm{B}} 
\let\H\relax 
\newcommand{\H}{\mathrm{H}}

\let\L\relax 
\newcommand{\L}{\mathrm{L}}

\newcommand{\V}{\mathrm{V}}

\newcommand{\scr}{\mathscr}

\newcommand{\Inv}{\mathrm{Inv}} 

\newcommand{\M}{\mathrm{M}}

\newcommand{\dist}{\mathrm{dist}} 
\newcommand{\rank}{\mathrm{rank}} 

\newcommand{\even}{\mathrm{even}} 
\newcommand{\odd}{\mathrm{odd}} 
\newcommand{\bi}{\mathrm{bi}} 

\let\cal\relax
\newcommand{\cal}{\mathcal}
\newcommand{\Z}{\ensuremath{\mathbb{Z}}}
\newcommand{\R}{\ensuremath{\mathbb{R}}}

\newcommand{\W}{\mathrm{W}}

\newcommand{\Id}{\mathrm{Id}}

\newcommand{\la}{\langle}
\newcommand{\ra}{\rangle}

\renewcommand{\leq}{\ensuremath{\leqslant}}
\renewcommand{\geq}{\ensuremath{\geqslant}}
\newcommand{\qed}{\hfill \vrule height6pt  width6pt depth0pt}
\newcommand{\bnorm}[1]{ \big\| #1  \big\|}
\newcommand{\Bnorm}[1]{ \Big\| #1  \Big\|}

\newcommand{\norm}[1]{\left\Vert#1\right\Vert}

\newcommand{\xra}{\xrightarrow}
\newcommand{\co}{\colon}

\newcommand{\ot}{\otimes}
\newcommand{\ovl}{\overline}
\newcommand{\otvn}{\ovl\ot}
\DeclareMathOperator{\sgn}{\mathrm{sgn}} 

\newcommand{\fin}{\mathrm{fin}} 

\newcommand{\Ch}{\mathrm{Ch}}
\newcommand{\HdR}{\mathrm{HdR}} 

\let\i\relax 
\newcommand{\i}{\mathrm{i}}
\newcommand{\ov}{\overset}

\newcommand{\vol}{\mathrm{vol}} 

\newcommand{\K}{\mathrm{K}} 

\newcommand{\Vol}{\mathrm{Vol}}

\newcommand{\UMD}{\mathrm{UMD}} 

\newcommand{\epsi}{\varepsilon}
\renewcommand{\d}{\mathop{}\mathopen{}\mathrm{d}} 
\newcommand{\e}{\mathrm{e}} 
\renewcommand{\d}{\mathop{}\mathopen{}\mathrm{d}}

\DeclareMathOperator{\Lip}{\mathrm{Lip}} 
\DeclareMathOperator{\Index}{Index} 
\let\ker\relax 
\DeclareMathOperator{\ker}{Ker} 
\DeclareMathOperator{\Ran}{Ran} 
\DeclareMathOperator{\dom}{dom} 
\DeclareMathOperator{\diag}{diag} 
\let\Re\relax 
\DeclareMathOperator{\Re}{Re} 

\DeclareMathOperator{\Hom}{Hom}
\DeclareMathOperator{\coker}{Coker} 
\newcommand{\LC}{\mathrm{LC}} 

\DeclareMathOperator{\card}{card} 
\DeclareMathOperator*{\esssup}{ess\,sup} 
\DeclareMathOperator*{\essinf}{ess\,inf} 
 
\DeclareMathOperator{\diam}{diam} 

\DeclareMathOperator{\Curv}{Curv}
\newtheorem{thm}{Theorem}[section]
\newtheorem{defi}[thm]{Definition}
\newtheorem{prop}[thm]{Proposition}

\newtheorem{cor}[thm]{Corollary}
\newtheorem{lemma}[thm]{Lemma}

\newtheorem{remark}[thm]{Remark}
\newtheorem{example}[thm]{Example}

\newtheorem{ass}[thm]{Assumption}

\newenvironment{proof}[1][]{\noindent {\it Proof #1} : }{\hbox{~}\qed
\smallskip
}

\usepackage{tocloft}
\numberwithin{equation}{section}
\usepackage[nottoc,notlot,notlof]{tocbibind}

\let\OLDthebibliography\thebibliography
\renewcommand\thebibliography[1]{
  \OLDthebibliography{#1}
  \setlength{\parskip}{0pt}
  \setlength{\itemsep}{0pt plus 0.3ex}
}

\newcommand\reallywidehat[1]{\arraycolsep=0pt\relax%
\begin{array}{c}
\stretchto{
  \scaleto{
    \scalerel*[\widthof{\ensuremath{#1}}]{\kern-.5pt\bigwedge\kern-.5pt}
    {\rule[-\textheight/2]{1ex}{\textheight}} 
  }{\textheight} %
}{0.5ex}\\           
#1\\                 
\rule{-1ex}{0ex}
\end{array}
}

\begin{document}
\selectlanguage{english}
\title{\bfseries{Curvature, Dolbeault--Dirac operators, and an $\L^p$-index theorem on compact K\"ahler manifolds}}
\date{}
\author{\bfseries{C\'edric Arhancet}}
\maketitle


\begin{abstract}
We develop an $\mathrm{L}^p$-Banach noncommutative-geometric framework for Dolbeault--Dirac operators on compact K\"ahler manifolds with coefficients in a Hermitian holomorphic vector bundle $E$. For every $p \in (1,\infty)$, we prove that the closed $\mathrm{L}^p$-realization $\mathcal{D}_{E,p}$ of the Dolbeault--Dirac operator is bisectorial and admits a bounded $\mathrm{H}^\infty$ functional calculus on $\mathrm{L}^p(\Omega^{0,\bullet}(M,E))$. We also show an $\mathrm{L}^p$-Gaffney-type estimate, obtain $\mathrm{L}^p$-Hodge decompositions, and prove that $\mathcal{D}_{E,p}$ gives rise to an even compact Banach spectral triple over the algebra $\mathrm{C}(M)$, graded by form parity. The index of the associated Fredholm operator is equal to the holomorphic Euler characteristic $\chi(M,E)$. In particular, it is independent of $p$. A central tool is an abstract notion of Ricci curvature lower bound for strongly continuous semigroups on $\mathrm{UMD}$ Banach spaces, formulated as a semigroup-level intertwining relation. Under this condition, together with natural Riesz equivalences and bounded $\mathrm{H}^\infty$ functional calculi for the relevant generators, the associated Hodge--Dirac operator is bisectorial and admits a bounded $\mathrm{H}^\infty$ functional calculus. The framework also applies to heat semigroups on Riemannian manifolds, $q$-Ornstein--Uhlenbeck semigroups and  semigroups of Schur multipliers. This provides a unified Banach-space approach to curvature, functional calculus, Riesz transforms and index theory beyond the Hilbert space setting.
\end{abstract}


\makeatletter
 \renewcommand{\@makefntext}[1]{#1}
 \makeatother

{
  \hypersetup{linkcolor=blue}
 \tableofcontents
}

\section{Introduction}
\label{sec:Introduction}

Let $E$ be a Hermitian holomorphic vector bundle of finite rank over a compact K\"ahler manifold $M$. The Dolbeault--Dirac operator
\begin{equation}
\label{Def-Dirac-DE-bis}
\D_E
\ov{\mathrm{def}}{=} \sqrt{2}(\bar\partial_{E,0,\bullet} + \bar\partial_{E,0,\bullet}^*)
\end{equation}
is a first-order elliptic differential operator on the space $\Omega^{0,\bullet}(M,E)$ of smooth anti-holomorphic differential forms, whose square $\D_E^2 = 2\Delta_{\bar\partial_E,0,\bullet}$ is twice the Kodaira Laplacian $\Delta_{\bar\partial_E,0,\bullet}$ acting on the subspace $\Omega^{0,\bullet}(M,E)$ and  defined by
\begin{equation*}
\Delta_{\bar\partial_E,0,\bullet}
\ov{\mathrm{def}}{=} \bar{\partial}_{E,0,\bullet}\bar{\partial}_{E,0,\bullet}^*+\bar{\partial}_{E,0,\bullet}^*\bar{\partial}_{E,0,\bullet}.
\end{equation*}
On the Hilbert space $\L^2(\Omega^{0,\bullet}(M,E))$, its functional-analytic and index-theoretic properties are part of the classical Hodge theory of compact K\"ahler manifolds. The present paper presents the corresponding $\L^p$-theory. For every $p \in (1,\infty)$, we observe that the operator $\D_E$ is closable on the Banach space $\L^p(\Omega^{0,\bullet}(M,E))$. We denote by $\D_{E,p}$ its closure. With respect to the canonical grading $\gamma \co \L^p(\Omega^{0,\bullet}(M,E))\to \L^p(\Omega^{0,\bullet}(M,E))$ by form parity, $\D_{E,p}$ decomposes as
\[
\D_{E,p}
=
\begin{bmatrix}
0 & \D_{E,-,p} \\
\D_{E,+,p} & 0
\end{bmatrix},
\]
with $\D_{E,+,p} \co \W^{1,p}_{\bar\partial_E}(\Omega^{0,\even}(M,E)) \to \L^p(\Omega^{0,\odd}(M,E))$. Our main result is the following Banach noncommutative-geometric formulation of the $\L^p$-theory of the Dolbeault--Dirac operator.

\begin{thm}
\label{thmA-intro}
Let $E$ be a Hermitian holomorphic vector bundle of finite rank over a compact K\"ahler manifold $M$. Suppose that $1 < p < \infty$.
\begin{enumerate}
\item The closed operator $\D_{E,p}$ is bisectorial on $\L^p(\Omega^{0,\bullet}(M,E))$ and admits a bounded $\H^\infty(\Sigma^\bi_\theta)$ functional calculus for some angle $\theta \in (0,\frac{\pi}{2})$.

\item The quadruple $(\C(M),\L^p(\Omega^{0,\bullet}(M,E)),\D_{E,p},\gamma)$ is an even compact Banach spectral triple in the sense of \cite{Arh26a}.

\item The Fredholm operator $\D_{E,+,p}$ has index
\[
\Index \D_{E,+,p}
= \chi(M,E),
\]
where $\chi(M,E)$ is the holomorphic Euler characteristic of $E$. In particular, $\Index \D_{E,+,p}$ is independent of $p$.
\end{enumerate}
\end{thm}

The Hilbert case $p=2$ is classical: $\D_{E,2}$ is selfadjoint, the index formula relies on Hodge theory and \eqref{Def-Dirac-DE-bis} is a classical Dirac-type operator giving rise, in the Hilbertian setting, to a spectral triple over $\C(M)$ in the sense of Connes \cite{Con94}.

For $p \neq 2$, the picture changes substantially: selfadjointness is lost, $\L^p$-Hodge decompositions are no longer automatic, and the spectral analysis must instead be formulated in terms of bisectoriality and bounded $\H^\infty$ functional calculus. Theorem~\ref{thmA-intro} shows that these difficulties can be overcome within the framework of Banach noncommutative geometry. The operator $\D_{E,p}$ retains enough of the analytic structure of the Hilbertian case to define a compact Banach spectral triple, while its spectral analysis is governed by bisectoriality and bounded $\H^\infty$ functional calculus rather than selfadjointness. Thus the result is not merely an $\L^p$-boundedness statement for a geometric operator. It identifies the analytic structure that allows the Dolbeault--Dirac operator to enter the machinery of Banach noncommutative geometry \cite{Arh26a} \cite{Laf02}.

The bulk of the paper is devoted to the proof of Theorem~\ref{thmA-intro}, which rests on two main ingredients of independent interest: an abstract curvature condition for semigroups on $\UMD$ Banach spaces, which yields bounded $\H^\infty$ functional calculus for Hodge--Dirac type operators (Theorem~\ref{thmB-intro} below) and an $\L^p$-Gaffney-type estimate for $\bar\partial_E$, which identifies the graph domain of $\D_{E,p}$ with a classical Sobolev space. In this paper, the symbols $\approx$ and $\lesssim$ denote an equality or an inequality up to multiplicative constants.

\begin{thm}
\label{thmB-intro}
Let $X$ and $Y$ be $\UMD$ Banach spaces. Let $-A$ be the generator of a bounded strongly continuous semigroup $(T_t)_{t \geq 0}$ of operators on $X$. Let $\partial \co \dom \partial \subset X \to Y$ and $\partial^\dagger \co \dom \partial^\dagger \subset Y \to X$ be densely defined closed operators with $A = \partial^\dagger \partial$. Assume that:
\begin{enumerate}
\item $(T_t)_{t \geq 0}$ satisfies the curvature condition $\Curv_{\partial,\H^\infty}(0)$ of Definition~\ref{curvature-H-infty}.
\item $A$ admits the $\partial$-Riesz equivalence and $A^*$ admits the $(\partial^\dagger)^*$-Riesz equivalence, that is,
\[
\bnorm{A^{\frac{1}{2}}x}_X \approx \norm{\partial x}_Y,
\qquad x \in \dom \partial,
\]
and
\[
\bnorm{(A^*)^{\frac{1}{2}}z}_{X^*}
\approx \norm{(\partial^\dagger)^*z}_{Y^*},
\qquad z \in \dom(\partial^\dagger)^*.
\]
\end{enumerate}
Then the Hodge--Dirac operator $
D
\ov{\mathrm{def}}{=}
\begin{bmatrix}
0 & \partial^\dagger|_{\ovl{\Ran \partial}} \\
\partial & 0
\end{bmatrix}$ is bisectorial on the reduced space $X \oplus_2 \ovl{\Ran \partial}$ and admits a bounded $\H^\infty(\Sigma^\bi_\omega)$ functional calculus for some angle $\omega \in (0,\frac{\pi}{2})$. Under the additional compatibility Assumption~\ref{ass-projection-reduction}, the same conclusion holds on the full space $X \oplus_2 Y$.
\end{thm}

For any $\lambda \in \R$, the curvature condition $\Curv_{\partial,\H^\infty}(\lambda)$ is a semigroup-level intertwining: there exists a bounded strongly continuous semigroup $(\tilde T_t)_{t \geq 0}$ on $Y$ with generator $-\tilde A$ such that
\begin{equation}
\label{commut-lambda}
\partial T_t x 
= \e^{-\lambda t} \tilde T_t \partial x,
\quad t \geq 0, x \in \dom \partial,
\end{equation}
together with the requirement that $A$ and $\tilde A$ both admit a bounded $\H^\infty$ functional calculus. Identity \eqref{commut-lambda} is the abstract counterpart of the commutation rule $\d\, \e^{t\Delta} = \e^{t\Delta_1} \d$ between heat semigroups and exterior derivatives, with an additional factor $\e^{-\lambda t}$ of curvature. We have the following  monotonicity property: if $\Curv_{\partial,\H^\infty}(\lambda)$ holds, then $\Curv_{\partial,\H^\infty}(\lambda')$ holds if $\lambda' \leq \lambda$.

Theorem~\ref{thmB-intro} provides a unified mechanism for verifying the analytic input required in Banach noncommutative geometry, in the sense of \cite{Arh26a} and \cite{ArK22}. It shows that a curvature-type semigroup intertwining, combined with Riesz equivalences and bounded $\H^\infty$ calculi for the underlying generators, automatically yields the bisectorial functional calculus
needed for Banach spectral triples and index pairings. Thus the abstract framework separates the construction into two conceptually distinct parts: a curvature/intertwining principle at the level of semigroups, and analytic estimates, such as Riesz equivalences and bounded $\H^\infty$ functional calculi for the underlying generators, which can be checked by geometric or probabilistic methods in concrete examples.

The role of Theorem~\ref{thmB-intro} in the proof of Theorem~\ref{thmA-intro} is the following: the Dolbeault--Dirac operator $\D_{E,p}$ fits the abstract setting with $X = \L^p(\Omega^{0,\even}(M,E))$, $Y = \L^p(\Omega^{0,\odd}(M,E))$, $\partial = \D_{E,+,p}$, $\partial^\dagger = \D_{E,-,p}$. In this realization, the corresponding generator on $X$ is $
\partial^\dagger \partial
=
\D_{E,-,p}\D_{E,+,p}
=
2\Delta_{\bar\partial_E,0,\even,p}$. Thus $(T_t)_{t \geq 0}$ is the heat semigroup generated by twice the even part of the Kodaira Laplacian. The required curvature intertwining reduces to the classical commutation relations between the Kodaira Laplacian and the Dolbeault differential and its formal adjoint, equivalently $
\D_{E,+}\Delta_{\bar\partial_E,0,\even}
=
\Delta_{\bar\partial_E,0,\odd} \D_{E,+}$, see Lemma~\ref{lem:commutation-Kodaira}.

\paragraph{Comparison with the literature} 
Bisectoriality and bounded $\H^\infty$ functional calculus for Hodge--Dirac type operators have a long history, notably through the first-order approach to the Kato square root problem developed in \cite{AKM06}. We additionally refer to \cite{AAM10}, \cite{AHLMT1} and \cite{Tch01} for the original solution of the Kato square root problem. Subsequent developments in various contexts include \cite{MaN09}, \cite{HMP08}, \cite{HMP11}, \cite{McM16}, \cite{NeV18}, \cite{McM18}, \cite{FMP18}, \cite{ArK22}, see also the survey \cite{Ban19}. The closest precedent for the curvature-intertwining viewpoint is the $\lambda$-Ricci curvature condition $\lambda$-GRic of Brannan, Gao and Junge \cite{BGJ22, BGJ23}, which requires the auxiliary semigroup $(\tilde T_t)$ to live on a von Neumann algebra $\tilde{\cal M}$ containing $\cal M$ and to satisfy the rigid restriction identity $\tilde T_t|_{\cal{M}} = T_t$. Our condition $\Curv_{\partial,\H^\infty}(\lambda)$ drops both requirements: $Y$ need not be a noncommutative $\L^p$-space, and no restriction identity is imposed. This is what permits the monotonicity property (not available in the rigid frameworks of \cite[Definition 3.26 p.~30]{BGJ22} and \cite{BGJ23}) and what allows direct application to operators such as $\D_{E,p}$, where the tangent space is the $\L^p$-space of a vector bundle of forms rather than a von Neumann algebra. When $\tilde{\cal{M}}$ is finite, $\lambda$-GRic implies $\Curv_{\partial,\H^\infty}(\lambda)$ on $X = \L^p(\cal M)$ and $Y = \L^p(\tilde{\cal M})$ for a canonical pair $(\partial,\partial^\dagger)$. It is worth noting that the $\lambda$-GRic framework was originally developed in a rather different direction, namely in connection with noncommutative curvature and functional inequalities, rather than index-theoretic or spectral-triple questions.

Theorem~\ref{thmA-intro} appears to be the first formulation of the Dolbeault--Dirac theory on compact K\"ahler manifolds as a compact Banach spectral triple on the full scale of spaces $\L^p$, $1<p<\infty$. Analysis on $\L^p$-spaces over complete K\"ahler manifolds was studied in \cite{Li10} and \cite{LaT15}, but without any spectral-triple or index-theoretic perspective. 

\paragraph{Further applications of the abstract framework} Beyond the Dolbeault--Dirac case, Theorem~\ref{thmB-intro} applies to several classical and noncommutative situations, which we collect in Section~\ref{sec-examples}: heat semigroups on Riemannian manifolds (recovering results of \cite{NeV18}, obtained there under suitable positive curvature assumptions), $q$-Ornstein--Uhlenbeck semigroups, and symmetric Markovian semigroups of Schur multipliers (recovering a result of \cite{ArK22}). These examples show that the curvature-intertwining principle provides a unified framework encompassing situations that had previously been studied separately. Further developments of the framework, including quantum groups and compact Riemannian manifolds without curvature assumptions, are pursued in subsequent works, the first of which is \cite{Arh26b}.

\paragraph{Methodological ingredients} The proof of Theorem~\ref{thmA-intro} combines three independent technical steps. First, an $\L^p$-Gaffney-type estimate for $\bar\partial_E$,
\begin{equation*}
\norm{\nabla \omega}_{\L^p(\Omega^1(M,\Lambda^{0,q}\mathrm{T}^*M \ot E))}
\lesssim \norm{\omega}_{\L^p(\Omega^{0,q}(M,E))}
+\norm{\bar\partial_E \omega}_{\L^p(\Omega^{0,q+1}(M,E))}
+\norm{\bar\partial_E^*\omega}_{\L^p(\Omega^{0,q-1}(M,E))}.
\end{equation*}
This estimate is an analogue of the $\L^p$-Gaffney inequality for Riemannian manifolds proved by Scott \cite{Sco95}. Our proof is based on an elliptic parametrix argument, rather than on the method used in \cite{Sco95}, which relies in part on Riesz transforms. As a byproduct, we obtain a new global proof of the classical $\L^p$-Gaffney inequality stated in \eqref{Gaffney-Lp}, see Remark~\ref{remark-Gaffney}.

Second, the boundedness of Dolbeault--Riesz transforms and the bounded $\H^\infty$ functional calculus of the closure $\Delta_{\bar\partial_E,p}$ of the Kodaira Laplacian on the Banach space $\L^p(\Omega^{0,\bullet}(M,E))$, obtained from pseudo-differential calculus, Gaussian heat kernel domination, and a variant of an extrapolation principle of \cite{DuR96} (see also \cite{Hal05}). 

Third, a strong $\L^p$-Hodge decomposition for the Dolbeault complex, constructed from a parametrix Green operator. The compactness of the resolvent of $\D_{E,p}$ follows from a classical compact Sobolev embedding via the identification of the space $\dom \D_{E,p}$ with the Sobolev space $\W^{1,p}_\nabla(\Omega^{0,\bullet}(M,E))$.

\paragraph{Structure of the paper.} Section~\ref{sec-preliminaries} collects background on sectorial and bisectorial operators and their bounded $\H^\infty$ functional calculus. Section~\ref{sec-Hodge-Dirac} develops the abstract curvature condition and proves Theorem~\ref{thmB-intro}. Section \ref{Back-complex} and Section \ref{sec-Hinfty-calculus-Dolbeault}  establish the analytic input on compact K\"ahler manifolds: the grading on $\L^p$, an $\L^p$-Gaffney estimate, the bounded $\H^\infty$ functional calculus of $\Delta_{\bar\partial_E,p}$ and of $\D_{E,p}$, and an $\L^p$-Hodge decomposition. Section~\ref{sec-Banach-spectral-triples} proves that $(\C(M),\L^p(\Omega^{0,\bullet}(M,E)),\D_{E,p},\gamma)$ is an even compact Banach spectral triple. Section~\ref{sec:dolbeault-index-pairing} establishes the index formula via the K-theory/K-homology pairing of Banach noncommutative geometry. Section~\ref{sec-examples} contains the further applications of Theorem~\ref{thmB-intro}. Finally, in Section \ref{Appendix}, we prove a variant of an extrapolation principle of Duong, Robinson and Haller-Dintelmann.

\section{Preliminaries}
\label{sec-preliminaries}

\subsection{Operator theory}
\paragraph{Unbounded operators} Let $X$ and $Y$ be two Banach spaces. 
An unbounded operator $S \co \dom S \subset Y^* \to X^*$ is a formal adjoint of $T$ if we have
\begin{equation}
\label{formal-adjoint}
\la T(x),y \ra_{Y,Y^*}
=\la x,S(y) \ra_{X,X^*}, \quad x \in \dom T, y \in \dom S.
\end{equation}
If $T$ is \textit{densely defined}, there exists, by \cite[p.~167]{Kat76}, a unique maximal formal adjoint $T^*$ and $\dom T^*$ is equal to 
\begin{equation}
\label{Def-domaine-adjoint}
\big\{ y \in Y^* : \text{there exists } z \in X^* \text{ such that } \langle T(x),y \rangle_{Y,Y^*}=\langle x,z \rangle_{X,X^*} \text{ for all } x \in \dom T\big\}.
\end{equation}
If $y \in \dom T^*$, the previous $z \in X^*$ is determined uniquely by $y$ and we let $T^*(y) \ov{\mathrm{def}}{=} z$. For any $x \in \dom T$ and any $y \in \dom T^*$, we have the equality
\begin{equation}
\label{crochet-duality}
\langle T(x),y \rangle_{Y,Y^*}
=\langle x,T^*(y) \rangle_{X,X^*}.
\end{equation}
According to \cite[Problem 5.27 p.~168]{Kat76}, if the operator $T$ is densely defined we have
\begin{equation}
\label{lien-ker-image}
\ker T^*
=(\Ran T)^\perp.
\end{equation}

The following result is a variation of \cite[Proposition 10.30 p.~321]{Nee22}, which will be used in the proof of Theorem \ref{prop:closure-kodaira-square-dirac}.

\begin{lemma}
\label{lemma-extension-surjective-injective}
Let $A$ and $B$ be unbounded operators on a Banach space $X$ such that $A \subset B$. Suppose that there exists $\lambda \in \mathbb{C}$ such that $\lambda\Id_X+A \co \dom A \to X$ is surjective and $\lambda\Id_X+B \co \dom B \to X$ is injective. Then $A=B$.
\end{lemma}

\begin{proof}
Let $x \in \dom B$. By the surjectivity of the operator $\lambda\Id_X+A$, there exists $y \in \dom A$ such that
\begin{equation}
\label{inter-89}
(\lambda\Id_X+A)y
=
(\lambda\Id_X+B)x.
\end{equation}
Since $A \subset B$, we have \(y\in\dom B\) and $
(\lambda\Id_X+B)y
=
(\lambda\Id_X+A)y
\ov{\eqref{inter-89}}{=}
(\lambda\Id_X+B)x$. The injectivity of $\lambda\Id_X+B$ yields $x=y$. Hence $x \in \dom A$, and consequently $
\dom B \subset \dom A$. Since $A \subset B$, we conclude that $A=B$.
\end{proof}

\paragraph{Semigroup theory} 
We refer to \cite{EnN00}, \cite{Haa06} and \cite{HvNVW18} for background on semigroup theory. Let $-A$ be the generator of a bounded strongly continuous semigroup $(T_t)_{t \geq 0}$ on $X$, i.e.~$T_t=\e^{-tA}$ for any $t \geq 0$. For any $x \in X$ and any complex number $\lambda \not\in \ovl{\Sigma_{\frac{\pi}{2}}}$, we have by \cite[p.~55]{EnN00} and \cite[(3.2) p.~25]{JMX06} 
the following expression of the resolvent as a Laplace transform
\begin{equation}
\label{Resolvent-Laplace}
R(\lambda,A)x
=(\lambda\, \Id-A)^{-1}x
=-\int_{0}^{\infty} \e^{\lambda s} T_s(x) \d s,
\end{equation}
where the integral is improper. 
Moreover, by \cite[Corollary 5.5 p.~223]{EnN00}, for any $t > 0$, we have 
\begin{equation}
\label{Widder-Resolvent}
T_t(x)
=\lim_{n \to \infty} \big[-\tfrac{n}{t}R(-\tfrac{n}{t},A)\big]^n x,\quad  \quad x \in X.
\end{equation}

\subsection{Sectorial operators}
For any angle $\theta \in (0,\pi)$, we introduce the open sector symmetric around the positive real half-axis with opening angle $2\theta$
\begin{equation}
\label{def-sigma-omega}
\Sigma_{\theta} 
\ov{\mathrm{def}}{=} \big\{ z \in \mathbb{C} \backslash \{ 0 \} : \: | \arg z | < \theta \big\}.
\end{equation}
It is useful to put $\Sigma_0 \ov{\mathrm{def}}{=} (0,\infty)$. 

Background material on sectorial operators and the associated $\H^\infty$ functional calculus can be found in the books \cite{Haa06} and \cite{HvNVW18}.
Let $A \co \dom A \subset X \to X$ be a closed densely defined linear operator acting on a Banach space $X$. We say that $A$ is a $\theta$-sectorial operator for some angle $\theta \in (0,\pi)$ if its spectrum $\sigma(A)$ is a subset of the closed sector $\ovl{\Sigma_\theta}$ and if the set $\big\{zR(z,A) : z \in \mathbb{C} \backslash \ovl{\Sigma_\theta}\big\}$ is bounded in the algebra $\B(X)$ of bounded operators acting on $X$, where $R(z,A) \ov{\mathrm{def}}{=} (z\,\Id-A)^{-1}$ is the resolvent operator. We caution the reader that this definition may differ across the literature. 
The operator $A$ is said to be sectorial if it is a $\theta$-sectorial operator for some $\theta \in (0,\pi)$. In this situation, we can introduce the angle of sectoriality
$\omega_{\sec}(A) 
\ov{\mathrm{def}}{=} \inf\{ \theta \in (0,\pi) : A \textrm{ is $\theta$-sectorial} \}$. 
According to \cite[Example 10.1.2 p.~362]{HvNVW18}, if $-A$ is the generator of a strongly continuous semigroup $(T_t)_{t \geq 0}$ of bounded operators then the operator $A$ is sectorial with $\omega_{\sec}(A) \leq \frac{\pi}{2}$. Furthermore, by \cite[Example 10.1.3 p.~362]{HvNVW18} and \cite[Proposition 3.4.4 p.~79]{Haa06}, the operator $A$ is sectorial with $\omega_{\sec}(A) < \frac{\pi}{2}$ if and only if $-A$ generates a bounded holomorphic (equivalently, bounded analytic) strongly continuous semigroup $(T_t)_{t \geq 0}$, that is, there exist an angle $\theta \in (0,\frac{\pi}{2})$ and a bounded holomorphic extension $\Sigma_\theta \to \B(X)$, $z \mapsto T_z$.

\begin{example}\normalfont
\label{Example-analytic-positive-selfadjoint}
If $A$ is a positive selfadjoint operator with dense domain on a Hilbert space, then \cite[Example 3.7.5 p.~150]{ABHN11} shows that $(\e^{-tA})_{t \geq 0}$ is a bounded holomorphic strongly continuous semigroup.
\end{example}

If $A$ is a sectorial operator on a \textit{reflexive} Banach space $X$, we have by \cite[Proposition 2.1.1 (h) p.~21]{Haa06} or \cite[Proposition 10.1.9 p.~367]{HvNVW18} a topological decomposition
\begin{equation}
\label{decompo-reflexive}
X
=\ker A \oplus \ovl{\Ran A}.
\end{equation}

Following \cite[Definition 10.3.1 p.~399]{HvNVW18}, a sectorial operator $A$ is said to be $R$-sectorial if for some angle $\theta \in (\omega_{\sec}(A),\pi)$ the set
$
\big\{zR(z,A) : z \not\in \ovl{\Sigma_\theta}\big\}
$
is $R$-bounded. The infimum of all $\theta \in (\omega_{\sec}(A),\pi)$ such that $A$ is $R$-sectorial is called the angle of $R$-sectoriality of $A$ and denoted by $\omega_{R}(A)$.

For any angle $\theta \in (0,\pi)$, we consider the algebra $\H^{\infty}(\Sigma_\theta)$\label{algebra-Hinfty} of all bounded analytic functions $f \co \Sigma_\theta \to \mathbb{C}$, equipped with the supremum norm 
$$
\norm{f}_{\H^{\infty}(\Sigma_\theta)}
\ov{\mathrm{def}}{=} \sup\bigl\{\vert f(z)\vert  : z \in \Sigma_\theta\bigr\}.
$$ 
Let $\H^{\infty}_{0}(\Sigma_\theta)$\label{algebra-Hinfty0} be the subalgebra of bounded analytic functions $f \co \Sigma_\theta \to \mathbb{C}$ for which there exist $s,C>0$ such that 
\begin{equation}
\label{ine-Hinfty0}
\vert f(z)\vert
\leq C\min\{|z|^s,|z|^{-s}\}, \quad z \in \Sigma_\theta,
\end{equation} 
as discussed in \cite[Section 2.2]{Haa06}.

Let $A$ be a sectorial operator acting on a Banach space $X$. Consider some angle $\theta \in (\omega_{\sec}(A), \pi)$ and some function $f \in \H^\infty_0(\Sigma_\theta)$. Following \cite[p.~30]{Haa06}, \cite[p.~5]{LM99} (see also \cite[p.~369]{HvNVW18}), for any angle $\nu \in (\omega_{\sec}(A),\theta)$ 
 we introduce the operator
\begin{equation}
\label{2CauchySec}
f(A)
\ov{\mathrm{def}}{=} \frac{1}{2\pi \i}\int_{\partial\Sigma_\nu} f(z) R(z,A) \d z, 
\end{equation}
acting on $X$, with a Cauchy integral, where the boundary $\partial\Sigma_\nu$ is parametrized by
\begin{equation*}
\label{3contour}
\partial\Sigma_\nu(t)
\ov{\mathrm{def}}{=} \begin{cases}
-t \e^{\i\nu} &\text{if } t\in (-\infty,0]\\
t \e^{-\i\nu} &\text{if } t\in [0,\infty)\\
\end{cases}.
\end{equation*}
The sectoriality condition ensures that this integral is absolutely convergent and defines a bounded operator on the Banach space $X$. Using Cauchy's theorem, it is possible to show that this definition does not depend on the choice of the angle $\nu$. The resulting map $\H^\infty_0(\Sigma_\theta) \to \B(X)$, $f \mapsto f(A)$ is an algebra homomorphism.

Following \cite[Definition 2.6 p.~6]{LM99} (see also \cite[p.~114]{Haa06}), we say that the operator $A$ admits a bounded $\H^\infty(\Sigma_\theta)$ functional calculus if the latter homomorphism is bounded, i.e., if there exists a constant $C \geq 0$ such that 
$$
\norm{f(A)}_{X \to X} 
\leq C\norm{f}_{\H^\infty(\Sigma_\theta)}
$$ 
for any function $f \in \H^\infty_0(\Sigma_\theta)$. 
In this context, we can introduce the $\H^\infty$-angle
\begin{equation*}
\label{angle-Hinfty}
\omega_{\H^\infty}(A) 
\ov{\mathrm{def}}{=} \inf\{\theta \in (\omega_{\sec}(A),\pi) : A \text{ admits a bounded $\H^\infty(\Sigma_\theta)$ functional calculus} \}.
\end{equation*}
If the operator $A$ has dense range and admits a bounded $\H^\infty(\Sigma_\theta)$ functional calculus, then the previous homomorphism naturally extends to a bounded homomorphism $f \mapsto f(A)$ from the algebra $\H^\infty(\Sigma_\theta)$ into the algebra $\B(X)$ of bounded linear operators on $X$. 

\begin{example} \normalfont
\label{positive-selfadjoint}
By \cite[Proposition 10.2.23 p.~388]{HvNVW18}, a positive selfadjoint operator $A$ with dense domain on a complex Hilbert space $H$ is sectorial and admits a bounded $\H^\infty(\Sigma_\theta)$ functional calculus for any angle $\theta >0$, i.e.~$\omega_{\H^\infty}(A)=0$.
\end{example}

\paragraph{Fractional powers} 
References on fractional powers include \cite{ABHN11}, \cite{Haa06}, \cite{Haa18}, \cite{MCSA01} and \cite{HvNVW23}. If $A$ is a sectorial operator on a Banach space $X$ and $\alpha \in (0,\frac{\pi}{\omega_{\sec}(A)})$, then $A^\alpha$ is sectorial and $\omega_{\sec}(A^\alpha) = \alpha\omega_{\sec}(A)$ by \cite[Proposition 3.1.2]{Haa06} and \cite[Theorem 15.2.7 p.~440]{HvNVW23}. For any complex numbers $\alpha$ and $\beta$ with $\Re \alpha, \Re \beta > 0$, we have $A^\alpha A^{\beta} = A^{\alpha+\beta}$ according to \cite[Theorem 15.2.5 p.~438]{HvNVW23}. By \cite[Proposition 3.1.1 (d) p.~61]{Haa06} combined with a duality argument relying on \eqref{lien-ker-image} and \cite[Proposition 1.10.15 (c) p.~93]{Meg98}, we have for any complex number $\alpha \in \mathbb{C}$ with $\Re \alpha > 0$ the equalities
\begin{equation}
\label{inclusion-range}
\ker A^\alpha=\ker A
\quad \text{and} \quad
\ovl{\Ran A^\alpha} 
=\ovl{\Ran A}.
\end{equation}
Finally, whenever $A$ is densely defined and $0 < \Re\alpha <1$, the subspace $\dom A$ is a core for the unbounded operator $A^\alpha$ by \cite[Proposition 3.1.1 (h) p.~61]{Haa06}. Note the following useful observation. For any $x \in \dom A$, we have $
A^{\frac12}x
= A^{\frac12}(\Id+A)^{-1}(\Id+A)x$. Since the holomorphic function $z \mapsto \frac{z^{\frac12}}{1+z}$ satisfies \eqref{ine-Hinfty0} by \cite[Example 2.2.5 p.~29]{Haa18}, the operator $A^{\frac12}(\Id+A)^{-1}$ is bounded on $X$. Hence there exists a constant $K \geq 0$ such that
\begin{equation}
\label{majo-graph}
\bnorm{A^{\frac12}x}_X 
\leq K\big(\norm{Ax}_X+\norm{x}_X\big),
\quad x \in \dom A.
\end{equation}
Therefore the graph norm of $A$ dominates the graph norm of $A^{\frac12}$ on $\dom A$.

\subsection{Bisectorial operators}

We refer to \cite{Ege15} and to the monographs \cite{HvNVW18} and \cite{HvNVW23} for background on bisectorial operators and their functional calculus. These operators may be regarded as Banach-space counterparts of unbounded selfadjoint operators on Hilbert spaces. For any angle $\theta \in (0,\frac{\pi}{2})$, we consider the open bisector $\Sigma_\theta^\bi \ov{\mathrm{def}}{=} \Sigma_\theta  \cup (-\Sigma_\theta)$ where the sector $\Sigma_{\theta}$ is defined in \eqref{def-sigma-omega} and $\Sigma_0^\bi \ov{\mathrm{def}}{=} (-\infty,\infty)$. Following \cite[Definition 10.6.1 p.~447]{HvNVW18}, 
we say that a closed densely defined operator $D$ on a Banach space $X$ is $\theta$-bisectorial for some angle $\theta \in (0,\frac{\pi}{2})$ if its spectrum $\sigma(D)$ is a subset of the closed bisector $\ovl{\Sigma^\bi_{\theta}}$ and if the subset $\big\{z R(z,D) : z \in \mathbb{C} \backslash \ovl{\Sigma_{\theta}^\bi} \big\}$ is bounded in the space $\B(X)$ of bounded operators acting on $X$.  Here, as usual, $R(z,D)\ov{\mathrm{def}}{=}(z\,\Id-D)^{-1}$
denotes the resolvent operator.  
The infimum of all $\theta \in (0,\frac{\pi}{2})$ such that $D$ is $\theta$-bisectorial is called the angle of bisectoriality of $D$ and denoted by $\omega_{\bi}(D)$. 

%
%
%
%
%
%
%
By \cite[Example 3.4.15 p.~163]{Ege15}, any selfadjoint operator $D$ is bisectorial with $\omega_{\bi}(D)=0$. Replacing the boundedness condition by an $R$-boundedness condition, we obtain the notion of $R$-bisectorial operator. By \cite[p.~447]{HvNVW18}, a linear operator $D$ is $R$-bisectorial if and only if  
\begin{equation}
\label{Def-R-bisectorial}
\i \R^* \subset \rho(D)
\quad \quad \text{and if the set} \quad
 \{t R(\i t, D) : t \in \R_+^* \} \text{ is $R$-bounded}.
\end{equation}
An important feature of bisectorial operators is that their squares are sectorial. More precisely, if $D$ is a $\theta$-bisectorial operator on a Banach space $X$ then by \cite[Proposition 10.6.2 (2) p.~448]{HvNVW18} its square $D^2$ is $2\theta$-sectorial and we have
\begin{equation}
\label{Bisec-Ran-Ker}
\ovl{\Ran D^2}
=\ovl{\Ran D}
\quad \text{and} \quad
\ker D^2
=\ker D.
\end{equation}
If the Banach space $X$ is reflexive, we have by \cite[Proposition 3.2.2 (iv)]{Ege15} a topological decomposition
\begin{equation}
\label{decomposition-de-X}
X
=\ker D \oplus \ovl{\Ran D}.
\end{equation}

\paragraph{Functional calculus} Consider a bisectorial operator $D$ on a Banach space $X$ of angle $\omega_\bi(D)$. For any angle $\theta \in (\omega_\bi(D), \frac{\pi}{2})$ and any function $f$ in the space 
$$
\H^{\infty}_0(\Sigma_\theta^\bi) 
\ov{\mathrm{def}}{=}  \left\{ f \in \H^\infty(\Sigma_\theta^\bi) :\: \exists C,s > 0 \: \forall \: z \in \Sigma_\theta^\bi : \: |f(z)| \leq C \min\{|z|^s, |z|^{-s} \}  \right\},
$$
we can define a bounded operator $f(D)$ acting on the space $X$ by integrating on the boundary $\partial \Sigma^\bi_{\nu}$ of the bisector $\Sigma^\bi_{\nu}$ for some angle $\nu \in (\omega_\bi(D),\theta)$ using a Cauchy integral
\begin{equation}
\label{def-f(D)}
f(D)
\ov{\mathrm{def}}{=} \frac{1}{2\pi \i}\int_{\partial \Sigma^\bi_{\nu}} f(z)R(z,D) \d z.
\end{equation}
The integration contour is oriented counterclockwise, so that the interiors of the two sectors $\Sigma_\nu$ and $-\Sigma_\nu$ are always to its left. The integral in \eqref{def-f(D)} converges absolutely thanks to the decay of the function $f$ and is independent of the particular choice of the angle $\nu$ by Cauchy's integral theorem. Further details can be found in \cite[Section 3.2.1]{Ege15} and \cite[Theorem 10.7.10 p.~449]{HvNVW18}.

The operator $D$ is said to admit a bounded $\H^\infty(\Sigma_\theta^\bi)$ functional calculus if there exists a constant $C \geq 0$ such that 
$\bnorm{f(D)}_{X \to X} 
\leq C \norm{f}_{\H^\infty(\Sigma_\theta^\bi)}$ for any function $f \in \H^{\infty}_0(\Sigma_\theta^\bi)$. If $D$ has dense range and admits a bounded $\H^\infty(\Sigma_\theta^\bi)$ functional calculus, then the previous algebra homomorphism extends in a canonical way to a bounded homomorphism $\H^\infty(\Sigma_\theta^\bi) \to \B(X)$, $f \mapsto f(D)$.

\begin{example} \normalfont
\label{ex-signe}
Using the function $\sgn\ov{\mathrm{def}}{=} 1_{\Sigma_\theta}-1_{-\Sigma_\theta}$ and the injective part $D|_{\ovl{\Ran D}}$, we can define, following \cite[p.~498]{HvNVW23}, the bounded operator $\sgn D \co \ovl{\Ran D} \to \ovl{\Ran D}$. Using the decomposition \eqref{decomposition-de-X}, we extend it to $X$ by putting $\sgn D|_{\ker D}=0$. This operator is analogous to the Hilbert transform in classical harmonic analysis. It will play an important role later, since it allows the construction of a Banach Fredholm module and therefore of a $\K$-homology class, see Proposition \ref{prop-triple-to-Fredholm}.
\end{example}

\section{Functional calculus of Hodge--Dirac operators}
\label{sec-Hodge-Dirac}
\subsection{Commutation relations}

We prove several equivalent formulations of an exponentially intertwining relation between two semigroups and a suitable unbounded operator.

\begin{prop}
\label{Prop-equivalences}
Let $(T_t)_{t \geq 0}$ and $(\tilde{T}_t)_{t \geq 0}$ be strongly continuous bounded semigroups of operators acting on Banach spaces $X$ and $Y$ with infinitesimal generators $-A$ and $-\tilde{A}$. Let $\partial \co \dom \partial \subset  X \to Y$ be a closed densely defined unbounded operator such that $\dom A \subset \dom \partial$. Let $\lambda \in \R$. The following conditions are equivalent. 
\begin{enumerate}
	\item  If $x \in \dom \partial$ and $t \geq 0$, then $T_{t}(x)$ belongs to the subspace $\dom \partial$ and we have 
\begin{equation}
\label{T_t-et-derivations}
\partial \circ T_{t}(x)
= \e^{-\lambda t}\tilde{T}_{t} \circ \partial(x).
\end{equation}
	\item If $s > \max\{-\lambda,0\}$ and $x \in \dom \partial$ then $R(-s,A)(x)$ belongs to the subspace $\dom \partial$ and
\begin{equation}
\label{commuting-deriv}
R(-s-\lambda,\tilde{A}) \circ \partial(x)
=\partial \circ R(-s,A)(x).
\end{equation}

	\item For any $z \in \dom \tilde{A}^*$ we have $z \in \dom \partial^*$ and the equality
\begin{equation}
\label{bracket-useful}
\la Ax,\partial^*z \ra_{X,X^*}
=\big\la \partial x,\tilde{A}^*z \big\ra_{Y,Y^*}+\lambda\la \partial x,z \ra_{Y,Y^*}, \quad x \in \dom A.
\end{equation}	
\end{enumerate}
\end{prop}

\begin{proof}
1. $\Rightarrow$ 2. Let $s>\max\{-\lambda,0\}$. For any $x \in X$ the continuous function $\R^+ \to X$, $t \mapsto \e^{-s t}T_{t}(x)$ is Bochner integrable since $s>0$ and since
$$
\norm{\e^{-s t}T_{t}(x)}_{X} 
\lesssim \e^{-s t} \norm{x}_{X}.
$$ 
If $t > 0$ and if $x \in \dom \partial$, taking Laplace transforms on both sides of \eqref{T_t-et-derivations} and using  \cite[Theorem 1.2.4 p.~15]{HvNVW18} and the closedness of $\partial$ in the penultimate equality, we obtain that $\int_{0}^{\infty} \e^{-s t}T_{t}(x) \d t$ belongs to $\dom \partial$ and that
\begin{align*}
\MoveEqLeft
R(-s-\lambda,\tilde{A})\partial(x) 
\ov{\eqref{Resolvent-Laplace}}{=} -\int_{0}^{\infty} \e^{(-s-\lambda) t} \tilde{T}_{t} \partial(x) \d t 
\ov{\eqref{T_t-et-derivations}}{=} -\int_{0}^{\infty} \e^{-s t}\partial T_{t}(x) \d t \\
&=-\partial\bigg(\int_{0}^{\infty} \e^{-s t}T_{t}(x) \d t\bigg) 
\ov{\eqref{Resolvent-Laplace}}{=} 
\partial R(-s,A)(x),
\end{align*}
where the function $\R^+ \to Y$, $t \mapsto \e^{(-\lambda-s) t}  \tilde{T}_{t} \partial(x) \ov{\eqref{T_t-et-derivations}}{=} \e^{-s t}\partial T_{t}(x)$ is Bochner integrable. 


2. $\Rightarrow$ 1. For any integer $n \geq 0$ and any $s > 0$, we have by induction $R(-s,A)^n(x) \in \dom \partial$ and
\begin{equation}
\label{div-3456}
R(-s-\lambda,\tilde{A})^n \circ \partial x
=\partial \circ R(-s,A)^n(x).
\end{equation}
Let $x \in \dom \partial$. We have $
\lim_{n \to \infty}\big[-\tfrac{n}{t}R(-\tfrac{n}{t},A)\big]^n x \ov{\eqref{Widder-Resolvent}}{=} T_t(x)$ if $t > 0$. Moreover, since $\tfrac{n}{t} > \max\{-\lambda,0\}$ for $n$ large enough, we have 
\begin{align*}
\MoveEqLeft
\partial\big[-\tfrac{n}{t}R(-\tfrac{n}{t},A)\big]^n  x
\ov{\eqref{div-3456}}{=} \big[-\tfrac{n}{t}R(-\tfrac{n}{t}-\lambda,\tilde{A})\big]^n \partial x \\
&=\big[-\tfrac{n}{t}R(-\tfrac{n}{t},\tilde{A}+\lambda)\big]^n \partial x
\xra[n \to +\infty]{\eqref{Widder-Resolvent}} \e^{-t(\tilde{A}+\lambda)}\partial x
=\e^{-\lambda t}\tilde{T}_{t} \circ \partial x.
\end{align*}
By the closedness of $\partial$, 
we deduce that $T_t(x)$ belongs to $\dom \partial$ and the relation \eqref{T_t-et-derivations}.

2. $\Rightarrow$ 3. We have $\dom A \subset \dom \partial$. Consequently, the operator $S \ov{\mathrm{def}}{=} \partial R(-s,A)$ is bounded from $X$ into $Y$ by \cite[Problem 5.22 p.~167]{Kat76}.
For any $u \in \dom \partial$ and any $\xi \in Y^*$, we deduce that
\begin{align}
\MoveEqLeft
\label{interminable-7789}
\la u, S^* \xi \ra_{X,X^*}
= \la Su, \xi \ra_{Y,Y^*}
= \big\la  \partial R(-s,A) u, \xi \big\ra_{Y,Y^*} \\
&\ov{\eqref{commuting-deriv}}{=} \big\la R(-s-\lambda,\tilde{A}) \circ \partial u, \xi \big\ra_{Y,Y^*}
=\big\la  \partial u, R(-s-\lambda,\tilde{A})^* \xi \big\ra_{Y,Y^*}. \nonumber
\end{align}
Now, using the vector $\xi \ov{\mathrm{def}}{=} (-s-\lambda -\tilde{A}^*)z$, for some $z \in \dom \tilde{A}^* $, we deduce that
$$
\big\la u, S^*((-s-\lambda - \tilde{A}^*)z) \big\ra_{X,X^*} 
\ov{\eqref{interminable-7789}}{=} \la \partial u, z \ra_{Y,Y^*}.
$$
From \eqref{Def-domaine-adjoint}, it follows that $z$ belongs to $\dom \partial^*$. Thus, for $z \in \dom \tilde{A}^*$, it holds that
$$
\big\la Su, (-s - \lambda - \tilde{A}^*)z \big\ra_{Y,Y^*} 
= \la u, \partial^* z \ra_{X,X^*}, \quad u \in \dom \partial.
$$
Since the subspace $\dom \partial$ is dense in the Banach space $X$, the previous identity holds for all $u \in X$. In particular, if we set $u \ov{\mathrm{def}}{=} (-s - A)x$ for some $x \in \dom A$, we have
$$
\big\la S(-s - A)x, (-s - \lambda - \tilde{A}^*)z \big\ra_{Y,Y^*} 
= \big\la (-s - A)x, \partial^*z \big\ra_{X,X^*}.
$$
By observing that $S(-s - A)x = \partial R(-s,A)(-s - A)x = \partial x$, we get 
$$
\big\la \partial x, (-s - \lambda - \tilde{A}^*)z \big\ra_{Y,Y^*} 
= \big\la (-s - A)x, \partial^*z \big\ra_{X,X^*}.
$$
Expanding both sides, we obtain
\begin{align*}
\MoveEqLeft
-s \la \partial x,z\ra_{Y,Y^*} - \lambda \la \partial x,z\ra_{Y,Y^*} - \la \partial x,\tilde{A}^*z\ra_{Y,Y^*} 
= -s \la x,\partial^*z\ra_{X,X^*} - \la Ax,\partial^*z\ra_{X,X^*} \\
&\ov{\eqref{formal-adjoint}}{=} -s \la \partial x,z\ra_{Y,Y^*} - \la Ax,\partial^*z\ra_{X,X^*}.
\end{align*}
Cancelling the terms $-s\la \partial x,z\ra_{Y,Y^*}$ yields 
\begin{equation*}
\big\la \partial x,\tilde{A}^*z \big\ra_{Y,Y^*}+\lambda\la \partial x,z \ra_{Y,Y^*}
=\la Ax,\partial^*z \ra_{X,X^*}, \quad x \in \dom A.
\end{equation*}

3. $\Rightarrow$ 2. For any $x \in \dom A$ and any $z \in \dom \tilde{A}^*$, we have $x \in \dom \partial$ and $z \in \dom \partial^*$. So
\begin{align*}
\MoveEqLeft
\big\la (-s - A)x, \partial^*z \big\ra_{X,X^*}
=-s \la x, \partial^*z \ra_{X,X^*}-\la Ax, \partial^*z \ra_{X,X^*} \\
&\ov{\eqref{crochet-duality}\eqref{bracket-useful}}{=} -s \la \partial x, z \ra_{Y,Y^*}-\la \partial x, \tilde{A}^*z \ra_{Y,Y^*} -\lambda\la \partial x,z \ra_{Y,Y^*}
=\big\la \partial x, (-s-\lambda - \tilde{A}^*)z \big\ra_{Y,Y^*}.
\end{align*}
Let $v \in X$ and let $\xi \in Y^*$. Replacing $x$ by the element $R(-s,A) v$ of $\dom A$ and $z$ by the element $R(-s-\lambda,\tilde{A})^* \xi$ of $\dom \tilde{A}^*$, we obtain
$$
\big\la v, \partial^* R(-s-\lambda,\tilde{A})^* \xi \big\ra 
= \big\la \partial R(-s,A)v , \xi \big\ra.
$$
If $v \in \dom \partial$, it follows from \eqref{crochet-duality} that $\big\la R(-s-\lambda,\tilde{A}) \partial v, \xi \big\ra = \big\la \partial R(-s,A) v, \xi \big\ra$. By duality, we obtain \eqref{commuting-deriv}.
\end{proof}
	
\begin{remark} \normalfont
\label{remark-useful}
Assume that $\dom A \subset \dom \partial$. Suppose that there exists a subspace $C$ of $\dom A$ such that $A(C) \subset \dom \partial$,  $\partial(C) \subset \dom \tilde{A}$, then \eqref{bracket-useful} implies that
\begin{equation}
\label{Curvature-generator}
\partial Ax
=\tilde{A}\partial x+\lambda\partial x, \quad x \in C.
\end{equation}
Indeed, for any $x \in C$ and any $z \in \dom \tilde{A}^*$, we have
\begin{align*}
\MoveEqLeft
\la \partial Ax,z \ra_{Y,Y^*} 
\ov{\eqref{crochet-duality}}{=} \la Ax,\partial^*z \ra_{X,X^*}
\ov{\eqref{bracket-useful}}{=} \big\la \partial x,\tilde{A}^*z \big\ra_{Y,Y^*}+\lambda\la \partial x,z \ra_{Y,Y^*} \\  
&=\big\la \tilde{A}\partial x,z \big\ra_{Y,Y^*}+\lambda\la \partial x,z \ra_{Y,Y^*}
=\big\la \tilde{A}\partial x + \lambda \partial x,z \big\ra_{Y,Y^*}.
\end{align*}
Since the densely defined operator $\tilde{A}$ is closed, the subspace $\dom \tilde{A}^*$ is weak* dense in the dual space $Y^*$ by \cite[Proposition G.1.6 p.~522]{HvNVW18}. Consequently, we can conclude by duality.
\end{remark}

\subsection{Riesz equivalences}
\label{sec-Riesz-equivalence}

We start with the following definition. 

\begin{defi}
Consider some Banach spaces $X$ and $Y$. Consider a sectorial operator $A$ with dense domain acting on $X$. Let $\partial \co \dom \partial \subset X \to Y$ be an unbounded operator such that $\dom \partial$ is a subspace of $\dom A^{\frac{1}{2}}$ and a core of the operator $A^{\frac{1}{2}}$. We say that the operator $A$ admits the $\partial$-Riesz equivalence if
\begin{equation}
\label{estim-Riesz-bis}
\bnorm{A^{\frac{1}{2}}(x)}_{X} 
\approx \norm{\partial(x)}_{Y}, \quad  x \in \dom \partial.
\end{equation}
\end{defi}

Now, we prove some elementary consequences of this equivalence.

\begin{prop}
\label{Prop-derivation-closable-sgrp-bis}
Assume that $A$ admits the $\partial$-Riesz equivalence.
\begin{enumerate}
\item The unbounded operator $\partial \co \dom \partial \subset X \to Y$ is closable.

	\item We have $\dom \ovl{\partial}=\dom A^{\frac{1}{2}}$. Moreover, for any $x \in \dom A^{\frac{1}{2}}$, we have
\begin{equation} 
\label{Equivalence-square-root-domaine-Schur-sgrp-bis}
\bnorm{A^{\frac12}(x)}_{X} 
\approx \bnorm{\ovl{\partial}(x)}_{Y}. 
\end{equation}	
Finally, for any $x \in \dom A^{\frac{1}{2}}$, there exists a sequence $(x_n)$ of elements of $\dom \partial$ such that $x_n \to x$, $A^{\frac{1}{2}}(x_n) \to A^{\frac{1}{2}}(x)$ and $\partial(x_n) \to \ovl{\partial}(x)$.
\end{enumerate} 
\end{prop}

\begin{proof}
1. We will check the criterion described in \cite[(5.6) p.~165]{Kat76}. Consider a sequence $(x_n)$ of elements of the subspace $\dom \partial$ such that $x_n \to 0$ and $\partial(x_n) \to y$ for some $y \in Y$. We have
\begin{align*}
\MoveEqLeft
\norm{x_n-x_m}_{X} + \bnorm{A^{\frac{1}{2}}(x_n)-A^{\frac{1}{2}}(x_m)}_{X} 
\ov{\eqref{estim-Riesz-bis}}{\lesssim} \norm{x_n-x_m}_{X} + \norm{\partial(x_n) - \partial(x_m)}_{Y},
\end{align*} 
which shows that $(x_n)$ is a Cauchy sequence in $\dom A^{\frac{1}{2}}$ equipped with the graph norm. Since $A^{\frac12}$ is closed, the space $\dom A^{\frac12}$ is complete for the graph norm. We infer that there exists $x' \in \dom A^{\frac{1}{2}}$ such that $x_n \to x'$ and $A^{\frac12}x_n \to A^{\frac12}x'$ in $X$. By uniqueness of the limit, we deduce that $x'=0$. Consequently, we have $A^{\frac12}x_n \to 0$. Now, we see that $\bnorm{\partial(x_n)}_{Y} \ov{\eqref{estim-Riesz-bis}}{\lesssim} \bnorm{A^{\frac12}(x_n)}_{X} $. Passing to the limit, we obtain $y=0$.

2. Let $x \in \dom A^{\frac{1}{2}}$. The subspace $\dom \partial$ is a core of the operator $A^{\frac{1}{2}}$. 
Thus we can find a sequence $(x_n)$ of elements of $\dom \partial$ such that $x_n \to x$ and $A^{\frac{1}{2}}(x_n) \to A^{\frac{1}{2}}(x)$. For any integers $n,m \geq 1$, we obtain
\begin{align*}
\MoveEqLeft
\norm{x_n-x_m}_{X} + \norm{\ovl{\partial}(x_n) - \ovl{\partial}(x_m)}_{Y} 
=\norm{x_n-x_m}_{X} + \norm{\partial(x_n) - \partial(x_m)}_{Y}\\
&\ov{\eqref{estim-Riesz-bis}}{\lesssim} \norm{x_n-x_m}_{X} + \bnorm{A^{\frac{1}{2}}(x_n)-A^{\frac{1}{2}}(x_m)}_{X}
\end{align*} 
which shows that $(x_n)$ is a Cauchy sequence in $\dom \ovl{\partial}$ equipped with the graph norm. Since $\ovl{\partial}$ is closed, the space $\dom \ovl{\partial}$ is complete for the graph norm. We infer that there exists $x' \in \dom \ovl{\partial}$ such that $x_n \to x'$ and $\partial(x_n) \to \ovl{\partial}(x')$. By uniqueness of the limit, we have $x=x'$. It follows that $x \in \dom \ovl{\partial}$. This proves the inclusion $\dom A^{\frac{1}{2}} \subset \dom \ovl{\partial}$. Moreover, for any integer $n \geq 1$, we have $
\norm{\partial(x_n)}_{Y} 
\ov{\eqref{estim-Riesz-bis}}{\lesssim} \bnorm{A^{\frac{1}{2}}(x_n)}_{X}$. Since $x_n \to x$ in $\dom\ovl{\partial}$ and in $\dom A^{\frac{1}{2}}$ both equipped with the graph norm, we conclude that $
\bnorm{\ovl{\partial}(x)}_{Y} 
\lesssim \bnorm{A^{\frac{1}{2}}(x)}_{X}$. 
The proofs of the reverse inclusion and of the reverse estimate are similar. Let $x \in \dom \ovl{\partial}$. The subspace $\dom \partial$ is a core of the operator $\ovl{\partial}$. 
So we can find a sequence $(x_n)$ of elements of $\dom \partial$ such that $x_n \to x$ and $\partial(x_n) \to \ovl{\partial}(x)$. For any integers $n,m \geq 1$, we obtain
\begin{align*}
\MoveEqLeft
\norm{x_n-x_m}_{X} + \bnorm{A^{\frac{1}{2}}(x_n) - A^{\frac{1}{2}}(x_m)}_{X} 
\ov{\eqref{estim-Riesz-bis}}{\lesssim} \norm{x_n-x_m}_{X} + \norm{\partial(x_n)-\partial(x_m)}_{Y}
\end{align*} 
which shows that $(x_n)$ is a Cauchy sequence in $\dom A^{\frac{1}{2}}$ equipped with the graph norm. Since $A^{\frac{1}{2}}$ is closed, the space $\dom A^{\frac{1}{2}}$ is complete for the graph norm. We infer that there exists $x' \in \dom A^{\frac{1}{2}}$ such that $x_n \to x'$ and $A^{\frac{1}{2}}(x_n) \to A^{\frac{1}{2}}(x')$. By uniqueness of the limit, we have $x=x'$. It follows that $x \in \dom A^{\frac{1}{2}}$. This proves the inclusion $\dom \ovl{\partial} \subset \dom A^{\frac{1}{2}}$. Moreover, for any integer $n \geq 1$, we have  $
\bnorm{A^{\frac{1}{2}}(x_n)}_{X} 
\ov{\eqref{estim-Riesz-bis}}{\lesssim} \norm{\partial(x_n)}_{Y}$. 
Since $x_n \to x$ in $\dom\ovl{\partial}$ and in $\dom A^{\frac{1}{2}}$ both equipped with the graph norm, we conclude that $
\bnorm{A^{\frac{1}{2}}(x)}_{X} 
\lesssim \bnorm{\ovl{\partial}(x)}_{Y}$. 
\end{proof}

Now, we prove a result which gives some $R$-boundedness of a family of operators. Recall that a $\UMD$ Banach space $X$ has the triangular contraction property $(\Delta)$ by \cite[Theorem 7.5.9 p.~137]{HvNVW18} and is reflexive by \cite[Theorem 4.3.3 p.~306]{HvNVW16}.

\begin{prop}
\label{Prop-R-gradient-bounds-Fourier} 
Assume that the Banach space $X$ has the triangular contraction property $(\Delta)$. Assume that the unbounded operator $A$ admits a bounded $\H^\infty(\Sigma_\theta)$ functional calculus on $X$ for some angle $\theta \in (0,\frac{\pi}{2})$ and the $\partial$-Riesz equivalence \eqref{estim-Riesz-bis}. Then the family 
\begin{equation}
\label{R-gradient-bounds-groups-bis}
\big\{t\partial R(-t^2,A) : t > 0 \big\}
\end{equation} 
of operators of $\B(X,Y)$ is $R$-bounded.
\end{prop}

\begin{proof}
According to \eqref{estim-Riesz-bis}, we can consider the bounded extension $\cal{R}_\partial \co X \to Y$ of $\partial A^{-\frac{1}{2}}$ from $\Ran A^{\frac{1}{2}}$ to $\ovl{\Ran A}$, extended by $0$ on $\ker A$, along the decomposition $X \ov{\eqref{decompo-reflexive}}{=} \ker A \oplus \ovl{\Ran A}$. For any $t > 0$, we have 
\begin{align}
\label{Divers-987-bis}
\MoveEqLeft
t\partial R(-t^2,A)  
=t\partial (-t^2-A)^{-1}         
=-\cal{R}_\partial\big((\tfrac{1}{t^2}A)^{\frac{1}{2}}(\Id+\tfrac{1}{t^2}A)^{-1} \big).
\end{align} 
Since the operator $A$ admits a bounded $\H^\infty(\Sigma_\theta)$ functional calculus for some angle $\theta \in (0,\frac{\pi}{2})$ on a Banach space $X$ with the triangular contraction property $(\Delta)$, it is $R$-sectorial by \cite[Theorem 10.3.4 (2) p.~402]{HvNVW18}. With \cite[Example 10.3.5 p.~402]{HvNVW18} applied with $\alpha=\frac{1}{2}$ and $\beta=1$, we deduce that the set
$$
\big\{ (\tfrac{1}{t^2}A)^{\frac{1}{2}}(\Id+\tfrac{1}{t^2}A)^{-1} : t > 0 \big\}
$$
of operators of $\B(X)$ is $R$-bounded. Since a singleton is $R$-bounded by \cite[Example 8.1.7 p.~170]{HvNVW18}, we obtain by composition \cite[Proposition 8.1.19 (3) p.~178]{HvNVW18} that the set
$$
\big\{\cal{R}_\partial\big((\tfrac{1}{t^2}A)^{\frac{1}{2}}(\Id+\tfrac{1}{t^2}A)^{-1} \big) : t > 0 \big\}
$$
of operators of $\B(X,Y)$ is $R$-bounded. With \eqref{Divers-987-bis} we conclude that the subset \eqref{R-gradient-bounds-groups-bis} is also $R$-bounded.
\end{proof}

\subsection{Pairs of divergence and gradient operators}
\label{sec-pairs}

We start with the following definition.

\begin{defi}
\label{def-compatible-pair-correct}
Consider some Banach spaces $X$ and $Y$. Consider a sectorial operator $A$ with dense domain acting on $X$. An abstract pair $(\partial,\partial^\dagger)$ of gradient and divergence operators compatible with $A$ consists of
a closed operator $\partial \co \dom \partial \subset X \to Y$ with dense domain and a closed operator
$\partial^\dagger \co \dom \partial^\dagger \subset Y \to X$ with dense domain such that
\begin{equation}
\label{facto-with-diamond}
\dom A
=\big\{x \in \dom \partial : \partial x \in \dom \partial^\dagger \big\},
\qquad
Ax
=\partial^\dagger \partial x
\quad \text{for all } x \in \dom A.
\end{equation}
\end{defi}

In concrete $\L^p$-examples, $\partial^\dagger$ acts as a divergence on a $\L^p$-space and $\partial$ as a gradient on a second $\L^p$-space. 
%
We introduced the $\partial$-Riesz equivalence for the operator $A$ in Section \ref{sec-Riesz-equivalence}. In Section \ref{sec-functional-full}, we will also use the $(\partial^\dagger)^*$-Riesz equivalence for the adjoint operator $A^*$. The next result relates these two properties. In this context, recall that $\dom A$ is a core of $A^{\frac{1}{2}}$ by \cite[Proposition 3.1.1 (h) p.~61]{Haa06}.

\begin{prop}
\label{prop-duality-Riesz}
Suppose that $X$ is reflexive. Consider an abstract pair $(\partial,\partial^\dagger)$ of gradient and divergence operators compatible with a sectorial operator $A$. Let $C'$ be a subspace of $\dom (A^*)^\frac{1}{2} \cap \dom(\partial^\dagger)^*$ which is a core
of $(A^*)^\frac{1}{2}$. We denote by $P \co X \to X$ the bounded projection onto the subspace $\ker A$ associated to the decomposition $X \ov{\eqref{decompo-reflexive}}{=} \ker A \oplus \ovl{\Ran A}$. Then the estimate
\begin{equation}
\label{Riesz-555}
\norm{(\partial^\dagger)^*(z)}_{Y^*}
\leq K\bnorm{(A^*)^{\frac{1}{2}}(z)}_{X^*}, \quad z \in C'
\end{equation}
implies the estimate 
$$
\bnorm{A^{\frac{1}{2}}(x)}_{X}
\leq K \norm{\Id-P}_{X \to X} \norm{\partial(x)}_{Y}, \quad x \in \dom A.
$$
\end{prop}

\begin{proof}
We first extend the Riesz estimate \eqref{Riesz-555} from the core $C'$ to the whole domain of $(A^*)^\frac{1}{2}$. We have $\ker P=\ovl{\Ran A}$. By \cite[Theorem 3.2.6 p.~297]{Meg98} combined with \cite[5.10 p.~148]{FHHMPZ01}, 
the adjoint $P^*$ is a bounded projection with
\[
\Ran P^*=(\ker P)^\perp
=(\ovl{\Ran A})^\perp
=(\Ran A)^\perp
\ov{\eqref{lien-ker-image}}{=} \ker A^*.
\]
Let $u \in \dom (A^*)^\frac{1}{2}$. Since $C'$ is a core of $(A^*)^{\frac{1}{2}}$, there exists a sequence $(u_n)$ of elements of $C'$ such that $u_n \to u$ and $(A^*)^{\frac{1}{2}}u_n \to (A^*)^{\frac{1}{2}}u$ in $X^*$. By \eqref{Riesz-555}, the sequence $((\partial^\dagger)^*u_n)$ is Cauchy in $Y^*$, hence converges to some $v \in Y^*$. Since $(\partial^\dagger)^*$ is closed, this implies that $u \in \dom(\partial^\dagger)^*$ and $(\partial^\dagger)^*u=v$, and passing to the limit in $\norm{(\partial^\dagger)^*(u_n)}_{Y^*} \ov{\eqref{Riesz-555}}{\leq} K \bnorm{(A^*)^{\frac{1}{2}}(u_n)}_{X^*}$ yields
\begin{equation}
\label{Riesz-555-extended}
\norm{(\partial^\dagger)^*u}_{Y^*}
\leq K \bnorm{(A^*)^{\frac{1}{2}}u}_{X^*},
\quad u \in \dom (A^*)^{\frac{1}{2}}.
\end{equation}

Fix $x \in \dom A$. Recall that $\dom A \subset \dom A^{\frac{1}{2}}$. By Hahn--Banach theorem, there exists $y \in X^*$ with $\norm{y}_{X^*}=1$ and $
\bnorm{A^{\frac{1}{2}}x}_X
=\la y, A^{\frac{1}{2}}x\ra_{X^*,X}$. Set $y_0 \ov{\mathrm{def}}{=} y-P^*y$. Then $y_0$ belongs to $ \ker P^*=(\Ran P)^\perp=(\ker A)^\perp$ and
\[
\la y_0, A^{\frac{1}{2}}x\ra_{X^*,X}
=\la y, A^{\frac{1}{2}}x\ra_{X^*,X}-\la P^*y, A^{\frac{1}{2}}x\ra_{X^*,X}.
\]
Since $\ovl{\Ran A^{\frac{1}{2}}} \ov{\eqref{inclusion-range}}{=} \ovl{\Ran A}$, the element $A^{\frac{1}{2}}x \in \Ran A^{\frac{1}{2}}$ belongs to $\ovl{\Ran A}=\ker P$. Hence $P(A^{\frac{1}{2}}x)=0$ and therefore $\la P^*y, A^{\frac{1}{2}}x\ra=\la y, P(A^{\frac{1}{2}}x)\ra_{X^*,X}=0$. Thus
\begin{equation}
\label{reduce-y0}
\bnorm{A^{\frac{1}{2}}x}_X
=\la y_0, A^{\frac{1}{2}}x\ra_{X^*,X}.
\end{equation}
Moreover, we have
\begin{equation}
\label{bound-y0}
\norm{y_0}_{X^*}
=\norm{y-P^*y}_{X^*}
\leq \norm{\Id-P^*}_{X^* \to X^*} \norm{y}_{X^*}
=\norm{\Id-P}_{X \to X}.
\end{equation}
Note that $y_0$ belongs to $(\ker A)^\perp \ov{\eqref{lien-ker-image}}{=} \ovl{\Ran A^*} \ov{\eqref{inclusion-range}}{=} \ovl{\Ran (A^*)^{\frac{1}{2}}}$, where we use the reflexivity of $X$ and the results \cite[Proposition 2.6.6 p.~225]{Meg98} and \cite[Theorem 2.5.16 p.~216]{Meg98} in the first equality. Hence there exists a sequence $(v_n)$ of elements in $\dom (A^*)^{\frac{1}{2}}$ such that $
(A^*)^{\frac{1}{2}}v_n \to y_0 \quad \text{in } X^*$. Since the subspace $C'$ is a core of the operator $(A^*)^{\frac{1}{2}}$, for each integer $n$ we can choose $u_n \in C'$ such that 
$$
\norm{u_n-v_n}_{X^*}+\bnorm{(A^*)^{\frac{1}{2}}(u_n-v_n)}_{X^*} \leq \frac1n.
$$ 
Then $(A^*)^{\frac{1}{2}}u_n \to y_0$ in $X^*$. We get
\begin{align*}
\bnorm{A^{\frac{1}{2}}x}_X
&\ov{\eqref{reduce-y0}}{=} \la y_0, A^{\frac{1}{2}}x\ra_{X^*,X}
=\lim_{n \to \infty} \la (A^*)^{\frac{1}{2}}u_n, A^{\frac{1}{2}}x\ra_{X^*,X}.
\end{align*}
Since $(A^\frac{1}{2})^*=(A^*)^{\frac{1}{2}}$ by \cite[Corollary 5.2.4 p.~116]{MCSA01} and since $x$ belongs to $\dom A \subset \dom A^{\frac{1}{2}}$, we have
\begin{align*}
\MoveEqLeft
\la (A^*)^{\frac{1}{2}}u_n, A^{\frac{1}{2}}x\ra_{X^*,X}
=\la u_n, A^{\frac{1}{2}}A^{\frac{1}{2}}x\ra_{X^*,X}
=\la u_n, Ax\ra_{X^*,X} \\
&\la u_n, Ax\ra_{X^*,X}
\ov{\eqref{facto-with-diamond}}{=}\la u_n, \partial^\dagger\partial x\ra_{X^*,X}
\ov{\eqref{crochet-duality}}{=}\la (\partial^\dagger)^*u_n, \partial x\ra_{Y^*,Y},
\end{align*}
where we used \cite[Theorem 15.2.5 (4) p.~438]{HvNVW23} in the second equality. Consequently, we conclude that
\begin{align*}
\bnorm{A^{\frac{1}{2}}x}_X
&\leq \limsup_{n \to \infty}\norm{(\partial^\dagger)^*u_n}_{Y^*}\, \norm{\partial x}_Y
\ov{\eqref{Riesz-555-extended}}{\leq}
K \limsup_{n\to\infty}\bnorm{(A^*)^{\frac{1}{2}}u_n}_{X^*}\, \norm{\partial x}_Y \\
&=\norm{y_0}_{X^*} \norm{\partial x}_Y 
\ov{\eqref{bound-y0}}{\leq} K \norm{\Id-P}_{X \to X} \norm{\partial x}_Y.
\end{align*}
\end{proof}

\subsection{Curvature via commutation relations and bounded $\H^\infty$ functional calculus}
\label{sec-Curvature}

This subsection introduces an abstract curvature condition expressed as an intertwining relation between the semigroup $(T_t)_{t \geq 0}$ on $X$ and an auxiliary suitable semigroup $(\tilde T_t)_{t \geq 0}$ on the ``tangent space" $Y$. It should be viewed as a semigroup-level analogue of the commutation rules satisfied by gradients under lower Ricci curvature bounds. The identity \eqref{Ricci-def} will be the key input for deriving resolvent commutation formulas and $R$-bounded estimates later on.

\begin{defi}
\label{curvature-H-infty}
Let $(T_t)_{t \geq 0}$ be a strongly continuous semigroup of operators on a Banach space $X$ with infinitesimal generator $-A$. Consider a closed operator $\partial \co \dom \partial \subset X \to Y$ with dense domain. If $\lambda \in \R$, we say that $(T_t)_{t \geq 0}$ satisfies $\Curv_{\partial}(\lambda)$ if there exists a strongly continuous bounded semigroup $(\tilde{T}_t)_{t \geq 0}$ of operators with generator $-\tilde{A}$ acting on the Banach space $Y$ such that 
\begin{equation}
\label{inclusion-def}
T_t(\dom\partial) \subset \dom\partial, \quad t \geq 0
\end{equation}
and
\begin{equation}
\label{Ricci-def} 
\partial T_t x 
= \e^{-\lambda t}\tilde T_t\partial x,\qquad t\geq 0, x \in \dom\partial,
\end{equation}
and if there exists a core $C$ of the closed operator $A$ such that 
\begin{equation}
\label{inclusions-minimales}
A(C) \subset \dom \partial,
\quad
\partial(C) \subset \dom \tilde{A}
\quad \text{and} \quad T_t(C) \subset C, \quad t \geq 0.
\end{equation}
We say that $(T_t)_{t \geq 0}$ satisfies the $\H^\infty$-enhanced curvature condition $\Curv_{\partial,\H^\infty}(\lambda)$ if in addition the operators $\tilde{A}$ and $A$ admit a bounded $\H^\infty(\Sigma_\theta)$ functional calculus for some angle $\theta \in (0, \frac{\pi}{2})$.
\end{defi}

Finally, we record a simple monotonicity property: if the curvature condition holds for some $\lambda$, then it automatically holds for any smaller $\lambda'$. This matches the geometric intuition that a lower curvature bound implies any weaker lower bound.

\begin{prop}
\label{prop-curv-lambda-bis}
Let $(T_t)_{t \geq 0}$ be a strongly continuous semigroup of operators on a Banach space $X$. If $\lambda$ and $\lambda'$ are real numbers such that $\lambda \geq \lambda'$ and if the semigroup $(T_t)_{t \geq 0}$ satisfies $\Curv_{\partial}(\lambda)$ (resp.~$\Curv_{\partial,\H^\infty}(\lambda)$) then the semigroup $(T_t)_{t \geq 0}$ also satisfies the property $\Curv_{\partial}(\lambda')$ (resp.~$\Curv_{\partial,\H^\infty}(\lambda')$).
\end{prop}

\begin{proof}
We can suppose that $\lambda > \lambda'$. It suffices to write $
\partial \circ T_t
=\e^{-\lambda' t}\big(\e^{-(\lambda-\lambda') t}\tilde{T}_t\big) \circ \partial$ for any $t \geq 0$
and to observe that $(\e^{-(\lambda-\lambda') t}\tilde{T}_t)_{t \geq 0}$ is a strongly continuous semigroup by \cite[p.~43]{EnN00} with infinitesimal generator $-(\tilde{A}+(\lambda-\lambda')\,\Id) $ according to \cite[2.2 p.~60]{EnN00}. By \cite[Lemma 3.5 p.~28]{JMX06}
(applied with $\epsi=\lambda-\lambda'>0$), the shifted operator $\tilde{A}+(\lambda-\lambda')\,\Id$ admits a bounded $\H^\infty$ functional calculus  if the operator $\tilde{A}$ admits a bounded $\H^\infty$ functional calculus.
\end{proof}

\subsection{From curvature to commutation: resolvents, cores, and the reduced semigroup}
\label{sec-consequence-curvature}

The purpose of this section is to turn the curvature assumption $\Curv_{\partial,\H^\infty}(0)$ and the two Riesz equivalences into concrete operator-theoretic consequences on the reduced space $\ovl{\Ran\partial}$. In particular, we derive commutation relations for the resolvents and the semigroups, and we identify natural dense cores that will later serve to define the Hodge--Dirac operator.

Here, we assume that $X$ and $Y$ are reflexive. In the sequel, we suppose that the semigroup $(T_t)_{t \geq 0}$ satisfies the condition $\Curv_{\partial,\H^\infty}(0)$. We assume that  the operator $A$ satisfies the $\partial$-Riesz equivalence \eqref{estim-Riesz-bis} and that the operator $A^*$ satisfies the $(\partial^\dagger)^*$-Riesz equivalence \eqref{estim-Riesz-bis}. Recall that \eqref{facto-with-diamond} holds and that $\dom \partial=\dom A^{\frac{1}{2}}$. Since the subspace $\dom \partial^\dagger$ is dense in $Y$, we can consider the Banach adjoint 
$
(\partial^\dagger)^* \co \dom (\partial^\dagger)^* \subset X^* \to Y^*$. We assume that the operator $A^*$ admits the $(\partial^\dagger)^*$-Riesz equivalence \eqref{estim-Riesz-bis}. We have
\begin{align}
\label{Ricci-0}  
\partial \circ T_{t}x
\ov{\eqref{Ricci-def}}{=} \tilde{T}_t \circ \partial x, \quad t \geq 0, x \in \dom \partial.
\end{align}
Identity \eqref{Ricci-0} expresses that the semigroup $(T_t)_{t\geq 0}$ transports the first-order structure: applying $\partial$ after $T_t$ agrees with applying the tangent semigroup $\tilde T_t$ after $\partial$. This intertwining principle is the starting point for all subsequent commutation arguments.

According to Proposition \ref{Prop-equivalences}, for any $s > 0$ and any $x \in \dom \partial$, the element $R(-s,A)(x)$ belongs to the subspace $\dom \partial$ and we obtain the following commutation rule between the resolvents and the abstract gradient
\begin{equation}
\label{commuting-deriv-Fourier-bis}
\partial \circ R(-s,A)(x)
=R(-s,\tilde{A}) \circ \partial(x), \quad x \in \dom \partial.
\end{equation}
The resolvent commutation rule \eqref{commuting-deriv-Fourier-bis} can be viewed as a Laplace-transform version of \eqref{Ricci-0}. 

To obtain core statements later on, we fix a dense subspace $C \subset \dom A$ on which the compatibility relations are available and stable under the semigroup. This will allow us to approximate arbitrary elements in the relevant domains by ``smooth" ones.

Recall that there exists a core $C$ of the operator $A$ such that $A(C) \subset \dom \partial$, $\partial(C) \subset \dom \tilde{A}$ and $T_t(C) \subset C$ for any $t \geq 0$. In particular, we have $C \subset \dom\partial$ by \eqref{facto-with-diamond}. By Remark \ref{remark-useful}, we have
\begin{equation}
\label{commute-AAA}
\partial Ax
\ov{\eqref{Curvature-generator}}{=}\tilde{A}\partial x, \quad x \in C.
\end{equation}
Formula \eqref{commute-AAA} is the infinitesimal counterpart of \eqref{Ricci-0} on the core $C$. It provides the basic identification $\tilde A\partial=\partial A$ at the level of generators, which is crucial when comparing second-order operators built from $(\partial,\partial^\dagger)$ with $\tilde A$.

The next lemma shows that $\partial(C)$ is large enough inside $\ovl{\Ran\partial}$. This density property is the key approximation device allowing us to extend identities first proved on $\partial(C)$ to arbitrary vectors in $\ovl{\Ran\partial}$ by closure arguments.

\begin{lemma}
\label{lem-density-range-partial-domA}
The subspace $C$ is a common core of the operators $A$, $A^{\frac12}$ and $\partial$. Moreover, the subspace $\partial(C)$ is dense in the Banach space $\ovl{\Ran \partial}$, i.e.~for any $y \in \ovl{\Ran \partial}$ there exists a sequence $(x_n)$ of elements of $C$ such that $\partial x_n \to y$ in $Y$.
\end{lemma}

\begin{proof}
Since $\dom A$ is a core of $A^{\frac12}$, if $x\in\dom A^{\frac12}$, choose a sequence $(u_n)$ in $\dom A$ such that $
u_n\to x$ and $A^{\frac12}u_n\to A^{\frac12}x$ in $X$. For each $n$, since $C$ is a core of $A$, choose $c_n \in C$ such that $
\norm{c_n-u_n}_X+\norm{Ac_n-Au_n}_X\leq \frac1n$. By \eqref{majo-graph}, this implies $
\norm{c_n-u_n}_X+\bnorm{A^{\frac12}c_n-A^{\frac12}u_n}_X \to 0$. Consequently, $c_n\to x$ and $A^{\frac12}c_n\to A^{\frac12}x$. Hence $C$ is a core of $A^{\frac12}$. 
Applying \eqref{Equivalence-square-root-domaine-Schur-sgrp-bis} to $c_n-x$, we get
\[
\norm{\partial c_n-\partial x}_Y
=
\norm{\partial(c_n-x)}_Y
\ov{\eqref{Equivalence-square-root-domaine-Schur-sgrp-bis}}{\lesssim}
\bnorm{A^{\frac12}(c_n-x)}_X
\to 0.
\]
Thus, for any $x \in \dom A^{\frac12}=\dom \partial$, there exists a sequence $(c_n)$ in $C$ such that $
c_n\to x$, $A^{\frac12}c_n \to A^{\frac12}x$, $\partial c_n\to\ovl{\partial}x$. In particular, $C$ is a core of both $A^{\frac12}$ and $\ovl{\partial}$.

Finally, if $x \in \dom A$, we may choose the approximating sequence $(c_n)$ directly from the fact that $C$ is a core of $A$, so that $c_n \to x$ and $Ac_n\to Ax$. Then \eqref{majo-graph} gives $A^{\frac12}c_n\to A^{\frac12}x$, and the Riesz equivalence gives $\partial c_n \to \partial x$. Thus the same sequence approximates $x$ simultaneously for the graph norms of $A$, $A^{\frac12}$ and $\partial$, whenever $x\in\dom A$.

Let $y \in \ovl{\Ran\partial}$. For each integer $n \geq 1$, choose $x_n \in \dom \partial$ such that $\norm{y-\partial x_n}_Y \leq \frac1n$. Since $C$ is a core for $\partial$, there exists $c_n \in C$ such that $\norm{\partial c_n-\partial x_n}_Y
\leq \frac1n$. Then 
$$
\norm{y-\partial c_n}_Y
\leq
\norm{y-\partial x_n}_Y
+
\norm{\partial x_n-\partial c_n}_Y
\leq 
\frac2n.
$$ 
Thus $\partial c_n \to y$ in $Y$.
\end{proof}

Now, we exploit the Laplace representation of resolvents to obtain commutation rules for $\partial^\dagger$ with $R(-s,\tilde A)$ and with $\tilde T_t$ on the subspace $\ovl{\Ran\partial}$.

\begin{prop}
\label{prop-commuting-deriv2-correct}
\begin{enumerate}
\item For any $s > 0$, the operator $R(-s,A)\partial^\dagger$ induces a bounded operator on the Banach space $\ovl{\Ran \partial}$.

\item For any $s > 0$ and any $x \in \dom A$ we have $R(-s,\tilde A)\partial x \in \dom \partial^\dagger$ and
\begin{equation}
\label{eq-commute-partialdiamond-on-core}
\partial^\dagger R(-s,\tilde A)\partial x = R(-s,A)\partial^\dagger\partial x.
\end{equation}
\item For any $s > 0$ and any $y \in \ovl{\Ran \partial}\cap \dom \partial^\dagger$ we have
$R(-s,\tilde A)y \in \dom \partial^\dagger$ and
\begin{equation}
\label{eq-commute-partialdiamond-general}
\partial^\dagger R(-s,\tilde A) y = R(-s,A)\partial^\dagger y.
\end{equation}
\item For any $y \in \ovl{\Ran \partial} \cap \dom \partial^\dagger$ and any $t > 0$, we have $\tilde{T}_t y \in \dom \partial^\dagger$ and
\begin{equation}
\label{eq-commute-partial-semigroup-tangent}
\partial^\dagger \tilde{T}_t y 
= T_t\partial^\dagger y.
\end{equation}
\end{enumerate} 
\end{prop}

\begin{proof}
1. Recall that the space $X$ is reflexive. Note that $R(-s,A \big)\partial^\dagger = R(-s,A \big)^{**}\partial^{\dagger**} \subset \big((\partial^\dagger)^* R(-s,A^*)\big)^*$ where we use \cite[Problem 5.26 p.~168]{Kat76}. 
Furthermore, by Proposition \ref{Prop-R-gradient-bounds-Fourier} applied with $A^*$ and $(\partial^\dagger)^*$ instead of $A$ and $\partial$, the operator $(\partial^\dagger)^* R(-s,A^*)$ is bounded, hence $\big((\partial^\dagger)^* R(-s,A^*)\big)^*$ also. By Lemma \ref{lem-density-range-partial-domA}, the subspace $\partial(C)$ of $\dom \partial^\dagger$ is dense in the Banach space $\ovl{\Ran \partial}$. Now, the conclusion is immediate.

2. Fix $s>0$ and let $x \in \dom A$. Then $x \in \dom \partial$ and $\partial x \in \dom \partial^\dagger$ by
\eqref{facto-with-diamond}. Moreover, for all $t > 0$ we have using \cite[(1.5) p.~50]{EnN00} in the second equality
\begin{align}
\MoveEqLeft
\label{Inter-3455}
T_t \partial^\dagger \partial(x) 
\ov{\eqref{facto-with-diamond}}{=} T_t A(x)
= A T_t(x) 
\ov{\eqref{facto-with-diamond}}{=} \partial^\dagger \partial T_t(x) 
\ov{\eqref{Ricci-0}  }{=} \partial^\dagger \tilde{T}_t \partial(x).
\end{align}
Using \eqref{Resolvent-Laplace} with $\lambda=-s$ and this commutation identity, we obtain
\begin{align}
\label{eq-resolvent-commutation-proof}
R(-s,A)\partial^\dagger\partial x
&\ov{\eqref{Resolvent-Laplace}}{=} -\int_0^\infty \e^{-st} T_t\partial^\dagger\partial x \d t
\ov{\eqref{Inter-3455}}{=} -\int_0^\infty \e^{-st} \partial^\dagger \tilde T_t \partial x \d t.
\end{align}
Since $\partial^\dagger$ is closed and the Bochner integral is well-defined, we may apply \cite[Theorem 1.2.4 p.~15]{HvNVW16} to interchange $\partial^\dagger$ and the integral, which yields
\[
R(-s,A)\partial^\dagger\partial x
= \partial^\dagger\Big(-\int_0^\infty \e^{-st}\,\tilde T_t \partial x \d t\Big)
\ov{\eqref{Resolvent-Laplace}}{=} \partial^\dagger R(-s,\tilde A)\partial x.
\]
This proves \eqref{eq-commute-partialdiamond-on-core}.

3. Now, consider some element $y \in \ovl{\Ran \partial}\cap \dom \partial^\dagger$. By Lemma \ref{lem-density-range-partial-domA}, there exists a sequence $(x_n)$ of elements of $C$ such that $\partial x_n \to y$ in $Y$. By boundedness of the operator $R(-s,\tilde A)$, we deduce that $R(-s,\tilde A)\partial x_n \to R(-s,\tilde A)y$ in $Y$. By the first point, the operator $R(-s,A)\partial^\dagger$ extends to a bounded operator on $\ovl{\Ran\partial}$. We obtain
\[
\partial^\dagger R(-s,\tilde A)\partial x_n
\ov{\eqref{eq-commute-partialdiamond-on-core}}{=} R(-s,A)\partial^\dagger\partial x_n
\to R(-s,A)\partial^\dagger y
\]
in $X$. Since the operator $\partial^\dagger$ is closed, it follows 
that $R(-s,\tilde A)y \in \dom \partial^\dagger$ and that
$$
\partial^\dagger R(-s,\tilde A) y 
= R(-s,A)\partial^\dagger y.
$$

4. Fix $t > 0$ and let $y \in \ovl{\Ran \partial}\cap \dom \partial^\dagger$. For $n \geq 1$, set $s_n \ov{\mathrm{def}}{=} \frac{n}{t}$ and define the bounded operators $
J_n 
\ov{\mathrm{def}}{=} 
-s_n R(-s_n,\tilde A)$ and $
K_n \ov{\mathrm{def}}{=} 
-s_n R(-s_n,A)$. By the third point, for every $s > 0$ and every $y \in \ovl{\Ran \partial}\cap \dom \partial^\dagger$ we have $R(-s,\tilde A)y \in \dom \partial^\dagger$ and
\[
\partial^\dagger R(-s,\tilde A)y 
= R(-s,A)\partial^\dagger y.
\]
Taking $s=s_n$ and multiplying by $-s_n$, we obtain $
\partial^\dagger J_n y = K_n \partial^\dagger y$. In particular, $J_n y \in \dom \partial^\dagger$. By iteration, we get for any integer $m \geq 1$,
\[
\partial^\dagger J_n^m y = K_n^m \partial^\dagger y,
\]
and hence $J_n^m y \in \dom \partial^\dagger$ for all $m$. Note that by the formula \eqref{Widder-Resolvent}, we have
\[
J_n^n y \to \tilde T_t y \quad \text{in } Y,
\qquad
K_n^n x \to T_t x \quad \text{in } X \ \text{for all } x \in X.
\]
Applying this with $x=\partial^\dagger y$ and using the identity above with $m=n$, we infer
\[
\partial^\dagger J_n^n y 
= K_n^n \partial^\dagger y \to T_t \partial^\dagger y \quad \text{in } X.
\]
Since $\partial^\dagger$ is closed and $J_n^n y \to \tilde T_t y$ in $Y$, it follows that $\tilde T_t y \in \dom \partial^\dagger$ and $\partial^\dagger \tilde T_t y = T_t \partial^\dagger y$.
\end{proof}

Our Hodge--Dirac operator in \eqref{Def-Dirac-operator-Fourier} below will be constructed out of $\partial$ and the unbounded operator $\partial^\dagger|_{\ovl{\Ran \partial}}$. Note that the latter is by definition an unbounded operator on the Banach space $\ovl{\Ran \partial}$ with values in $X$ with domain $\dom \partial^\dagger \cap \ovl{\Ran \partial}$.  The next lemma records the basic functional-analytic properties needed for the construction of the Hodge--Dirac operator.

\begin{lemma}
\label{Lemma-rest-closed-dense-Fourier-bis}
The operator $\partial^\dagger|_{\ovl{\Ran \partial}}$ is densely defined and is closed. 
\end{lemma}

\begin{proof}
By Lemma \ref{lem-density-range-partial-domA}, the subspace $\partial(C)$ of the space $\dom \partial^\dagger$ is dense in the Banach space $\ovl{\Ran \partial}$. 
Since the unbounded operator $\partial^\dagger$ is closed, the second assertion on the closedness is obvious.
\end{proof}

According to \eqref{Ricci-0}, for any $t \geq 0$ the bounded operator $\tilde T_t$ leaves the subspace $\Ran\partial$ invariant. Hence, by continuity, $\tilde{T}_t$ also leaves the closure $\ovl{\Ran\partial}$ invariant. Therefore the family $(\tilde{T}_t|_{\ovl{\Ran\partial}})_{t \geq 0}$ is a strongly continuous semigroup on the Banach space $\ovl{\Ran\partial}$, and $\tilde{A}|_{\ovl{\Ran\partial}}$, with domain $\dom \tilde{A} \cap \ovl{\Ran\partial}$, is the opposite of its infinitesimal generator, see \cite[pp.~60--61]{EnN00}. Finally, we identify a convenient core for the operator $\tilde A|_{\ovl{\Ran\partial}}$. This will later allow us to compare $\tilde A|_{\ovl{\Ran\partial}}$ with the second-order operator $\partial\partial^\dagger$ on the same space.

\begin{prop}
\label{Prop-core-1-group-bis}
The subspace $\partial(C)$ is a core of the unbounded operator $\tilde{A}|_{\ovl{\Ran \partial}}$.
\end{prop}

\begin{proof}
By Lemma \ref{lem-density-range-partial-domA}, the subspace $\partial(C)$ is dense in the Banach space $\ovl{\Ran \partial}$. Moreover, it is invariant under each operator $\tilde T_t|_{\ovl{\Ran\partial}}$. Indeed, if $x \in C$ then $T_t x \in C$ for all $t \geq 0$, hence $\partial(T_t x) \in \partial(C)$. Using the intertwining relation \eqref{Ricci-0}, we get
\[
\tilde T_t(\partial x)
\ov{\eqref{Ricci-0}}{=} \partial(T_t x) \in \partial(C),
\qquad t \geq 0.
\]
Thus the subspace $\partial(C)$ is a dense invariant subspace of $\ovl{\Ran\partial}$ for the semigroup $(\tilde T_t|_{\ovl{\Ran\partial}})_{t \geq 0}$. By \cite[Proposition G.2.4 p.~526]{HvNVW18}, it follows that $\partial(C)$ is a core of the operator $\tilde A|_{\ovl{\Ran\partial}}$.
%
\end{proof}

\subsection{Regularization on $\ovl{\Ran\partial}$ and identification of the reduced generator}
\label{sec-regul}

The goal of this section is to justify, on the reduced space $\ovl{\Ran\partial}$, the formal identity
$\tilde{A} = \partial\partial^\dagger$ by proving that $\partial\partial^\dagger|_{\ovl{\Ran\partial}}$ is closable and that its closure coincides with the restriction $\tilde{A}|_{\ovl{\Ran\partial}}$ of the operator $\tilde{A}$, under some regularization assumptions. Recall that the subspace $C$ of $X$ is a core of the unbounded operator $A$ satisfying
\[
A(C)\subset \dom\partial,
\quad
\partial(C)\subset \dom\tilde A, \quad T_t(C)\subset C, \quad t \geq 0.
\]
The technical issue is that $\partial(C)$, although natural, is not a priori known to be a core for the operator $\partial\partial^\dagger$.

We first formulate a minimal smoothing assumption ensuring that, for positive times, the tangent semigroup sends the reduced range into the ``smooth" subspace $\partial(C)$.

\begin{defi}
\label{def-regularization-coreC}
We say that the semigroup $(\tilde T_t)_{t \geq 0}$ has the regularization property with respect to $C$ if
\begin{equation}
\label{axiom-regularization-C}
\tilde T_t(\ovl{\Ran\partial}) \subset \partial(C),
\quad t > 0.
\end{equation}
\end{defi}
Condition \eqref{axiom-regularization-C} should be understood as a regularization on $\ovl{\Ran\partial}$: applying $\tilde T_t$ for $t>0$ produces vectors of the form $\partial x$ with $x\in C$. In particular, it provides a convenient core candidate for operators acting on $\ovl{\Ran\partial}$.

Under this regularization property, the second order operator $\partial\partial^\dagger$ becomes well behaved on $\ovl{\Ran\partial}$ and one can identify its closure with the reduced generator $\tilde A|_{\ovl{\Ran\partial}}$. This is the analogue, in the present abstract setting, of the classical identity $\d \Delta_\HdR f=\Delta_\HdR(\d f) $ of \cite[(2.7.16) p.~166]{RuS17}, which can be written $\d \d^*\d f = \Delta_\HdR(\d f)$.

\begin{prop}
\label{Prop-fundamental-Fourier}
Assume the regularization property with respect to $C$ on $\ovl{\Ran\partial}$. Then $\partial\partial^\dagger|_{\ovl{\Ran\partial}}$ is closable and
\begin{equation}
\label{relation-dur-fund-Fourier}
\ovl{\partial\partial^\dagger|_{\ovl{\Ran\partial}}}
=\tilde A|_{\ovl{\Ran\partial}}.
\end{equation}
Moreover, the subspace $\partial(C)$ is a core of the operators $\partial\partial^\dagger|_{\ovl{\Ran\partial}}$ and $\partial^\dagger|_{\ovl\Ran\partial}$.
\end{prop}

\begin{proof}
Let $y \in \dom \partial\partial^\dagger|_{\ovl{\Ran\partial}}$. Let $(t_j)$ be a sequence with $t_j > 0$ and $t_j \to 0$ and set $y_j \ov{\mathrm{def}}{=}\tilde T_{t_j}y$. By the regularization property, we have $y_j\in \partial(C)$ for all $j$. Moreover $y_j \to y$ in $Y$ by strong continuity. Since $y\in \dom\partial^\dagger\cap \ovl{\Ran\partial}$, Proposition \ref{prop-commuting-deriv2-correct} gives $\partial^\dagger y_j=T_{t_j}\partial^\dagger y$. As $y \in \dom \partial\partial^\dagger|_{\ovl{\Ran\partial}}$, we have $\partial^\dagger y\in \dom\partial$, hence $T_{t_j}\partial^\dagger y\in \dom\partial$. Hence $y_j\in \ovl{\Ran\partial}\cap\dom\partial^\dagger$ and $\partial^\dagger y_j\in\dom\partial$, i.e.~$y_j\in\dom(\partial\partial^\dagger|_{\ovl{\Ran\partial}})$ and
\[
\partial\partial^\dagger|_{\ovl{\Ran\partial}} y_j
=\partial\partial^\dagger \tilde{T}_{t_j}y
\ov{\eqref{eq-commute-partial-semigroup-tangent}}{=}\partial(T_{t_j}\partial^\dagger y)
\ov{\eqref{Ricci-0}}{=} \tilde T_{t_j}\partial\partial^\dagger y
=\tilde T_{t_j} \partial\partial^\dagger|_{\ovl{\Ran\partial}}y
\to \partial\partial^\dagger|_{\ovl{\Ran\partial}}y.
\]
We have proved that $y_j \to y$ in the graph norm of $\partial\partial^\dagger|_{\ovl{\Ran\partial}}$. So $\partial(C)$ is a core of $\partial\partial^\dagger|_{\ovl{\Ran\partial}}$.

Let $x \in C$. Then $\partial x \in \ovl{\Ran\partial}$, $\partial x \in \dom\partial^\dagger$ and
\[
\partial\partial^\dagger|_{\ovl{\Ran\partial}}(\partial x)
=\partial\partial^\dagger\partial x
\ov{\eqref{facto-with-diamond}}{=} \partial Ax
\ov{\eqref{commute-AAA}}{=}\tilde A \partial x.
\]
Since the operators $\partial\partial^\dagger|_{\ovl{\Ran\partial}}$ and $\tilde A|_{\ovl{\Ran\partial}}$ coincide on the core $\partial(C)$ of $\partial\partial^\dagger|_{\ovl{\Ran\partial}}$ and since $\tilde A|_{\ovl{\Ran\partial}}$ is closed, $\partial\partial^\dagger|_{\ovl{\Ran\partial}}$ is closable and its closure satisfies $\ovl {\partial\partial^\dagger|_{\ovl{\Ran\partial}}} \subset \tilde A|_{\ovl{\Ran\partial}}$. On the other hand, by Proposition \ref{Prop-core-1-group-bis}, the subspace $\partial(C)$ is a core of $\tilde A|_{\ovl{\Ran\partial}}$, and $\ovl{\partial\partial^\dagger|_{\ovl{\Ran\partial}}}$ is closed and coincides with $\tilde A|_{\ovl{\Ran\partial}}$ on $\partial(C)$, hence $\tilde A|_{\ovl{\Ran\partial}}\subset \ovl{\partial\partial^\dagger|_{\ovl{\Ran\partial}}}$. Therefore $\ovl{\partial\partial^\dagger|_{\ovl{\Ran\partial}}}=\tilde A|_{\ovl{\Ran\partial}}$.

Let $y \in \dom \partial^\dagger|_{\ovl{\Ran\partial}}
=
\dom\partial^\dagger\cap\ovl{\Ran\partial}$. Choose a sequence $(t_j)$ of positive numbers such that $t_j\to0$, and set $y_j\ov{\mathrm{def}}{=}\tilde T_{t_j}y$. By the regularization property, we have $y_j \in \partial(C)$ for any $j \geq 1$. Moreover, by strong continuity of $(\tilde T_t)_{t\geq0}$ on $\ovl{\Ran\partial}$, we have $y_j \to y$ in $Y$. By Proposition~\ref{prop-commuting-deriv2-correct}, we also have $y_j=\tilde T_{t_j}y \in \dom\partial^\dagger$ and $
\partial^\dagger y_j
=
\partial^\dagger\tilde T_{t_j}y
\ov{\eqref{eq-commute-partial-semigroup-tangent}}{=}
T_{t_j}\partial^\dagger y
\to \partial^\dagger y$,
since $(T_t)_{t \geq 0}$ is strongly continuous on $X$. Thus $y_j \to y$ in the graph norm of $\partial^\dagger|_{\ovl{\Ran\partial}}$. Therefore $\partial(C)$ is a core of $\partial^\dagger|_{\ovl{\Ran\partial}}$.
\end{proof}

The regularization property \eqref{axiom-regularization-C} is convenient but may fail in some situations. In such cases, one can often construct bounded approximations of the identity that commute with $\partial$ and $\partial^\dagger$ and that map $\ovl{\Ran\partial}$ into $\partial(C)$. This motivates the next definition.

\begin{defi}
\label{Def-regul}
Consider some Banach spaces $X$ and $Y$. Consider a sectorial operator $A$ with dense domain acting on $X$.
Consider an abstract pair $(\partial,\partial^\dagger)$ compatible with $A$ in the sense of \eqref{facto-with-diamond}. We say that two nets $(R_j)$ and $(\tilde{R}_j)$ of bounded linear maps $R_j \co X \to X$ and $\tilde R_j \co Y \to Y$ define a couple of regularizing nets if
\begin{enumerate}
\item $R_j(X)\subset \dom\partial$ for all $j$,
\item $\tilde R_j(Y)\subset \dom\partial^\dagger$ for all $j$,
\item $\tilde R_j(\ovl{\Ran\partial})\subset \partial(C)$ for all $j$,
\item for all $x\in \dom\partial$ we have
\begin{equation}
\label{commute-regul}
\partial R_j(x)
=\tilde R_j\partial(x),
\end{equation}
\item for all $y \in \dom\partial^\dagger$ we have
\begin{equation}
\label{commute-regul-bis}
R_j\partial^\dagger(y)
=\partial^\dagger\tilde R_j(y),
\end{equation}
\item $R_j(x) \to x$ in $X$ for all $x \in X$, and $\tilde R_j(y) \to y$ in $Y$ for all $y \in Y$.
\end{enumerate}
\end{defi}

The third point is the crucial ``core regularization" on the reduced range. The fourth and fifth points encode compatibility with the first-order structure, and the last point ensures that the nets approximate the identity strongly on $X$ and $Y$. In the case of semigroups of operators acting on noncommutative $\L^p$-spaces associated with von Neumann algebras, these operators will be connected to approximations properties, see e.g.~\cite{BrO08}.
In Section~\ref{sec-Schur}, we will briefly describe an example of such a net consisting of matrix truncations.

The point of regularizing nets is that they replace the smoothing action of the semigroup by a purely algebraic and bounded approximation scheme. The next proposition shows that this suffices to obtain the core properties and the identification $\ovl{\partial\partial^\dagger|_{\ovl{\Ran\partial}}}=\tilde A|_{\ovl{\Ran\partial}}$, without assuming any regularization property for the semigroup $(\tilde T_t)_{t \geq 0}$.

\begin{prop}
\label{prop-cores}
Assume that we have a couple of regularizing nets in the sense of Definition \ref{Def-regul}.
Then:
\begin{enumerate}
\item $\partial(C)$ is a core of the unbounded operator $\partial^\dagger|_{\ovl{\Ran\partial}}$.
\item $\partial(C)$ is a core of the unbounded operator $\partial\partial^\dagger|_{\ovl{\Ran\partial}}$.
\item The operator $\partial\partial^\dagger|_{\ovl{\Ran\partial}}$ is closable and we have $
\ovl{\partial\partial^\dagger|_{\ovl{\Ran\partial}}}
=\tilde A|_{\ovl{\Ran\partial}}$.
\end{enumerate}
\end{prop}

\begin{proof}
1. Let $y \in \dom(\partial^\dagger|_{\ovl{\Ran\partial}})=\dom\partial^\dagger\cap \ovl{\Ran\partial}$ and set
$y_j\ov{\mathrm{def}}{=}\tilde R_j y$.
By Definition \ref{Def-regul}, we have $y_j\in \partial(C) \subset \ovl{\Ran\partial}$ and $y_j \in \dom \partial^\dagger$. Moreover, since $y \in \dom\partial^\dagger$, the identity \eqref{commute-regul-bis} implies that 
\[
\partial^\dagger(y_j)
=\partial^\dagger(\tilde R_j y)
\ov{\eqref{commute-regul-bis}}{=} R_j(\partial^\dagger y)
\to \partial^\dagger y.
\]
Moreover, we have $y_j \to y$ in $Y$. Hence $y_j \to y$ in the graph norm of $\partial^\dagger|_{\ovl{\Ran\partial}}$. This proves that $\partial(C)$ is a core.

2. Let $y \in \dom(\partial\partial^\dagger|_{\ovl{\Ran\partial}})$. Then $y\in \dom\partial^\dagger\cap \ovl{\Ran\partial}$ and $\partial^\dagger y\in \dom\partial$. Set again $y_j \ov{\mathrm{def}}{=} \tilde R_j y$. We have $y_j\in \partial(C) \subset \ovl{\Ran\partial}$ and $y_j \in \dom\partial^\dagger$. Moreover, we have
\[
\partial^\dagger(y_j)
=\partial^\dagger(\tilde R_j y)
\ov{\eqref{commute-regul-bis}}{=} R_j(\partial^\dagger y).
\]
Since $R_j(X)\subset \dom\partial$, we have $\partial^\dagger(y_j)\in \dom\partial$, hence $y_j \in \dom(\partial\partial^\dagger|_{\ovl{\Ran\partial}})$.

Next, using \eqref{commute-regul} with $x=\partial^\dagger y \in \dom\partial$, we obtain
\[
\partial\partial^\dagger(y_j)
= \partial\partial^\dagger(\tilde R_j y)
\ov{\eqref{commute-regul-bis}}{=}\partial(R_j\partial^\dagger y)
\ov{\eqref{commute-regul}}{=} \tilde R_j(\partial\partial^\dagger y)
\to \partial\partial^\dagger y.
\]
Furthermore, we have $y_j \to y$ in $Y$. Consequently, we have $y_j\to y$ in the graph norm of $\partial\partial^\dagger|_{\ovl{\Ran\partial}}$. This proves that $\partial(C)$ is a core.

3. By the second point, the subspace $\partial(C)$ is a core of $\partial\partial^\dagger|_{\ovl{\Ran\partial}}$.
For any $x \in C$, we have
\[
\partial\partial^\dagger|_{\ovl{\Ran\partial}}(\partial x)=\partial\partial^\dagger\partial x
\ov{\eqref{facto-with-diamond}}{=} \partial Ax
\ov{\eqref{commute-AAA}}{=}\tilde A \partial x.
\]
Hence $\partial\partial^\dagger|_{\ovl{\Ran\partial}}$ and $\tilde A|_{\ovl{\Ran\partial}}$ coincide on $\partial(C)$. Since $\tilde A|_{\ovl{\Ran\partial}}$ is closed, $\partial\partial^\dagger|_{\ovl{\Ran\partial}}$ is closable and $\ovl{\partial\partial^\dagger|_{\ovl{\Ran\partial}}} \subset \tilde A|_{\ovl{\Ran\partial}}$. On the other hand, by Proposition~\ref{Prop-core-1-group-bis}, the subspace $\partial(C)$ is a core of $\tilde A|_{\ovl{\Ran\partial}}$. Since $\ovl{\partial\partial^\dagger|_{\ovl{\Ran\partial}}}$ is closed and coincides with $\tilde A|_{\ovl{\Ran\partial}}$ on this core, we obtain $\tilde A|_{\ovl{\Ran\partial}}\subset \ovl{\partial\partial^\dagger|_{\ovl{\Ran\partial}}}$. Therefore $\ovl{\partial\partial^\dagger|_{\ovl{\Ran\partial}}}=\tilde A|_{\ovl{\Ran\partial}}$.
\end{proof}

\subsection{Boundedness of the functional calculus on the reduced space $X \oplus_2 \ovl{\Ran \partial}$}
\label{sec-boundedness-reduced-space}

In this section, we assume Definition \ref{def-regularization-coreC} or  Definition \ref{Def-regul}. We introduce the unbounded operator
\begin{equation}
\label{Def-Dirac-operator-Fourier}
D 
\ov{\mathrm{def}}{=}\begin{bmatrix} 
0 & \partial^\dagger|_{\ovl{\Ran \partial}} \\ 
\partial & 0 
\end{bmatrix}
\end{equation}
on the Banach space $X \oplus_2 \ovl{\Ran \partial}$ defined by
\begin{equation}
\label{Def-D-psi}
D(x,y)
\ov{\mathrm{def}}{=}
(\partial^\dagger(y),
\partial(x)\big), \quad x \in \dom \partial, y \in \dom \partial^\dagger \cap \ovl{\Ran \partial}.
\end{equation}
Its domain $\dom \partial \oplus (\dom \partial^\dagger \cap \ovl{\Ran \partial})$ is dense by Lemma \ref{Lemma-rest-closed-dense-Fourier-bis}. We call it the Hodge--Dirac operator associated with the pair $(\partial,\partial^\dagger)$.

\begin{prop}
\label{Prop-carre-Dirac-Fourier}
The operator $D^2$ is closable on the Banach space $X \oplus_2 \ovl{\Ran\partial}$ and we have
\begin{equation}
\label{D-alpha-carre-egal-Fourier}
\ovl{D^2}
=\begin{bmatrix}
A & 0\\
0 & \tilde{A}|_{\ovl{\Ran\partial}}
\end{bmatrix}.
\end{equation}
\end{prop}

\begin{proof}
We have
\begin{equation}
\label{eq-29932}
D^2
\ov{\eqref{Def-Dirac-operator-Fourier}}{=}
\begin{bmatrix} 
0 & \partial^\dagger|_{\ovl{\Ran \partial}} \\ 
\partial& 0 
\end{bmatrix}^2 
=
\begin{bmatrix} 
\partial^\dagger\partial & 0 \\ 
0 & \partial \partial^\dagger|_{\ovl{\Ran \partial}}
\end{bmatrix} 
\ov{\eqref{facto-with-diamond}}{=}
\begin{bmatrix} 
A & 0 \\ 
0 & \partial\partial^\dagger|_{\ovl{\Ran\partial}}
\end{bmatrix}.  
\end{equation}
By Proposition~\ref{Prop-fundamental-Fourier} in the case of the regularization property, and by Proposition~\ref{prop-cores} in the case of regularizing nets, the operator
$\partial\partial^\dagger|_{\ovl{\Ran\partial}}$ is closable and its closure is $\tilde{A}|_{\ovl{\Ran\partial}}$. Since $A$ is closed, the operator $D^2$ is closable and
we have $
\ovl{D^2}
\ov{\eqref{eq-29932}}{=}\begin{bmatrix}
A & 0\\
0 & \ovl{\partial\partial^\dagger|_{\ovl{\Ran\partial}}}
\end{bmatrix}
=\begin{bmatrix}
A & 0\\
0 & \tilde{A}|_{\ovl{\Ran\partial}}
\end{bmatrix}$.
\end{proof}

\begin{lemma}
\label{lem-resolvent-invariant}
Let $Z$ be a closed subspace of $Y$ such that $\tilde{T}_t(Z) \subset Z$ for any $t \geq 0$. Then $R(-s,\tilde A)(Z) \subset Z$ for all $s > 0$.
\end{lemma}

\begin{proof}
Fix $s > 0$ and $y \in Z$. By the Laplace formula for the resolvent, we have $R(-s,\tilde A)y \ov{\eqref{Resolvent-Laplace}}{=} -\int_0^\infty \e^{-st} \tilde T_t y \d t$. Since $\tilde T_t y \in Z$ for any $t \geq 0$ and since the subspace $Z$ is closed, the integral belongs to $Z$.
\end{proof}

Immediately after Lemma \ref{Lemma-rest-closed-dense-Fourier-bis}, we observed that $\tilde{T}_t$  leaves the space $\ovl{\Ran\partial}$ invariant. So, we can use Lemma \ref{lem-resolvent-invariant} with $Z=\ovl{\Ran\partial}$. With the argument following \cite[Proposition 10.2.18 p.~386]{HvNVW18}, we obtain the invariance of $\ovl{\Ran\partial}$ under the resolvent $R(\lambda,\tilde{A})$ for any $\lambda$ in the connected component of $\rho(\tilde A)$ containing $(-\infty,0)$ and we have $R(\lambda,\tilde{A}|_{\ovl{\Ran\partial}})
=
R(\lambda,\tilde A)|_{\ovl{\Ran\partial}}$ for such $\lambda$.

\begin{thm}
\label{Th-D-R-bisectorial-Fourier}
Assume that the Banach spaces $X$ and $Y$ are $\UMD$. We suppose that the semigroup $(T_t)_{t \geq 0}$ satisfies the condition $\Curv_{\partial,\H^\infty}(0)$. We assume that the operator $A$ satisfies the $\partial$-Riesz equivalence \eqref{estim-Riesz-bis} and that the operator $A^*$ satisfies the $(\partial^\dagger)^*$-Riesz equivalence \eqref{estim-Riesz-bis}. We assume the regularization property (Definition \ref{def-regularization-coreC} or Definition \ref{Def-regul}). The Hodge--Dirac operator $D$ is $R$-bisectorial on the Banach space $X \oplus_2 \ovl{\Ran \partial}$. 
\end{thm}

\begin{proof}
We will start by showing that the set $\i\mathbb{R}^*$ is contained in the resolvent set $\rho(D)$ of the Hodge--Dirac operator $D$. We will do this by proving that for any $t \in \mathbb{R}^*$ the operator $\i t\Id-D$ has a two-sided bounded inverse $R(\i t,D)$ given by
\begin{equation}
\label{Resolvent-Fourier}
\begin{bmatrix} 
\i t R(-t^2,A)               & R(-t^2,A)\partial^\dagger|_{\ovl{\Ran \partial}} \\ 
\partial R(-t^2,A) & \i t R(-t^2,\tilde{A}|_{\ovl{\Ran \partial}})
\end{bmatrix}
\co X \oplus_2 \ovl{\Ran \partial} \to X \oplus_2 \ovl{\Ran \partial}.
\end{equation}
Note that the operators $A$ and $\tilde{A}$ admit a bounded $\H^\infty(\Sigma_\theta)$ functional calculus for some $\theta \in (0, \frac{\pi}{2})$. By \cite[Proposition 10.2.18 p.~386]{HvNVW18}, this implies that $\tilde{A}|_{\ovl{\Ran \partial}}$ also admits a bounded $\H^\infty(\Sigma_\theta)$ functional calculus. So these operators are $R$-sectorial with $\omega_R(A) < \frac{\pi}{2}$ and $\omega_R(\tilde{A}) < \frac{\pi}{2}$ by \cite[Theorem 10.3.4 (2) p.~402]{HvNVW18} since $X$ and $Y$ have the triangular contraction property. Consequently, the subsets $\big\{z R(z,A) : z \not\in \ovl{\Sigma_\theta}\big\}$ and $\big\{z R(z,\tilde{A}) : z \not\in \ovl{\Sigma_\theta}\big\}$ are $R$-bounded (hence bounded) for any suitable $\theta > 0$. So the diagonal entries of \eqref{Resolvent-Fourier} are bounded. By Proposition \ref{Prop-R-gradient-bounds-Fourier}, the other entries are bounded. 

It only remains to check that this matrix defines a two-sided inverse of $\i t\Id- D$. On the subspace $C \oplus \partial(C)$, 
we have the following equalities of operators
\begin{align*}
\MoveEqLeft
\begin{bmatrix} 
\i t R(-t^2,A)               & R(-t^2,A)\partial^\dagger|_{\ovl{\Ran \partial}} \\ 
\partial R(-t^2,A) & \i t R(-t^2,\tilde{A}|_{\ovl{\Ran \partial}})
\end{bmatrix} (\i t\Id-D) \\
&\ov{\eqref{Def-Dirac-operator-Fourier}}{=}
\begin{bmatrix} 
\i t R(-t^2,A)               & R(-t^2,A)\partial^\dagger|_{\ovl{\Ran \partial}} \\ 
\partial R(-t^2,A) & \i t R(-t^2,\tilde{A}|_{\ovl{\Ran \partial}})
\end{bmatrix}
\begin{bmatrix} 
\i t\Id_{X} & -\partial^\dagger|_{\ovl{\Ran \partial}} \\ 
-\partial  & \i t\Id_{\ovl{\Ran \partial}}
\end{bmatrix} \\
&=\begin{bmatrix} 
-t^2 R(-t^2,A)-R(-t^2,A)\partial^\dagger\partial &-\i t R(-t^2,A)\partial^\dagger|_{\ovl{\Ran \partial}}+\i t R(-t^2,A)\partial^\dagger|_{\ovl{\Ran \partial}} \\
\i t\partial R(-t^2,A)-\i t R(-t^2,\tilde{A})\partial &-\partial R(-t^2,A)\partial^\dagger|_{\ovl{\Ran \partial}}-t^2 R(-t^2,\tilde{A}|_{\ovl{\Ran \partial}})
\end{bmatrix} \\
&\ov{\eqref{facto-with-diamond} \eqref{commuting-deriv-Fourier-bis} \eqref{eq-commute-partialdiamond-general}}{=}
\begin{bmatrix} 
-t^2 R(-t^2,A)-R(-t^2,A)A & 0 \\ 
\i t\partial R(-t^2,A)-\i t\partial R(-t^2,A)  
&(-t^2-\partial \partial^\dagger) R(-t^2,\tilde{A}|_{\ovl{\Ran \partial}})
\end{bmatrix} \\
&\ov{\eqref{relation-dur-fund-Fourier}}{=}  
\begin{bmatrix} 
\Id_{X} & 0 \\ 
0 & \Id_{\ovl{\Ran \partial}}
\end{bmatrix}.
\end{align*} 
Note that $C$ is a core of $\partial$ by Lemma \ref{lem-density-range-partial-domA}. Moreover, $\partial(C)$ is a core of $\partial^\dagger|_{\ovl{\Ran\partial}}$. In the case of the regularization property this follows from Proposition~\ref{Prop-fundamental-Fourier}, and in the case of regularizing nets it follows from Proposition~\ref{prop-cores}. Therefore the subspace $C\oplus\partial(C)$ is a core of the operator $D$.
So this identity extends to $\dom D$.

Moreover, on $X \oplus_2 \ovl{\Ran \partial}$, using $\dom A \subset \dom A^{\frac{1}{2}}=\dom \partial$ and Proposition \ref{prop-commuting-deriv2-correct} combined with \eqref{commuting-deriv-Fourier-bis}, we see that
\begin{align*}
\MoveEqLeft
(\i t\Id-D) 
\begin{bmatrix} 
\i t R(-t^2,A)               & R(-t^2,A)\partial^\dagger|_{\ovl{\Ran \partial}} \\ 
\partial R(-t^2,A) & \i t R(-t^2,\tilde{A}|_{\ovl{\Ran \partial}})
\end{bmatrix} \\
&\ov{\eqref{Def-Dirac-operator-Fourier}}{=} \begin{bmatrix} 
\i t\Id_{X} & -\partial^\dagger|_{\ovl{\Ran \partial}} \\ 
-\partial  & \i t\Id_{\ovl{\Ran \partial}}
\end{bmatrix}
\begin{bmatrix} 
\i t R(-t^2,A)               & R(-t^2,A)\partial^\dagger|_{\ovl{\Ran \partial}} \\ 
\partial R(-t^2,A) & \i t R(-t^2,\tilde{A}|_{\ovl{\Ran \partial}})
\end{bmatrix}
\\
&= \begin{bmatrix} 
-t^2 R(-t^2,A) - \partial^\dagger\partial R(-t^2,A) & \i t R(-t^2,A)\partial^\dagger|_{\ovl{\Ran \partial}}-\i t\partial^\dagger|_{\ovl{\Ran \partial}} R(-t^2,\tilde{A}|_{\ovl{\Ran \partial}}) \\ 
-\i t\partial R(-t^2,A)+\i t\partial R(-t^2,A) & -\partial R(-t^2,A)\partial^\dagger|_{\ovl{\Ran \partial}}- t^2 R(-t^2,\tilde{A}|_{\ovl{\Ran \partial}})
\end{bmatrix} \\
&\ov{\eqref{facto-with-diamond} \eqref{commuting-deriv-Fourier-bis} \eqref{eq-commute-partialdiamond-general}}{=}
\begin{bmatrix} 
-t^2 R(-t^2,A) - A R(-t^2,A) & \i t \partial^\dagger R(-t^2,\tilde{A}|_{\ovl{\Ran \partial}})-\i t\partial^\dagger R(-t^2,\tilde{A}|_{\ovl{\Ran \partial}}) \\ 
0 & -\partial \partial^\dagger R(-t^2,\tilde{A}|_{\ovl{\Ran \partial}})- t^2 R(-t^2,\tilde{A}|_{\ovl{\Ran \partial}})
\end{bmatrix} \\
&=\begin{bmatrix} 
\Id_{X} & 0 \\ 
0 & \Id_{\ovl{\Ran \partial}}
\end{bmatrix}.
\end{align*} 
It remains to show that the set $\{\i t R(\i t, D) : t > 0 \}$ of operators is $R$-bounded. For any $t>0$, note that
$$
\i t R(\i t,D)
=\i t\begin{bmatrix} 
\i t R(-t^2,A)               & R(-t^2,A)\partial^\dagger|_{\ovl{\Ran \partial}} \\ 
\partial R(-t^2,A) & \i t R(-t^2,\tilde{A}|_{\ovl{\Ran \partial}})
\end{bmatrix}
=\begin{bmatrix} 
- t^2 R(-t^2,A)               & \i t R(-t^2,A)\partial^\dagger|_{\ovl{\Ran \partial}} \\ 
\i t\partial R(-t^2,A) & - t^2 R(-t^2,\tilde{A}|_{\ovl{\Ran \partial}})
\end{bmatrix}.
$$
Now, observe that the diagonal entries are $R$-bounded by the $R$-sectoriality of $A$ and $\tilde{A}$. The $R$-boundedness of the other entries follows from the $R$-gradient bounds of Proposition \ref{Prop-R-gradient-bounds-Fourier} applied to $(A,\partial)$ and to $(A^*,(\partial^\dagger)^*)$, combined with \cite[Problem 5.26 p.~168]{Kat76}. Note that a set of operator matrices is $R$-bounded precisely when each entry is $R$-bounded. This follows from stability of $R$-boundedness under sums and compositions and the bounded coordinate projections on $X \oplus_2 Y$. We conclude that the operator $D$ is $R$-bisectorial.
\end{proof}

\begin{thm}
\label{Th-functional-calculus-bisector-Fourier}
Assume that the Banach spaces $X$ and $Y$ are $\UMD$. We suppose that the semigroup $(T_t)_{t \geq 0}$ satisfies the condition $\Curv_{\partial,\H^\infty}(0)$. We assume that the operator $A$ satisfies the $\partial$-Riesz equivalence \eqref{estim-Riesz-bis} and that the operator $A^*$ satisfies the $(\partial^\dagger)^*$-Riesz equivalence \eqref{estim-Riesz-bis}. We assume the regularization property (Definition \ref{def-regularization-coreC} or  Definition \ref{Def-regul}). The Hodge--Dirac operator $D$  admits a bounded $\mathrm{H}^\infty(\Sigma^\bi_\omega)$ functional calculus for some angle $\omega \in (0,\frac{\pi}{2})$. 
\end{thm}

\begin{proof}
Since the operators $A$ and $\tilde{A}$ admit a bounded $\H^\infty(\Sigma_\theta)$ functional calculus of angle $\theta \in (0, \frac{\pi}{2})$, the operator 
$$
\ovl{D^2} 
\ov{\eqref{D-alpha-carre-egal-Fourier}}{=}\begin{bmatrix} 
A & 0 \\ 
0 & \tilde{A}|_{\ovl{\Ran \partial}}
\end{bmatrix}
$$ 
also admits a bounded $\H^\infty(\Sigma_\theta)$ functional calculus of angle $\theta \in (0, \frac{\pi}{2})$. Since the operator $D$ is $R$-bisectorial by Theorem \ref{Th-D-R-bisectorial-Fourier}, we deduce by \cite[Theorem 10.6.7 p.~450]{HvNVW18} that the unbounded operator $D^2$ is $R$-sectorial, hence closed, since any operator with non-empty resolvent set is closed \cite[p.~360]{HvNVW18}. This implies that $\ovl{D^2}=D^2$. The result \cite[Theorem 10.6.7 p.~450]{HvNVW18} also implies that $D$ admits a bounded $\H^\infty(\Sigma^\bi_\omega)$ functional calculus for $\omega=\frac{\theta}{2}$.
\end{proof}

\begin{remark} \normalfont
\label{remark-sgn}
The bounded $\H^\infty$ functional calculus for $D$ yields the Riesz equivalence, and can therefore be regarded as a strengthening of the equivalence \eqref{estim-Riesz-bis}. To see this, let $\sgn \in \H^\infty(\Sigma_\omega^\bi)$ be defined by $
\sgn(z)
\ov{\mathrm{def}}{=}
1_{\Sigma_\omega^+}(z)-1_{\Sigma_\omega^-}(z)$, where $z \in \Sigma_\omega^\bi$. Assume that $D$ admits a bounded $\H^\infty(\Sigma_\omega^\bi)$ functional calculus on $X \oplus_2 \ovl{\Ran \partial}$. By Example \ref{ex-signe}, the operator $\sgn(D) \co X \oplus_2 \ovl{\Ran\partial} \to X \oplus_2 \ovl{\Ran\partial}$ is bounded. Moreover, we have
\begin{equation}
\label{lien-sign-abs-Fourier-bis}
|D|=\sgn(D)D
\quad \text{and} \quad
D=\sgn(D)|D|.
\end{equation}
We have $
\ovl{\Ran D}
=
\ovl{\Ran\big(\partial^\dagger|_{\ovl{\Ran\partial}}\big)}^{X}
\oplus_2
\ovl{\Ran\partial}
=
\ovl{\Ran A}\oplus_2\ovl{\Ran\partial}$, 
where the last equality, left to the reader, follows from the identity $A=\partial^\dagger\partial$ and from the fact that $\partial(C)$ is a core of $\partial^\dagger|_{\ovl{\Ran\partial}}$. Thus, for every $\xi$ belonging to the subspace $\dom D=\dom |D|$, we have
\[
\bnorm{D\xi}_{X \oplus_2 \ovl{\Ran \partial}}
\ov{\eqref{lien-sign-abs-Fourier-bis}}{=} 
\bnorm{\sgn(D)|D|\xi}_{X \oplus_2 \ovl{\Ran \partial}}
\lesssim
\bnorm{|D|\xi}_{X \oplus_2 \ovl{\Ran \partial}},
\]
and conversely
\[
\bnorm{|D|\xi}_{X \oplus_2 \ovl{\Ran \partial}}
\ov{\eqref{lien-sign-abs-Fourier-bis}}{=} 
\bnorm{\sgn(D)D\xi}_{X \oplus_2 \ovl{\Ran \partial}}
\lesssim
\bnorm{D\xi}_{X \oplus_2 \ovl{\Ran \partial}}.
\]
On the space $X \oplus_2 \ovl{\Ran \partial}$, we also have
$|D|
=(D^2)^{\frac{1}{2}}
\ov{\eqref{D-alpha-carre-egal-Fourier}}{=}
\begin{bmatrix}
A^{\frac{1}{2}} & 0 \\
0 & (\tilde{A}|_{\ovl{\Ran \partial}})^{\frac{1}{2}}
\end{bmatrix}$, 
where we use \cite[Theorem 3.2.20]{Ege15} in the first equality. Using 
$X
\ov{\eqref{decompo-reflexive}}{=} \ker A \oplus \ovl{\Ran A}$ and applying the preceding estimates to vectors of the form $(x,0)$ with $x \in \dom A^{\frac12} \cap \ovl{\Ran A}$ gives the desired Riesz equivalence, since the operators $A^{\frac12}$ and $\partial$ vanish on $\ker A$.
\end{remark}

\subsection{Boundedness of the functional calculus on the full space $X \oplus Y$}
\label{sec-functional-full}

The goal of this section is to extend the bisectorial functional calculus obtained on the reduced space $X \oplus_2 \ovl{\Ran \partial}$ to the full space $X \oplus_2 Y$ under some precise assumptions. Recall that we have two densely defined closed operators $
\partial \co \dom\partial \subset X \to Y$ and $\partial^\dagger \co \dom\partial^\dagger \subset Y \to X$. We can introduce the block operator
\begin{equation}
\label{def-scrD}
\scr{D}
\ov{\mathrm{def}}{=}
\begin{bmatrix}
0 & \partial^\dagger \\
\partial & 0
\end{bmatrix}
\end{equation}
defined on the Banach space $X \oplus_2 Y$ with domain $\dom \partial \oplus_2 \dom \partial^\dagger$. Now, we make the following assumption.

\begin{ass}
\label{ass-projection-reduction}
There exists a bounded projection $Q \co Y\to Y$ such that
\begin{enumerate}
\item $\Ran Q=\ovl{\Ran\partial}$,
\item $\ker Q \subset \dom\partial^\dagger$ and $\partial^\dagger|_{\ker Q}=0$ (in other words $\ker Q \subset \ker \partial^\dagger$).
\end{enumerate}
\end{ass}

This projection $Q \co Y \to Y$ onto the subspace $\ovl{\Ran \partial}$ yields a topological decomposition 
\begin{equation}
\label{decompo-6567}
Y
=\ovl{\Ran \partial} \oplus \ker Q
\end{equation}
and allows us to split the previous block operator $\scr{D}$ into the direct sum of the reduced Hodge--Dirac operator $D$ on $X \oplus_2 \ovl{\Ran \partial}$ and a trivial part on $\ker Q$. The bisectoriality and the bounded $\H^\infty$ functional calculus on $X \oplus_2 Y$ follow by block-diagonal reduction. Recall that the operator $D$ is defined in \eqref{Def-Dirac-operator-Fourier}.

\begin{thm}
\label{thm-full-bisectorial}
Assume Assumption \ref{ass-projection-reduction}. If the operator $D$ is bisectorial and admits a bounded $\H^\infty(\Sigma^\bi_\omega)$ functional calculus on $X\oplus_2 \ovl{\Ran\partial}$ for some angle $\omega \in (0,\frac{\pi}{2})$, then the operator $\scr{D}$ is bisectorial and admits a bounded $\H^\infty(\Sigma^\bi_\omega)$ functional calculus on $X\oplus_2 Y$.
\end{thm}

\begin{proof}
Let $y \in \dom \partial^\dagger$. We can write $y=Qy+(\Id-Q)y$, where $Qy \in \ovl{\Ran\partial}$ and $(\Id-Q)y \in \ker Q$. By the second point of Assumption \ref{ass-projection-reduction}, the element $(\Id-Q)y$ belongs to the subspace $\dom \partial^\dagger$ and $\partial^\dagger((\Id-Q)y)=0$. Hence by linearity $Qy=y-(\Id-Q)y$ belongs to the subspace $\dom \partial^\dagger$. We conclude that $Qy \in \dom\partial^\dagger\cap\ovl{\Ran\partial}$.

Conversely, if $u \in \dom\partial^\dagger\cap\ovl{\Ran\partial}$ and $v \in \ker Q$, then by Assumption \ref{ass-projection-reduction} we have $v \in \dom\partial^\dagger$ and $\partial^\dagger v=0$. Hence $u+v \in \dom\partial^\dagger$ and $\partial^\dagger(u+v)=\partial^\dagger u$. Therefore, under the decomposition $Y \ov{\eqref{decompo-6567}}{=} \ovl{\Ran\partial} \oplus \ker Q$, we have
\[
\dom\partial^\dagger
=
(\dom\partial^\dagger\cap\ovl{\Ran\partial}) \oplus \ker Q.
\]
Now, consider some $(x,y) \in \dom\scr{D}$. By definition, we have $x \in \dom \partial$ and $y \in \dom \partial^\dagger$. We obtain
\[
\scr{D}(x,y)
\ov{\eqref{def-scrD}}{=} \big(\partial^\dagger y,\partial x\big)
=\big(\partial^\dagger\big(Q y+(\Id-Q)y\big),\partial x\big)
=\big(\partial^\dagger(Q y),\partial x\big),
\]
which identifies to the element $\begin{bmatrix} 
0 & \partial^\dagger|_{\ovl{\Ran\partial}} & 0 \\ 
\partial & 0 & 0 \\ 
0 & 0 & 0 
\end{bmatrix}\begin{bmatrix}
    x    \\
    Q(y)    \\
	(\Id-Q)(y)	 \\
\end{bmatrix}$ in the space $X \oplus_2 Y
\ov{\eqref{decompo-6567}}{=} X \oplus_2 \ovl{\Ran\partial} \oplus_2 \ker Q$. Moreover, the element $\partial x$ obviously belongs to the subspace $\ovl{\Ran\partial}=\Ran Q$. Therefore, with respect to the decomposition $
X \oplus_2 \ovl{\Ran\partial} \oplus_2 \ker Q$, 
the operator $\scr{D}$ takes the block-diagonal form
\[
\scr{D}
=
\begin{bmatrix} 
0 & \partial^\dagger|_{\ovl{\Ran\partial}} & 0 \\ 
\partial & 0 & 0 \\ 
0 & 0 & 0 
\end{bmatrix}
\ov{\eqref{Def-Dirac-operator-Fourier}}{=}
\begin{bmatrix}
D & 0 \\
0 & 0
\end{bmatrix}.
\]
Hence $\scr{D}$ is the direct sum of $D$ and the zero operator on $\ker Q$. Consequently, bisectoriality and the boundedness of the $\H^\infty$ functional calculus transfers from $D$ to $\scr{D}$.
\end{proof}

\subsection{A bounded projection on the closed subspace $\ovl{\Ran\partial}$}
\label{sec-projection-RSstar}

Assumption \ref{ass-projection-reduction} isolates the only additional ingredient needed to pass from the reduced operator $D$ on $X \oplus_2 \ovl{\Ran\partial}$ to the full block operator $\scr{D}$ on $X\oplus_2 Y$, namely a bounded projection $Q$ onto $\ovl{\Ran\partial}$ which is compatible with $\partial^\dagger$ in the sense that $\partial^\dagger$ vanishes on the complementary subspace $\ker Q$. This is reminiscent of Hodge-type splittings, where one expects a topological decomposition of $Y$ into a ``range part" and a ``co-closed part" contained in $\ker\partial^\dagger$.

In this section, we provide a canonical candidate for such a projection. This bounded projection $Q \co Y \to Y$ always satisfies $
\Ran Q
=\ovl{\Ran\partial}$, 
while the annihilation property on $\ker Q$ may require an additional argument. When such a compatibility holds (by a separate argument in concrete situations or Proposition \ref{prop-kernelQ-kernelpartialdagger}), one may use Theorem \ref{thm-full-bisectorial}.

We assume that the Banach spaces $X$ and $Y$ are reflexive. Recall that $A$ is a sectorial operator on $X$. We suppose that the densely defined operator $\partial \co \dom \partial \subset X \to Y$ is closed and that $A$ admits the $\partial$-Riesz equivalence \eqref{estim-Riesz-bis} and that $A^*$ admits the $(\partial^\dagger)^*$-Riesz equivalence \eqref{estim-Riesz-bis}. By Proposition \ref{Prop-derivation-closable-sgrp-bis}, we have $\dom \ovl{\partial}=\dom A^{\frac{1}{2}}$. Since $\partial$ is closed, $\ovl{\partial}=\partial$, hence $\dom\partial=\dom A^{\frac{1}{2}}$.  Recall that $A=\partial^\dagger\partial$ in the sense of \eqref{facto-with-diamond} and that the operator $\partial^\dagger \co \dom \partial^\dagger \subset Y \to X$ has dense domain and is closed. It is well-known \cite[Proposition 3.8.2 p.~165]{ABHN11} that the subspace $\dom A$ is a core of the closed operator $A^{\frac12}$.

\paragraph{Riesz transform from $X$ into $Y$} We define the Riesz transform $R$ as follows. First, set
\begin{equation}
\label{first-Riesz}
R_0 \ov{\mathrm{def}}{=} \partial A^{-\frac12}\co \Ran A^{\frac12} \to Y,
\qquad
R_0(A^{\frac12}x) 
\ov{\mathrm{def}}{=} \partial x, \quad x \in \dom\partial.
\end{equation}
The $\partial$-Riesz equivalence implies that $R_0$ is well-defined and bounded on the subspace $\Ran A^{\frac12}$. Hence it extends uniquely to a bounded operator
\begin{equation}
\label{def-Riesz-34}
R\co \ovl{\Ran A} \ov{\eqref{inclusion-range}}{=} \ovl{\Ran A^{\frac12}} \to Y.
\end{equation}
We extend $R$ to all of $X$ by setting $R|_{\ker A}=0$ along the topological decomposition $X \ov{\eqref{decompo-reflexive}}{=} \ker A \oplus \ovl{\Ran A}$. Note that $\Ran R\subset \ovl{\Ran\partial}$.

\paragraph{A dual Riesz transform from $X^*$ into $Y^*$} Since the subspace $\dom \partial^\dagger$ is dense in $Y$, we can consider the Banach adjoint 
$
(\partial^\dagger)^* \co \dom (\partial^\dagger)^* \subset X^* \to Y^*$. We assume that the operator $A^*$ admits the $(\partial^\dagger)^*$-Riesz equivalence \eqref{estim-Riesz-bis}. By Proposition \ref{Prop-derivation-closable-sgrp-bis}, this implies that $\dom (\partial^\dagger)^* = \dom (A^*)^{\frac12}$ and that the operator 
\begin{equation}
\label{def-S0}
S_0 \ov{\mathrm{def}}{=} (\partial^\dagger)^*(A^*)^{-\frac{1}{2}}
\co \Ran (A^*)^{\frac12} \to Y^*,
\qquad
S_0\big((A^*)^{\frac12} z\big) 
\ov{\mathrm{def}}{=} (\partial^\dagger)^*(z), \quad z \in \dom (\partial^\dagger)^*
\end{equation}
 extends uniquely to a bounded operator
\begin{equation*}
S \co \ovl{\Ran A^*} \ov{\eqref{inclusion-range}}{=} \ovl{\Ran (A^*)^{\frac12}} \to Y^*.
\end{equation*}
We extend $S$ to all of $X^*$ by setting $S|_{\ker A^*} = 0$ along the topological decomposition $X^* \ov{\eqref{decompo-reflexive}}{=} \ker A^* \oplus \ovl{\Ran A^*}$.

\paragraph{A bounded projection onto $\ovl{\Ran\partial}$}
Since $Y$ is reflexive, we identify $Y$ with $Y^{**}$. The adjoint operator $S^* \co Y^{**} \to X^{**}$ is therefore viewed as a bounded operator $S^* \co Y \to X$. We define the bounded operator
\begin{equation}
\label{eq-def-projection-P}
Q 
\ov{\mathrm{def}}{=} R S^* \co Y \to Y.
\end{equation}

\begin{prop}
\label{prop-P-projection}
The operator $Q$ is a bounded projection onto the subspace $\ovl{\Ran\partial}$. In particular, we have the topological decomposition $Y=\ovl{\Ran\partial}\oplus \ker Q$. Moreover, we have the inclusion $\ker\partial^\dagger \subset \ker Q$.
\end{prop}

\begin{proof}
Fix $x \in \dom A$ and note that the element $\partial x$ belongs to $Y$. Let $u \in \Ran(A^*)^{\frac12}$. We can write $u=(A^*)^{\frac{1}{2}}z$ for some $z \in \dom(A^*)^{\frac{1}{2}}$. Since $S$ extends $S_0$, we have  
\begin{equation}
\label{inter-ERTY}
S(u)
=S_0(u)
=S\big((A^*)^{\frac{1}{2}}z\big)
\ov{\eqref{def-S0}}{=} (\partial^\dagger)^*(z)
= (\partial^\dagger)^*(A^*)^{-\frac{1}{2}}u.
\end{equation}
We deduce that
\begin{align*}
\MoveEqLeft
\big\la S^*(\partial x),u \big\ra_{X,X^*}
=\big\la \partial x, S(u)\big\ra_{Y,Y^*}
\ov{\eqref{inter-ERTY}}{=} \big\la \partial x, (\partial^\dagger)^*(A^*)^{-\frac{1}{2}}u\big\ra_{Y,Y^*} \\
&\ov{\eqref{crochet-duality}}{=} \big\la \partial^\dagger\partial x,(A^*)^{-\frac{1}{2}}u\big\ra_{X,X^*}
\ov{\eqref{facto-with-diamond}}{=} \big\la Ax,(A^*)^{-\frac{1}{2}}u \big\ra_{X,X^*} 
=\big\la A^{\frac12}x,u\big\ra_{X,X^*}.
\end{align*}
We have shown that $
\big\la S^*(\partial x)-A^{\frac12}x, u \big\ra_{X,X^*}
=0$ for any $u \in \Ran(A^*)^{\frac12}$. Since the subspace $\Ran(A^*)^{\frac12}$ is dense in the Banach space $\ovl{\Ran (A^*)^{\frac12}} \ov{\eqref{inclusion-range}}{=} \ovl{\Ran A^*}$ we infer that
\[
\big\la S^*(\partial x)-A^{\frac12}x, u \big\ra_{X,X^*}
=0,
\qquad u \in \ovl{\Ran A^*}.
\]
Moreover, for any $x \in \dom A$ and any $w \in \ker A^*$, we have using the standard identity $
(A^{\frac{1}{2}})^*=(A^*)^{\frac{1}{2}}$ provided by \cite[Corollary 5.2.4 p.~116]{MCSA01}
\[
\big\la A^{\frac12}x,w\big\ra_{X,X^*}
\ov{\eqref{crochet-duality}}{=} \big\la x,(A^*)^{\frac12}w \big\ra_{X,X^*}
=0.
\]
On the other hand, since $w \in \ker A^*$, we also have by definition of the operator $S$
\[
\big\la S^*(\partial x),w \big\ra_{X,X^*}
\ov{\eqref{crochet-duality}}{=} \big\la \partial x, S(w) \big\ra_{Y,Y^*}
=0.
\]
Since $A$ is densely defined and sectorial, its Banach adjoint $A^*$ is sectorial by \cite[p.~21]{Haa06}. Consequently, we have by \eqref{decompo-reflexive} the topological decomposition $X^* = \ovl{\Ran(A^*)} \oplus \ker A^*$. We conclude that
\begin{equation}
\label{eq:Sstar-partialx}
S^*(\partial x)
=A^{\frac12}x,\quad x \in \dom A.
\end{equation}

Let $x \in \dom A$. Using \eqref{eq:Sstar-partialx} and the defining property \eqref{first-Riesz} of $R$ on the subspace $\Ran(A)^{\frac{1}{2}}$, we obtain the equality
\[
Q(\partial x)
\ov{\eqref{eq-def-projection-P}}{=} RS^*(\partial x)
\ov{\eqref{eq:Sstar-partialx}}{=} R(A^{\frac12}x)
=\partial x.
\]
Thus $Q$ is the identity on the subspace $\partial(\dom A)$. By Lemma \ref{lem-density-range-partial-domA}, the subspace $\partial(\dom A)$ is dense in the Banach space $\ovl{\Ran\partial}$. Since $Q$ is bounded, it follows that
\begin{equation}
\label{eq:Q-Id-range}
Q(y)
=y,\quad y\in \ovl{\Ran\partial}.
\end{equation}
In particular, $\ovl{\Ran\partial}\subset \Ran Q$. Conversely, by construction $R$ takes values in $\ovl{\Ran\partial}$, hence
\[
\Ran Q
\ov{\eqref{eq-def-projection-P}}{=} \Ran(RS^*)\subset \Ran R \subset \ovl{\Ran\partial}.
\]
Therefore $\Ran Q = \ovl{\Ran\partial}$. Finally, since $Q$ is the identity on its range by \eqref{eq:Q-Id-range}, we have $Q^2 = Q$.

Let $y \in \ker\partial^\dagger$. Fix $u \in \Ran(A^*)^{\frac{1}{2}}$. We may write $u=(A^*)^{\frac{1}{2}}z$ for some element $z$ in $\dom(A^*)^{\frac{1}{2}} = \dom (\partial^\dagger)^*$. We have $S(u) \ov{\eqref{def-S0}}{=} (\partial^\dagger)^*z$. Hence
\[
\big\la S^*(y),u \big\ra_{X,X^*}
=\big\la y,S(u) \big\ra_{Y,Y^*}
=\big\la y,(\partial^\dagger)^*z \big\ra_{Y,Y^*}
=\la \partial^\dagger y,z \ra_{X,X^*}
=0.
\]
Thus $S^*(y)$ annihilates $\Ran(A^*)^{\frac{1}{2}}$. Since $(A^{\frac{1}{2}})^*=(A^*)^{\frac{1}{2}}$ by \cite[Corollary 5.2.4 p.~116]{MCSA01}, we have $\big(\Ran (A^*)^{\frac{1}{2}}\big)^\perp 
=\big(\Ran (A^{\frac{1}{2}})^*\big)^\perp \ov{\eqref{lien-ker-image}}{=} \ker A^{\frac{1}{2}} \ov{\eqref{inclusion-range}}{=} \ker A$. We conclude that the element $S^*(y)$ belongs to $\ker A$. But $R$ vanishes on the subspace $\ker A$ by definition, so $Q(y) \ov{\eqref{eq-def-projection-P}}{=} RS^*(y) = 0$. We conclude that $\ker\partial^\dagger \subset \ker Q$.
\end{proof}

Now, we describe a sufficient condition for the reverse inclusion.

\begin{prop}
\label{prop-kernelQ-kernelpartialdagger}
Assume that
\begin{equation}
\label{eq-hodge-splitting-Y}
Y
=\ovl{\Ran\partial}+\ker\partial^\dagger.
\end{equation}
Then $
\ker Q=\ker\partial^\dagger$. In particular, $\ker Q\subset \dom\partial^\dagger$ and $\partial^\dagger|_{\ker Q}=0$. So Assumption \ref{ass-projection-reduction} holds.
\end{prop}

\begin{proof}
According to Proposition \ref{prop-P-projection} we have $\Ran Q=\ovl{\Ran\partial}$ and $\ker\partial^\dagger\subset \ker Q$. It remains to show that $\ker Q\subset \ker\partial^\dagger$. Let $y \in \ker Q$. By \eqref{eq-hodge-splitting-Y}, we may write $
y=y_1+y_2$ with $y_1 \in \ovl{\Ran\partial}$ and $y_2 \in \ker\partial^\dagger$. Since $Q$ is the identity on its range $\ovl{\Ran\partial}$ and since $y_2$ belongs to $\ker\partial^\dagger \subset \ker Q$, we obtain $
0
=Qy
=Qy_1+Qy_2
=y_1+0$. Hence $y_1=0$ and therefore $y=y_2 \in \ker\partial^\dagger$. This shows $\ker Q\subset \ker\partial^\dagger$.
\end{proof}

\section{Complex geometry and analysis}
\label{Back-complex}

\subsection{{Background and notations}}
\label{Back-complex-sub}

\paragraph{$\L^p$-spaces of sections} We briefly recall the definition of $\L^p$-spaces of sections on a manifold. Let $E$ be a smooth Hermitian vector bundle of finite rank over a smooth compact manifold $M$ equipped with a Borel measure $\mu$. Suppose that $1 \leq p \leq \infty$. For a measurable section $f \co M \to E$, following \cite[p.~481]{Nic21} and \cite[Chapter I]{Gun17} we define
\begin{equation}
\label{Lp-norm-vector-bundle}
\norm{f}_{\L^p(M,E)} 
\ov{\mathrm{def}}{=} 
\begin{cases}
\left( \int_M \norm{f(x)}_{E_x}^p \d\mu(x) \right)^{\frac{1}{p}}, & 1 \leq p < \infty\\
\esssup_{x \in M} \norm{f(x)}_{E_x}, & p=\infty
\end{cases},
\end{equation}
and we denote by $\L^p(M,E)$ the corresponding Banach space of measurable sections modulo equality $\mu$-almost everywhere. 

\paragraph{Complex geometry}
Let $E$ be a holomorphic vector bundle over a compact complex manifold $M$ of complex dimension $n$. For any integers $r,q \in \{0,\ldots,n\}$, we denote by $\Omega^{r,q}(M,E)$ the space of smooth differential forms of bidegree $(r,q)$ on the manifold $M$ with values in $E$ and we let $\Omega^{\bullet,\bullet}(M,E) \ov{\mathrm{def}}{=} \oplus_{0 \leq r,q \leq n} \Omega^{r,q}(M,E)$.

Let $E$ be a Hermitian holomorphic vector bundle over a compact Hermitian manifold $M$ of complex dimension $n$. For any integers $r,q \in \{0,\ldots, n\}$, we set $
\L^p(\Omega^{r,q}(M,E)) 
\ov{\mathrm{def}}{=} \L^p\big(M,\Lambda^{r,q}\mathrm{T}^*M \ot E\big)$, 
where the norm on each fiber $\Lambda^{r,q}\mathrm{T}_x^*M \ot E_x$ is the one induced by the Hermitian metric on $E$ and the pointwise inner product on $\Lambda^{r,q}\mathrm{T}_x^*M$. We also set $
\L^p(\Omega^{\bullet,\bullet}(M,E)) \ov{\mathrm{def}}{=} \L^p\big(M,\Lambda^{\bullet,\bullet}\mathrm{T}^*M \ot E\big)$, where the fiber norm is given by the pointwise $\ell^2$-sum of bidegrees. Thus, if $\omega=\sum_{r,q=0}^n \omega_{r,q}$ with $\omega_{r,q} \in \Omega^{r,q}(M,E)$, then
\begin{equation}
\label{norm-Lp-Df}
\norm{\omega}_{\L^p(\Omega^{\bullet,\bullet}(M,E))}
=
\bigg(\int_M \bigg(\sum_{r,q=0}^n \norm{\omega_{r,q}(x)}_{\Lambda^{r,q} \mathrm{T}_x^*M \ot E_x}^2\bigg)^{\frac{p}{2}} \d\mu_g(x)\bigg)^{\frac{1}{p}},
\qquad 1 \leq p < \infty,
\end{equation}
where $\mu_g$ is the Riemannian measure.

Now, assume that $M$ is Hermitian and that $E$ is endowed with a Hermitian metric. We will use the Dolbeault operator $\bar\partial_E \co \Omega^{\bullet,\bullet}(M,E) \to \Omega^{\bullet,\bullet}(M,E)$ of \cite[Proposition 4.16 p.~116]{Lee24}, which satisfies by \cite[p.~59]{Voi07} and \cite[Lemma 2.6.23 p.~109]{Huy05} the equality
\begin{equation}
\label{square-derivation}
(\bar \partial_E)^2
=0.
\end{equation}
Its restriction on the subspace $\Omega^{r,q}(M,E)$ is denoted by $\bar\partial_{E,r,q} \co \Omega^{r,q}(M,E) \to \Omega^{r,q+1}(M,E)$. Its restriction on the subspace $\Omega^{0,\bullet}(M,E) \ov{\mathrm{def}}{=} \oplus_{q=0}^n \Omega^{0,q}(M,E)$ of anti-holomorphic differential forms is denoted $\bar\partial_{E,0,\bullet} \co \Omega^{0,\bullet}(M,E) \to \Omega^{0,\bullet}(M,E)$. Let $\bar{\partial}_{E}^*$ be the formal adjoint of $\bar\partial_E$ with respect to the scalar product
$$
\la s_1,s_2 \ra
=\int_M  \la s_1(x),s_2(x) \ra_{E_x} \d \mu_g(x).
$$
By restriction, we have a linear operator $\bar{\partial}_{E}^* \co \Omega^{r,q+1}(M,E) \to \Omega^{r,q}(M,E)$. We also have an operator $\partial_E \co \Omega^{\bullet,\bullet}(M,E) \to \Omega^{\bullet,\bullet}(M,E)$ with formal adjoint $\partial_E^* \co \Omega^{\bullet,\bullet}(M,E) \to \Omega^{\bullet,\bullet}(M,E)$. The restrictions are denoted $\partial_{E,r,q} \co \Omega^{r,q}(M,E) \to \Omega^{r+1,q}(M,E)$ and $\partial_{E}^* \co \Omega^{r,q}(M,E) \to \Omega^{r-1,q}(M,E)$.

Following \cite[(1.4.3) p.~30]{MaM07} or \cite[(9.28) p.~277]{Lee24}, we introduce the holomorphic Kodaira Laplacian $\Delta_{\bar\partial_E} \co \Omega^{\bullet,\bullet}(M,E) \to \Omega^{\bullet,\bullet}(M,E) $ defined by
\begin{equation}
\label{holo-Laplacian}
\Delta_{\bar\partial_E}
\ov{\mathrm{def}}{=} \bar{\partial}_E\bar{\partial}_{E}^*+\bar{\partial}_E^*\bar{\partial}_E
\end{equation}
(also denoted $\Box_E$), which is elliptic by \cite[Proposition 9.34 p.~277]{Lee24}. We denote by $\Delta_{\bar\partial_E,0,\bullet}$ its restriction to the space $\Omega^{0,\bullet}(M,E)$. Following \cite[(1.4.3) p.~30]{MaM07}, we introduce the Dolbeault--Dirac operator $\D_E \co \Omega^{0,\bullet}(M,E) \to \Omega^{0,\bullet}(M,E)$ defined by
\begin{equation}
\label{Def-Dirac-DE}
\D_E
\ov{\mathrm{def}}{=} \sqrt{2}(\bar{\partial}_{E,0,\bullet} +\bar{\partial}_{E,0,\bullet}^*).
\end{equation}
This is a first order elliptic differential operator. According to \cite[(1.4.4) p.~30]{MaM07}, we have the equality $
\D_E^2
= 2 \Delta_{\bar\partial_E,0,\bullet}$. We have the following commutation relations.

\begin{lemma}
\label{lem:commutation-Kodaira}
Let $E$ be a holomorphic vector bundle over a compact complex manifold $M$. We have
\begin{equation}
\label{commutation-rules}
\bar{\partial}_E \Delta_{\bar\partial_E} 
= \Delta_{\bar\partial_E} \bar{\partial}_E
\quad \text{and} \quad
\bar{\partial}_E^* \Delta_{\bar\partial_E} 
= \Delta_{\bar\partial_E} \bar{\partial}_{E}^*.
\end{equation}
\end{lemma}

\begin{proof}
Since $\bar\partial_E^2=0$, on the one hand we have $
\bar\partial_E\Delta_{\bar\partial_E}
\ov{\eqref{holo-Laplacian}}{=} \bar\partial_E(\bar\partial_E\bar\partial_E^*+\bar\partial_E^*\bar\partial_E) 
=\bar\partial_E\bar\partial_E^*\bar\partial_E$. 
On the other hand, we have $
\Delta_{\bar\partial_E}\bar\partial_E
\ov{\eqref{holo-Laplacian}}{=}(\bar\partial_E\bar\partial_E^*+\bar\partial_E^*\bar\partial_E)\bar\partial_E 
=\bar\partial_E\bar\partial_E^*\bar\partial_E$. 
So the first equality is proved. Since $(\bar\partial_E^*)^2=(\bar\partial_E^2)^*=0$, we see that 
$
\bar\partial_E^*\Delta_{\bar\partial_E}
\ov{\eqref{holo-Laplacian}}{=}\bar\partial_E^*(\bar\partial_E\bar\partial_E^*+\bar\partial_E^*\bar\partial_E) 
=\bar\partial_E^*\bar\partial_E\bar\partial_E^*$  
and  
$
\Delta_{\bar\partial_E}\bar\partial_E^*
\ov{\eqref{holo-Laplacian}}{=}(\bar\partial_E\bar\partial_E^*+\bar\partial_E^*\bar\partial_E)\bar\partial_E^* 
=\bar\partial_E^*\bar\partial_E\bar\partial_E^*$. 
We obtained the second equality.
\end{proof}

Suppose that $1 < p < \infty$. Let $E$ be a Hermitian holomorphic vector bundle of finite rank over a compact Hermitian manifold $M$. As a direct consequence of \cite[Theorem 5.28 p.~168]{Kat76} and the symmetry of $\D_E$, the Dolbeault--Dirac $\D_E \co \Omega^{0,\bullet}(M,E) \to \Omega^{0,\bullet}(M,E)$ is closable on the Banach space $\L^p(\Omega^{0,\bullet}(M,E))$. 
We denote by $\D_{E,p} \co \dom \D_{E,p} \subset \L^p(\Omega^{0,\bullet}(M,E)) \to \L^p(\Omega^{0,\bullet}(M,E))$ its closure. It is known that the closed operator $\D_{E,2}$ is selfadjoint, see \cite[Proposition 12.2 pp.~58-59 and p.~61]{BDIP02}. 

\subsection{Canonical grading by form degree}
\label{sec-grading}

Now, we introduce a standard grading by form degree.

\begin{prop}
\label{prop:even-Dolbeault-Dirac}
Let $E$ be a Hermitian holomorphic vector bundle of finite rank over a compact Hermitian manifold $M$ with complex dimension $n$. Suppose that $1 < p < \infty$. Then the maps $\Omega^{0,k}(M,E) \to \Omega^{0,k}(M,E)$, $\omega \mapsto (-1)^{k}\omega$, where $k \in \{0,\ldots,n\}$, induce an isometric isomorphism $\gamma \co \L^p(\Omega^{0,\bullet}(M,E))\to \L^p(\Omega^{0,\bullet}(M,E))$ and we have $\gamma \D_{E,p} = -\D_{E,p}\gamma$.
\end{prop}

\begin{proof}
We define a linear map $\gamma \co \Omega^{0,\bullet}(M,E) \to \Omega^{0,\bullet}(M,E)$ by prescribing its action on homogeneous forms:
\begin{equation}
\label{def-gamma-123}
\gamma\omega \ov{\mathrm{def}}{=} (-1)^q \omega,
\qquad
\omega \in \Omega^{0,q}(M,E),\ q = 0,\dots,n,
\end{equation}
and extending linearly to arbitrary elements of $\Omega^{0,\bullet}(M,E)$. It is immediate that $
\gamma^2 = \Id_{\Omega^{0,\bullet}(M,E)}$. If $\omega \in \Omega^{0,\bullet}(M,E)$, we can write $\omega = \sum_{q=0}^n \omega_q$ with $\omega_q \in \Omega^{0,q}(M,E)$. Observe that
\begin{align*}
\norm{\gamma\omega}_{\L^p(\Omega^{0,\bullet}(M,E))}
&\ov{\eqref{norm-Lp-Df}\eqref{def-gamma-123}}{=}
\bigg(\int_M \bigg(\sum_{q=0}^n |(-1)^q\omega_q(x)|_{\Lambda^{0,q} \mathrm{T}_x^*M \ot E_x}^2\bigg)^{\frac{p}{2}} \d\mu_g(x)\bigg)^{\frac1p}
\\
&=
\bigg(\int_M \bigg(\sum_{q=0}^n |\omega_q(x)|_{\Lambda^{0,q} \mathrm{T}_x^*M \ot E_x}^2\bigg)^{\frac{p}{2}} \d\mu_g(x)\bigg)^{\frac1p}
\ov{\eqref{norm-Lp-Df}}{=} \norm{\omega}_{\L^p(\Omega^{0,\bullet}(M,E))}.
\end{align*}
Since $\Omega^{0,\bullet}(M,E)$ is dense in the Banach space $\L^p(\Omega^{0,\bullet}(M,E))$, the operator $\gamma \co \Omega^{0,\bullet}(M,E) \to \Omega^{0,\bullet}(M,E)$ clearly extends by continuity to an isometric isomorphism $\gamma \co \L^p(\Omega^{0,\bullet}(M,E)) \to \L^p(\Omega^{0,\bullet}(M,E))$ satisfying $\gamma^2=\Id_{\L^p(\Omega^{0,\bullet}(M,E))}$. 

Let $f \in \C(M)$ and let $\omega \in \Omega^{0,\bullet}(M,E)$. Write $\omega = \sum_{q=0}^n \omega_q$ with $\omega_q \in \Omega^{0,q}(M,E)$. Then $\pi(f)\omega = f\omega = \sum_{q=0}^n f\omega_q$, and each $f\omega_q$ belongs to $\Omega^{0,q}(M,E)$. Consequently, we have
\[
\gamma\big(\pi(f)\omega\big)
=
\gamma\bigg(\sum_{q=0}^n f\omega_q\bigg)
=
\sum_{q=0}^n (-1)^q f\omega_q
=
f\sum_{q=0}^n (-1)^q \omega_q
=
\pi(f)\gamma\omega.
\]
Hence $\gamma\pi(f) = \pi(f)\gamma$ on $\Omega^{0,\bullet}(M,E)$, and by density and continuity we obtain on $\L^p(\Omega^{0,\bullet}(M,E))$ the equality $
\gamma\pi(f) = \pi(f)\gamma$ for any $f \in \C(M)$.

Finally, we check the anticommutation relation with $\D_E$. Recall that $\D_E \ov{\mathrm{def}}{=} \sqrt{2}(\bar{\partial}_{E,0,\bullet} +\bar{\partial}_{E,0,\bullet}^*)$, where $\bar{\partial}_{E}$ increases the degree by one and $\bar{\partial}_{E}^*$ decreases the degree by one. Let $\omega_q \in \Omega^{0,q}(M,E)$. Then
\begin{equation}
\label{inter-ABC}
\D_E\omega_q
\ov{\eqref{Def-Dirac-DE}}{=} \sqrt{2}\big(\bar{\partial}_{E}(\omega_q) + \bar{\partial}_{E}^*(\omega_q)\big)
=\sqrt{2}\alpha_{q+1} + \sqrt{2}\beta_{q-1},
\end{equation}
where $\alpha_{q+1} \ov{\mathrm{def}}{=} \bar{\partial}_{E}\omega_q$ belongs to $\Omega^{0,q+1}(M,E)$ and $\beta_{q-1} \ov{\mathrm{def}}{=} \bar{\partial}_{E}^*\omega_q$ belongs to $\Omega^{0,q-1}(M,E)$ (with the convention $\alpha_{n+1} = 0$ if $q=n$ and $\beta_{-1} = 0$ if $q=0$). On the one hand, since the map $\gamma$ acts by $(-1)^{\deg}$ on homogeneous forms, we have
\begin{equation}
\label{inter-UYIIII}
\gamma \D_E\omega_q
\ov{\eqref{inter-ABC}}{=}
\sqrt{2}\gamma(\alpha_{q+1} + \beta_{q-1})
\ov{\eqref{def-gamma-123}}{=}
(-1)^{q+1}\sqrt{2}\alpha_{q+1} + (-1)^{q-1}\sqrt{2}\beta_{q-1}.
\end{equation}
On the other hand, we have
\begin{equation}
\label{inter-DFGT}
\D_E\gamma\omega_q
\ov{\eqref{def-gamma-123}}{=} (-1)^q \D_E\omega_q
\ov{\eqref{inter-ABC}}{=} (-1)^q \sqrt{2}\alpha_{q+1} + (-1)^q \sqrt{2}\beta_{q-1}.
\end{equation}
Therefore
\begin{align*}
\MoveEqLeft
(\gamma \D_E + \D_E\gamma)\omega_q
\ov{\eqref{inter-DFGT} \eqref{inter-UYIIII}}{=}
\sqrt{2}\big((-1)^{q+1} + (-1)^q\big)\alpha_{q+1}
+\sqrt{2}\big((-1)^{q-1} + (-1)^q\big)\beta_{q-1}
= 0.
\end{align*} 
We conclude by linearity that $\gamma \D_E = -\D_E \gamma$ on $\Omega^{0,\bullet}(M,E)$. Recall that the operator $\D_E$ acting on the subspace $\Omega^{0,\bullet}(M,E)$ is closable and that by definition $\D_{E,p}$ is the closure of $\D_E$. In particular, $\Omega^{0,\bullet}(M,E)$ is a core for the operator $\D_{E,p}$. 

To pass from the anticommutation relation on $\Omega^{0,\bullet}(M,E)$ to the domain of $\D_{E,p}$, consider some $u \in \dom \D_{E,p}$. By definition, there exists a sequence $(u_k)$ of elements of $\Omega^{0,\bullet}(M,E)$ such that $
u_k \to u$ and $\D_E u_k \to \D_{E,p}u$ in $\L^p(\Omega^{0,\bullet}(M,E))$. Since the map $\gamma$ is bounded on $\L^p(\Omega^{0,\bullet}(M,E))$, we have $\gamma u_k \to \gamma u$ in $\L^p(\Omega^{0,\bullet}(M,E))$. Moreover, using $\gamma \D_E = - \D_E \gamma$ on $\Omega^{0,\bullet}(M,E)$, we obtain
\[
\D_E(\gamma u_k)
=
-\gamma(\D_E u_k)
\to
-\gamma(\D_{E,p}u)
\qquad\text{in }\L^p(\Omega^{0,\bullet}(M,E)).
\]
Therefore $(\gamma u_k)$ is an approximating sequence in the graph norm for $\gamma u$, which shows that $\gamma u \in \dom \D_{E,p}$ and $
\D_{E,p}(\gamma u)
=
-\gamma(\D_{E,p}u)$. Hence $\gamma \D_{E,p} + \D_{E,p}\gamma = 0$ on the space $\dom \D_{E,p}$.
\end{proof}

Let $E$ be a Hermitian holomorphic vector bundle of finite rank over a compact Hermitian manifold $M$. Suppose that $1 < p < \infty$. Recall that we have the canonical decomposition $\Omega^{0,\bullet}(M,E) = \Omega^{0,\even}(M,E) \oplus \Omega^{0,\odd}(M,E)$ of the space of smooth forms, where $\Omega^{0,\even}(M,E) \ov{\mathrm{def}}{=} \bigoplus_{q \text{ even}} \Omega^{0,q}(M,E)$ and $\Omega^{0,\odd}(M,E) \ov{\mathrm{def}}{=} \bigoplus_{q \text{ odd}} \Omega^{0,q}(M,E)$. 
By definition of the $\L^p$-norms on forms, this decomposition induces by completion a direct sum decomposition
\begin{equation}
\label{decompo-Lp-even-odd}
\L^p(\Omega^{0,\bullet}(M,E))
= \L^p(\Omega^{0,\even}(M,E)) \oplus \L^p(\Omega^{0,\odd}(M,E)).
\end{equation}
With respect to this decomposition, by \cite[Proposition 8.1]{Arh26b} the operator $\D_{E,p}$ admits 
the block matrix form 
\begin{equation}
\label{Dp-decompo-block}
\D_{E,p} 
= \begin{bmatrix}
  0   & \D_{E,-,p} \\
 \D_{E,+,p}  &  0 \\
\end{bmatrix},
\end{equation}
where the unbounded operators $
\D_{E,+,p} \co \dom \D_{E,+,p} \subset \L^p(\Omega^{0,\even}(M,E)) \to \L^p(\Omega^{0,\odd}(M,E))$ and $\D_{E,-,p} \co \dom \D_{E,-,p} \subset \L^p(\Omega^{0,\odd}(M,E)) \to \L^p(\Omega^{0,\even}(M,E))$ are the off-diagonal parts of $\D_{E,p}$ relative to the decomposition \eqref{decompo-Lp-even-odd}. Consider the restrictions
\[
\bar\partial_E^{\even} \colon \Omega^{0,\even}(M,E)\to\Omega^{0,\odd}(M,E)
\quad \text{and} \quad
(\bar\partial_E^*)^{\odd} \colon \Omega^{0,\odd}(M,E)\to\Omega^{0,\even}(M,E)
\]
of the operators $\bar\partial_E$ and $\bar\partial_E^*$ and similarly $\bar\partial_E^{\odd}$ and $(\bar\partial_E^*)^{\even}$. Now, we prove that the operators $\D_{E,+,p}$ and $\D_{E,-,p}$ are the closure of the closable operators $\sqrt{2}(\bar\partial_E^{\even}+(\bar\partial_E^*)^{\even})$ and $\sqrt{2}(\bar\partial_E^{\odd}+(\bar\partial_E^*)^{\odd})$.

\begin{prop}
\label{prop:blocks-closure}
Let $E$ be a Hermitian holomorphic vector bundle of finite rank over a compact Hermitian manifold $M$ with complex dimension $n$. Define the operators
\begin{equation}
\label{def-DEplus}
\D_{E,+}
\ov{\mathrm{def}}{=} 
\sqrt{2}\big(\bar\partial_E^{\even}+(\bar\partial_E^*)^{\even}\big)
\co
\Omega^{0,\even}(M,E) \to \Omega^{0,\odd}(M,E),
\end{equation}
and
\begin{equation}
\label{def-DE-moins}
\D_{E,-}
\ov{\mathrm{def}}{=} 
\sqrt{2}\big(\bar\partial_E^{\odd}+(\bar\partial_E^*)^{\odd}\big)
\co
\Omega^{0,\odd}(M,E) \to \Omega^{0,\even}(M,E).
\end{equation}
Then $\D_{E,+}$ and $\D_{E,-}$ are closable on $\L^p(\Omega^{0,\even}(M,E))$ and $\L^p(\Omega^{0,\odd}(M,E))$, respectively, and their closures coincide with the off-diagonal parts $\D_{E,+,p}$ and $\D_{E,-,p}$ appearing in \eqref{Dp-decompo-block}.
\end{prop}

\begin{proof}
Consider the bounded projections $
P_{\even} 
\ov{\mathrm{def}}{=} \frac{\Id+\gamma}{2}$ and $
P_{\odd} 
\ov{\mathrm{def}}{=} \frac{\Id-\gamma}{2}$.  Then $\Ran P_{\even}=\L^p(\Omega^{0,\even}(M,E))$ and $\Ran P_{\odd}=\L^p(\Omega^{0,\odd}(M,E))$. Let $\omega \in \Omega^{0,\even}(M,E)$. Since $\bar\partial_E$ raises the form degree by one and $\bar\partial_E^*$ lowers it by one, we have $\bar\partial_E \omega \in \Omega^{0,\odd}(M,E)$ and $
\bar\partial_E^* \omega \in \Omega^{0,\odd}(M,E)$. 
Hence $\D_E\omega \in \Omega^{0,\odd}(M,E)$, and by definition of $\D_{E,+}$ we have $
\D_E\big|_{\Omega^{0,\even}(M,E)} = \D_{E,+}$. Similarly, on $\Omega^{0,\odd}(M,E)$ we have
\[
\D_E\big|_{\Omega^{0,\odd}(M,E)} = \D_{E,-},
\qquad
\D_E\omega \in \Omega^{0,\even}(M,E) \text{ for } \omega \in \Omega^{0,\odd}(M,E).
\]
We conclude that $\D_{E,+}$ and $\D_{E,-}$ are closable since they are restrictions of the closable operator $\D_{E}$.

Since $\D_{E,p}$ is the closure of $\D_E$ on $\L^p(\Omega^{0,\bullet}(M,E))$, the restriction of $\D_{E,p}$ to smooth forms coincides with $\D_E$. In particular, for $\omega \in \Omega^{0,\even}(M,E)$ we have
\[
\D_{E,p}\omega 
= \D_E\omega = \D_{E,+} \omega \in \Omega^{0,\odd}(M,E).
\]
Therefore $\D_{E,+} \subset \D_{E,+,p}$ as operators from $\L^p(\Omega^{0,\even}(M,E))$ to $\L^p(\Omega^{0,\odd}(M,E))$, where $\D_{E,+,p}$ is the off-diagonal part of $\D_{E,p}$ relative to $\L^p(\Omega^{0,\even}(M,E)) \oplus \L^p(\Omega^{0,\odd}(M,E))$. Since $\D_{E,+,p}$ is closed (as a restriction of the closed operator $\D_{E,p}$ on a closed subspace), we get $\ovl{\D_{E,+}} \subset \D_{E,+,p}$. The same argument yields $\overline{\D_{E,-}} \subset \D_{E,-,p}$.

Now, we prove the reverse inclusion. We again treat the $+$ case. Let $u \in \dom \D_{E,+,p}$. By definition, $u \in \L^p(\Omega^{0,\even}(M,E))$ and $u \in \dom \D_{E,p}$, and $\D_{E,p}u \in \L^p(\Omega^{0,\odd}(M,E))$ with $\D_{E,+,p}u = \D_{E,p}u$. Since $\Omega^{0,\bullet}(M,E)$ is a core for $\D_{E,p}$, there exists a sequence $(w_k)$ in $\Omega^{0,\bullet}(M,E)$ such that $
w_k \to u$ and $
\D_E w_k \to \D_{E,p}u$ in $\L^p(\Omega^{0,\bullet}(M,E))$.
Set $
u_k \ov{\mathrm{def}}{=} P_{\even} w_k$. Then $u_k \in \Omega^{0,\even}(M,E)$ and $u_k \to P_{\even}u = u$ in $\L^p(\Omega^{0,\even}(M,E))$ since $P_{\even}$ is bounded and $u$ is even.
Moreover, on smooth forms we have the identity 
$$
\D_E P_{\even}
\D_E\tfrac{\Id+\gamma}{2}
=
\tfrac{1}{2}(\D_E + \D_E\gamma)
=
\tfrac{1}{2}(\D_E - \gamma\D_E)
=
\tfrac{\Id-\gamma}{2}\D_E
=
P_{\odd}\D_E.
$$ 
Consequently, we have
\[
\D_E u_k
=
\D_E(P_{\even}w_k)
=
P_{\odd}(\D_E w_k).
\]
Passing to the limit and using boundedness of $P_{\odd}$, we obtain $
\D_E u_k \to P_{\odd}(\D_{E,p}u)$. Since $\D_{E,p}u \in \L^p(\Omega^{0,\odd}(M,E))$, we have $P_{\odd}(\D_{E,p}u)=\D_{E,p}u$. Thus $\D_E u_k \to \D_{E,p}u$ in the space $\L^p(\Omega^{0,\odd}(M,E))$. Hence $\D_E u_k = \D_{E,+} u_k$ because $u_k$ is even. So
\[
u_k \to u \text{ in } \L^p(\Omega^{0,\even}(M,E)),
\qquad
\D_{E,+}u_k \to \D_{E,p}u 
= \D_{E,+,p}u \text{ in } \L^p(\Omega^{0,\odd}(M,E)).
\]
This shows that $(u,\D_{E,+,p}u)$ belongs to the closure of the graph of $\D_{E,+}$, i.e.~$\D_{E,+,p} \subset \ovl{\D_{E,+}}$. We conclude that $\overline{\D_{E,+}} = \D_{E,+,p}$. The proof for the $-$ block is identical and yields $\overline{\D_{E,-}} = \D_{E,-,p}$.
\end{proof}

\section{$\mathrm{H}^\infty$ calculus of the Dolbeault--Dirac operator on a compact K\"ahler manifold}
\label{sec-Hinfty-calculus-Dolbeault}

\subsection{Some Sobolev spaces and actions of pseudo-differential operators}
\label{sec-Sobolev-spaces}
We start by recalling some information on connections.

\paragraph{Connections on vector bundles}
Let $E$ be a complex (or real) vector bundle over a differentiable manifold $M$. Following \cite[Definition 3.3.1 p.~95]{Nic21}, a (complex or real) connection (or covariant derivative) on $E$ is a $\mathbb{C}$-linear (or $\mathbb{R}$-linear) differential operator $\nabla \co \Omega^0(M,E) \to \Omega^1(M,E)$ that satisfies the Leibniz rule 
$$
\nabla(f \sigma) 
= \d f \ot \sigma + f \nabla \sigma
$$ 
for any $f \in \C^\infty(M)$ and any $\sigma \in \Omega^0(M,E)$. By \cite[Proposition 3.3.8 p.~101]{Nic21} and \cite[Theorem 12.57 p.~536]{Lee09}, a connection $\nabla$ has a natural extension to an operator $\nabla \co \Omega^{\bullet}(M,E) \to \Omega^{\bullet}(M,E)$ such that $\nabla|_{\Omega^{0}(M,E)}=\nabla$, $\nabla(\Omega^{k}(M,E)) \subset \Omega^{k+1}(M,E)$ satisfying the Leibniz rule
$$
\nabla(\alpha \wedge \beta)
=\d \alpha \wedge \beta + (-1)^k\alpha \wedge \nabla\beta, \quad \alpha \in \Omega^k(M), \beta \in \Omega^l(M,E).
$$
Following \cite[Definition 6.29 p.~177]{BlB13}, 
we say that a connection $\nabla \co \Omega^0(M,E) \to \Omega^1(M,E)$ on a Hermitian vector bundle $E$ is a Hermitian connection (or metric connection) if $
\d \la s_1,s_2 \ra
=\la \nabla(s_1),s_2 \ra + \la s_1,\nabla(s_2) \ra
$  
for any smooth sections $s_1,s_2 \co M \to E$. 

\paragraph{Tensor products of connections} Consider two vector bundles $E$ and $F$ over a compact manifold $M$. Let $\tau \co E \ot \mathrm{T}^*M \ot F \to \mathrm{T}^*M \ot E \ot F$ be the natural isomorphism permuting the first two factors. Consider some connections $\nabla^E$ and $\nabla^F$ on $E$ and $F$. Then we have a connection $\nabla^{E \ot F} \ov{\mathrm{def}}{=} \nabla^E \ot \Id_{\Omega^0(M,F)} + \tau \circ (\Id_{\Omega^0(M,E)} \ot \nabla^F)$ on te tensor product vector bundle $E \ot F$, which we shall write, by abuse of notation, in the form
\begin{equation}
\label{tensor-product-connections}
\nabla^{E \ot F} 
= \nabla^{E} \ot \Id + \Id \ot \nabla^{F}.
\end{equation}
Recall that the Levi-Civita connection $\nabla^\LC \co \Omega^0(M,\mathrm{T}M) \to \Omega^1(M,\mathrm{T}M)$ induces by \cite[p.~96]{Nic21} a connection $\nabla^{\mathrm{T}^*M}$ on the cotangent  bundle $\mathrm{T}^*M$. Now, we work with complexified bundles and keep the same notation. If $\mathrm{T}^{*\ot j}M \ov{\mathrm{def}}{=} (\mathrm{T}^*M)^{\ot j}$ denote the repeated tensor products of the complexified cotangent bundle $\mathrm{T}^*M$, we can introduce the notation $\nabla^0=\Id_{\mathrm{C}^\infty(M,E)}$ and for any integer $j \geq 1$ the iterated connection map
\begin{equation*}
\nabla^j \ov{\mathrm{def}}{=} \nabla^{\mathrm{T}^{* \ot (j-1)} M \ot E} \circ
\cdots \circ \nabla^{\mathrm{T}^{*} M \ot E} \circ \nabla^{E} \co \Omega^0(M,E) \to \C^\infty(M,\mathrm{T}^{*\ot j}M \ot E).
\end{equation*}
Observe that we have a canonical isomorphism $\Omega^{1}(M,F) \simeq \mathrm{T}^*M \ot \Omega^{0}(M,F)$.

\paragraph{Sobolev spaces via connections} Let $E$ be a Hermitian vector bundle of finite-rank over a compact Riemannian manifold $M$ equipped with a Hermitian connection $\nabla^E$. Suppose that $1 \leq p <\infty$. For any integer $k \geq 0$, we can introduce $\L^p$--type $\nabla$--Sobolev space $\W^{k,p}_{\nabla}(M,E)$ of order $k$, which is the completion of the space $\Omega^0(M,E)$ 
for the norm
\begin{equation}
\label{Sobolev-norm}
\norm{u}_{\W^{k, p}_{\nabla}(M, E)} 
\ov{\mathrm{def}}{=} \sum_{j=0}^{k}
\bigg(\norm{\nabla^j (u)}_{\L^p(M,\mathrm{T}^{* \ot j} M \ot E)}^p\bigg)^{\frac{1}{p}}.
\end{equation}
Furthermore, \cite[Theorem 10.2.36 p.~483]{Nic21} says that this space is independent up to isomorphism of the metric and the metric preserving connection. Note the isometric identification $\W_\nabla^{0,p}(M,E) = \L^p(M,E)$. 

\paragraph{Sobolev spaces via atlas}
Let $E$ be a Hermitian vector bundle of rank $r$ over a smooth compact Riemannian manifold $(M,g)$ of dimension $n$. A (bundle) trivialization of $E$ is a family $
\cal{T}_E
=(U_i,\varphi_i,\xi_i,h_i)_{i \in I}$, 
where $(U_i)_{i \in I}$ is an open finite cover of $M$, $\varphi_i \co U_i \to V_i \subset \R^n$ are charts, $h_i$ is a smooth partition of unity subordinate to $(U_i)$, and $\xi_i \co U_i \times \mathbb{C}^r \to E|_{U_i}$ is a local trivialization over $U_i$. Given a section $u$ of $E$, we denote by $(\xi_i)^*(u)$ its local representation as an $\mathbb{C}^r$-valued function (or distribution) on $V_i$. As usual, $(\xi_i)^*(h_i u)$ is extended by $0$ on $\R^n$.

Let $s \in \R$. Suppose that $1 \leq p < \infty$. Following \cite[Definition 42]{GrS13}, given a finite bundle trivialization $\cal{T}_E=(U_i,\varphi_i,\xi_i,h_i)_{i \in I}$, define $\W^{s,p}_{\cal{T}_E}(M,E)$ as the completion of $\C^\infty(M,E)$ for the norm
\[
\norm{u}_{\W^{s,p}_{\cal{T}_E}(E)}
\ov{\mathrm{def}}{=}
\Big( \sum_{i \in I}
\norm{ (\xi_i)^*(h_i u) }_{\W^{s,p}(\R^n,\mathbb{C}^r)}^p
\Big)^{\frac{1}{p}}
<\infty.
\]
By essentially \cite[Theorem 44]{GrS13} or implicitly in \cite[p.~483]{Nic21}, for any trivialization $\cal{T}_E$ of $E$ and any integer $k \geq 0$ we have
\begin{equation}
\label{iso-Sobolev}
\W^{k,p}_{\cal{T}_E}(M,E)
=\W^{k,p}_\nabla(M,E),
\end{equation}
with equivalence of norms. If $k_0 > k_1 \geq 0$ the embedding 
$
\W^{k_0,p}_{\cal{T}_E}(M,E) \xhookrightarrow{} \W^{k_1,p}_{\cal{T}_E}(M,E)$ is compact. With the identification \eqref{iso-Sobolev}, this implies the following important point for us in Section \ref{sec-compact-spectral-triple}, which is described in \cite[Theorem 10.2.36 p.~483]{Nic21} (without proof). If $k_0 > k_1 \geq 0$ the embedding
\begin{equation}
\label{Sobolev-compact-embedding}
\W^{k_0,p}_\nabla(M,E) \xhookrightarrow{} \W^{k_1,p}_\nabla(M,E) \text{ is compact}. 
\end{equation}
If $E$ is a Hermitian holomorphic vector bundle over a compact Hermitian manifold $M$, we will exploit this property to establish that the spectral triple $(\C(M), \L^p(\Omega^{0,\bullet}(M,E)), \D_{E,p})$ is ``compact", i.e.~that the closed operator $\D_{E,p}$ has compact resolvent. The goal of this section is to show that the Sobolev space $\W^{1,p}_\nabla(\Omega^{0,q}(M,E))$ is isomorphic to the domain $\W^{1,p}_{\bar\partial_E}(\Omega^{0,q}(M,E))$ of the Dolbeault--Dirac operator $\D_{E,p}$ in order to use the compact Sobolev embedding \eqref{Sobolev-compact-embedding}. We refer to \cite{Nic10}, \cite[Section 10.2.4]{Nic21}, \cite{GrS13}, \cite{Gun17}, \cite{KoV22}, \cite{Shu92} and \cite[Chapter 7]{BlB13} (for the case $p=2$) for more information on the spaces $\W^{k, p}_{\nabla}(M,E)$.  

\paragraph{Action of pseudo-differential operators} The following results describe the action of pseudo-differential operators on Sobolev spaces of vector bundles. It generalizes \cite[Theorem 5.2.23 p.~420]{RuT10}, \cite[p.~267]{Tay81} and \cite[Theorem 8.5 p.~206]{Gru09}.

\begin{prop}
\label{prop-elliptic-Sobolev}
Let $E$ and $F$ be Hermitian smooth complex vector bundles of finite rank over a compact smooth Riemannian manifold $M$. Suppose that $1 < p < \infty$. A pseudo-differential operator $P \co \C^\infty(M,E) \to \C^\infty(M,F)$ of order $m$ induces a bounded operator $P \co \W^{k,p}_\nabla(M,E) \to \W^{k-m,p}_\nabla(M,F)$ if $m \leq k$.
\end{prop}

\begin{proof}
Using the identification \eqref{iso-Sobolev}, it suffices to prove that $P$ induces a bounded operator $P \co \W^{k,p}_{\cal{T}_E}(M,E) \to \W^{k-m,p}_{\cal{T}_F}(M,F)$ for some $\cal{T}_E =(U_i,\varphi_i,\xi_i,h_i)_{i \in I}$ and $\cal{T}_F=(U_j',\psi_j,\eta_j,g_j)_{j \in J}$. For any $u \in \C^\infty(M,E)$, we can write 
$$
P u
=\sum_{j \in J} g_j P\big(\sum_{i \in I} h_i u\big)
=\sum_{i \in I,j \in J} g_j P(h_i u)
$$  
and
$
(\eta_j)^*(g_j P(h_i u)) 
= P_{j i}\big( (\xi_i)^*(h_i u)\big)$ 
for some pseudo-differential operators $P_{j i}$. By the standard boundedness \cite[Theorem 12.9 p.~91]{Won14} of pseudo-differential operators on $\R^n$ (and a finite-dimensional vector-valued argument), for each pair $(j,i)$ there exists a constant $C_{j i} \geq 0$ such that
\[
\norm{P_{j i} f}_{\W^{k-m,p}(\R^n,\mathbb{C}^{r_F})}
\leq C_{j i}  \norm{f}_{\W^{k,p}(\R^n,\mathbb{C}^{r_E})}, \quad f \in \C_c^\infty(\R^n,\mathbb{C}^{r_E}).
\]
Hence, for any $u \in \C^\infty(M,E)$, we obtain
\[
\norm{(\eta_j)^*(g_j P u)}_{\W^{k-m,p}(\R^n,\mathbb{C}^{r_F})}
\leq
\sum_{i \in I} C_{j i}\norm{(\xi_i)^*(h_i u)}_{\W^{k,p}(\R^n,\mathbb{C}^{r_E})}.
\]
Taking the $\ell^p$-sum over $j \in J$, we deduce that 
$$
\norm{P u}_{\W^{k-m,p}_{\cal{T}_F}(M,F)}
\leq C \norm{u}_{\W^{k,p}_{\cal{T}_E}(M,E)}
$$ 
for some constant $C$ depending on $P$, $\cal{T}_E$, and $\cal{T}_F$. By density of the subspace $\C^\infty(M,E)$ in the space $\W^{k,p}_{\cal{T}_E}(M,E)$ and completeness, $P$ extends uniquely and boundedly from $\W^{k,p}_{\cal{T}_E}(M,E)$ with values in $\W^{k-m,p}_{\cal{T}_F}(M,F)$.
\end{proof}

Using \cite[Theorem 5.28 p.~168]{Kat76} and the symmetry of the Kodaira Laplacian $\Delta_{\bar\partial_E}$, we see that the unbounded operators $\Delta_{\bar\partial_E}$ and are closable on the Banach spaces $\L^p(\Omega^{\bullet,\bullet}(M,E))$ and $\L^p(\Omega^{0,\bullet}(M,E))$, whose closures are denoted by $\Delta_{\bar\partial_E,p}$ and $\Delta_{\bar\partial_E,0,\bullet,p}$. We need the following lemma for the proof of Proposition \ref{prop:closure-kodaira-square-dirac}.

\begin{lemma}
\label{lemma-I-plus-Kodaira-isomorphism-W2p}
Let $E$ be a Hermitian holomorphic vector bundle of finite rank over a compact Hermitian manifold $M$. Suppose that $1 < p < \infty$. Then
\begin{equation}
\label{eq-domain-Kodaira-W2p}
\dom\Delta_{\bar\partial_E,p}
=
\W^{2,p}_\nabla(\Omega^{\bullet,\bullet}(M,E)).
\end{equation}
%
Moreover, the operator
\begin{equation}
\label{eq-I-plus-Kodaira-closure-isomorphism}
\Id+\Delta_{\bar\partial_E,p}
\co
\W^{2,p}_\nabla(\Omega^{\bullet,\bullet}(M,E))
\to
\L^p(\Omega^{\bullet,\bullet}(M,E))
\end{equation}
is an isomorphism of Banach spaces. Finally, $-1 \in \rho(\Delta_{\bar\partial_E,p})$ and the graph norm of the operator $\Delta_{\bar\partial_E,p}$ is equivalent to the
Sobolev norm, i.e.
\begin{equation}
\label{eq-Kodaira-W2p-graph-norm}
\norm{u}_{\W^{2,p}_\nabla(\Omega^{\bullet,\bullet}(M,E))}
\approx
\norm{u}_{\L^p(\Omega^{\bullet,\bullet}(M,E))}
+
\norm{\Delta_{\bar\partial_E,p}u}
      _{\L^p(\Omega^{\bullet,\bullet}(M,E))},
\qquad
u\in\dom\Delta_{\bar\partial_E,p}.
\end{equation}
The same conclusions hold for the restriction $\Delta_{\bar\partial_E,0,\bullet}$ acting on $\Omega^{0,\bullet}(M,E)$.
\end{lemma}

\begin{proof}
Note that the operator $\Id+\Delta_{\bar\partial_E}$ is an elliptic differential operator of order $2$. Using Proposition \ref{prop-elliptic-Sobolev}, we obtain a well-defined bounded operator $\Id+\Delta_{\bar\partial_E}
\co
\W^{2,p}_\nabla(\Omega^{\bullet,\bullet}(M,E))
\to
\L^p(\Omega^{\bullet,\bullet}(M,E))$. We first prove that $\Id+\Delta_{\bar\partial_E}$ is injective. Consider some $u \in \W^{2,p}_\nabla(\Omega^{\bullet,\bullet}(M,E))$ such that $(\Id+\Delta_{\bar\partial_E})u=0$. By elliptic regularity \cite[(23.22.8)]{Dieu88}, 
$u$ is smooth. Taking the $\L^2$ inner product with $u$, we get
\begin{align*}
0
&=
\big\la(\Id+\Delta_{\bar\partial_E})u,u
\big\ra_{\L^2(\Omega^{\bullet,\bullet}(M,E))}
=
\norm{u}_{\L^2(\Omega^{\bullet,\bullet}(M,E))}^2
+
\big\la\Delta_{\bar\partial_E}u,u
\big\ra_{\L^2(\Omega^{\bullet,\bullet}(M,E))}
\\
&\ov{\eqref{holo-Laplacian}}{=}
\norm{u}_{\L^2(\Omega^{\bullet,\bullet}(M,E))}^2
+
\norm{\bar\partial_Eu}
      _{\L^2(\Omega^{\bullet,\bullet}(M,E))}^2
+
\norm{\bar\partial_E^*u}
      _{\L^2(\Omega^{\bullet,\bullet}(M,E))}^2.
\end{align*}
Consequently, $u=0$. Hence the operator $\Id+\Delta_{\bar\partial_E}$ is injective.

Now, we prove that $\Id+\Delta_{\bar\partial_E}$ has a closed range. By \cite[(23.30.6)]{Dieu88} and \cite[Theorem 8.6 p.~207]{Gru09}, there exists a pseudo-differential operator $Q \co \Omega^{\bullet,\bullet}(M,E) \to \Omega^{\bullet,\bullet}(M,E)$ of order $-2$ and a smoothing operator $R \co \Omega^{\bullet,\bullet}(M,E) \to \Omega^{\bullet,\bullet}(M,E)$ such that
\begin{equation}
\label{eq:parametrix-elliptic-bis}
\Id 
= Q(\Id+\Delta_{\bar\partial_E}) + R.
\end{equation}
Hence for any $u \in \W^{2,p}_\nabla(\Omega^{\bullet,\bullet}(M,E))$ we have using Proposition \ref{prop-elliptic-Sobolev} and following \cite[p.~193]{LaM89}.
\begin{align}
\MoveEqLeft
\label{eq:elliptic-estimate-P}
\norm{u}_{\W^{2,p}_\nabla(\Omega^{\bullet,\bullet}(M,E))}
\ov{\eqref{eq:parametrix-elliptic-bis}}{=} \norm{Q (\Id+\Delta_{\bar\partial_E}) u+Ru}_{\W^{2,p}_\nabla(\Omega^{\bullet,\bullet}(M,E))} \\
&\lesssim \norm{(\Id+\Delta_{\bar\partial_E}) u}_{\L^p}+\norm{Ru}_{\W^{2,p}_\nabla(\Omega^{\bullet,\bullet}(M,E))}
\lesssim \norm{(\Id+\Delta_{\bar\partial_E}) u}_{\L^p}+\norm{u}_{\L^p}.  \nonumber       
\end{align}
We claim that there exists a constant \(C>0\) such that
\begin{equation}
\label{eq-coercive-estimate-I-plus-Kodaira}
\norm{u}_{\W^{2,p}_\nabla(\Omega^{\bullet,\bullet}(M,E))}
\leq
C
\norm{(\Id+\Delta_{\bar\partial_E})u}_{\L^p(\Omega^{\bullet,\bullet}(M,E))},
\quad
u\in
\W^{2,p}_\nabla(\Omega^{\bullet,\bullet}(M,E)).
\end{equation}
Suppose that it is not the case. There exists a sequence $(u_j)$ in $\W^{2,p}_\nabla(\Omega^{\bullet,\bullet}(M,E))$ such that
\[
\norm{u_j}_{\W^{2,p}_\nabla(\Omega^{\bullet,\bullet}(M,E))}=1
\quad \text{and} \quad
\norm{(\Id+\Delta_{\bar\partial_E})u_j}_{\L^p(\Omega^{\bullet,\bullet}(M,E))} \to 0.
\]
By compactness \eqref{Sobolev-compact-embedding} of the embedding $\W^{2,p}_\nabla(\Omega^{\bullet,\bullet}(M,E)) \hookrightarrow \L^p(\Omega^{\bullet,\bullet}(M,E))$, after passing to a subsequence, we may assume that $(u_j)$ converges in $\L^p(\Omega^{\bullet,\bullet}(M,E))$ to some $u$. Using \eqref{eq:elliptic-estimate-P} applied to $u_j-u_k$, we infer that the sequence $(u_j)$ is Cauchy in $\W^{2,p}_\nabla(\Omega^{\bullet,\bullet}(M,E))$. Since the space $\W^{2,p}_\nabla(\Omega^{\bullet,\bullet}(M,E))$ is complete, there exists $v \in \W^{2,p}_\nabla(\Omega^{\bullet,\bullet}(M,E))$ such that $
u_j \to v
\quad \text{in } \W^{2,p}_\nabla(\Omega^{\bullet,\bullet}(M,E))$. By the continuous embedding $
\W^{2,p}_\nabla(\Omega^{\bullet,\bullet}(M,E))
\hookrightarrow
\L^p(\Omega^{\bullet,\bullet}(M,E))$, we also have $u_j \to v$ in $\L^p(\Omega^{\bullet,\bullet}(M,E))$. On the other hand, by construction of the subsequence,  $u_j \to u$ in $\L^p(\Omega^{\bullet,\bullet}(M,E))$. By uniqueness of the limit in $\L^p(\Omega^{\bullet,\bullet}(M,E))$, we obtain $u=v$. Thus $u_j \to u$ in $\W^{2,p}_\nabla(\Omega^{\bullet,\bullet}(M,E))$. Since $(\Id+\Delta_{\bar\partial_E})u_j \to 0$ in $\L^p(\Omega^{\bullet,\bullet}(M,E))$, we get $(\Id+\Delta_{\bar\partial_E})u=0$. By injectivity, $u=0$, contradicting
\[
\norm{u}_{\W^{2,p}_\nabla(\Omega^{\bullet,\bullet}(M,E))}
=
\lim_j \norm{u_j}_{\W^{2,p}_\nabla(\Omega^{\bullet,\bullet}(M,E))}
=
1.
\]
This proves \eqref{eq-coercive-estimate-I-plus-Kodaira}. Consequently, by \cite[Theorem 2.5 p.~70]{AbA02}, the subspace $\Ran (\Id+\Delta_{\bar\partial_E})$ is closed in the space $\L^p(\Omega^{\bullet,\bullet}(M,E))$.

It remains to prove that the subspace $\Ran (\Id+\Delta_{\bar\partial_E})$ is dense. Let $g \in \L^{p^*}(\Omega^{\bullet,\bullet}(M,E))$ annihilate $\Ran (\Id+\Delta_{\bar\partial_E})$, where $p^*$ is the conjugate exponent of $p$. Then
\[
\big\la g,(\Id+\Delta_{\bar\partial_E})u \big\ra_{\L^{p^*},\L^p}
=0,
\quad
u \in \W^{2,p}_\nabla(\Omega^{\bullet,\bullet}(M,E)).
\]
In particular, this holds for every smooth form $u$. Since $\Id+\Delta_{\bar\partial_E}$ is formally selfadjoint,  we have, by definition of the distribution $(\Id+\Delta_{\bar\partial_E})g$,
\[
\big\la (\Id+\Delta_{\bar\partial_E})g,u\big\ra_{\cal{D}',\cal{D}}
=
\big\la g,(\Id+\Delta_{\bar\partial_E})u\big\ra_{\cal{D}',\cal{D}}
=
\big\la g, (\Id+\Delta_{\bar\partial_E})u \big\ra_{\L^{p^*},\L^p}
=0
\]
for every $u\in\Omega^{\bullet,\bullet}(M,E)$. 
Hence $
(\Id+\Delta_{\bar\partial_E})g=0$ distributionally. By elliptic regularity of $\Id+\Delta_{\bar\partial_E}$ \cite[(23.22.8)]{Dieu88}, $g$ is smooth. Taking the $\L^2$ inner product with $g$, the same positivity argument as previously gives $g=0$. Thus the annihilator of $\Ran (\Id+\Delta_{\bar\partial_E})$ in $\L^{p^*}(\Omega^{\bullet,\bullet}(M,E))$ is trivial. Hence the subspace $\Ran (\Id+\Delta_{\bar\partial_E})$ is dense in $\L^p(\Omega^{\bullet,\bullet}(M,E))$.

Since $\Ran (\Id+\Delta_{\bar\partial_E})$ is both closed and dense in $\L^p(\Omega^{\bullet,\bullet}(M,E))$, it is equal to $\L^p(\Omega^{\bullet,\bullet}(M,E))$. Consequently $\Id+\Delta_{\bar\partial_E}$ is bijective. Since it is a bounded bijection between Banach spaces, its inverse is bounded by \cite[Corollary 1.6.6 p.~44]{Meg98}.

Let $\Delta_{\bar\partial_E,p}$ be the closure of the unbounded operator $\Delta_{\bar\partial_E}$. We first prove that $
\W^{2,p}_\nabla(\Omega^{\bullet,\bullet}(M,E))
\subset
\dom\Delta_{\bar\partial_E,p}$. Let \(u\in\W^{2,p}_\nabla(\Omega^{\bullet,\bullet}(M,E))\). Choose a sequence \((u_j)\) in \(\Omega^{\bullet,\bullet}(M,E)\) such that $u_j\to u$ in \(\W^{2,p}_\nabla(\Omega^{\bullet,\bullet}(M,E))\). By
Proposition~\ref{prop-elliptic-Sobolev}, we have $
\Delta_{\bar\partial_E}u_j
\to
\Delta_{\bar\partial_E}u$ in \(\L^p(\Omega^{\bullet,\bullet}(M,E))\). Hence $u \in \dom \Delta_{\bar\partial_E,p}$ and $
\Delta_{\bar\partial_E,p}u
=
\Delta_{\bar\partial_E}u$.

Conversely, let $u \in \dom\Delta_{\bar\partial_E,p}$. By the definition of the closure, there exists a sequence \((u_j)\) in $\Omega^{\bullet,\bullet}(M,E)$ such that $u_j \to u$ and $\Delta_{\bar\partial_E}u_j
\to
\Delta_{\bar\partial_E,p}u$ in $\L^p(\Omega^{\bullet,\bullet}(M,E))$. Consequently, we have 
$(\Id+\Delta_{\bar\partial_E})u_j
\to
u+\Delta_{\bar\partial_E,p}u$ in \(\L^p(\Omega^{\bullet,\bullet}(M,E))\). Applying the bounded operator $S  \ov{\mathrm{def}}{=}
(\Id+\Delta_{\bar\partial_E})^{-1}
\co
\L^p(\Omega^{\bullet,\bullet}(M,E))
\to
\W^{2,p}_\nabla(\Omega^{\bullet,\bullet}(M,E))$, we obtain
\[
u_j
=
S(\Id+\Delta_{\bar\partial_E})u_j
\to
S(u+\Delta_{\bar\partial_E,p}u)
\]
in $\W^{2,p}_\nabla(\Omega^{\bullet,\bullet}(M,E))$. Since the embedding of $\W^{2,p}_\nabla(\Omega^{\bullet,\bullet}(M,E))$ into $\L^p(\Omega^{\bullet,\bullet}(M,E))$ is continuous and $u_j \to u$ in $\L^p(\Omega^{\bullet,\bullet}(M,E))$, the uniqueness of the limit gives $
u
=S(u+\Delta_{\bar\partial_E,p}u)$. Thus $
u\in\W^{2,p}_\nabla(\Omega^{\bullet,\bullet}(M,E))$. This proves \eqref{eq-domain-Kodaira-W2p}.

It also shows that $\Id+\Delta_{\bar\partial_E,p}$ coincides, under the identification of its domain with the Sobolev space $\W^{2,p}_\nabla(\Omega^{\bullet,\bullet}(M,E))$, with the isomorphism $\Id+\Delta_{\bar\partial_E}$ constructed previously. Hence \eqref{eq-I-plus-Kodaira-closure-isomorphism} holds and $-1 \in \rho(\Delta_{\bar\partial_E,p})$.

Finally, Proposition~\ref{prop-elliptic-Sobolev} gives
\[
\norm{u}_{\L^p(\Omega^{\bullet,\bullet}(M,E))}
+
\norm{\Delta_{\bar\partial_E,p}u}_{\L^p(\Omega^{\bullet,\bullet}(M,E))}
\lesssim
\norm{u}_{\W^{2,p}_\nabla(\Omega^{\bullet,\bullet}(M,E))},
\]
whereas \eqref{eq-coercive-estimate-I-plus-Kodaira} gives
\begin{align*}
\norm{u}_{\W^{2,p}_\nabla(\Omega^{\bullet,\bullet}(M,E))}
&\ov{\eqref{eq-coercive-estimate-I-plus-Kodaira}}{\lesssim}
\norm{(\Id+\Delta_{\bar\partial_E,p})u}_{\L^p(\Omega^{\bullet,\bullet}(M,E))}
\leq
\norm{u}_{\L^p(\Omega^{\bullet,\bullet}(M,E))}
+
\norm{\Delta_{\bar\partial_E,p}u}_{\L^p(\Omega^{\bullet,\bullet}(M,E))}.
\end{align*}
This proves \eqref{eq-Kodaira-W2p-graph-norm}.
\end{proof}

\begin{prop}
\label{prop:closure-kodaira-square-dirac}
Let \(E\) be a Hermitian holomorphic vector bundle of finite rank over a
compact K\"ahler manifold \(M\). Suppose that \(1<p<\infty\). Assume that
the operator \(\D_{E,p}\) is bisectorial. Then
\begin{equation}
\label{useful-sans-fin}
\Delta_{\bar\partial_E,0,\bullet,p}
=
\frac12\D_{E,p}^2.
\end{equation}
\end{prop}

\begin{proof}
We have $\D_E^2=2\Delta_{\bar\partial_E}$. Since $\D_{E,p}$ is bisectorial, the operator $\D_{E,p}^2$ is sectorial by \cite[Proposition 10.6.2 (2) p.~448]{HvNVW18}, hence closed. So $\frac12\D_{E,p}^2$ is a closed extension of the operator $\Delta_{\bar\partial_E}$. Consequently, we have $\Delta_{\bar\partial_E,0,\bullet,p} \subset \frac12\D_{E,p}^2$. 

By Lemma~\ref{lemma-I-plus-Kodaira-isomorphism-W2p}, the operator $\Id_{\L^p(\Omega^{0,\bullet}(M,E))}+\Delta_{\bar\partial_E,0,\bullet,p} \co \dom \Delta_{\bar\partial_E,0,\bullet,p} \to \L^p(\Omega^{0,\bullet}(M,E))$ is surjective. On the other hand, $\frac12\D_{E,p}^2$ is sectorial, and therefore $-1 \in \rho(\frac12\D_{E,p}^2)$. In particular,
\[
\Id_{\L^p(\Omega^{0,\bullet}(M,E))} + \tfrac12\D_{E,p}^2 \co \dom \tfrac12\D_{E,p}^2 \to \L^p(\Omega^{0,\bullet}(M,E))
\]
is injective. Applying Lemma~\ref{lemma-extension-surjective-injective} with $\lambda=1$, $
A
\ov{\mathrm{def}}{=}
\Delta_{\bar\partial_E,0,\bullet,p}$ and $B
\ov{\mathrm{def}}{=}
\frac12\D_{E,p}^2$, we obtain \eqref{useful-sans-fin}.
\end{proof}

\subsection{An $\L^p$-Gaffney-type inequality for the operator $\bar\partial_E$}
\label{sec-Gaffney}

If $M$ is a compact Riemannian manifold of dimension $d$, with exterior derivative $\d \co \Omega^k(M) \to \Omega^{k+1}(M)$ where $0 \leq k \leq d$, recall that the classical Gaffney inequality \cite{Gaf51} \cite{Gaf55} 
\begin{equation*}
\norm{\omega}_{\W^{1,2}(\Omega^k(M))} 
\lesssim \norm{\omega}_{\L^2(\Omega^k(M))} + \norm{\d \omega}_{\L^2(\Omega^{k+1}(M))} + \norm{\d^* \omega}_{\L^2(\Omega^{k-1}(M))}
\end{equation*}
holds for any smooth $k$-differential form $\omega \in \Omega^k(M)$. We refer to the introduction of \cite{Li22}, \cite{CDK12}, \cite{CDS18}, \cite{Mor66} and \cite[p.~66]{Sch95} and references therein for more information and extensions. This inequality was extended to $\L^p$-spaces for $1 \leq p < \infty$ by Scott in \cite{Sco95} (although stated differently; see also \cite{ISS99}) which provides the extension
\begin{equation}
\label{Gaffney-Lp}
\norm{\omega}_{\W^{1,p}(\Omega^k(M))} 
\lesssim \norm{\omega}_{\L^p(\Omega^k(M))}+\norm{\d \omega}_{\L^p(\Omega^{k+1}(M))} + \norm{\d^* \omega}_{\L^p(\Omega^{k-1}(M))}.
\end{equation}
We also refer to \cite[Theorem A p.~1430]{Li22} for a recent proof.

We prove a complex analogue of this inequality. Our approach relies on the following observation.

\begin{thm}
\label{thm:gaffney-elliptic-Lp}
Consider a connection $\nabla \co \Omega^0(M,E) \to \Omega^1(M,E)$ on a Hermitian vector bundle $E$ of finite rank over a compact Riemannian manifold $M$. Let $D \co \Omega^{0}(M,E) \to \Omega^{0}(M,E)$ be a first order elliptic differential operator. Suppose that $1 < p < \infty$. Then for any $u \in \Omega^{0}(M,E)$, we have the estimate
\begin{equation}
\label{eq:order0-bound-elliptic}
\norm{\nabla u}_{\L^p(\Omega^{1}(M,E))}
\lesssim
\norm{u}_{\L^p(M,E)}+\norm{Du}_{\L^p(M,E)}.
\end{equation}
\end{thm}

\begin{proof}
Since $D$ is a first order elliptic differential operator on the compact manifold $M$, there exists, by \cite[(23.30.6)]{Dieu88} and \cite[Theorem 8.6 p.~207]{Gru09}, a parametrix, i.e., a pseudo-differential operator $Q \co \Omega^{0}(M,E) \to \Omega^{0}(M,E)$ of order $-1$ and smoothing operators $R_1,R_2 \co \Omega^{0}(M,E) \to \Omega^{0}(M,E)$ such that
\begin{equation}
\label{eq:parametrix-elliptic}
QD
= \Id - R_1
\quad \text{and} \quad
DQ 
= \Id - R_2.
\end{equation}
Fix $u \in \Omega^{0}(M,E)$. From the first identity in \eqref{eq:parametrix-elliptic} we obtain
$u
= Q(Du)+R_1u$. Applying the connection $\nabla \co \Omega^0(M,E) \to \Omega^1(M,E)$ gives
\begin{equation}
\label{eq:nabla-splitting-elliptic}
\nabla u
=
(\nabla \circ Q)(Du)
+(\nabla \circ R_1)u.
\end{equation}
Since $\nabla$ is a differential operator of order 1 by \cite[Remark (a) p.~70]{Wel08}, the operator $\nabla \circ Q$ is a pseudo-differential operator of order $0$ by composition \cite[17.13.5 p.~298]{Dieu72}. Hence it extends by Proposition \ref{prop-elliptic-Sobolev} 
to a bounded operator $
\L^p(M,E) \to
\L^p(\Omega^{1}(M,E))$. Consequently, we have the estimate
\begin{equation}
\label{eq:order0-bound-elliptic-aux}
\norm{(\nabla \circ Q)(Du)}_{\L^p(\Omega^{1}(M,E))}
\lesssim
\norm{Du}_{\L^p(M,E)}.
\end{equation}
Moreover, the operator $\nabla \circ R_1$ is smoothing since a composition is smoothing when one of the operators is
smoothing by \cite[p.~261]{Dieu88}. Hence it extends to a bounded operator $
\L^p(M,E)
\to
\L^p\big(\Omega^{1}(M,E)\big)$. In particular, we have
\begin{equation}
\label{eq:smoothing-bound-elliptic}
\norm{(\nabla \circ R_1)u}_{\L^p(\Omega^{1}(M,E))}
\lesssim
\norm{u}_{\L^p(M,E)}.
\end{equation}
Combining \eqref{eq:nabla-splitting-elliptic}, \eqref{eq:order0-bound-elliptic-aux} and \eqref{eq:smoothing-bound-elliptic} yields 
$$
\norm{\nabla u}_{\L^p(\Omega^{1}(M,E))}
\lesssim
\norm{Du}_{\L^p(M,E)}+\norm{u}_{\L^p(M,E)}.
$$
This proves \eqref{eq:order0-bound-elliptic}.
\end{proof}

\begin{remark} \normalfont
If the connection $\nabla \co \Omega^0(M,E) \to \Omega^1(M,E)$ is Hermitian, this result implies the estimate $
\norm{u}_{\W_\nabla^{1,p}(M,E)}
\lesssim \norm{u}_{\W^{\{D\},p}(M,E)}$ for any $u \in \Omega^0(M,E)$, 
where $\W^{\{D\},p}(M,E)$ is the $\{D\}$-Sobolev space of $\L^p$-sections of \cite[Definition I.18 p.~29]{Gun17} (the original notation is $\Gamma_{\W^{\{D\},p}}(M,E)$) defined by the norm 
$$
\norm{u}_{\W^{\{D\},p}(M,E)} 
\ov{\mathrm{def}}{=} \big(\norm{u}_{\L^p(M,E)}^p+\norm{Du}_{\L^p(M,E)}^p\big)^{\frac{1}{p}}.
$$ 
\end{remark}

We deduce an $\L^p$-Gaffney inequality for the operator $\bar\partial_E$.

\begin{cor}
\label{cor:gaffney-dolbeault-Lp}
Let $E$ be a Hermitian holomorphic vector bundle of finite rank over a compact Hermitian manifold $M$ of complex dimension $n$. Suppose that $1 < p < \infty$ and let $q \in \{0,\ldots,n\}$. Consider a connection $\nabla$ on the vector bundle $\Lambda^{0,\bullet}\mathrm{T}^*M \ot E$. For any form $\omega \in \Omega^{0,q}(M,E)$, we have the estimate
\begin{equation}
\label{eq:order0-bound-Dolbeault}
\norm{\nabla \omega}_{\L^p(\Omega^{1}(M,\Lambda^{0,q}\mathrm{T}^*M \ot E))}
\lesssim \norm{\omega}_{\L^p(\Omega^{0,q}(M,E))}
+\norm{\bar\partial_E \omega}_{\L^p(\Omega^{0,q+1}(M,E))}
+\norm{\bar\partial_E^*\omega}_{\L^p(\Omega^{0,q-1}(M,E))}.
\end{equation}
\end{cor}

\begin{proof}
Using Theorem \ref{thm:gaffney-elliptic-Lp} with the vector bundle $\Lambda^{0,\bullet}\mathrm{T}^*M \ot E$ instead of $E$ and the operator $\D_E$ in place of $D$, we obtain for any form $\omega \in \Omega^{0,q}(M,E)$ the estimate
\begin{align*}
\MoveEqLeft
\norm{\nabla \omega}_{\L^p(\Omega^{1}(M,\Lambda^{0,q}\mathrm{T}^*M \ot E))}
\ov{\eqref{eq:order0-bound-elliptic}}{\lesssim} \norm{\omega}_{\L^p(\Omega^{0,q}(M,E))} + \norm{\D_E \omega}_{\L^p(\Omega^{0,\bullet}(M,E))} \\       
&\ov{\eqref{Def-Dirac-DE}}{=}
\norm{\omega}_{\L^p(\Omega^{0,q}(M,E))}+\sqrt{2}\norm{\bar\partial_E \omega+\bar\partial_E^*\omega}_{\L^p(\Omega^{0,\bullet}(M,E))} \\
&\lesssim
\norm{\omega}_{\L^p(\Omega^{0,q}(M,E))}
+\norm{\bar\partial_E \omega}_{\L^p(\Omega^{0,q+1}(M,E))}
+\norm{\bar\partial_E^*\omega}_{\L^p(\Omega^{0,q-1}(M,E))}.
\end{align*}
\end{proof}

Now, we introduce convenient Sobolev spaces adapted to the Dolbeault--Dirac operator. More precisely, if $E$ is a Hermitian holomorphic vector bundle over a compact Hermitian manifold $M$ and if $1 \leq p < \infty$, we introduce the completion $\W^{1,p}_{\bar\partial_E}(\Omega^{r,q}(M,E))$ of the space $\Omega^{r,q}(M,E)$ equipped with the norm
\begin{equation}
\label{eq:def-W1p-dolbeault-norm}
\norm{\omega}_{\W^{1,p}_{\bar\partial_E}(\Omega^{r,q}(M,E))}
\ov{\mathrm{def}}{=}
\norm{\omega}_{\L^p(\Omega^{r,q}(M,E))}
+\norm{\bar\partial_E\omega}_{\L^p(\Omega^{r,q+1}(M,E))}
+\norm{\bar\partial_E^*\omega}_{\L^p(\Omega^{r,q-1}(M,E))}.
\end{equation}
If $p=2$, this space was investigated in the paper \cite[Section 4]{Pio23}, under the notation $\W^{1,2}_2$. Since the norm \eqref{eq:def-W1p-dolbeault-norm} dominates the $\L^p$-norm, the identity map on $\Omega^{r,q}(M,E)$ extends to a continuous injective map $\W^{1,p}_{\bar\partial_E}(\Omega^{r,q}(M,E)) \hookrightarrow \L^p(\Omega^{r,q}(M,E))$.  In the sequel, we identify $\W^{1,p}_{\bar\partial_E}(\Omega^{r,q}(M,E))$ with its image in $\L^p(\Omega^{r,q}(M,E))$. Under this identification, the space coincides with a space studied in \cite[Chapter I, Section 3]{Gun17}.

Now, we obtain the following identification result.

\begin{cor}
\label{cor:Gaffney-dolbeault-Lp}
Let $E$ be a Hermitian holomorphic vector bundle of finite rank over a compact Hermitian manifold $M$. Suppose that $1 < p < \infty$.  Consider a Hermitian connection $\nabla$ on the vector bundle $\Lambda^{0,\bullet}\mathrm{T}^*M \ot E$. The Sobolev space $\W^{1,p}_{\bar\partial_E}(\Omega^{0,q}(M,E))$ coincides with the Sobolev space $\W^{1,p}_\nabla(\Omega^{0,q}(M,E))$, with equivalent norms.
\end{cor}

\begin{proof}
For any $\omega \in \Omega^{0,q}(M,E)$, we have
\begin{align*}
\MoveEqLeft
\norm{\omega}_{\W_\nabla^{1,p}(\Omega^{0,q}(M,E))}
\ov{\eqref{Sobolev-norm}}{=}
\norm{\omega}_{\L^p(\Omega^{0,q}(M,E))}
+\norm{\nabla \omega}_{\L^p(\Omega^{1}(M,\Lambda^{0,q}\mathrm{T}^*M \ot E))}  \\
&\ov{\eqref{eq:order0-bound-Dolbeault}}{\lesssim}
\norm{\omega}_{\L^p(\Omega^{0,q}(M,E))}
+\norm{\bar\partial_E \omega}_{\L^p(\Omega^{0,q+1}(M,E))}
+\norm{\bar\partial_E^*\omega}_{\L^p(\Omega^{0,q-1}(M,E))}.
\end{align*}
The converse estimate
\[
\norm{\omega}_{\L^p(\Omega^{0,q}(M,E))}
+\norm{\bar\partial_E \omega}_{\L^p(\Omega^{0,q+1}(M,E))}
+\norm{\bar\partial_E^*\omega}_{\L^p(\Omega^{0,q-1}(M,E))}
\lesssim
\norm{\omega}_{\W^{1,p}_\nabla(\Omega^{0,q}(M,E))}
\]
is immediate by Proposition \ref{prop-elliptic-Sobolev}, applied with $k=1$ and $m=1$, since $\bar\partial_E$ and $\bar\partial_E^*$ are first order differential operators. Hence, we have $\W^{1,p}_{\bar\partial_E}(\Omega^{0,q}(M,E))=\W^{1,p}_\nabla(\Omega^{0,q}(M,E))$ with equivalent norms.
\end{proof}

\begin{remark} \normalfont
\label{remark-Gaffney}
A similar proof  using the Hodge-Dirac operator $D \ov{\mathrm{def}}{=} \d+\d^*$ on a compact Riemannian manifold $M$ provides a proof of the $\L^p$-Gaffney inequality \eqref{Gaffney-Lp}.
\end{remark}

Following the distributional framework for first order differential operators on manifolds, we say that a form $\omega \in \L^1(\Omega^{0,q}(M,E))$ admits a distributional $\bar\partial_E$ if there exists a  form $\eta \in \L^1(\Omega^{0,q+1}(M,E))$ such that $\la \eta, \varphi\ra_{\L^2(\Omega^{0,q+1}(M,E))}  = \la \omega,\bar\partial_E^*\varphi \ra_{\L^2(\Omega^{0,q}(M,E))}$ for every differential form $\varphi \in \Omega^{0,q+1}(M,E)$. In this case, $\eta$ is uniquely determined and we let $
\bar\partial_E \omega
\ov{\mathrm{def}}{=} \eta$. We have
\begin{equation}
\label{int-by-part-1}
\big\la \bar\partial_E \omega, \varphi \big\ra_{\L^2(\Omega^{0,q+1}(M,E))}
=
\big\la \omega, \bar\partial_E^* \varphi \big\ra_{\L^2(\Omega^{0,q}(M,E))}, \quad \varphi \in \Omega^{0,q+1}(M,E).
\end{equation}
In that case, $\bar\partial_E \omega$ is uniquely determined. Similarly, we say that $\omega$ admits a distributional $\bar\partial_E^*$ if there exists a form $\eta \in \L^1(\Omega^{0,q-1}(M,E))$ such that 
$$
\la \eta, \psi \ra_{\L^2(\Omega^{0,q-1}(M,E))}  = \big\la \omega, \bar\partial_E \psi \big\ra_{\L^2(\Omega^{0,q}(M,E))}
$$ 
for every differential form $\psi \in \Omega^{0,q-1}(M,E)$. In this case, $\eta$ is uniquely determined and we let $\bar\partial_E^* \omega
\ov{\mathrm{def}}{=} \eta$. We have
\begin{equation}
\label{int-by-part-2}
\big\la \bar\partial_E^* \omega, \psi \big\ra_{\L^2(\Omega^{0,q-1}(M,E))}
=
\big\la \omega, \bar\partial_E \psi \big\ra_{\L^2(\Omega^{0,q}(M,E))}, \quad \psi \in \Omega^{0,q-1}(M,E).
\end{equation}
In that case, $\bar\partial_E^* \omega$ is uniquely determined. 
Moreover, we can view $\bar\partial_E$ and $\bar\partial_E^*$ as densely defined operators on $\L^p(\Omega^{r,\bullet}(M,E))$ with initial domain $\Omega^{r,\bullet}(M,E)$. By \cite[Theorem 5.28 p.~168]{Kat76}, these operators are closable. So we can consider their closures
\[
\bar\partial_{E,p}\co \dom \bar\partial_{E,p}\subset \L^p(\Omega^{r,\bullet}(M,E))\to \L^p(\Omega^{r,\bullet+1}(M,E))
\]
and
\[
\bar\partial_{E,p}^*\co \dom \bar\partial_{E,p}^*\subset \L^p(\Omega^{r,\bullet}(M,E))\to \L^p(\Omega^{r,\bullet-1}(M,E)).
\]

Now, we give an explicit description of the domains of these operators.

\begin{prop}
\label{prop-domain-dbar-dbarstar-DEp}
Let $E$ be a Hermitian holomorphic vector bundle of finite rank over a compact K\"ahler manifold $M$. Suppose that $1 < p < \infty$. The closures $\bar\partial_{E,p}$ and $\bar\partial_{E,p}^*$ satisfy
\[
\dom \bar\partial_{E,p}
=
\{u \in \L^p(\Omega^{r,\bullet}(M,E)) :
\bar\partial_E u \text{ exists in the distributional sense and belongs to }
\L^p\},
\]
\[
\dom \bar\partial_{E,p}^*
=
\{u \in \L^p(\Omega^{r,\bullet}(M,E)) :
\bar\partial_E^* u \text{ exists in the distributional sense and belongs to }
\L^p\},
\]
and
\begin{equation}
\label{Def-Sobolev-dolbeault}
\W_{\bar\partial_E}^{1,p}(\Omega^{r,\bullet}(M,E))
=
\dom \bar\partial_{E,p}
\cap
\dom \bar\partial_{E,p}^*.
\end{equation}
Moreover, under these identifications, the closed operators $\bar\partial_{E,p}$ and $\bar\partial_{E,p}^*$ act as the corresponding distributional differential operators. 
\end{prop}

\begin{proof}
By definition of the closure, if $u \in \dom \bar\partial_{E,p}$, then there exists a sequence $(u_n)$ in $\Omega^{r,\bullet}(M,E)$ such that $
u_n \to u$ and $\bar\partial_E u_n \to \bar\partial_{E,p}u$ in $\L^p(\Omega^{r,\bullet}(M,E))$. Passing to the limit in the distributional identity
\[
\la \bar\partial_E u_n,\varphi \ra
=
\la u_n,\bar\partial_E^*\varphi \ra,
\qquad \varphi \in \Omega^{r,\bullet}(M,E),
\]
gives $
\la \bar\partial_{E,p}u,\varphi \ra
=
\la u,\bar\partial_E^*\varphi \ra$. Thus $\bar\partial_E u$ exists in the distributional sense and equals $\bar\partial_{E,p}u$. Therefore
\[
\dom \bar\partial_{E,p}
\subset
\{u \in \L^p(\Omega^{r,\bullet}(M,E)) :
\bar\partial_E u \text{ exists distributionally and belongs to }
\L^p(\Omega^{r,\bullet}(M,E))\}.
\]
Conversely, assume that $u \in \L^p(\Omega^{r,\bullet}(M,E))$ and that $\bar\partial_E u$ exists in the distributional sense and belongs to $\L^p(\Omega^{r,\bullet}(M,E))$. By the generalization \cite[Theorem I.19 p.~14]{Gun17} of Meyers--Serrin's theorem for first-order differential operators with smooth coefficients on compact manifolds without boundary, there exists a sequence $(u_n)$ in $\Omega^{r,\bullet}(M,E)$ such that $
u_n \to u$ and $\bar\partial_E u_n \to \bar\partial_E u$ in $\L^p(\Omega^{r,\bullet}(M,E))$. Hence $u \in \dom \bar\partial_{E,p}$ and $\bar\partial_{E,p}u=\bar\partial_E u$. This proves
\[
\dom \bar\partial_{E,p}
=
\{u \in \L^p(\Omega^{r,\bullet}(M,E)) :
\bar\partial_E u \text{ exists distributionally and belongs to }
\L^p(\Omega^{r,\bullet}(M,E))\}.
\]
The same argument can be applied to the first-order differential operator
$\bar\partial_E^*$. It remains to identify the intersection of these two domains with
$\W_{\bar\partial_E}^{1,p}(\Omega^{r,\bullet}(M,E))$. Consider some $
u \in \dom \bar\partial_{E,p}
\cap
\dom \bar\partial_{E,p}^*$. By the preceding domain identifications, both $\bar\partial_E u$ and
$\bar\partial_E^*u$ exist in the distributional sense and belong to $\L^p(\Omega^{r,\bullet}(M,E))$. Now, we apply \cite[Theorem I.19 p.~14]{Gun17} to the finite family of first-order differential operators $\{\bar\partial_E,\bar\partial_E^*\}$. There exists a sequence $(u_n)$ in $\Omega^{r,\bullet}(M,E)$ such that $u_n \to u$, $\bar\partial_E u_n \to \bar\partial_E u$, $\bar\partial_E^* u_n \to \bar\partial_E^* u$ in $\L^p(\Omega^{r,\bullet}(M,E))$. Therefore $
u \in
\W_{\bar\partial_E}^{1,p}(\Omega^{r,\bullet}(M,E))$. The converse inclusion is immediate from the definition of $\W_{\bar\partial_E}^{1,p}(\Omega^{r,\bullet}(M,E))$ and the closedness of
$\bar\partial_{E,p}$ and $\bar\partial_{E,p}^*$. Thus 
$$
\W_{\bar\partial_E}^{1,p}(\Omega^{r,\bullet}(M,E))
=
\dom \bar\partial_{E,p}
\cap
\dom \bar\partial_{E,p}^*.
$$ 
\end{proof}

\subsection{Green's operator and $\L^p$-Dolbeault Hodge decompositions}
\label{sec-Green-Lp-Hodge}

Let $E$ be a Hermitian vector bundle of finite rank over a compact Riemannian manifold $M$. Consider a selfadjoint elliptic differential operator $A \co \C^\infty(M,E) \to \C^\infty(M,E)$ of order $m$. If $P \co \C^\infty(M,E) \to \C^\infty(M,E)$ is the restriction of the orthogonal projection from the Hilbert space $\L^2(M,E)$ on the subspace $\ker A$ then by \cite[Theorem 3.15 p.~15]{BDIP02} (see also \cite[Proposition 18 p.~81]{Mel06}) there exists a unique pseudo-differential operator $G \co \C^\infty(M,E) \to \C^\infty(M,E)$ such that
\begin{equation*}
GA
= \Id_{\C^\infty(M,E)} - P
\quad \text{and} \quad 
AG
= \Id_{\C^\infty(M,E)} - P
\end{equation*}
In this case, the order of $G$ is $-m$ and it is called Green's operator of $A$. We warn the reader that, in some references (see for instance \cite[Section 23.31.11]{Dieu82}), the term Green’s operator is used to denote the resolvent $R(\lambda,A)$, where $\lambda$ does not belong to the spectrum of $A$ on $\L^2(M,E)$. Applying this result to $\Delta_{\bar\partial_E}$ and to the Bergman projection $P=P_{r,q} \co \Omega^{r,q}(M,E) \to \Omega^{r,q}(M,E)$ (which is the restriction of the orthogonal projection from the Hilbert space $\L^2(\Omega^{r,q}(M,E))$ on the subspace $\cal{H}^{r,q}(M,E) \ov{\mathrm{def}}{=} \{\omega \in \Omega^{r,q}(M,E) : \Delta_{\bar\partial_E} \omega = 0\}$), we obtain the following proposition. We refer also to \cite[p.~133]{MaM07}. 

\begin{prop}
\label{prop-existence-Green}
Let $E$ be a Hermitian holomorphic vector bundle of finite rank over a compact Hermitian manifold $M$ of complex dimension $n$. For any integers $r,q \in \{0,\ldots,n\}$, there exists a unique pseudo-differential operator $G=G_{r,q} \co \Omega^{r,q}(M,E) \to \Omega^{r,q}(M,E)$ of order $-2$ such that
\begin{equation}
\label{parametrix}
G\Delta_{\bar\partial_E} 
= \Id - P
\quad \text{and} \quad
\Delta_{\bar\partial_E}G
= \Id - P,
\end{equation}
where $P$ is the Bergman projection on the space of $\Delta_{\bar\partial_E}$-harmonic forms.
\end{prop}

Let $E$ be a holomorphic Hermitian bundle of finite rank over a compact Hermitian manifold $M$ of complex dimension $n$. For any integers $r,q \in \{0,\ldots,n\}$, we consider the space 
\begin{equation*}
\cal{H}^{r,q}(M,E)
\ov{\mathrm{def}}{=} \{\omega \in \Omega^{r,q}(M,E) : \Delta_{\bar\partial_E} \omega = 0\}
\end{equation*}
of smooth $\Delta_{\bar\partial_E}$-harmonic forms of bi-degree $(r,q)$, where $\Delta_{\bar\partial_E}$ is defined in \eqref{holo-Laplacian}. By \cite[(7.2) p.~352]{Dem12} or \cite[Corollary 4.1.14 p.~170]{Huy05},  the Dolbeault cohomology group $\H^{r,q}(M,E)$ is finite-dimensional and there exists a canonical isomorphism
\begin{equation}
\label{harmonic-space-and-Dolbeault}
\H^{r,q}(M,E)
\cong \cal{H}^{r,q}(M,E).
\end{equation}
In particular, every cohomology class in the Dolbeault cohomology group $\H^{r,q}(M,E)$ admits a unique harmonic representative. For any differential form $\omega \in \Omega^{r,q}(M,E)$ we have by \cite[Lemma 9.32 p.~277]{Lee24} the equality $\Delta_{\bar\partial_E} \omega = 0$ if and only if $\bar\partial_E\omega = 0$ and $\bar\partial_E^* \omega = 0$.
By \cite[Theorem 7.1 p.~352]{Dem12} or \cite[Theorem 4.1.13 p.~170]{Huy05}, we have the orthogonal decomposition
\begin{equation}
\label{eq:Hodge-decomp}
\Omega^{r,q}(M,E)
=
\bar\partial_E\Omega^{r,q-1}(M,E)
\oplus
\bar\partial_E^*\Omega^{r,q+1}(M,E)
\oplus
\cal{H}^{r,q}(M,E).
\end{equation}
Recall that $\Delta_{\bar\partial_E,p}$ denote the closure on $\L^p(\Omega^{\bullet,\bullet}(M,E))$ of the Kodaira Laplacian $\Delta_{\bar\partial_E}$ initially defined on $\Omega^{\bullet,\bullet}(M,E)$. We introduce its kernel
\begin{equation*}
\cal{H}^{r,q}_p(M,E)
\ov{\mathrm{def}}{=} \{\omega \in \dom \Delta_{\bar\partial_E,p} : \Delta_{\bar\partial_E,p} \omega = 0\}.
\end{equation*}

\begin{prop}
\label{prop-harmonic-Lp}
Let $E$ be a Hermitian holomorphic vector bundle of finite rank over a compact Hermitian manifold $M$ of complex dimension $n$. Suppose that $1 < p < \infty$. For any integers $r,q \in \{0,\ldots,n\}$, we have $
\cal H^{r,q}_p(M,E)
=\cal H^{r,q}(M,E)$. 
\end{prop}

\begin{proof}
It suffices to prove the inclusion $\cal H^{r,q}_p(M,E)\subset \cal H^{r,q}(M,E)$. Let $\omega \in \cal H^{r,q}_p(M,E)$. By definition, we have $\omega\in \dom \Delta_{\bar\partial_E,p}$ and the equality $\Delta_{\bar\partial_E,p}\omega=0$ in $\L^p(\Omega^{r,q}(M,E))$, hence in the sense of distributions. By Proposition \ref{prop-existence-Green}, there exists a pseudo-differential operator $G$ of order $-2$ such that $P \ov{\eqref{parametrix}}{=}\Id-G\Delta_{\bar\partial_E}$. Since $G$ is a pseudo-differential operator, it extends continuously to the space of distributional sections by \cite[(23.28.8)]{Dieu88}. Hence the identity $P=\Id-G\Delta_{\bar\partial_E}$ also holds in the distributional sense. Applying this identity to $\omega$ in the distributional sense gives $P(\omega)=\omega$. According to \cite[p.~31]{MaM07} the Bergman projection $P$ is a smoothing operator (its Schwartz kernel is smooth), hence $\omega=P(\omega)$ belongs to the space $\Omega^{r,q}(M,E)$.
\end{proof}

Now, we prove a uniqueness result for the next $\L^p$-Hodge decomposition, where we use the usual convention $\Omega^{r,-1}(M,E)=\Omega^{r,n+1}(M,E)=\{0\}$ for any integer $r \in \{0,\ldots,n\}$. 

\begin{lemma}
\label{lem:uniqueness-dolbeault}
Let $E$ be a Hermitian holomorphic vector bundle of finite rank over a compact K\"ahler manifold $M$ of complex dimension $n$. Suppose that $1 < p < \infty$. Consider some integers $r,q \in \{0,\ldots,n\}$. If $a \in \W_{\bar\partial_E}^{1,p}(\Omega^{r,q-1}(M,E))$, $b \in \W^{1,p}_{\bar\partial_E}(\Omega^{r,q+1}(M,E))$ and $h \in \cal{H}^{r,q}(M,E)$ satisfy $
\bar\partial_{E,p} a + \bar\partial_{E,p}^* b + h 
= 0$ in $\L^p(\Omega^{r,q}(M,E))$, then $\bar\partial_{E,p} a=0$, $\bar\partial_{E,p}^*b=0$ and $h=0$.
\end{lemma}

\begin{proof}
Let $\varphi \in \Omega^{r,q}(M,E)$. By \cite[Theorem 4.1.13 p.~170]{Huy05}, we can write its Hodge decomposition
\begin{equation}
\label{Hodge-45}
\varphi 
= \bar\partial_E \eta + \bar\partial_E^* \theta + \tau
\end{equation}
for some smooth forms $\eta \in \Omega^{r,q-1}(M,E)$, $\theta \in \Omega^{r,q+1}(M,E)$ and $\tau \in \cal{H}^{r,q}(M,E)$. Now, with Proposition \ref{prop-domain-dbar-dbarstar-DEp}, we obtain
\begin{equation}
\label{inter-3566}
\big\la \bar\partial_{E,p}^*b, \bar\partial_E\eta\big\ra
\ov{\eqref{int-by-part-2}}{=} \big\la b, \bar\partial_E\bar\partial_E\eta\big\ra 
\ov{\eqref{square-derivation}}{=} 0
\quad \text{and} \quad
\big\la \bar\partial_{E,p} a, \bar\partial_E^*\theta \big\ra
\ov{\eqref{int-by-part-1}}{=} \big\la a, \bar\partial_E^*\bar\partial_E^*\theta\big\ra 
= 0.
\end{equation}
Since $\tau$ is harmonic, we have $\bar\partial_E^*\tau = 0$ and $\bar\partial_E \tau = 0$. We infer that
\begin{equation}
\label{tau-ortho}
\big\la \bar\partial_{E,p} a,\tau\big\ra
\ov{\eqref{int-by-part-1}}{=}
\big\la a,\bar\partial_E^*\tau\big\ra
=0 
\quad \text{and} \quad
\big\la \bar\partial_{E,p}^*b,\tau\big\ra
\ov{\eqref{int-by-part-2}}{=}
\big\la b,\bar\partial_E\tau\big\ra
=0.
\end{equation}
Since $h$ is harmonic, we have $\bar\partial_E^*h = 0$ and $\bar\partial_E h = 0$. We deduce that
\begin{equation}
\label{inter-456}
\big\la h, \bar\partial_E\eta\big\ra 
= \big\la \bar\partial_E^*h,\eta\big\ra 
= 0
\quad \text{and} \quad
\big\la h, \bar\partial_E^*\theta\big\ra 
= \big\la \bar\partial_E h,\theta\big\ra 
= 0.
\end{equation}
We obtain
\begin{align*}
\MoveEqLeft
\big\la \bar\partial_{E,p} a, \varphi\big\ra
\ov{\eqref{Hodge-45}}{=}
\big\la \bar\partial_{E,p} a, \bar\partial_E\eta\big\ra + \big\la \bar\partial_{E,p} a, \bar\partial_E^*\theta\big\ra 
+ \la \bar\partial_{E,p} a, \tau\big\ra 
\ov{\eqref{inter-3566} \eqref{tau-ortho}}{=} \big\la \bar\partial_{E,p} a, \bar\partial_E\eta\big\ra.        
\end{align*}
On the other hand, pairing the identity $\bar\partial_E a + \bar\partial_E^* b + h = 0$ with $\bar\partial_E\eta$ gives
\[
\big\la \bar\partial_{E,p} a, \bar\partial_E\eta\big\ra
=
-\big\la \bar\partial_{E,p}^*b, \bar\partial_E\eta\big\ra - \big\la h, \bar\partial_E\eta\big\ra
\ov{\eqref{inter-3566}\eqref{inter-456}}{=}
0.
\]
Hence $\big\la \bar\partial_{E,p} a,\varphi \big\ra = 0$ for all $\varphi \in \Omega^{r,q}(M,E)$. We conclude that $\bar\partial_{E,p} a=0$  in $\L^p(\Omega^{r,q}(M,E))$. The same argument yields $\bar\partial_{E,p}^*b=0$. Indeed, the preceding identities give
\[
\big\la \bar\partial_{E,p}^*b,\varphi\big\ra
\ov{\eqref{Hodge-45}}{=}
\big\la \bar\partial_{E,p}^*b, \bar\partial_E\eta\big\ra + \big\la \bar\partial_{E,p}^*b, \bar\partial_E^*\theta\big\ra 
+ \la \bar\partial_{E,p}^*b, \tau\big\ra 
\ov{\eqref{inter-3566}\eqref{tau-ortho}}{=}
\big\la \bar\partial_{E,p}^*b,\bar\partial_E^*\theta\big\ra,
\]
and pairing $\bar\partial_{E,p} a+\bar\partial_{E,p}^*b+h=0$ with $\bar\partial_E^*\theta$ yields 
$$
\big\la \bar\partial_{E,p}^*b,\bar\partial_E^*\theta\big\ra
=
-\big\la \bar\partial_{E,p} a,\bar\partial_E^*\theta\big\ra
-\big\la h,\bar\partial_E^*\theta\big\ra
\ov{\eqref{inter-3566}\eqref{inter-456}}{=} 0.
$$ 
Thus $\bar\partial_{E,p}^*b=0$. Then $h=0$ follows immediately.
\end{proof}

The next result gives a strong $\L^p$-Hodge decomposition, by analogy with this notion for Riemannian manifolds described in \cite[p.~115]{Li11}. We refer to \cite[Section A.3]{Mar05} for a nice discussion of $\L^2$-Hodge decomposition. 

\begin{thm}
\label{thm:Lp-dolbeault-hodge}
Let $E$ be a Hermitian holomorphic vector bundle of finite rank over a compact K\"ahler manifold $M$ of complex dimension $n$. Suppose that $1 < p < \infty$. Let $r,q \in \{0,\ldots,n\}$. The Bergman projection $P$ onto the subspace $\cal{H}^{r,q}(M,E)$ extends to a bounded projection $P_p \co \L^p(\Omega^{r,q}(M,E)) \to \L^p(\Omega^{r,q}(M,E))$ with range $\cal H^{r,q}(M,E)$ and the maps $\bar\partial_E \bar\partial_E^*G$ and $\bar\partial_E^*\bar\partial_E G$ induce bounded operators on $\L^p(\Omega^{r,q}(M,E))$. Moreover, we have the topological direct sum decomposition
\begin{equation}
\label{eq:Lp-Hodge-decomp}
\L^p(\Omega^{r,q}(M,E))
=
\bar\partial_{E,p} \W^{1,p}_{\bar\partial_E}(\Omega^{r,q-1}(M,E))
\oplus
\bar\partial_{E,p}^* \W^{1,p}_{\bar\partial_E}(\Omega^{r,q+1}(M,E))
\oplus
\cal{H}^{r,q}(M,E).
\end{equation}
Moreover, for any form $\omega \in \L^p(\Omega^{r,q}(M,E))$ the decomposition is given explicitly by
\begin{equation}
\label{la-der}
\omega 
=\bar\partial_E\bar\partial_E^* G \omega
+
\bar\partial_E^*\bar\partial_E G \omega
+
P_p \omega.
\end{equation}
Finally, the maps $\bar\partial_E^* G$ and $\bar\partial_E G$ are bounded from the Banach space $\L^p(\Omega^{r,q}(M,E))$ into the space $\W^{1,p}_{\bar\partial_E}(\Omega^{r,q-1}(M,E))$ and from $\L^p(\Omega^{r,q}(M,E))$ into $\W^{1,p}_{\bar\partial_E}(\Omega^{r,q+1}(M,E))$.
\end{thm}

\begin{proof}
We start by proving the case $r=0$. Consider some form $\omega \in \Omega^{0,q}(M,E)$. By Proposition \ref{prop-existence-Green}, we have
\[
\bar\partial_E \bar\partial_E^* G \omega + \bar\partial_E^*\bar\partial_E G \omega + P \omega
\ov{\eqref{holo-Laplacian}}{=} \Delta_{\bar\partial_E} G \omega +P \omega
\ov{\eqref{parametrix}}{=} \omega.
\]
Note that by composition \cite[17.13.5 p.~298]{Dieu72}, the linear maps $\bar\partial_E \bar\partial_E^*G$ and $\bar\partial_E^*\bar\partial_E G$ are pseudo-differential operators of order $0$. By Proposition \ref{prop-elliptic-Sobolev}, these operators induce bounded operators on the Banach space $\L^p(\Omega^{0,q}(M,E))$. Since
\begin{equation}
\label{identity-788}
P \omega
=\omega-\bar\partial_E \bar\partial_E^* G \omega - \bar\partial_E^*\bar\partial_E G \omega,
\end{equation}
we conclude that the Bergman projection $P$ induces a bounded projection $P_p \co \L^p(\Omega^{0,q}(M,E)) \to \L^p(\Omega^{0,q}(M,E))$, as a sum of three bounded operators (alternatively, one can argue more directly using that $P$ is smoothing). 

Now, we prove that $\Ran P_p = \cal H^{0,q}(M,E)$. For any $\omega \in \Omega^{0,q}(M,E)$ we have $\Delta_{\bar\partial_E}P\omega=0$. Since $P_p$ is the continuous extension of $P$ to $\L^p(\Omega^{0,q}(M,E))$ and $\Delta_{\bar\partial_E,p}$ is the closure of $\Delta_{\bar\partial_E}$, it follows that for any $\omega \in \L^p(\Omega^{0,q}(M,E))$ we have $\Delta_{\bar\partial_E,p} P_p \omega = 0$. Indeed, fix some $\omega \in \L^p(\Omega^{0,q}(M,E))$ and choose a sequence $(\omega_n)$ of elements in $\Omega^{0,q}(M,E)$ with $\omega_n \to \omega$ in $\L^p(\Omega^{0,q}(M,E))$. Then $P\omega_n \to P_p\omega$ in $\L^p(\Omega^{0,q}(M,E))$ and $\Delta_{\bar\partial_E}(P\omega_n)=0$ for all integer $n$, so by closedness,  
we obtain $\Delta_{\bar\partial_E,p}(P_p\omega)=0$. By Proposition \ref{prop-harmonic-Lp}, we obtain $P_p\omega \in \cal H^{0,q}(M,E)$.
Hence $\Ran P_p \subset \cal{H}^{0,q}(M,E)$. Conversely, if $h \in \cal{H}^{0,q}(M,E)$, then $P(h)=h$, hence $P_p(h)=h$. Therefore $\cal{H}^{0,q}(M,E) \subset \Ran P_p$ and finally $\Ran P_p =\cal H^{0,q}(M,E)$. 

Note that by composition \cite[17.13.5 p.~298]{Dieu72}, the linear maps $\bar\partial_E^*G$ and $\bar\partial_E G$ are pseudo-differential operators of order $-1$. By Proposition \ref{prop-elliptic-Sobolev} and Corollary \ref{cor:Gaffney-dolbeault-Lp}, we obtain two bounded linear operators $\bar\partial_E^* G \co \L^p(\Omega^{0,q}(M,E)) \to \W^{1,p}_{\bar\partial_E}(\Omega^{0,q-1}(M,E))$ and $\bar\partial_E G \co \L^p(\Omega^{0,q}(M,E)) \to \W^{1,p}_{\bar\partial_E}(\Omega^{0,q+1}(M,E))$. If $\omega \in \L^p(\Omega^{0,q}(M,E))$, we deduce that
\[
\bar\partial_E(\bar\partial_E^*G \omega) \in \bar\partial_{E,p} \W^{1,p}_{\bar\partial_E}(\Omega^{0,q-1}(M,E))
\quad \text{and} \quad
\bar\partial_{E,p}^*(\bar\partial_E G \omega) \in \bar\partial_{E,p}^* \W^{1,p}_{\bar\partial_E}(\Omega^{0,q+1}(M,E)).
\]
Since the subspace $\Omega^{0,q}(M,E)$ is dense in the Banach space $\L^p(\Omega^{0,q}(M,E))$ and since the operators $\bar\partial_E\bar\partial_E^*G$, $\bar\partial_E^*\bar\partial_EG$ and $P_p$ are bounded on $\L^p(\Omega^{0,q}(M,E))$, the identity \eqref{identity-788} extends to all $\omega \in \L^p(\Omega^{0,q}(M,E))$ by continuity, yielding \eqref{la-der}. Thus every $\omega \in \L^p(\Omega^{0,q}(M,E))$ belongs to the sum of the three stated subspaces. 

To see that the sum is direct, suppose that $\bar\partial_{E,p} a + \bar\partial_{E,p}^* b + h 
= 0$ with $a \in \W^{1,p}_{\bar\partial_E}(\Omega^{0,q-1}(M,E))$, $b \in \W^{1,p}_{\bar\partial_E}(\Omega^{0,q+1}(M,E))$ and $h \in \cal{H}^{0,q}(M,E)$. Lemma \ref{lem:uniqueness-dolbeault} yields $\bar\partial_{E,p} a=\bar\partial_{E,p}^*b=h=0$, so the sum is direct.

Now, we reduce the general case to the previous case. Let $r \in \{0,\ldots,n\}$ and set $
E_r \ov{\mathrm{def}}{=} \Lambda^{r,0}\mathrm{T}^*M \ot E$. Then $E_r$ is a Hermitian holomorphic vector bundle and we have canonical identifications $
\Omega^{r,q}(M,E) \simeq \Omega^{0,q}(M,E_r)$ and $\L^p(\Omega^{r,q}(M,E)) \simeq \L^p(\Omega^{0,q}(M,E_r))$. Under these identifications, the Dolbeault operators correspond, i.e.
\[
\bar\partial_{E_r}\big|_{\Omega^{0,q}(M,E_r)}
=
\bar\partial_E\big|_{\Omega^{r,q}(M,E)},
\quad \text{and} \quad
\bar\partial_{E_r}^*\big|_{\Omega^{0,q}(M,E_r)}
=
\bar\partial_E^*\big|_{\Omega^{r,q}(M,E)}.
\]
Consequently, the first part of the proof applied to the bundle $E_r$ yields the decomposition
\[
\L^p(\Omega^{r,q}(M,E))
=
\bar\partial_{E,p} \W^{1,p}_{\bar\partial_E}(\Omega^{r,q-1}(M,E))
\oplus
\bar\partial_{E,p}^* \W^{1,p}_{\bar\partial_E}(\Omega^{r,q+1}(M,E))
\oplus
\cal H^{r,q}(M,E),
\]
with the explicit formula
\[
\omega
=
\bar\partial_E\bar\partial_E^* G\omega
+
\bar\partial_E^*\bar\partial_E G \omega
+
P_{p}\omega,
\qquad
\omega \in \L^p(\Omega^{r,q}(M,E)),
\]
where $G$ is the Green operator of $\Delta_{\bar\partial_E}$ on $\Omega^{r,q}(M,E)$ and $P_{p}$ is the bounded extension to $\L^p(\Omega^{r,q}(M,E))$ of the orthogonal projection onto $\cal H^{r,q}(M,E)$.
\end{proof}

\subsection{Domain of the Dolbeault--Dirac operator $\D_{E,p}$}
\label{sec-domain-dolbeault}

In this section,  we identify the domain of the closure  $\D_{E,p}$ of the Dolbeault--Dirac operator under standard analytic assumptions. If $E$ is a Hermitian holomorphic vector bundle over a compact Hermitian manifold $M$ and if $1 \leq p < \infty$, recall that the Sobolev space $\W^{1,p}_{\bar\partial_E}(\Omega^{0,q}(M,E))$ is defined in \eqref{eq:def-W1p-dolbeault-norm}.

\begin{prop}
\label{prop:graph-norm-dolbeault-Lp}
Let $E$ be a Hermitian holomorphic vector bundle of finite rank over a compact Hermitian manifold $M$. Suppose that $1 < p < \infty$. Assume that the following conditions hold.
\begin{enumerate}
\item The Dolbeault--Dirac operator $\D_{E,p}$ is bisectorial on the Banach space $\L^p(\Omega^{0,\bullet}(M,E))$ and admits a bounded $\H^\infty(\Sigma_\theta^\bi)$ functional calculus for some angle $\theta \in (0,\frac{\pi}{2})$.

\item The reduced Dolbeault--Riesz transforms, initially defined on $\Delta_{\bar\partial_E,p}^{\frac{1}{2}}(\Omega^{0,\bullet}(M,E))$ by
\begin{equation}
\label{Dol-Riesz}
R_{\bar\partial_E,0}\bigl(\Delta_{\bar\partial_E,p}^{\frac{1}{2}}\omega\bigr)
\ov{\mathrm{def}}{=}
\bar\partial_E\omega,
\quad \text{and} \quad
R_{\bar\partial_E^*,0}\bigl(\Delta_{\bar\partial_E,p}^{\frac{1}{2}}\omega\bigr)
\ov{\mathrm{def}}{=}
\bar\partial_E^*\omega,
\end{equation}
for $\omega \in \Omega^{0,\bullet}(M,E)$, extend to bounded operators
$R_{\bar\partial_E},R_{\bar\partial_E^*}
\co
\ovl{\Ran\Delta_{\bar\partial_E,p}^{\frac{1}{2}}}
\to
\L^p(\Omega^{0,\bullet}(M,E))$.
\end{enumerate}
Then we have
\begin{equation}
\label{eq:domain-D-dolbeault-precise}
\dom \D_{E,p}
=
\W^{1,p}_{\bar\partial_E}(\Omega^{0,\bullet}(M,E))
\end{equation}
and the graph norm of $\D_{E,p}$ is equivalent to the Sobolev norm \eqref{eq:def-W1p-dolbeault-norm}, namely
\begin{equation}
\label{eq:graph-norm-dolbeault-equivalence}
\norm{\omega}_{\W^{1,p}_{\bar\partial_E}(\Omega^{0,\bullet}(M,E))}
\approx_p
\norm{\omega}_{\L^p(\Omega^{0,\bullet}(M,E))}+\norm{\D_{E,p}\omega}_{\L^p(\Omega^{0,\bullet}(M,E))},
\quad
\omega \in \dom \D_{E,p}.
\end{equation}
\end{prop}

\begin{proof}
Assume that the operator $\D_{E,p}$ is bisectorial on the Banach space $\L^p(\Omega^{0,\bullet}(M,E))$ and admits a bounded $\H^\infty(\Sigma_\theta^\bi)$ functional calculus for some $\theta \in (0,\pi)$. The abstract theory of bisectorial operators yields that the operator $\Delta_{\bar\partial_E,p} \ov{\mathrm{def}}{=} \frac{1}{2}\D_{E,p}^2$ is sectorial and that $\dom \D_{E,p}
=
\dom \Delta_{\bar\partial_E,p}^{\frac12}$ together with the norm equivalence
\begin{equation}
\label{eq:equiv-D-Delta-half-proof-dolbeault}
\norm{\D_{E,p}\omega}_{\L^p(\Omega^{0,\bullet}(M,E))}
\approx
\bnorm{\Delta_{\bar\partial_E,p}^{\frac12}\omega}_{\L^p(\Omega^{0,\bullet}(M,E))},
\qquad
\omega \in \dom \D_{E,p}.
\end{equation}
For any differential form $\omega \in \Omega^{0,\bullet}(M,E)$, the triangle inequality gives
\begin{align*}
\MoveEqLeft
\norm{\omega}_{\L^p(\Omega^{0,\bullet}(M,E))}
+
\norm{\D_E\omega}_{\L^p(\Omega^{0,\bullet}(M,E))}
\\
&\ov{\eqref{Def-Dirac-DE}}{=}
\norm{\omega}_{\L^p(\Omega^{0,\bullet}(M,E))}
+
\sqrt{2}\norm{\bar\partial_E\omega+\bar\partial_E^*\omega}_{\L^p(\Omega^{0,\bullet}(M,E))}
\\
&\lesssim
\norm{\omega}_{\L^p(\Omega^{0,\bullet}(M,E))}
+
\norm{\bar\partial_E\omega}_{\L^p(\Omega^{0,\bullet}(M,E))}
+
\norm{\bar\partial_E^*\omega}_{\L^p(\Omega^{0,\bullet}(M,E))}
\ov{\eqref{eq:def-W1p-dolbeault-norm}}{=}
\norm{\omega}_{\W^{1,p}_{\bar\partial_E}(\Omega^{0,\bullet}(M,E))}.
\end{align*}

Now, we check that the operators in \eqref{Dol-Riesz} are well-defined. Suppose that $
\Delta_{\bar\partial_E,p}^{\frac12}\omega_1
=
\Delta_{\bar\partial_E,p}^{\frac12}\omega_2$ for some smooth forms $\omega_1,\omega_2 \in \Omega^{0,\bullet}(M,E)$. Set $
\omega
\ov{\mathrm{def}}{=}
\omega_1-\omega_2$. Then $
\Delta_{\bar\partial_E,p}^{\frac12}\omega=0$. Using \eqref{eq:equiv-D-Delta-half-proof-dolbeault}, we obtain $\D_{E,p}\omega=0$.  Since $\omega$ is smooth, this means that $
\D_E\omega=0$. Applying $\D_E$ once more and using the identity $
\D_E^2=2\Delta_{\bar\partial_E}$ on smooth forms, we obtain $
\Delta_{\bar\partial_E}\omega=0$. By \cite[Lemma 9.32 p.~277]{Lee24}, we infer that $
\bar\partial_E\omega=0$ and $
\bar\partial_E^*\omega=0$. We conclude that $\bar\partial_E\omega_1=\bar\partial_E\omega_2$ and $
\bar\partial_E^*\omega_1=\bar\partial_E^*\omega_2$.

For any differential form $\omega \in \Omega^{0,\bullet}(M,E)$ we write
\begin{align}
\MoveEqLeft
\label{eq:dolbeault-estimate-Riesz}
\norm{\bar\partial_E\omega}_{\L^p(\Omega^{0,\bullet}(M,E))}
+
\norm{\bar\partial_E^*\omega}_{\L^p(\Omega^{0,\bullet}(M,E))}
\\
&\ov{\eqref{Dol-Riesz}}{=}
\bnorm{R_{\bar\partial_E}\bigl(\Delta_{\bar\partial_E}^{\frac12}\omega\bigr)}_{\L^p(\Omega^{0,\bullet}(M,E))}
+
\bnorm{R_{\bar\partial_E^*}\bigl(\Delta_{\bar\partial_E}^{\frac12}\omega\bigr)}_{\L^p(\Omega^{0,\bullet}(M,E))}
\nonumber
\\
&\leq
\bigl(
\bnorm{R_{\bar\partial_E}}_{\B(\L^p(\Omega^{0,\bullet}(M,E)))}
+
\bnorm{R_{\bar\partial_E^*}}_{\B(\L^p(\Omega^{0,\bullet}(M,E)))}
\bigr)
\bnorm{\Delta_{\bar\partial_E}^{\frac12}\omega}_{\L^p(\Omega^{0,\bullet}(M,E))}
\nonumber
\\
&\ov{\eqref{eq:equiv-D-Delta-half-proof-dolbeault}}{\lesssim}
\norm{\D_E\omega}_{\L^p(\Omega^{0,\bullet}(M,E))}.
\nonumber
\end{align}
We infer that
\begin{align*}
\norm{\omega}_{\W^{1,p}_{\bar\partial_E}(\Omega^{0,\bullet}(M,E))}
&\ov{\eqref{eq:def-W1p-dolbeault-norm}}{=}
\norm{\omega}_{\L^p(\Omega^{0,\bullet}(M,E))}
+\bnorm{\bar\partial_E\omega}_{\L^p(\Omega^{0,\bullet}(M,E))}
+\bnorm{\bar\partial_E^*\omega}_{\L^p(\Omega^{0,\bullet}(M,E))}
\\
&\ov{\eqref{eq:dolbeault-estimate-Riesz}}{\lesssim}
\norm{\omega}_{\L^p(\Omega^{0,\bullet}(M,E))}
+
\norm{\D_E\omega}_{\L^p(\Omega^{0,\bullet}(M,E))}.
\end{align*}
We conclude the proof by using standard arguments.
\end{proof}

\begin{cor}
\label{prop-DEp-sum-dbar-dbarstar}
Under the assumptions of Proposition~\ref{prop:graph-norm-dolbeault-Lp}, we have $
\D_{E,p}
=
\sqrt{2}\big(\bar\partial_{E,p}+\bar\partial_{E,p}^*\big)$ 
as operators on $\L^p(\Omega^{0,\bullet}(M,E))$ with common domain
$\dom \D_{E,p}=\dom \bar\partial_{E,p}\cap\dom \bar\partial_{E,p}^*$.
\end{cor}

\begin{proof}
By Proposition~\ref{prop:graph-norm-dolbeault-Lp} and Proposition~\ref{prop-domain-dbar-dbarstar-DEp}, we have $
\dom \D_{E,p}
\ov{\eqref{eq:domain-D-dolbeault-precise}}{=}
\W_{\bar\partial_E}^{1,p}(\Omega^{0,\bullet}(M,E))
=
\dom \bar\partial_{E,p}
\cap
\dom \bar\partial_{E,p}^*$. Let $u \in \dom \D_{E,p}$. Since $
\dom \D_{E,p}
\ov{\eqref{eq:domain-D-dolbeault-precise}}{=}
\W_{\bar\partial_E}^{1,p}(\Omega^{0,\bullet}(M,E))$, 
there exists a sequence $(u_n)$ in $\Omega^{0,\bullet}(M,E)$ such that $
u_n \to u$, $
\bar\partial_E u_n \to \bar\partial_{E,p}u$, 
$\bar\partial_E^* u_n \to \bar\partial_{E,p}^*u$
in $\L^p(\Omega^{0,\bullet}(M,E))$. Hence
\[
\D_E u_n
\ov{\eqref{Def-Dirac-DE-bis}}{=}
\sqrt{2}\big(\bar\partial_E u_n+\bar\partial_E^*u_n\big)
\to
\sqrt{2}\big(\bar\partial_{E,p}u+\bar\partial_{E,p}^*u\big)
\]
in $\L^p(\Omega^{0,\bullet}(M,E))$. Since $\D_{E,p}$ is the closure of the operator $\D_E$, we conclude that $
\D_{E,p}u
=
\sqrt{2}\big(\bar\partial_{E,p}u+\bar\partial_{E,p}^*u\big)$. The equality of operators follows from the equality of domains.
\end{proof}

\subsection{Boundedness of the $\mathrm{H}^\infty$ functional calculus of the Dolbeault--Dirac operator}
\label{sec-Hinfty}

The following result is a variant of \cite[Theorem 5.3 p.~634]{Li10}. Note that if the reduced Dolbeault--Riesz transforms introduced in \eqref{Dol-Riesz} are bounded, we can extend it by setting them equal to zero on
$\ker\Delta_{\bar\partial_E,p}$ in the decomposition $
\L^p(\Omega^{0,\bullet}(M,E))
\ov{\eqref{decompo-reflexive}}{=}
\ker\Delta_{\bar\partial_E,p}
\oplus
\ovl{\Ran\Delta_{\bar\partial_E,p}}$. 

\begin{thm}
\label{thm-Riesz-bounded}
Let $E$ be a Hermitian holomorphic vector bundle of finite rank over a compact Hermitian manifold $M$. Suppose that $1 < p < \infty$. Then the Dolbeault--Riesz transforms $\bar{\partial}_E (\Delta_{\bar\partial_E})^{-\frac{1}{2}}$ and $\bar{\partial}_{E}^* (\Delta_{\bar\partial_E})^{-\frac{1}{2}}$ are bounded on the Banach space $\L^p(\Omega^{\bullet,\bullet}(M,E))$. 
\end{thm}

\begin{proof}
Observe that the map $\bar{\partial}_E$ is a differential operator of order $1$ and recall that $\Delta_{\bar\partial_E}$ is an elliptic differential operator of order 2. In particular by \cite[Theorem 1 p.~124]{Bur68} \cite{See67} \cite{Shu01}, we see that $(\Delta_{\bar\partial_E})^{-\frac{1}{2}}$ is a pseudo-differential operator of order $-1$ (defined as 0 on the subspace of $\Delta_{\bar\partial_E}$-harmonic forms). So, by the composition rule \cite[17.13.5 p.~298]{Dieu72}, $\bar{\partial}_E (\Delta_{\bar\partial_E})^{-\frac{1}{2}}$ is a pseudo-differential operator of order $0$. Hence, by Proposition \ref{prop-elliptic-Sobolev}, it extends to a bounded operator on the space $\L^p(\Omega^{\bullet,\bullet}(M,E))$. A similar reasoning shows that the Riesz transform $\bar{\partial}_{E}^* (\Delta_{\bar\partial_E})^{-\frac{1}{2}}$ is also bounded on the Banach space $\L^p(\Omega^{\bullet,\bullet}(M,E))$.
\end{proof}

The following result of semigroup domination is proved in \cite[Theorem 4.3 p.~632]{Li10}. Here $\W_{r,q}(E)$ is the Weitzenb\"ock curvature operator.

\begin{thm}
\label{th-domination-Li}
Let $E$ be a holomorphic Hermitian vector bundle of finite rank over a complete K\"ahler manifold $M$ of complex dimension $n$. Suppose that $\W_{r,q}(E) \geq -a$, where $a \in \R$ is a constant for some integers $r,q \in \{0,\ldots,n\}$. For any form $\omega \in \Omega^{r,q}(M,E)$ and any $x \in M$, we have
\begin{equation}
\label{domination-43}
|\e^{-t \Delta_{\bar\partial_E}}\omega(x)|
\leq \e^{a t}\e^{t \Delta} |\omega|(x), \quad t > 0.
\end{equation}
\end{thm}

We need the following variant of \cite[Lemma 3.3]{Arh26b}.

\begin{lemma}
\label{lemma-kernel-domination-Dolbeault}
Let $E$ be a holomorphic Hermitian vector bundle of finite rank over a Hermitian manifold $M$ of complex dimension $n$. Let $r,q \in \{0,\ldots,n\}$. Assume that for some $a \in \R$ we have for any $\eta \in \L^2(\Omega^{r,q}(M,E))$ the estimate
\begin{equation}
\label{domin-bis-Dolbeault}
\bigl|\e^{-t\Delta_{\bar\partial_E}}\eta\bigr|
\leq \e^{at} \e^{t\Delta}(|\eta|), \quad t > 0.
\end{equation}
Let $p_t^{\Delta_{\bar\partial_E}}(x,y)$ and $p_t(x,y)$ be the smooth integral kernels of the operators $\e^{-t\Delta_{\bar\partial_E}}$ and $\e^{t\Delta}$. Then
\[
\norm{p_t^{\Delta_{\bar\partial_E}}(x,y)}
\leq \e^{at} p_t(x,y),\qquad t > 0,\ x,y\in M.
\]
\end{lemma}

\begin{proof}
For any $\eta \in \L^2(\Omega^{r,q}(M,E))$ we deduce that
\begin{align*}
\MoveEqLeft
\bigl|(\e^{-t\Delta_{\bar\partial_E}}\eta)(x)\bigr|
=
\left|\int_M p_t^{\Delta_{\bar\partial_E}}(x,z)\eta(z)\d\mu_g(z) \right|
\leq
\int_M \norm{p_t^{\Delta_{\bar\partial_E}}(x,z)}\,|\eta(z)| \d\mu_g(z).
\end{align*}
In a similar way, we have the equality
\[
(\e^{t\Delta}|\eta|)(x)
=
\int_M p_t(x,z)\,|\eta(z)| \d\mu_g(z),
\]
where $p_t(x,z) \geq 0$.

Let $y \in M$, $t>0$ and let $v \in \Lambda^{r,q}\mathrm{T}_y^*M \ot E_y$ be a unit vector. Choose a smooth section $\sigma \in \Omega^{r,q}(M,E)$ satisfying $\sigma(y)=v$. Such a section can be constructed by taking a local trivialization of the bundle $\Lambda^{r,q}\mathrm{T}^*M \ot E$ on a coordinate chart $U$ around $y$, defining a constant local section equal to $v$, in this trivialization, multiplying by a cutoff function equal to $1$ around $y$, and extending by zero outside the coordinate neighborhood.

Let $(\rho_n)$ be a sequence of positive smooth functions supported in the geodesic balls $B(y,\frac{1}{n})$ and satisfying $
\int_M \rho_n(z) \d\mu_g(z) 
= 1$. Thus $(\rho_n)$ forms an approximate identity centered at $y$. Define the sequence $(\eta_n)$ of test forms $
\eta_n \ov{\mathrm{def}}{=} \rho_n\,\sigma \in \Omega^{r,q}(M,E)$. For any integer $n$, applying the domination estimate \eqref{domin-bis-Dolbeault} to $\eta_n$, we obtain 
\begin{align*}
\MoveEqLeft
\left|\int_M p_t^{\Delta_{\bar\partial_E}}(x,z)\eta_n(z)\d\mu_g(z)\right|
=
\bigl|(\e^{-t\Delta_{\bar\partial_E}}\eta_n)(x)\bigr|
\ov{\eqref{domin-bis-Dolbeault}}{\leq}
\e^{at} (\e^{t\Delta}|\eta_n|)(x) \\
&=
\e^{at}\int_M p_t(x,z)\,|\eta_n(z)| \d\mu_g(z).
\end{align*}
Passing to the limit as $n \to \infty$, using the continuity of $z\mapsto p_t^{\Delta_{\bar\partial_E}}(x,z)$ and $z \mapsto p_t(x,z)$ together with the approximation properties of $(\rho_n)$ and  the identity $\sigma(y)=v$, yield
\[
|p_t^{\Delta_{\bar\partial_E}}(x,y)v|
\leq
\e^{at}p_t(x,y)|v|
=
\e^{at}p_t(x,y).
\]
Finally, taking the supremum over all unit vectors $v \in \Lambda^{r,q}\mathrm{T}_y^*M \ot E_y$ gives
\[
\bnorm{p_t^{\Delta_{\bar\partial_E}}(x,y)}
\leq
\e^{at}p_t(x,y).
\]
Since both kernels are smooth, the estimate holds pointwise for every $x,y\in M$.
\end{proof}

If $E$ is a smooth Hermitian vector bundle of finite rank over a compact Riemannian manifold $M$, we can introduce the tensor bundle $E \boxtimes E^*$ over $M \times M$ with fiber
\[
(E \boxtimes E^*)_{(x,y)} 
= E_x \ot E_y^* 
= \Hom(E_y,E_x), \quad x,y \in M.
\]
Further information can be found in \cite[p.~147]{BaC17}, \cite[(23.4.4)]{Dieu88} and \cite[p.~74]{BGV04}. If $K \co M \times M \to E \boxtimes E^*$ is a continuous section then the function $M \times M  \to [0,\infty)$, $(x,y) \mapsto \norm{K(x,y)}_{\Hom(E_y,E_x)}$ is continuous. Furthermore, for any $f \in \L^1(M,E)$ and any $x \in M$, we can define
\begin{equation}
\label{integral-operator}
(T_K f)(x) 
\ov{\mathrm{def}}{=} \int_M K(x,y) f(y) \d\mu_g(y),
\end{equation}
where the integral is a Bochner integral in the finite-dimensional space $E_x$. 
It is well-known (see e.g.~\cite[Proposition 4.4]{Arh26b} for a proof) that $T_K$ extends uniquely to a bounded operator $T_K \co \L^1(M,E) \to \L^\infty(M,E)$, and its operator norm is given by
\begin{equation}
\label{Dunford}
\norm{T_K}_{\L^1(M,E) \to \L^\infty(M,E)}
=
\sup_{(x,y) \in M \times M} \norm{K(x,y)}_{\Hom(E_y,E_x)}.
\end{equation}
It is a generalization of Dunford--Pettis theorem \cite[p.~528]{Rob91}. Now, we establish some Gaussian estimates. 

\begin{prop}
\label{prop:Gq-compact-Dolbeault}
Consider a holomorphic Hermitian vector bundle $E$ of finite rank over a compact K\"ahler manifold $M$ of complex dimension $n$. Denote by $p_t^{\Delta_{\bar\partial_E}}(x,y)$ the smooth integral kernel of the heat semigroup $\e^{-t\Delta_{\bar\partial_E}}$ acting on $\L^2(\Omega^{r,q}(M,E))$. Then there exist constants $C,c>0$ such that, for all $x,y \in M$,
\begin{equation}
\label{eq:Gq-compact-Dolbeault}
\norm{p_t^{\Delta_{\bar\partial_E}}(x,y)}
\leq
\frac{C}{V(x,\sqrt t)}
\exp\Bigl(-c\frac{\dist(x,y)^2}{t}\Bigr), \quad t > 0,
\end{equation}
where $\norm{\cdot}$ is the operator norm on $\Hom(\Lambda^{r,q}\mathrm{T}_y^*M \ot E_y,\Lambda^{r,q}\mathrm{T}_x^*M \ot E_x)$ and $V(x,r)$ is the volume of the geodesic ball $B(x,r)$.
\end{prop}

\begin{proof}
According to \cite[Corollary 5.3.5 p.~142]{Hsu02} the heat kernel $p_t(x,y)$ of the operator $\e^{t\Delta}$ acting on functions satisfies an estimate of the form
$$
p_t(x,y) 
\lesssim \frac{1}{t^{\frac{d}{2}}}, \quad x,y \in M, 0<t<1.
$$
Using \cite[Exercise 16.5 p.~424]{Gri09} and \cite[Theorem 1.1 p.~35]{Gri97}, this yields a Gaussian upper bound, i.e.~there exists $c>0$ such that
\begin{equation}
\label{eq:G0-compact}
p_t(x,y) 
\lesssim \frac{1}{V(x,\sqrt{t})}\exp\Bigl(-c\frac{\dist(x,y)^2}{t}\Bigr), \quad x,y \in M, 0<t<1.
\end{equation}
The compactness of the manifold $M$ gives the existence of a constant $a_{r,q} \geq 0$ such that $\W_{r,q}(E) \geq -a_{r,q}$. Theorem \ref{th-domination-Li} provides the domination estimate
\begin{equation}
\label{eq:domination-compact}
\bigl|\e^{-t\Delta_{\bar\partial_E}}\omega\bigr|
\ov{\eqref{domination-43}}{\leq}
\e^{a_{r,q} t}\e^{t\Delta}(|\omega|),
\qquad t > 0, \omega \in \L^2(\Omega^{r,q}(M,E)).
\end{equation}
Using Lemma \ref{lemma-kernel-domination-Dolbeault}, we obtain the pointwise kernel domination
\begin{equation}
\label{eq:kernel-domination}
\norm{p_t^{\Delta_{\bar\partial_E}}(x,y)}
\leq
\e^{a_{r,q} t}\,p_t(x,y),
\qquad x,y\in M,t > 0.
\end{equation}
Fix $t \in(0,1]$. We have $\e^{a_{r,q} t}\leq \e^{a_{r,q}}$. Combining \eqref{eq:G0-compact} and \eqref{eq:kernel-domination}, we obtain \eqref{eq:Gq-compact-Dolbeault} if $t \in (0,1]$.

It remains to consider large times. Let $P \co \L^2(\Omega^{r,q}(M,E)) \to \L^2(\Omega^{r,q}(M,E))$ be the orthogonal projection onto the finite-dimensional space $\cal{H}^{r,q}(M,E)$ of $\Delta_{\bar\partial_E}$-harmonic $(r,q)$-forms. Then $\e^{-t\Delta_{\bar\partial_E}}=\e^{-t\Delta_{\bar\partial_E}}P+\e^{-t\Delta_{\bar\partial_E}}(\Id-P)$. Note that $\Delta_{\bar\partial_E}$ is an elliptic pseudo-differential operator of order 2 with a principal symbol of bijective type. We deduce by \cite[(23.35.2)]{Dieu88} that its spectrum is discrete, entirely composed of eigenvalues. Hence, on the subspace $(\Id-P)\L^2(\Omega^{r,q}(M,E))$ the spectrum of $\Delta_{\bar\partial_E}$ is contained in $[\lambda_{r,q},\infty)$ for some $\lambda_{r,q} > 0$. Now, for any $t \geq 2$, spectral theory provides the estimate
\begin{align}
\MoveEqLeft
\label{inter98}
\bnorm{\e^{-(t-2)\Delta_{\bar\partial_E}}(\Id-P)}_{\L^2(\Omega^{r,q}(M,E)) \to \L^2(\Omega^{r,q}(M,E))}
\leq \e^{-\lambda_{r,q}(t-2)} 
\lesssim \e^{-\lambda_{r,q} t}.        
\end{align}
For a fixed $x \in M$, Cauchy--Schwarz inequality gives for any form $\omega \in \Omega^{r,q}(M,E)$
\begin{align*}
\MoveEqLeft
\norm{e^{-\Delta_{\bar\partial_E}}\omega(x)}_{\Lambda^{r,q}\mathrm{T}_x^*M \ot E_x}
\ov{\eqref{integral-operator}}{=} \norm{\int_M p_1^{\Delta_{\bar\partial_E}}(x,y)\omega(y) \d\mu_g(y) }_{\Lambda^{r,q}\mathrm{T}_x^*M \ot E_x} \\
&\leq \int_M \norm{p_1^{\Delta_{\bar\partial_E}}(x,y)} \norm{\omega(y)}_{_{\Lambda^{r,q}\mathrm{T}_y^*M \ot E_y}}
\d\mu_g(y) \\
&\ov{\eqref{Lp-norm-vector-bundle}}{\leq} \left( \int_M \norm{p_1^{\Delta_{\bar\partial_E}}(x,y)}^2 \d\mu_g(y) \right)^{\frac{1}{2}} \norm{\omega}_{\L^2(\Omega^{r,q}(M,E))}.
\end{align*} 
Since $p_1^{\Delta_{\bar\partial_E}}$ is smooth on the compact manifold $M \times M$, the quantity $$
\sup_{x \in M} \bigg( \int_M \norm{p_1^{\Delta_{\bar\partial_E}}(x,y)}^2 \d\mu_g(y)
\bigg)^{\frac{1}{2}}
$$ 
is finite. Thus $\e^{-\Delta_{\bar\partial_E}}$ is bounded from $\L^2(\Omega^{r,q}(M,E))$ into $\L^\infty(\Omega^{r,q}(M,E))$. By duality, using the symmetry of $\Delta_{\bar\partial_E}$, we have a bounded operator $\e^{-\Delta_{\bar\partial_E}} \co \L^1(\Omega^{r,q}(M,E)) \to \L^2(\Omega^{r,q}(M,E))$. 
For any $t \geq 2$, we obtain
\begin{align*}
\MoveEqLeft
\bnorm{\e^{-t\Delta_{\bar\partial_E}}(\Id-P)}_{\L^1(\Omega^{r,q}(M,E)) \to \L^\infty(\Omega^{r,q}(M,E))} \\
&\leq \bnorm{\e^{-\Delta_{\bar\partial_E}}}_{\L^2 \to \L^\infty} \bnorm{\e^{-(t-2)\Delta_{\bar\partial_E}}(\Id-P)}_{\L^2 \to \L^2} \bnorm{\e^{-\Delta_{\bar\partial_E}}}_{\L^1 \to \L^2} 
\ov{\eqref{inter98}}{\lesssim} \e^{-\lambda_{r,q} t}. 
\end{align*}
Furthermore, the map $P=\e^{-t\Delta_{\bar\partial_E}}P \co \L^1(\Omega^{r,q}(M,E)) \to \L^\infty(\Omega^{r,q}(M,E))$ is bounded by \eqref{Dunford}. Indeed, by \cite[(24.48.7) p.~321]{Dieu82} or \cite[Remark 1.4.3 p.~31]{MaM07}, it is a smoothing integral operator. We infer that the operator $\e^{-t\Delta_{\bar\partial_E}} \co \L^1(\Omega^{r,q}(M,E)) \to \L^\infty(\Omega^{r,q}(M,E))$ is bounded with norm $\leq C$ for any $t \geq 2$. By \eqref{Dunford}, this implies that $\bnorm{p_t^{\Delta_{\bar\partial_E}}(x,y)} \leq C$ for any $t \geq 2$ and any $x,y \in M$. For $t \geq 2$, compactness of $M$ gives $V(x,\sqrt t) \simeq \Vol(M)$ and finally
$$
\exp(-c\diam(M)^2) 
\leq \exp\Bigl(-c\frac{\dist(x,y)^2}{t}\Bigr).
$$ 
\end{proof}

Now, we obtain a fundamental property of functional calculus of the Kodaira Laplacian relying on the extrapolation result Proposition \ref{prop:Duong-Robinson-vector-bundle}.

\begin{thm}
\label{Th-funct-Kodaira}
Consider a holomorphic Hermitian vector bundle $E$ of finite rank over a compact K\"ahler manifold $M$. Suppose that $1 < p < \infty$. The closure $\Delta_{\bar\partial_E,p}$ of the Kodaira Laplacian $\Delta_{\bar\partial_E}$ is sectorial on the Banach space $\L^p(\Omega^{\bullet,\bullet}(M,E))$ and admits a bounded $\H^\infty(\Sigma_\theta)$ functional calculus for some angle $\theta \in (0,\frac{\pi}{2})$.
\end{thm}

\begin{proof}
By Proposition \ref{prop:Gq-compact-Dolbeault}, the semigroup generated by the operator $\Delta_{\bar\partial_E,2}$ admits Gaussian kernel estimates on each summand $\Omega^{0,q}(M,E)$. Hence the kernel domination assumption \eqref{eq:kernel-domination-vector-bundle} of Proposition \ref{prop:Duong-Robinson-vector-bundle} is satisfied. Moreover, according to \cite[Proposition 3.1.2 p.~128]{MaM07} and \cite[Corollary 3.3.4 p.~149]{MaM07} the closed operator $\Delta_{\bar\partial_E,2}$ is a positive selfadjoint operator on the complex Hilbert space $\L^2(\Omega^{0,\bullet}(M,E))$. By Example~\ref{Example-analytic-positive-selfadjoint}, its negative generates a bounded holomorphic strongly continuous semigroup on the space $\L^2(\Omega^{0,\bullet}(M,E))$. Furthermore, by Example \ref{positive-selfadjoint} it admits a bounded $\H^\infty(\Sigma_\theta)$ functional calculus on the Hilbert space $\L^2(\Omega^{0,\bullet}(M,E))$ for any angle $\theta > 0$. We can therefore apply Proposition \ref{prop:Duong-Robinson-vector-bundle} and conclude that the closed operator $\Delta_{\bar\partial_E,p}$ admits a bounded $\H^\infty(\Sigma_\theta)$ functional calculus for some angle $\theta \in (0,\frac{\pi}{2})$ on the Banach space $\L^p(\Omega^{0,\bullet}(M,E))$.

Fix $r \in \{0,\ldots,n\}$ and set $E_r \ov{\mathrm{def}}{=} \Lambda^{r,0}\mathrm{T}^*M \ot E$. Then $E_r$ is a Hermitian holomorphic vector bundle. Using the canonical bundle isomorphism
\[
\Lambda^{r,q}\mathrm{T}^*M \ot E
\simeq
\Lambda^{0,q}\mathrm{T}^*M \ot (\Lambda^{r,0}\mathrm{T}^*M \ot E)
=
\Lambda^{0,q}\mathrm{T}^*M \ot E_r,
\]
we obtain identifications $
\Omega^{r,q}(M,E) \simeq \Omega^{0,q}(M,E_r)$ and $\L^p(\Omega^{r,q}(M,E)) \simeq \L^p(\Omega^{0,q}(M,E_r))$.
= Under these identifications, the Dolbeault operators correspond, namely for each $q$
\[
\bar\partial_{E_r}\big|_{\Omega^{0,q}(M,E_r)}
=
\bar\partial_E\big|_{\Omega^{r,q}(M,E)},
\qquad
\bar\partial_{E_r}^*\big|_{\Omega^{0,q}(M,E_r)}
=
\bar\partial_E^*\big|_{\Omega^{r,q}(M,E)}.
\]
Consequently, the Kodaira Laplacians correspond on each bidegree, i.e.~$
\Delta_{\bar\partial_{E_r}}\big|_{\Omega^{0,q}(M,E_r)}
=
\Delta_{\bar\partial_E}\big|_{\Omega^{r,q}(M,E)}$ and the same identity holds for their $\L^p$-closures by conjugation with the above isometric identifications. Applying the first part to the holomorphic Hermitian bundle $E_r$, we obtain that for each $r$ the operator $\Delta_{\bar\partial_E,p}$ restricted to $\L^p(\Omega^{r,\bullet}(M,E))$ is sectorial and admits a bounded $\H^\infty(\Sigma_{\theta_r})$ functional calculus for some $\theta_r \in (0,\frac{\pi}{2})$. We conclude by a direct sum argument.
\end{proof}

According to Theorem \ref{thm:Lp-dolbeault-hodge}, we have the $\L^p$-Bergman projection $
P_p \co \L^p(\Omega^{0,\bullet}(M,E))\to \L^p(\Omega^{0,\bullet}(M,E))$ with range $\Ran P_p=\cal H^{0,\bullet}(M,E)$. We introduce the restrictions 
\[
P_p^{\even}\co \L^p(\Omega^{0,\even}(M,E))\to \L^p(\Omega^{0,\even}(M,E)),
\quad
P_p^{\odd}\co \L^p(\Omega^{0,\odd}(M,E))\to \L^p(\Omega^{0,\odd}(M,E))
\]
of $P_p$ on spaces of even and odd degrees. We have $\Ran P_p^{\odd}=\cal H^{0,\odd}(M,E)$.

\begin{prop}
\label{prop-ker-HD}
Let $E$ be a Hermitian holomorphic vector bundle of finite rank over a compact K\"ahler manifold $M$. Suppose that $1 < p < \infty$. We have 
\begin{equation}
\label{description-kernel}
\ker \D_{E,+,p} 
= \cal{H}^{0,\even}(M,E)
\quad \text{and} \quad
\ker \D_{E,-,p} 
= \cal{H}^{0,\odd}(M,E).
\end{equation}
\end{prop}

\begin{proof}
%
Using Proposition \ref{prop-harmonic-Lp} in the third equality, we have
\begin{align*}
\MoveEqLeft
\ker \D_{E,+,p} \oplus \ker \D_{E,-,p}
\ov{\eqref{Dp-decompo-block}}{=}\ker \D_{E,p}
\ov{\eqref{Bisec-Ran-Ker}}{=} \ker \D_{E,p}^2 \\
&\ov{\eqref{useful-sans-fin}}{=}\ker \Delta_{\bar\partial_E,0,\bullet,p}
=\cal{H}^{0,\bullet}(M,E)
=\cal{H}^{0,\even}(M,E) \oplus \cal{H}^{0,\odd}(M,E).         
\end{align*}
\end{proof}

Now, we identify the ranges.

\begin{prop}
\label{prop:range-Laplacian-harmonic-projection}
Let $E$ be a Hermitian holomorphic vector bundle of finite rank over a compact K\"ahler manifold $M$. Suppose that $1 < p < \infty$. We have
\begin{equation}
\label{eq:range-Laplacian-kernel-projection}
\Ran \Delta_{\bar\partial_E,p}
=\ker P_p, \quad \Ran \D_{E,+,p} 
= \ker P_p^{\odd}
\quad \text{and} \quad 
\Ran \D_{E,-,p} 
= \ker P_p^{\even}.
\end{equation}
\end{prop}

\begin{proof}
Let $G$ be the Green operator of Proposition~\ref{prop-existence-Green}. By Proposition~\ref{prop-elliptic-Sobolev} and the continuous embedding $\W^{2,p}_\nabla(\Omega^{\bullet,\bullet}(M,E))
\hookrightarrow
\L^p(\Omega^{\bullet,\bullet}(M,E))$, 
the operator $G$ extends to a bounded operator on $\L^p(\Omega^{\bullet,\bullet}(M,E))$. Similarly, the harmonic projection $P$ extends to the bounded projection $P_p$.

We first prove that $
\ker P_p \subset \Ran \Delta_{\bar\partial_E,p}$. Let $\eta \in \ker P_p$. Choose a sequence
$(\eta_j)$ in $\Omega^{\bullet,\bullet}(M,E)$ such that $
\eta_j \to \eta$ in $\L^p(\Omega^{\bullet,\bullet}(M,E))$. Set $
u_j \ov{\mathrm{def}}{=} G\eta_j$. Since $G$ maps smooth forms to smooth forms, we have
$u_j \in \Omega^{\bullet,\bullet}(M,E)$. Moreover, $
u_j \to G\eta$ in $\L^p(\Omega^{\bullet,\bullet}(M,E))$. On smooth forms, the Green identity gives $
\Delta_{\bar\partial_E}u_j=\Delta_{\bar\partial_E}G\eta_j
\ov{\eqref{parametrix}}{=}
(\Id-P)\eta_j$. Since $P_p$ is the bounded extension of $P$, we have
\[
(\Id-P)\eta_j
=
(\Id-P_p)\eta_j
\to
(\Id-P_p)\eta
=
\eta
\quad\text{in } \L^p(\Omega^{\bullet,\bullet}(M,E)),
\]
since $\eta \in \ker P_p$. Thus $u_j \to G\eta$ and $\Delta_{\bar\partial_E}u_j \to \eta$ in $\L^p(\Omega^{\bullet,\bullet}(M,E))$. Since $\Delta_{\bar\partial_E,p}$ is the closure of $\Delta_{\bar\partial_E}$ it follows that $G\eta \in \dom \Delta_{\bar\partial_E,p}$ and $\Delta_{\bar\partial_E,p}G\eta=\eta$. Therefore $\eta \in \Ran \Delta_{\bar\partial_E,p}$, and so $\ker P_p \subset \Ran \Delta_{\bar\partial_E,p}$.

Conversely, we prove that $\Ran \Delta_{\bar\partial_E,p} \subset \ker P_p$. Let $u \in \dom \Delta_{\bar\partial_E,p}$. By definition of the closure, there exists a sequence $(u_j)$ in $\Omega^{\bullet,\bullet}(M,E)$ such that $u_j \to u$ and $\Delta_{\bar\partial_E}u_j \to \Delta_{\bar\partial_E,p}u$ in $\L^p(\Omega^{\bullet,\bullet}(M,E))$. If $h \in \cal{H}^{\bullet,\bullet}(M,E)$, then by formal selfadjointness we have 
$$
\big\la \Delta_{\bar\partial_E}u_j,h\big\ra_{\L^2}
=
\big\la u_j,\Delta_{\bar\partial_E}h\big\ra_{\L^2}
=
0
$$ for any $j$. Thus $\Delta_{\bar\partial_E}u_j$ is orthogonal to the harmonic forms, and hence $P\Delta_{\bar\partial_E}u_j=0$. Since $P_p$ extends $P$ and is bounded on $\L^p$, passing to the limit gives
\[
P_p\Delta_{\bar\partial_E,p}u
=
\lim_j P_p\Delta_{\bar\partial_E}u_j
=
\lim_j P\Delta_{\bar\partial_E}u_j
=0.
\]
Therefore $\Delta_{\bar\partial_E,p}u \in \ker P_p$. Hence $\Ran \Delta_{\bar\partial_E,p} \subset \ker P_p$. Combining both inclusions, we obtain $\Ran \Delta_{\bar\partial_E,p} = \ker P_p$. In particular, this space is closed. Finally, we have 
\begin{align}
\MoveEqLeft
\label{infinite}
\ker P_p^\odd \oplus \ker P_p^\even
=\ker P_p
=\Ran \Delta_{\bar\partial_E,0,\bullet,p}
=\ovl{\Ran \Delta_{\bar\partial_E,0,\bullet,p}}
\ov{\eqref{useful-sans-fin}}{=} \ovl{\Ran \D_{E,p}^2} \\
&\ov{\eqref{Bisec-Ran-Ker}}{=} \ovl{\Ran \D_{E,p}}
\ov{\eqref{Dp-decompo-block}}{=} \ovl{\Ran \D_{E,+,p}} \oplus \ovl{\Ran \D_{E,-,p}}. \nonumber        
\end{align}
Let $z \in \ker P_p^{\odd}$. By the first part of the proof, applied on odd forms, we have $
\ker P_p^{\odd}
=
\Ran \Delta_{\bar\partial_E,0,\odd,p}$. Hence there exists $v \in \dom \Delta_{\bar\partial_E,0,\odd,p}$ such that $
\Delta_{\bar\partial_E,0,\odd,p}v=z$. Since $
\D_{E,p}^2 \ov{\eqref{useful-sans-fin}}{=}2\Delta_{\bar\partial_E,p}$, 
we have on odd forms $
\D_{E,+,p}\D_{E,-,p}
=
2\Delta_{\bar\partial_E,0,\odd,p}$. In particular, we have $v \in \dom \D_{E,-,p}$ and $\D_{E,-,p}v \in \dom \D_{E,+,p}$. Then the element $x \ov{\mathrm{def}}{=}\frac12\D_{E,-,p}v$ belongs to $\dom \D_{E,+,p}$ and
\[
\D_{E,+,p}x
=
\frac12\D_{E,+,p}\D_{E,-,p}v
=
\Delta_{\bar\partial_E,0,\odd,p}v
=
z.
\]
Thus $
\ker P_p^{\odd}
\subset
\Ran\D_{E,+,p}$. The reverse inclusion follows from $\Ran \D_{E,+,p}
\subset
\ovl{\Ran\D_{E,+,p}}
\ov{\eqref{infinite}}{=}
\ker P_p^{\odd}$. Consequently, we have $\Ran\D_{E,+,p}=\ker P_p^{\odd}$. The proof of the identity $\Ran\D_{E,-,p}
=
\ker P_p^{\even}$ is similar.
\end{proof}

Now, we will use the results of Section \ref{sec-Hodge-Dirac} with the spaces $
X\ov{\mathrm{def}}{=}\L^p(\Omega^{0,\even}(M,E))$, $Y \ov{\mathrm{def}}{=}\L^p(\Omega^{0,\odd}(M,E))$, $C \ov{\mathrm{def}}{=} \Omega^{0,\even}(M,E)$ and the operators $\partial\ov{\mathrm{def}}{=} \D_{E,+,p}$, $\partial^\dagger \ov{\mathrm{def}}{=} \D_{E,-,p}$, $A \ov{\mathrm{def}}{=} 2\Delta_{\bar\partial_E,0,\even,p}$ and $\tilde{A} \ov{\mathrm{def}}{=} 2\Delta_{\bar\partial_E,0,\odd,p}$. 

\begin{prop}
\label{prop-Q-equals-Hodge-projection}
The abstract projection $Q=RS^*$ of \eqref{eq-def-projection-P} coincides with the Hodge projection away from harmonic odd forms, namely 
\begin{equation}
\label{eq:def-P-odd}
Q \ov{\mathrm{def}}{=} \Id_{\L^p(\Omega^{0,\odd}(M,E))}-P_p^{\odd}
\co \L^p(\Omega^{0,\odd}(M,E)) \to \L^p(\Omega^{0,\odd}(M,E)).
\end{equation}
\end{prop}

\begin{proof}
By Proposition~\ref{prop-P-projection}, we have $
\Ran Q
=\ovl{\Ran\partial}
=\Ran\partial
\ov{\eqref{eq:range-Laplacian-kernel-projection}}{=} \ker P_p^{\odd}$. By Proposition \ref{prop-ker-HD}, we have $
\ker\partial^\dagger
=\ker\D_{E,-,p} 
\ov{\eqref{description-kernel}}{=} \cal{H}^{0,\odd}(M,E)$. Consequently, using Theorem \ref{thm:Lp-dolbeault-hodge}, we obtain the topological direct sum decomposition 
$$
Y \ov{\mathrm{def}}{=}\L^p(\Omega^{0,\odd}(M,E))
=
\ker P_p^{\odd}
\oplus
\cal H^{0,\odd}(M,E)
=\ovl{\Ran\partial}+\ker\partial^\dagger.
$$ 
Hence \eqref{eq-hodge-splitting-Y} is satisfied. By Proposition \ref{prop-kernelQ-kernelpartialdagger}, we deduce that $\ker Q=\ker \partial^\dagger=\cal{H}^{0,\odd}(M,E)$ and that Assumption \ref{ass-projection-reduction} holds. We conclude that $
Q
=\Id_{\L^p(\Omega^{0,\odd}(M,E))}-P_p^{\odd}$. 
\end{proof}

The following corollary is the main result of this section.

\begin{cor}
\label{cor-Dolbeault-Dirac-functional-calculus}
Consider a holomorphic Hermitian vector bundle $E$ of finite rank over a compact K\"ahler manifold $M$. Suppose that $1 < p < \infty$. The Dolbeault--Dirac operator $\D_{E,p}$ is bisectorial and admits a bounded $\H^\infty(\Sigma^\bi_\theta)$ functional calculus for some angle $\theta \in (0,\frac{\pi}{2})$ on the Banach space $\L^p(\Omega^{0,\bullet}(M,E))$.
\end{cor}

\begin{proof}
Note that $
A^*=2\Delta_{\bar\partial_E,0,\even,p^*}$, $\tilde A^*=2\Delta_{\bar\partial_E,0,\odd,p^*}$ and $\partial^*=\D_{E,+,p}^*=\D_{E,-,p^*}$. We first prove the Riesz estimates. By Theorem~\ref{thm-Riesz-bounded}, the Dolbeault--Riesz transforms $\bar\partial_E\Delta_{\bar\partial_E}^{-\frac12}$ and $\bar\partial_E^*\Delta_{\bar\partial_E}^{-\frac12}$ are bounded on the Banach space $\L^p(\Omega^{\bullet,\bullet}(M,E))$. Hence, for every smooth form $x \in \Omega^{0,\even}(M,E)$, we have
\begin{align*}
\MoveEqLeft
\norm{\D_{E,+}x}_{\L^p(\Omega^{0,\odd}(M,E))}
\ov{\eqref{def-DEplus}}{=}
\norm{\sqrt{2}\big(\bar\partial_E x+\bar\partial_E^*x\big)}_{\L^p(\Omega^{0,\odd}(M,E))} \\
&\lesssim \norm{\bar\partial_E x}_{\L^p(\Omega^{0,\odd}(M,E))}
+
\norm{\bar\partial_E^*x}_{\L^p(\Omega^{0,\odd}(M,E))}
\lesssim
\bnorm{\big(\Delta_{\bar\partial_E,0,\even,p}\big)^{\frac12}x}_{\L^p(\Omega^{0,\even}(M,E))}.         
\end{align*}
A similar argument gives $
\norm{\D_{E,-}x}_{\L^p(\Omega^{0,\even}(M,E))}
\lesssim \bnorm{\big(\Delta_{\bar\partial_E,0,\odd,p}\big)^{\frac12}x}_{\L^p(\Omega^{0,\odd}(M,E))}$ if $y \in \Omega^{0,\odd}(M,E)$. Applying the same argument with $p^*$ in place of $p$ and using four times Proposition~\ref{prop-duality-Riesz}, we obtain the $\partial$-Riesz equivalence for $A$, the $(\partial^\dagger)^*$-Riesz equivalence for $A^*$ and the $\partial^*$-Riesz equivalence for $\tilde{A}^*$. By Proposition~\ref{Prop-derivation-closable-sgrp-bis}, we deduce that $\dom A^{\frac12}=\dom\partial$ and $\dom(\tilde{A}^*)^{\frac12}
=
\dom\partial^*$. Since $\dom A \subset \dom A^{\frac12}$ and $\dom \tilde{A}^* \subset \dom (\tilde{A}^*)^{\frac12}$ by \cite[Theorem 15.2.5 p.~438]{HvNVW23}, we obtain $\dom A \subset\dom\partial$, which is an assumption of Proposition~\ref{Prop-equivalences} and $\dom \tilde{A}^* \subset \dom \partial^*$. 

For any form $\omega \in C$, the commutation relation for the Kodaira Laplacian of Lemma \ref{lem:commutation-Kodaira} gives
\begin{align}
\MoveEqLeft
\label{final}
\partial A\omega
=2\D_{E,+}\Delta_{\bar\partial_E,0,\even}\omega
=2(\bar\partial_E^{\even}+(\bar\partial_E^*)^{\even}) \Delta_{\bar\partial_E,0,\even} \omega \\
&\ov{\eqref{commutation-rules}}{=}
2\Delta_{\bar\partial_E,0,\odd} (\bar\partial_E^{\even}+(\bar\partial_E^*)^{\even})\omega 
=2\Delta_{\bar\partial_E,0,\odd} \D_{E,+}\omega         
=\tilde A\partial\omega. \nonumber
\end{align}
Recall $A(C) \subset \dom \partial$. If $z \in\dom\tilde A^*$ then $z \in \dom\partial^*$ and the preceding identity gives for $\omega \in C$
\begin{equation}
\label{final-fin-33}
\la A\omega,\partial^*z \ra_{X,X^*}
=\la \partial A\omega,z \ra_{Y,Y^*}
\ov{\eqref{final}}{=}
\la \tilde A\partial\omega,z \ra_{Y,Y^*}
=
\la \partial\omega,\tilde A^*z \ra_{Y,Y^*}.
\end{equation}
It remains to extend this identity from $C$ to $\dom A$. By Lemma \ref{lem-density-range-partial-domA}, for any $x \in \dom A$ there exists a sequence $(\omega_j)$ of elements in $C$ such that $
\omega_j \to x$, $A \omega_j \to Ax$ in $X$ and $\partial \omega_j \to \partial x$ in $Y$. 
Passing to the limit in $
\la A\omega_j,\partial^*z\ra_{X,X^*}
\ov{\eqref{final-fin-33}}{=} 
\la \partial\omega_j,\tilde A^*z\ra_{Y,Y^*}$, 
we obtain
\[
\la Ax,\partial^*z\ra_{X,X^*}
=
\la \partial x,\tilde A^*z\ra_{Y,Y^*},
\quad x\in\dom A.
\]
Thus the bracket condition \eqref{bracket-useful} holds with $\lambda=0$.

The regularization property of Definition~\ref{def-regularization-coreC} is satisfied as well.
Indeed, we claim that
\begin{equation}
\label{reg-a-prouver}
\tilde T_t(\ovl{\Ran \D_{E,+,p}})
\subset
\D_{E,+}(\Omega^{0,\even}(M,E)),
\qquad t>0.
\end{equation}
Let $y \in \ovl{\Ran \D_{E,+,p}}$. By Proposition~\ref{prop-Q-equals-Hodge-projection} and Proposition \ref{prop-ker-HD}, we have
\[
\ovl{\Ran \D_{E,+,p}}
=
\Ran \D_{E,+,p}
=
\Ran Q
\ov{\eqref{eq:def-P-odd}}{=}
\ker P_p^{\odd}.
\]
Hence \(P_p^{\odd}y=0\). Since $M$ is compact and \(\Delta_{\bar\partial_E,0,\odd,p}\) is elliptic, the heat operator $\tilde{T}_t=\e^{-2t\Delta_{\bar\partial_E,0,\odd}}$ is smoothing for every $t > 0$ by \cite[Chapter~2]{BGV04}. According to \cite[(23.26.10)]{Dieu88} it is a pseudo-differential operator of order $-k$ for any integer $k \geq 0$. By Proposition~\ref{prop-elliptic-Sobolev}, it induces a bounded operator $
\tilde{T}_t
\co
\L^p(\Omega^{0,\odd}(M,E))
=
\W^{0,p}_\nabla(\Omega^{0,\odd}(M,E))
\to
\W^{k,p}_\nabla(\Omega^{0,\odd}(M,E))$ for any $k \geq 0$. We deduce that $z \ov{\mathrm{def}}{=}\tilde T_t y$ belongs to the intersection $\cap_{k \geq 0}\W^{k,p}_\nabla(\Omega^{0,\odd}(M,E))$ of the Sobolev space ${\W^{k,p}_{\nabla}(M,E)}$. According to \cite[Theorem 10.2.36 (d) p.~483]{Nic21}, this implies that $z \in \Omega^{0,\odd}(M,E)$. 

Note that by Proposition \ref{prop:range-Laplacian-harmonic-projection} and Proposition \ref{prop-harmonic-Lp}, the map $P_p^{\odd}$ is the projection onto $\ker \Delta_{\bar\partial_E,0,\odd,p}$ along the subspace $\ovl{\Ran \Delta_{\bar\partial_E,0,\odd,p}}=\Ran \Delta_{\bar\partial_E,0,\odd,p}$. So, $P_p^{\odd}$ commutes with the heat semigroup by \cite[pp.~262-263]{ABHN11}. Thus
\[
P_p^{\odd}z
=
P_p^{\odd}\tilde{T}_t y
=
\tilde{T}_t P_p^{\odd}y
=
0.
\]
Thus $z$ is a smooth odd form with vanishing harmonic component. Let $G_{\odd}$ denote the Green operator of the Kodaira Laplacian on odd forms. Since $P_p^{\odd}z=0$, the parametrix identity gives 
$
z
\ov{\eqref{parametrix}}{=} \Delta_{\bar\partial_E,0,\odd}G_{\odd}z$. 
On odd forms, the identity \(\D_E^2=2\Delta_{\bar\partial_E}\) gives 
$\D_{E,+}\D_{E,-}
=
2\Delta_{\bar\partial_E,0,\odd}$. Hence 
$$
z
=
\frac12\D_{E,+}\D_{E,-}G_{\odd}z.
$$ 
Since $z$ is smooth, $G_{\odd}z$ belongs to $\Omega^{0,\odd}(M,E)$. Set 
$
x
\ov{\mathrm{def}}{=}
\frac12\D_{E,-}G_{\odd}z$. Then \(x \in \Omega^{0,\even}(M,E)\) and $
\D_{E,+}x
=
\frac12\D_{E,+}\D_{E,-}G_{\odd}z
=
z$. 
Consequently, the element $
\tilde T_t y=z$ belongs to the space $\D_{E,+}(\Omega^{0,\even}(M,E))$. This proves \eqref{reg-a-prouver}.

Note that by \cite[Remark I.17 p.~13]{Gun17}, the space $\L^p(M,E)$ may be identified isometrically with the Bochner space $\L^p(M,\mathbb{C}^N)$. Hence $\L^p(M,E)$ is a UMD Banach space by \cite[pp.~291-292]{HvNVW16}. By Theorem \ref{Th-functional-calculus-bisector-Fourier}, we obtain that the ``reduced" Hodge--Dirac operator $\D_{E,p}|_{\L^p(\Omega^{0,\even}(M,E)) \oplus \Ran \D_{E,+,p}}$ is bisectorial and admits a bounded $\H^\infty(\Sigma^\bi_\theta)$ functional calculus for some angle $\theta \in (0,\frac{\pi}{2})$ on the Banach space 
\[
\L^p(\Omega^{0,\even}(M,E)) \oplus_2 \Ran \D_{E,+,p}.
\]
Now, we explain how to pass from this reduced Hodge-Dirac operator to the full Dolbeault--Dirac operator $\D_{E,p}$ from \eqref{Dp-decompo-block} acting on the space $\L^p(\Omega^{0,\even}(M,E))\oplus_2 \L^p(\Omega^{0,\odd}(M,E))$. The key point is to use the bounded projection $Q \co Y \to Y$ of \eqref{eq:def-P-odd}. By Proposition~\ref{prop-Q-equals-Hodge-projection}, Assumption \ref{ass-projection-reduction} holds for $(\partial,\partial^\dagger)$ and $Q$. By Theorem \ref{thm-full-bisectorial}, we conclude that $\D_{E,p}$ is bisectorial and admits a bounded $\H^\infty(\Sigma^\bi_\theta)$ functional calculus. 
\end{proof}

\section{Banach spectral triples from Dolbeault--Dirac operators on $\L^p(\Omega^{0,\bullet}(M,E))$}
\label{sec-Banach-spectral-triples}

\subsection{Commutators of the Dolbeault--Dirac operator $\D_{E,p}$ with multiplication operators}
\label{sec-commutators}

We will use the following elementary observation proved in \cite{Arh26a}.

\begin{lemma}
\label{lem:matrix-multiplication-Lp}
Let $(\Omega,\mu)$ be a measure space and let $1 \leq p < \infty$.  Let $A \co \Omega \to \M_N$ be a strongly measurable function such that $
\esssup_{x \in \Omega} \norm{A(x)}_{\M_N}
< \infty$. 
Define the multiplication operator
\[
M_A \co \L^p(\Omega,\mathbb{C}^N) \to \L^p(\Omega,\mathbb{C}^N), 
\qquad
(M_A u)(x) \ov{\mathrm{def}}{=} A(x)u(x).
\]
Then $M_A$ is bounded and its operator norm is given by
\begin{equation}
\label{eq:norm-matrix-mult}
\norm{M_A}_{\L^p(\Omega,\mathbb{C}^N)\to \L^p(\Omega,\mathbb{C}^N)}
=
\esssup_{x \in \Omega} \norm{A(x)}_{\M_N}.
\end{equation}
\end{lemma}

Recall that the interior product is the adjoint of the wedge multiplication. More precisely, by essentially \cite[(5.6) p.~21]{Dem96}, if $\theta \in \Omega^{0,1}(M)$, $\omega \in \Omega^{0,q}(M,E)$ and $\eta \in \Omega^{0,q-1}(M,E)$ we have
\begin{equation}
\label{dual-interior}
\la i_{\theta^\sharp}\omega, \eta \ra
=\la \omega,\theta \wedge \eta \ra_{},
\end{equation}
where $\theta^\sharp \in \C^\infty(M,\mathrm{T}^{0,1}M)$ is the metric dual of $\theta$ (with respect to the Hermitian metric) and $i_{\theta^\sharp}$ denotes the interior product. For any $\alpha \in \Omega^{0,q}(M)$ and any $\beta \in \Omega^{0,q'}(M,E)$, we have by \cite[Lemma 2.35 p.~59]{Voi07} the Leibniz rule
\begin{equation}
\label{Leibniz-partialE}
\bar{\partial}_{E}(\alpha \wedge \beta)
=\bar{\partial}\alpha \wedge \beta+(-1)^q \alpha \wedge \bar{\partial}_{E}\beta.
\end{equation}

\begin{prop}
\label{prop:commutator-Dolbeault-Dirac}
Let $E$ be a Hermitian holomorphic vector bundle of finite rank over a compact K\"ahler manifold $M$. For any function $f \in \C^\infty(M)$ and any differential form $\omega \in \Omega^{0,\bullet}(M,E)$, we have
\begin{equation}
\label{eq:commutator-D-Mf}
[\D_E,M_f](\omega)
=\sqrt{2}\big(\bar{\partial}  f \wedge \omega - i_{(\bar{\partial} f)^\sharp}\omega\big).
\end{equation}
Suppose that $1 < p < \infty$. Moreover, the commutator $[\D_E,M_f]$ extends to a bounded operator on the Banach space $\L^p(\Omega^{0,\bullet}(M,E))$ with
\begin{equation}
\label{estimation-commutator-HD}
\norm{[\D_E,M_f]}_{\L^p(\Omega^{0,\bullet}(M,E)) \to \L^p(\Omega^{0,\bullet}(M,E))}
=\sqrt{2} \norm{\bar\partial f}_{\L^\infty(\Omega^{0,1}(M))}.
\end{equation}
\end{prop}

\begin{proof}
Let $f \in \C^\infty(M)$. We split the commutator into two parts:
\[
[\D_E,M_f]
\ov{\eqref{Def-Dirac-DE}}{=} \sqrt{2}[\bar{\partial}_{E,0,\bullet} +\bar{\partial}_{E,0,\bullet}^*,M_f]
=\sqrt{2}[\bar{\partial}_{E,0,\bullet},M_f]+\sqrt{2}[\bar{\partial}_{E,0,\bullet}^*,M_f].
\]
For any differential form $\omega \in \Omega^{0,\bullet}(M,E)$, we have
\begin{align}
\label{inter-45}
[\bar{\partial}_{E},M_f](\omega)
&=\bar{\partial}_{E}(f\omega)-f\bar{\partial}_{E}\omega
\ov{\eqref{Leibniz-partialE}}{=} \bar{\partial} f \wedge \omega.
\end{align}
Now, we compute $[\bar{\partial}_{E}^*,M_f]$. Let $\omega \in \Omega^{0,q}(M,E)$ and $\eta \in \Omega^{0,q-1}(M,E)$. Using the definition of the $\L^2$-adjoint we obtain
\begin{align*}
\big\la \bar{\partial}_E^*(f\omega),\eta \big\ra_{\L^2(\Omega^{0,q-1}(M,E))}
&=\big\la f\omega,\bar{\partial}_E\eta \big\ra_{\L^2(\Omega^{0,q}(M,E))}
=\big\la \omega, f\bar{\partial}_E\eta \big\ra_{\L^2(\Omega^{0,q}(M,E))}.
\end{align*}
By the Leibniz rule, we have $
\bar{\partial}_E(f\eta)
\ov{\eqref{Leibniz-partialE}}{=} \bar{\partial} f \wedge \eta + f\bar{\partial}_E\eta$. Hence $f\bar{\partial}_E\eta=\bar{\partial}_E(f\eta)-\bar{\partial} f \wedge \eta$. We deduce that
\begin{align}
\label{first-second}
\big\la \bar{\partial}_E^*(f\omega),\eta \big\ra_{\L^2(\Omega^{0,q-1}(M,E))}
&=\big\la \omega,\bar{\partial}_E(f\eta) \big\ra_{\L^2(\Omega^{0,q}(M,E))}
-\la \omega,\bar{\partial} f \wedge \eta \big\ra_{\L^2(\Omega^{0,q}(M,E))}.
\end{align}
For the first term, by definition of $\bar{\partial}_E^*$, we see that
\begin{equation}
\label{first-term}
\big\la \omega,\bar{\partial}_E(f\eta) \big\ra_{\L^2(\Omega^{0,q}(M,E))}
=\big\la \bar{\partial}_E^*\omega,f\eta \big\ra_{\L^2(\Omega^{0,q-1}(M,E))}
=\big\la f\bar{\partial}_E^*\omega,\eta \big\ra_{\L^2(\Omega^{0,q-1}(M,E))}.
\end{equation}
For the second term, we use \eqref{dual-interior} with $\theta=\bar{\partial} f$:
\begin{equation}
\label{second-term}
\big\la \omega,\bar{\partial} f \wedge \eta \big\ra_{\L^2(\Omega^{0,q}(M,E))}
\ov{\eqref{dual-interior}}{=} \big\la i_{(\bar{\partial} f)^\sharp}\omega,\eta \big\ra_{\L^2(\Omega^{0,q-1}(M,E))}.
\end{equation}
Consequently, we obtain 
\[
\big\la \bar{\partial}_E^*(f\omega),\eta \big\ra_{\L^2(\Omega^{0,q-1}(M,E))}
\ov{\eqref{first-second} \eqref{first-term} \eqref{second-term}}{=} \big\la f\bar{\partial}_E^*\omega,\eta \big\ra_{\L^2(\Omega^{0,q-1}(M,E))}
-\big\la i_{(\bar{\partial} f)^\sharp}\omega,\eta \big\ra_{\L^2(\Omega^{0,q-1}(M,E))}.
\]
Since this holds for all $\eta \in \Omega^{0,q-1}(M,E)$, we deduce that
\begin{equation}
\label{eq:commutator-dbar-star-Mf}
[\bar{\partial}_E^*,M_f](\omega)
=\bar{\partial}_E^*(f\omega)-f\bar{\partial}_E^*\omega
= - i_{(\bar{\partial} f)^\sharp}\omega.
\end{equation}
Combining this with \eqref{inter-45}, we obtain
\begin{equation}
\label{commutator-proof}
[\D_E,M_f](\omega)
=\sqrt{2}\big(\bar{\partial} f \wedge \omega - i_{(\bar{\partial} f)^\sharp}\omega\big).
\end{equation}
This is the desired formula \eqref{eq:commutator-D-Mf}.

Now, we prove the second part of the proposition. Fix $x \in M$. Using the musical isomorphisms, we can identify $\Lambda^k \mathrm{T}_x^*M$ with the exterior power $\Lambda^k \mathrm{T}_xM$, and we equip $\Lambda^k \mathrm{T}_x^*M$ with the inner product induced by $g_x$.

Following \cite[Section 3.4, (3.63), (3.67)]{VaR19}, for any $v \in \mathrm{T}_xM$ we introduce the left exterior multiplication
$E(v) \co \Lambda \mathrm{T}_xM \to \Lambda \mathrm{T}_xM$ by $v$ defined by $
E(v)(A)
\ov{\mathrm{def}}{=} v \wedge A$, 
and for any covector $\beta \in \mathrm{T}_x^*M$ we consider the left contraction
$I(\beta) \co \Lambda \mathrm{T}_xM \to \Lambda \mathrm{T}_xM $ by $\beta$ defined by $
I(\beta)(A)
\ov{\mathrm{def}}{=} \beta \lrcorner A$. If we denote by $\mathrm{T}_xM \to \mathrm{T}_x^*M$, $v \mapsto v^\flat$ the map defined by $v^\flat(w)=g_x(v,w)$, then $I(v^\flat)$ is precisely the interior product $i_v$ acting on differential forms.

In \cite[(3.76)–(3.77) p.~75]{VaR19}, the authors introduce the linear maps $\gamma_+(v) ,\gamma_-(v) \co \mathrm{T}_xM \to \mathrm{End}(\Lambda \mathrm{T}_xM)$ defined by
\begin{equation}
\label{eq:def-Clifford-mappings}
\gamma_+(v)
\ov{\mathrm{def}}{=} E(v) + I(v^\flat),
\qquad
\gamma_-(v)
\ov{\mathrm{def}}{=} E(v) - I(v^\flat),
\quad v \in \mathrm{T}_xM,
\end{equation}
and prove in \cite[Theorem 3.6 p.~76]{VaR19} that for all $u,v \in \mathrm{T}_xM$ these operators satisfy the Clifford relations
\begin{equation}
\label{eq:Clifford-relations-Gamma}
\gamma_+(v)\gamma_+(u) + \gamma_+(u)\gamma_+(v)
=2 g_x(v,u)\Id,
\qquad
\gamma_-(v)\gamma_-(u) + \gamma_-(u)\gamma_-(v)
=-2 g_x(v,u)\Id.
\end{equation}
In particular, taking $u=v$ and dividing by $2$, we obtain
\begin{equation}
\label{eq:Gamma-square}
\gamma_+(v)^2
=g_x(v,v)\Id,
\qquad
\gamma_-(v)^2
=-g_x(v,v)\Id,
\quad v \in \mathrm{T}_xM.
\end{equation}
In our notation, for each covector $\alpha \in \mathrm{T}_x^*M$ we set $v \ov{\mathrm{def}}{=} \alpha^\sharp \in \mathrm{T}_xM$. The endomorphism
$$
C_\alpha \co \Lambda^{0,\bullet} \mathrm{T}_x^*M \to \Lambda^{0,\bullet} \mathrm{T}_x^*M, \omega \mapsto \alpha \wedge \omega - i_{\alpha^\sharp}\omega
$$ 
corresponds, under the previous identification, exactly to the operator $\gamma_-(v)$. Hence, by \eqref{eq:Gamma-square} we have
\begin{equation}
\label{eq:Calpha-square}
C_\alpha^2
=\gamma_-(v)^2
\ov{\eqref{eq:Gamma-square}}{=} -g_x(v,v)\Id_{\Lambda^{0,\bullet} \mathrm{T}_x^*M}
=-|\alpha|_{\Lambda^{0,1} \mathrm{T}_x^*M}^2 \Id_{\Lambda^{0,\bullet} \mathrm{T}_x^*M}.
\end{equation}
On the other hand, the adjointness relation between wedge and interior product \eqref{dual-interior} implies that, for any $\eta,\omega \in \Lambda^{0,\bullet} \mathrm{T}_x^*M$,
\begin{align*}
\la C_\alpha \eta,\omega \ra_x
&=\la \alpha \wedge \eta,\omega \ra_x - \la i_{\alpha^\sharp}\eta,\omega \ra_x
\ov{\eqref{dual-interior}}{=} \la \eta,i_{\alpha^\sharp}\omega \ra_x - \la \eta,\alpha \wedge \omega \ra_x \\
&=-\big\la \eta,(\alpha \wedge \omega - i_{\alpha^\sharp}\omega) \big\ra_x
=-\la \eta,C_\alpha \omega \ra_x.
\end{align*}
Thus $C_\alpha$ is skew-adjoint on the finite-dimensional Hilbert space $\Lambda^{0,\bullet} \mathrm{T}_x^*M$, i.e.
\begin{equation}
\label{eq:Calpha-skew-adjoint}
C_\alpha^*
=-C_\alpha.
\end{equation}
Combining \eqref{eq:Calpha-square} and \eqref{eq:Calpha-skew-adjoint}, we obtain $
C_\alpha^* C_\alpha
\ov{\eqref{eq:Calpha-skew-adjoint}}{=} -C_\alpha^2
\ov{\eqref{eq:Calpha-square}}{=} |\alpha|_{\Lambda^{0,1} \mathrm{T}_x^*M}^2 \Id_{\Lambda^{0,\bullet} \mathrm{T}_x^*M}$. 
Hence, we infer that
\begin{equation}
\label{eq:Calpha-norm}
\norm{C_\alpha}_{\Lambda^{0,\bullet} \mathrm{T}_x^*M \to \Lambda^{0,\bullet} \mathrm{T}_x^*M}
=|\alpha|_{\Lambda^{0,1} \mathrm{T}_x^*M}.
\end{equation}
Now, we apply this fibrewise description to the commutator. For each $x \in M$ and any function $f \in \C^\infty(M)$, consider the element $\alpha_x \ov{\mathrm{def}}{=} \bar\partial f(x)$ of $\Lambda^{0,1}\mathrm{T}_x^*M$. By \eqref{eq:commutator-D-Mf} and the previous identification, for any $\omega \in \Omega^{0,\bullet}(M,E)$ we have
\begin{equation}
\label{inter-A-525}
[\D_E,M_f](\omega)(x)
\ov{\eqref{commutator-proof}}{=} \sqrt{2}(\bar\partial f(x) \wedge \omega(x) - i_{(\bar{\partial} f(x))^\sharp}\omega(x))
=\sqrt{2}(C_{\alpha_x} \ot \Id_{E_x})(\omega(x)).
\end{equation}
Using \eqref{eq:Calpha-norm} with $\alpha=\alpha_x$, we obtain
\begin{align}
\MoveEqLeft
\label{ref-88UY}
|[\D_E,M_f](\omega)(x)|_{\Lambda^{0,\bullet}\mathrm{T}_x^*M \ot E_x}
\ov{\eqref{inter-A-525}}{=} 
\sqrt{2}\bigl|(C_{\alpha_x}\ot \Id_{E_x})(\omega(x))\bigr|_{\Lambda^{0,\bullet}\mathrm{T}_x^*M \ot E_x} \\
&\leq \sqrt{2}\,\norm{C_{\alpha_x}}_{\Lambda^{0,\bullet}\mathrm{T}_x^*M \to \Lambda^{0,\bullet}\mathrm{T}_x^*M}\,
|\omega(x)|_{\Lambda^{0,\bullet}\mathrm{T}_x^*M \ot E_x} \nonumber\\
&\ov{\eqref{eq:Calpha-norm}}{\leq}\sqrt{2} |\alpha_x|_{\Lambda^{0,1} \mathrm{T}_x^*M} |\omega(x)|_{\Lambda^{0,\bullet} \mathrm{T}_x^*M \ot E_x} \nonumber\\
&=\sqrt{2}|\bar\partial f(x)|_{\Lambda^{0,1} \mathrm{T}_x^*M} |\omega(x)|_{\Lambda^{0,\bullet} \mathrm{T}_x^*M \ot E_x}
\leq \sqrt{2}\norm{\bar\partial f}_{\L^\infty(\Omega^{0,1}(M))} |\omega(x)|_{\Lambda^{0,\bullet} \mathrm{T}_x^*M \ot E_x}.  \nonumber
\end{align}
Taking the $\L^p$-norm, we deduce that
\begin{align*}
\MoveEqLeft
\norm{[\D_E,M_f](\omega)}_{\L^p(\Omega^{0,\bullet}(M,E))}
\ov{\eqref{norm-Lp-Df}}{=}
\bigg(\int_M |[\D_E,M_f](\omega)(x)|_{\Lambda^\bullet \mathrm{T}_x^*M \ot E_x}^{p} \d \mu_g(x)\bigg)^{\frac{1}{p}} \\
&\ov{\eqref{ref-88UY}}{\leq} \sqrt{2}\norm{\bar\partial f}_{\L^\infty(\Omega^{0,1}(M))}
\bigg(\int_M |\omega(x)|_{\Lambda^\bullet \mathrm{T}_x^*M \ot E_x}^{ p} \d \mu_g(x)\bigg)^{\frac{1}{p}} \\
&\ov{\eqref{norm-Lp-Df}}{=}
\sqrt{2}\norm{\bar\partial f}_{\L^\infty(\Omega^{0,1}(M))} \norm{\omega}_{\L^p(\Omega^{0,\bullet}(M,E))}.
\end{align*}
Hence $[\D_E,M_f]$ extends by density to a bounded operator on $\L^p(\Omega^{0,\bullet}(M,E))$ and satisfies
\begin{equation}
\label{ine-34}
\norm{[\D_E,M_f]}_{\L^p(\Omega^{0,\bullet}(M,E)) \to \L^p(\Omega^{0,\bullet}(M,E))}
\leq \sqrt{2}\norm{\bar\partial f}_{\L^\infty(\Omega^{0,1}(M))}.
\end{equation}
Fix a local orthonormal frame of the bundle $\Omega^{0,\bullet}(M,E)$ on a coordinate chart $U \subset M$ over which the bundle is trivialized. By \cite[Remark I.17]{Gun17}, this yields an isometric identification of the Banach space $\L^p(\Omega^{0,\bullet}(U,E))$ with the Bochner space $\L^p(U,\mathbb{C}^N)$, where $N \ov{\mathrm{def}}{=} \dim \Lambda^{0,\bullet} \mathrm{T}_x^*M \ot E_x$ does not depend on $x$. In this trivialization, the family of endomorphisms $C_{\alpha_x}$ is represented by a strongly measurable matrix-valued function $
A \co U \to \M_N$, $x \mapsto \sqrt{2} C_{\alpha_x}$, and the operator $[\D_E,M_f]$ restricted to $\L^p(\Omega^{0,\bullet}(U,E))$ corresponds exactly to the multiplication operator
\[
M_A \co \L^p(U,\mathbb{C}^N) \to \L^p(U,\mathbb{C}^N),
\qquad
(M_A u)(x) \ov{\mathrm{def}}{=} A(x)u(x).
\]
For any $x \in U$, we have
\begin{align*}
\MoveEqLeft
\norm{A(x)}_{\M_N}
=\norm{\sqrt{2}(C_{\alpha_x}\ot \Id_{E_x})}_{\Lambda^{0,\bullet}\mathrm{T}_x^*M \ot E_x \to \Lambda^{0,\bullet}\mathrm{T}_x^*M \ot E_x} \\
&=\sqrt{2}\norm{C_{\alpha_x}}_{\Lambda^{0,\bullet}\mathrm{T}_x^*M \to \Lambda^{0,\bullet}\mathrm{T}_x^*M}
\ov{\eqref{eq:Calpha-norm}}{=} \sqrt{2}|\bar\partial f(x)|_{\Lambda^{0,1} \mathrm{T}_x^*M}.
\end{align*}
Hence $
\esssup_{x \in U} \norm{A(x)}_{\M_N}
=\sqrt{2}\norm{\bar\partial f}_{\L^\infty(\Omega^{0,1}(U))}$. Applying Lemma~\ref{lem:matrix-multiplication-Lp} with $\Omega=U$ and this matrix field $A$, we obtain
\begin{equation}
\label{eq:norm-local-commutator}
\norm{[\D_E,M_f]}_{\L^p(\Omega^{0,\bullet}(U,E)) \to \L^p(\Omega^{0,\bullet}(U,E))}
= \norm{M_A}_{\L^p(U,\mathbb{C}^N) \to \L^p(U,\mathbb{C}^N)}
\ov{\eqref{eq:norm-matrix-mult}}{=} \sqrt{2}\norm{\bar\partial f}_{\L^\infty(\Omega^{0,1}(U))}.
\end{equation}
We denote by $
T \co \L^p(\Omega^{0,\bullet}(U,E)) \to \L^p(\Omega^{0,\bullet}(U,E))$ the local version of the commutator $[\D_E,M_f]$ acting on $\L^p(\Omega^{0,\bullet}(U,E))$. Then \eqref{eq:norm-local-commutator} tells us that
\begin{equation}
\label{eq:norm-Tj}
\norm{T}_{\L^p(\Omega^{0,\bullet}(U,E)) \to \L^p(\Omega^{0,\bullet}(U,E))}
=\sqrt{2}\norm{\bar\partial f}_{\L^\infty(\Omega^{0,1}(U))}.
\end{equation}
Next, consider the isometric embedding
\[
J \co \L^p(\Omega^{0,\bullet}(U,E)) \to \L^p(\Omega^{0,\bullet}(M,E)),
\qquad
(J\omega)(x) \ov{\mathrm{def}}{=}
\begin{cases}
\omega(x)&\text{if }x \in U,\\
0&\text{if }x \in M\setminus U.
\end{cases}
\]
Since the commutator $[\D_E,M_f]$ is given pointwise by the formula \eqref{eq:commutator-D-Mf}, it follows that for any $\omega \in \Omega^{0,\bullet}(M,E)$ supported in $U$ we have
\[
[\D_E,M_f](\omega)(x)
=
\begin{cases}
T(\omega|_{U})(x)&\text{if }x \in U,\\
0&\text{if }x \in M \setminus U
\end{cases}.
\]
In particular, for every $\eta \in \L^p(\Omega^{0,\bullet}(U,E))$ we obtain
\[
\norm{[\D_E,M_f](J\eta)}_{\L^p(\Omega^{0,\bullet}(M,E))}
=\norm{T\eta}_{\L^p(\Omega^{0,\bullet}(U,E))},
\qquad
\norm{J\eta}_{\L^p(\Omega^{0,\bullet}(M,E))}
=\norm{\eta}_{\L^p(\Omega^{0,\bullet}(U,E))}.
\]
Hence
\begin{align*}
\MoveEqLeft
\norm{[\D_E,M_f]}_{\L^p(\Omega^{0,\bullet}(M,E)) \to \L^p(\Omega^{0,\bullet}(M,E))}
\geq \sup_{\eta \neq 0} \frac{\norm{[\D_E,M_f](J\eta)}_{\L^p(\Omega^{0,\bullet}(M,E))}}{\norm{J\eta}_{\L^p(\Omega^{0,\bullet}(M,E))}}\\
&= \sup_{\eta \neq 0} \frac{\norm{T\eta}_{\L^p(\Omega^{0,\bullet}(U,E))}}{\norm{\eta}_{\L^p(\Omega^{0,\bullet}(U,E))}}
=\norm{T}_{\L^p(\Omega^{0,\bullet}(U,E)) \to \L^p(\Omega^{0,\bullet}(U,E))}
\ov{\eqref{eq:norm-Tj}}{=} \sqrt{2}\norm{\bar\partial f}_{\L^\infty(\Omega^{0,1}(U))}.
\end{align*}
Observing that $\norm{\bar\partial f}_{\L^\infty(\Omega^{0,1}(M))} = \sup_U \norm{\bar\partial f}_{\L^\infty(\Omega^{0,1}(U))}$, we obtain the lower bound
\begin{equation}
\label{eq:global-lower-bound}
\norm{[\D_E,M_f]}_{\L^p(\Omega^{0,\bullet}(M,E)) \to \L^p(\Omega^{0,\bullet}(M,E))}
\geq \sqrt{2}\norm{\bar\partial f}_{\L^\infty(\Omega^{0,1}(M))}.
\end{equation}
Combining, with \eqref{ine-34} we obtain the equality.
\end{proof}

\subsection{The even compact Banach spectral triple $(\C(M), \L^p(\Omega^{0,\bullet}(M,E)), \D_{E,p},\gamma)$}
\label{sec-compact-spectral-triple}

\paragraph{Compact Banach spectral triples}
We briefly review the notion of compact Banach spectral triple introduced in \cite{Arh26a}, see also \cite[Definition 5.10 p.~218]{ArK22} for a closely related definition. Consider a triple $(\cal{A},X,D)$ where $X$ is a Banach space, $D$ is a bisectorial operator on $X$ with dense domain $\dom D$, and $\cal{A}$ is a Banach algebra represented continuously on $X$ through a homomorphism $\pi \co \cal{A} \to \B(X)$. From a heuristic perspective, Banach spectral triples provide a Banach-space framework for noncommutative geometry. The algebra $\cal{A}$ is interpreted as an algebra of continuous functions on a noncommutative space, whereas the operator $D$ encodes a differential structure of first order. The boundedness of commutators $[D,\pi(a)]$ plays the role of a Lipschitz-type regularity condition on the elements of $\cal{A}$, and the density of the corresponding Lipschitz algebra ensures that this regularity structure is sufficiently rich. The compactness of the resolvent of $D$ is the analogue of compactness in the classical setting.

The triple $(\cal{A},X,D)$ is called a compact Banach spectral triple whenever the following conditions hold:
\begin{enumerate}
\item{} the operator $D$ admits a bounded $\H^\infty(\Sigma^\bi_\theta)$ functional calculus on a bisector $\Sigma^\bi_\theta$ for some angle $\theta \in (\omega_{\bi}(D),\frac{\pi}{2})$,
\item{} $D$ has compact resolvent,  meaning that there exists $\lambda$ belonging to the resolvent set $\rho(D)\ov{\mathrm{def}}{=} \mathbb{C}\backslash \sigma(D)$ such that the resolvent $R(\lambda,D)$ is compact,
\item{} the Lipschitz algebra
\begin{align}
\label{Lipschitz-algebra-def}
\MoveEqLeft
\Lip_D(\cal{A}) 
\ov{\mathrm{def}}{=} \big\{a \in \cal{A} : \pi(a) \cdot \dom D \subset \dom D 
\text{ and the unbounded operator } \\
&\qquad \qquad  [D,\pi(a)] \co \dom D \subset X \to X \text{ extends to an element of } \B(X)\big\}. \nonumber
\end{align} 
is dense in the algebra $\cal{A}$.
\end{enumerate}
If, in addition, there exists an involution $\gamma \co X \to X$ satisfying $\gamma^2=\Id$ together with
\begin{equation}
\label{grading-ST}
D\gamma=-\gamma D
\quad \text{and} \quad
\gamma\pi(a)=\pi(a)\gamma,
\quad a\in\cal{A},
\end{equation}
then the spectral triple is said to be even (or graded). Otherwise, it is called odd (or ungraded).

\paragraph{Chern connections}
Consider a holomorphic vector bundle $E$ of finite rank over a complex compact manifold $M$. Since $\Omega^{1}(M,E)=\Omega^{1,0}(M,E) \oplus \Omega^{0,1}(M,E)$, we can decompose any connection $\nabla$ on $E$ in its two components $\nabla^{1,0} \co \Omega^{0}(M,E) \to \Omega^{1,0}(M,E)$ and $\nabla^{0,1} \co \Omega^{0}(M,E) \to \Omega^{0,1}(M,E)$, i.e., $\nabla=\nabla^{1,0} \oplus \nabla^{0,1}$. Following \cite[Definition 4.2.12 p.~177]{Huy05} and \cite[p.~207]{Lee24}, we say that a connection $\nabla$ on $E$ is compatible with the holomorphic structure if $\nabla^{0,1}=\bar\partial_E$.

If $E$ is a Hermitian holomorphic vector bundle of finite rank, according to \cite[Theorem 7.17 p.~208]{Lee24} or \cite[Proposition 4.2.14 p.~177]{Huy05} there exists a unique Hermitian connection $\nabla^{E,\Ch}$ on $E$ which is compatible with the holomorphic structure. This connection is called the Chern connection on $E$.

\paragraph{The Banach spectral triple $(\C(M), \L^p(\Omega^{0,\bullet}(M,E)), \D_{E,p})$}
Now, we explain the construction of the spectral triple investigated in this paper. Let $E$ be a Hermitian holomorphic vector bundle of finite rank over a compact K\"ahler manifold $M$. For any function $f \in \C(M)$, we introduce the multiplication operator $M_f \co \L^p(\Omega^{0,\bullet}(M,E)) \to \L^p(\Omega^{0,\bullet}(M,E))$, $\omega \mapsto f\omega$. We can consider the representation $\pi \co \C(M) \to \B(\L^p(\Omega^{0,\bullet}(M,E)))$, $f \mapsto M_f$ by multiplication operators, i.e.
\begin{equation}
\label{def-de-pi}
\pi(f)\omega 
\ov{\mathrm{def}}{=} f\omega,\quad f \in \C(M), \omega \in \L^p(\Omega^{0,\bullet}(M,E)).
\end{equation}

Now, we prove our first main result.

\begin{thm}
\label{thm:locally-compact-Hodge-Dirac}
Let $E$ be a Hermitian holomorphic vector bundle of finite rank over a compact K\"ahler manifold $M$. Suppose that $1 < p < \infty$. The quadruple $(\C(M), \L^p(\Omega^{0,\bullet}(M,E)), \D_{E,p},\gamma)$ is an even compact Banach spectral triple.
\end{thm}

\begin{proof}
The first point of the definition of a spectral triple is obviously satisfied by Corollary \ref{cor-Dolbeault-Dirac-functional-calculus}. Moreover, the algebra $\C^\infty(M)$ is a subset of $\Lip_{\D_{E,p}}(\C(M))$ by Proposition \ref{prop:commutator-Dolbeault-Dirac}. Since the subspace $\C^\infty(M)$ is dense in the algebra $\C(M)$, the Lipschitz algebra $\Lip_{\D_{E,p}}(\C(M))$ is dense in $\C(M)$ and the third point is satisfied. 

Now, we prove the second point which is that the operator $\D_{E,p}$ has compact resolvent. Since $\D_{E,p}$ is bisectorial, we have $\pm \i \in \rho(\D_{E,p})$. We will show that $R(\i,\D_{E,p})$ is a compact operator on the Banach space $\L^p(\Omega^{0,\bullet}(M,E))$. Recall that by Proposition \ref{prop:graph-norm-dolbeault-Lp}, the domain of the closed operator $\D_{E,p}$ equipped with the graph norm identifies to the Sobolev space $\W^{1,p}_{\bar\partial_E}(\Omega^{0,\bullet}(M,E))$. The  operator 
$$
(\i-\D_{E,p})^{-1} \co \L^p(\Omega^{0,\bullet}(M,E)) \to \W^{1,p}_{\bar\partial_E}(\Omega^{0,\bullet}(M,E))
$$ 
is bounded by definition\footnote{\thefootnote. Consider a closed operator $T$ on a Banach space $X$. We equip $\dom T$ with the graph norm, which is a Banach space since $T$ is closed. For any $\lambda \in \rho(T)$, the bijection $\lambda - T \co \dom T \to X$ is continuous. By the bounded inverse theorem, the resolvent operator $R(\lambda,T)$ is continuous from $X$ onto the space $\dom T$.}. 

Recall that the (complexified) Levi--Civita connection $\nabla^\LC \co \C^\infty(M,\mathrm{T}M) \to \Omega^1(M,\mathrm{T}M)$ induces by \cite[p.~322]{GHV73} a Hermitian connection on the exterior bundle $\Lambda \mathrm{T}^*M$. Since $M$ is K\"ahlerian, the Levi--Civita connection preserves by \cite[Theorem 8.10 p.~230]{Lee24} the holomorphic tangent bundle, in the sense that
$\nabla^{\LC}(\Gamma(T' M)) \subset \Gamma(T' M)$ (equivalently, $\nabla^{\LC}J=0$). Consequently, the induced connection on $\Lambda T^*M$ preserves the type decomposition of the complexified cotangent bundle, hence it restricts to a Hermitian connection on each bundle $\Lambda^{p,q}T^*M$, and in particular on $\Lambda^{0,q}T^*M$. 
Using the construction described in \ref{tensor-product-connections} with this connection and the Chern connection $\nabla^{E,\Ch}$ (or any Hermitian connection $\nabla^E$ on the vector bundle $E$), we obtain a Hermitian connection $\nabla$ on the vector bundle $\Lambda^{0,q}T^*M \ot E$. 

Recall that by Corollary \ref{cor:Gaffney-dolbeault-Lp}, we have an isomorphism 
$
\W^{1,p}_{\bar\partial_E}(\Omega^{0,q}(M,E)) 
\approx \W^{1,p}_\nabla(\Omega^{0,q}(M,E)).
$  
Taking the direct sum over $q$, we obtain an isomorphism
$$
\W^{1,p}_{\bar\partial_E}(\Omega^{0,\bullet}(M,E)) 
\approx \W^{1,p}_\nabla(\Omega^{0,\bullet}(M,E)).
$$
We denote by $J \co \W^{1,p}_\nabla(\Omega^{0,\bullet}(M,E)) \to \L^p(\Omega^{0,\bullet}(M,E))$ the canonical injection, which is a compact operator by \eqref{Sobolev-compact-embedding}. Then the resolvent operator $R(\i,\D_{E,p})$ admits the factorization
\[
\L^p(\Omega^{0,\bullet}(M,E)) \xra{(\i-\D_{E,p})^{-1}} \W^{1,p}_{\bar\partial_E}(\Omega^{0,\bullet}(M,E))
\xra{\approx} \W^{1,p}_\nabla(\Omega^{0,\bullet}(M,E))
 \xra{J} \L^p(\Omega^{0,\bullet}(M,E)).
\]
By the ideal property \cite[Theorem 4.8 p.~158]{Kat76} of the space of compact operators, we obtain the compactness. By Proposition~\ref{prop:even-Dolbeault-Dirac}, the quadruple $(\C(M),\L^p(\Omega^{0,\bullet}(M,E)),\D_{E,p},\gamma)$ is an even compact Banach spectral triple.
\end{proof}

\section{Coupling with K-theory and index formula}
\label{sec:dolbeault-index-pairing}
\label{sec-coupling}

\subsection{Pairing between K-theory and Banach K-homology}

Let $E$ be a Hermitian holomorphic vector bundle of finite rank over a compact K\"ahler manifold $M$. We set $
\cal{H}^{0,\even}(M,E) 
\ov{\mathrm{def}}{=} \bigoplus_{k \text{ even}} \cal{H}^{0,k}(M,E)$ and $
\cal{H}^{0,\odd}(M,E) 
\ov{\mathrm{def}}{=} \bigoplus_{k \text{ odd}} \cal{H}^{0,k}(M,E)$.  
Moreover, we let
\begin{equation}
\label{def-de-scrD}
\scr{D}
\ov{\mathrm{def}}{=} |\D_{E,p}| + P_p.
\end{equation}
We first give an elementary decomposition lemma.

\begin{lemma}
\label{lem:decomp-W1p-with-H}
Let $X$ be a Banach space. Consider a finite-dimensional subspace $F$ of $X$. Assume that there exists a bounded projection $P \co X \to X$ on $F$. If $Z$ is any closed subspace of $X$ such that $F \subset Z$, then we have a topological direct sum 
$Z 
= F \oplus (Z \cap \ker P)$.
\end{lemma}

\begin{proof}
The restriction $P|_{Z} \co Z \to Z$ is a bounded projection on $F$ with kernel $Z \cap \ker P$.
\end{proof}

\begin{prop}
\label{prop-isomorphism}
Let $E$ be a Hermitian holomorphic vector bundle of finite rank over a compact K\"ahler manifold $M$. Suppose that $1 < p < \infty$. The linear map
\begin{equation}
\label{eq:S-iso}
\scr{D} \co \W^{1,p}_{\bar\partial_E}(\Omega^{0,\bullet}(M,E)) \xrightarrow{\ \cong\ } \L^p(\Omega^{0,\bullet}(M,E))
\end{equation}
is an isomorphism of Banach spaces.
\end{prop}

\begin{proof}
By \eqref{eq:Lp-Hodge-decomp}, we have a direct sum decomposition $\L^p(\Omega^{0,\bullet}(M,E))= \cal{H}^{0,\bullet}(M,E) \oplus \ker P_p$, where $\cal{H}^{0,\bullet}(M,E)$ is the finite-dimensional space of harmonic forms (which is a subset of $\W^{1,p}_{\bar\partial_E}(\Omega^{0,\bullet}(M,E))$). Now, we use Lemma \ref{lem:decomp-W1p-with-H} with $X \ov{\mathrm{def}}{=} \L^p(\Omega^{0,\bullet}(M,E))$, $F \ov{\mathrm{def}}{=} \cal{H}^{0,\bullet}(M,E)$ and $Z \ov{\mathrm{def}}{=} \W^{1,p}_{\bar\partial_E}(\Omega^{0,\bullet}(M,E))$. Using the notation
\begin{equation}
\label{def-Fp}
F_p
\ov{\mathrm{def}}{=} \W^{1,p}_{\bar\partial_E}(\Omega^{0,\bullet}(M,E)) \cap \ker P_p,
\end{equation}
this gives a topological direct sum $
\W^{1,p}_{\bar\partial_E}(\Omega^{0,\bullet}(M,E))
= \cal{H}^{0,\bullet}(M,E) \oplus F_p$. Since $\ker \D_{E,p} = \cal{H}^{0,\bullet}(M,E)$, the action of $\scr{D}$ on the first summand is given by
\[
\scr{D}\big|_{\cal{H}^{0,\bullet}(M,E)} 
\ov{\eqref{def-de-scrD}}{=} (|\D_{E,p}| + P_p )\big|_{\cal{H}^{0,\bullet}(M,E)} 
= \Id_{\cal{H}^{0,\bullet}(M,E)}.
\]

According to \cite[Theorem 13.6 a) p.~318]{BlB13} the Dolbeault--Dirac operator $\D_E$ is an elliptic differential operator of order 1. By \cite[17.13.5 p.~298]{Dieu72}, 
we deduce that its square $\D_E^2$ is a differential operator of order 2, which is elliptic by \cite[p.~277]{Dieu88}. Finally, by \cite[Theorem 1 p.~124]{Bur68} \cite{See67} \cite{Shu01}, we conclude that $|\D_E|=(\D_E^2)^{\frac{1}{2}}$ is an elliptic pseudo-differential operator of order $1$. Note that this operator coincides with the restriction of the operator $|\D_{E,p}|$. By Proposition \ref{prop-elliptic-Sobolev} and Corollary \ref{cor:Gaffney-dolbeault-Lp}, we obtain a continuous map $|\D_{E,p}| \co \W^{1,p}_{\bar\partial_E}(\Omega^{0,\bullet}(M,E)) \to \L^p(\Omega^{0,\bullet}(M,E))$. By restriction, we obtain a bounded map $
|\D_{E,p}|\big|_{F_p} \co F_p \to \L^p(\Omega^{0,\bullet}(M,E))$, 
which is injective, since 
$$
\ker |\D_{E,p}| =\ker \D_{E,p}= \cal{H}^{0,\bullet}(M,E).
$$ 
Now, we describe its range. Using \cite[Theorem 3.2.20]{Ege15} in the first equality, we obtain
\begin{equation}
\label{inter-end}
\ovl{\Ran |\D_{E,p}|} 
= \ovl{\Ran (\D_{E,p}^2)^{\frac{1}{2}}} 
\ov{\textrm{Prop.~}\ref{prop:closure-kodaira-square-dirac}}{=} \ovl{\Ran \Delta_{\bar\partial_E,0,\bullet,p}^{\frac{1}{2}}}
\ov{\eqref{inclusion-range}}{=} \ovl{\Ran \Delta_{\bar\partial_E,0,\bullet,p}} 
\ov{\eqref{eq:range-Laplacian-kernel-projection}}{=} \ker P_p,
\end{equation}
By Proposition \ref{prop-existence-Green}, there exists a Green operator $G$ satisfying \eqref{parametrix}. By Proposition \ref{prop-elliptic-Sobolev}, it induces a bounded operator $G \co \L^p(\Omega^{0,\bullet}(M,E)) \to \W^{2,p}_\nabla(\Omega^{0,\bullet}(M,E))$. Let $\eta \in \L^p(\Omega^{0,\bullet}(M,E))$ and choose $\eta_j \in \Omega^{0,\bullet}(M,E)$ such that $\eta_j\to\eta$ in $\L^p(\Omega^{0,\bullet}(M,E))$. Since $G$ is bounded, we have
$G\eta_j \to G\eta$ in $\W^{2,p}_\nabla$. Moreover, we have 
$$
\Delta_{\bar\partial_E,0,\bullet}G\eta_j=(\Id-P)\eta_j \to (\Id-P_p)\eta.
$$ 
Hence $G\eta \in \dom \Delta_{\bar\partial_E,0,\bullet,p}$ and
\[
\Delta_{\bar\partial_E,0,\bullet,p}G\eta=(\Id-P_p)\eta.
\]
We infer that
\begin{equation}
\label{parametrix-bis}
\Delta_{\bar\partial_E,0,\bullet,p} G  
\ov{\eqref{parametrix}}{=} \Id_{\L^p(\Omega^{0,\bullet}(M,E))}-P_p.
\end{equation}
Now, we show that the map $
|\D_{E,p}| \big|_{F_p} \co F_p \to \ker P_p$ is surjective. Let $\eta \in \ker P_p$, so that $P_p \eta = 0$. We consider the element $v \ov{\mathrm{def}}{=} G\eta$ of $\W^{2,p}_\nabla(\Omega^{0,\bullet}(M,E))$. Then 
$$
\Delta_{\bar\partial_E,0,\bullet,p} v
= \Delta_{\bar\partial_E,0,\bullet,p} G\eta
\ov{\eqref{parametrix-bis}}{=} (\Id - P_p)\eta
= \eta.
$$
Since $\D_{E,p}^2 = 2\Delta_{\bar\partial_E,0,\bullet,p}$ by Proposition \ref{prop:closure-kodaira-square-dirac}, this can be rewritten as $\D_{E,p}^2 v = 2\eta$. Since $|\D_{E,p}|$ is an elliptic pseudo-differential operator of order $1$, it induces a bounded operator $|\D_{E,p}| \co \W^{2,p}_\nabla(\Omega^{0,\bullet}(M,E)) \to \W^{1,p}_{\nabla}(\Omega^{0,\bullet}(M,E))$. So the element $u 
\ov{\mathrm{def}}{=} \frac{1}{2}|\D_{E,p}|v$ belongs to the space $\W^{1,p}_{\nabla}(\Omega^{0,\bullet}(M,E)) = \W^{1,p}_{\bar\partial_E}(\Omega^{0,\bullet}(M,E))$. Furthermore, we have
\[
|\D_{E,p}|u 
= \frac{1}{2}|\D_{E,p}|^2 v 
= \frac{1}{2}\D_{E,p}^2 v 
= \eta.
\]
Since \eqref{inter-end} implies that $\Ran |\D_{E,p}| \subset \ker P_p$, it follows that $
P_pu=\frac12 P_p|\D_{E,p}|v=0$. So $u$ belongs to $\W^{1,p}_{\bar\partial_E}(\Omega^{0,\bullet}(M,E)) \cap \ker P_p \ov{\eqref{def-Fp}}{=} F_p$. We conclude that the map $
|\D_{E,p}| \big|_{F_p} \co F_p \to \ker P_p$ is a continuous bijection. Since $P_p$ vanishes on $F_p$, we have $\scr{D}\big|_{F_p} = |\D_{E,p}|\big|_{F_p}$, which is also a continuous bijection. 

Combining the two pieces, we conclude that $\scr{D}$ is continuous and bijective from the space $\W^{1,p}_{\bar\partial_E}(\Omega^{0,\bullet}(M,E))$ onto the space $\L^p(\Omega^{0,\bullet}(M,E))$, and the open mapping theorem \cite[Corollary 1.6.6 p.~44]{Meg98} yields \eqref{eq:S-iso}.
\end{proof} 

\paragraph{Fredholm operators}
A bounded operator $T \co X \to Y$, acting between complex Banach spaces $X$ and $Y$, is said to be a Fredholm operator \cite[Definition 4.37 p.~156]{AbA02} if the subspaces $\ker T$ and $Y/\Ran T$ are finite-dimensional. In this case, we can introduce the index
\begin{equation}
\label{Fredholm-Index}
\Index T 
\ov{\mathrm{def}}{=} \dim \ker T-\dim \coker T,
\end{equation}
where $\coker T \ov{\mathrm{def}}{=} Y/\Ran T$. According to \cite[Lemma 4.38 p.~156]{AbA02}, every Fredholm operator has a closed range.

\paragraph{$\K$-theory}
We refer to the books \cite{Bl98}, \cite{CMR07}, \cite{Eme24}, \cite{GVF01}, 
and \cite{WeO93} for background on $\K$-theory. 
We briefly recall the construction of the group $\K_0$, which is defined for unital rings and, in particular, for Banach algebras. Let $\cal{A}$ be a unital ring. We denote by $\M_\infty(\cal{A})$ the algebraic direct limit of the matrix algebras $\M_n(\cal{A})$ with respect to the embeddings $x \mapsto \diag(x,0)$. Elements of $\M_\infty(\cal{A})$ may be viewed as infinite matrices with only finitely many nonzero entries. Two idempotents $e,f \in \M_\infty(\cal{A})$ are called (Murray--von Neumann) equivalent if there exist $v,w \in \M_\infty(\cal{A})$ such that $e=vw$ and $f=wv$. We write $[e]$ for the equivalence class of $e$ and we introduce the set $
\V(\cal{A})
\ov{\mathrm{def}}{=}
\{[e] : e \in \M_\infty(\cal{A}),\ e^2=e\}$ of all equivalence classes of idempotents in $\M_\infty(\cal{A})$. By \cite[pp.~3-4]{CMR07}, this set becomes an abelian semigroup with addition given by $[e]+[f]
\ov{\mathrm{def}}{=}
[\diag(e,f)]$. Its neutral element is the class $[0]$ of the zero idempotent. The group $\K_0(\cal{A})$ is then defined as the Grothendieck group associated with the semigroup $\V(\cal{A})$. Hence every element of $\K_0(\cal{A})$ can be represented as a formal difference $[e]-[f]$. Two such formal differences $[e_1] - [f_1]$ and $[e_2] - [f_2]$ are equal in $\K_0(\cal{A})$ if there exists $g \in \V(\cal{A})$ such that $[e_1] + [f_2] +  [g] = [f_1] + [e_2] +  [g]$.

\paragraph{Coupling between Banach K-homology with K-theory}
An even Banach Fredholm module $(X,\pi,F,\gamma)$ over an algebra $\cal{A}$ on a Banach space $X$ consists of a representation $\pi \co \cal{A} \to \B(X)$, a bounded operator $F \co X \to X$ and a bounded operator $\gamma \co X \to X$ with $\gamma^2=\Id_X$ such that
\begin{enumerate}
\item $F^2-\Id_X$ is a compact operator on $X$,
\item for any $a \in \cal{A}$ the commutator $[F, \pi(a)]$ is a compact operator on $X$,
\item $\gamma$ is a symmetry, i.e.~$\gamma^2=\Id_X$, and for any $a \in \cal{A}$ we have
\begin{equation}
\label{commuting-rules}
F\gamma
=-\gamma F
\quad \text{and} \quad  
[\pi(a),\gamma]=0 .
\end{equation}
\end{enumerate}
We say that $\gamma$ is the grading operator. In this case, $P \ov{\mathrm{def}}{=} \frac{\Id_X+\gamma}{2}$ is a bounded projection and we can write $X=X_+\oplus X_-$ where $X_+ \ov{\mathrm{def}}{=} \Ran P$ and $X_- \ov{\mathrm{def}}{=}\Ran (\Id_X-P)$. With respect to this decomposition, the equations of \eqref{commuting-rules} imply that we can write
\begin{equation}
\label{Fredholm-even}
\pi(a)
=\begin{bmatrix}
 \pi_+(a)    &  0 \\
  0   & \pi_-(a)  \\
\end{bmatrix}
\quad \text{and} \quad 
F=\begin{bmatrix}
  0   & F_-  \\
  F_+   & 0  \\
\end{bmatrix},
\end{equation}
where $\pi_+ \co \cal{A} \to \B(X_+)$ and $\pi_- \co \cal{A} \to \B(X_-)$ are representations of $\cal{A}$. The following coupling result is proved in \cite{Arh26a}. 

\begin{thm}
\label{th-pairing-even}
Let $\cal{A}$ be a unital algebra. Consider a projection $e$ of the matrix algebra $\M_n(\cal{A})$ and an even Banach Fredholm module $\bigg(X_+ \oplus X_-,\pi,\begin{bmatrix}
  0   & F_-  \\
   F_+  &  0 \\
\end{bmatrix},\gamma\bigg)$ over $\cal{A}$ with $\pi(1)=\Id_X$. We introduce the bounded operator $e_n \ov{\mathrm{def}}{=} (\Id \ot \pi)(e) \co X^{\oplus n} \to X^{\oplus n}$. Then the bounded operator
\begin{equation}
\label{eFe-90}
e_n (\Id \ot F_+) e_n \co e_n(X_+^{\oplus n}) \to e_n(X_-^{\oplus n}) 
\end{equation}
is Fredholm. 
\end{thm}

Consequently, we have a pairing $\la \cdot, \cdot \ra \to \Z$ defined by
\begin{equation}
\label{pairing-even-1}
\big\la [e], [(X,\pi,F,\gamma)] \big\ra 
=\Index e_n (\Id \ot F_+) e_n.
\end{equation}

The following result is proved in \cite{Arh26a}, which is a variant of \cite[Proposition 2.4 p.~4818]{CGIS14} and \cite[Proposition 4.4]{CPR11}.
\begin{prop} 
\label{prop-triple-to-Fredholm}
Let $(\cal{A},X, D)$ be a compact Banach spectral triple over an algebra $\cal{A}$ via a homomorphism $\pi \co \cal{A} \to \B(X)$. Then $(X,\pi,\sgn D)$ is a Banach Fredholm module over $\cal{A}$.
\end{prop}

Now, we prove the first main result of this section.

\begin{thm}
\label{thm:Euler-characteristic-Lp}
Let $E$ be a Hermitian holomorphic vector bundle of finite rank over a compact K\"ahler manifold $M$. Suppose that $1 < p < \infty$. Let $[1] \in \K_0(\C(M))$ be the $\K$-theory class of the constant function $1$. Then the linear operator $\D_{E,+,p} \co \W^{1,p}_{\bar\partial_E}(\Omega^{0,\even}(M,E)) \to \L^p(\Omega^{0,\odd}(M,E))$ is Fredholm and we have
\[
\big\langle [1],[\L^p(\Omega^{0,\bullet}(M,E)),\pi,\sgn \D_{E,p},\gamma] \big\rangle_{\K_0(\C(M)),\K^0(\C(M),\scr{L}^p)}
= \Index \D_{E,+,p}.
\]
\end{thm}

\begin{proof}
By Theorem~\ref{thm:locally-compact-Hodge-Dirac}, the quadruple $(\C(M),\L^p(\Omega^{0,\bullet}(M,E)),D_{E,p},\gamma)$ is an even compact Banach spectral triple. By Proposition~\ref{prop-triple-to-Fredholm}, it gives rise to an even Banach Fredholm module $\big(\L^p(\Omega^{0,\bullet}(M,E)),\pi,\sgn(\D_{E,p}),\gamma\big)$ over the algebra $\C(M)$ of continuous functions, where $\pi$ is the representation by pointwise multiplication. Moreover, we have
\begin{equation}
\label{eq:sign-Dp-factorization}
\sgn \D_{E,p} 
= \D_{E,p} |\D_{E,p}|^{-1},
\end{equation}
where $|\D_{E,p}| \ov{\mathrm{def}}{=} (\D_{E,p}^2)^{\frac12}$ is defined by the sectorial functional calculus of the sectorial operator $\D_{E,p}^2$. On the subspace $\ker \D_{E,p}$, we have $\sgn \D_{E,p} = 0$ and $|\D_{E,p}|^{-1}=0$. Observe that by Theorem \ref{prop:even-Dolbeault-Dirac} we have $\gamma \D_{E,p} = -\D_{E,p}\gamma$. Since the sign function is odd, we obtain by \cite[Proposition 3.2.10]{Ege15} the equality $
\gamma \sgn(\D_{E,p})
=
-\sgn(\D_{E,p})\gamma$. Hence, with respect to the decomposition \eqref{decompo-Lp-even-odd}, we can write with \cite[Proposition 8.1]{Arh26b} 
\[
\sgn \D_{E,p}
=
\begin{bmatrix}
0 & (\sgn \D_{E,p})_-\\
(\sgn \D_{E,p})_+ & 0
\end{bmatrix},
\]
where $
(\sgn \D_{E,p})_+ \co \L^p(\Omega^{0,\even}(M,E)) \to \L^p(\Omega^{0,\odd}(M,E))$ 
and $
(\sgn \D_{E,p})_- \co \L^p(\Omega^{0,\odd}(M,E)) \to \L^p(\Omega^{0,\even}(M,E))$ are bounded operators.

Now, we apply Theorem \ref{th-pairing-even} with $\cal{A} = \C(M)$, $X_+ \ov{\mathrm{def}}{=} \L^p(\Omega^{0,\even}(M,E))$ and $X_- \ov{\mathrm{def}}{=} \L^p(\Omega^{0,\odd}(M,E))$, the bounded operator $F \ov{\mathrm{def}}{=} \sgn \D_{E,p}$ and the projection $e = 1$ in $\M_1(\C(M))$. Since $e_1 = (\Id \ot \pi)(1) = \Id_{\L^p(\Omega^{0,\bullet}(M,E))}$ the operator
\[
(\sgn \D_{E,p})_+
=e_1(\Id \ot F_+)e_1
\co \L^p(\Omega^{0,\even}(M,E)) \to \L^p(\Omega^{0,\odd}(M,E))
\]
is Fredholm and we have
\begin{equation}
\label{eq:pairing-signDp-plus}
\big\langle [1],[\L^p(\Omega^{0,\bullet}(M,E)),\pi,\sgn \D_{E,p},\gamma] \big\rangle_{\K_0(\C(M)),\K^0(\C(M),\scr{L}^p)}
=
\Index(\sgn \D_{E,p})_+.
\end{equation}
Note that $
\D_{E,p}^2
\ov{\eqref{Dp-decompo-block}}{=}
\begin{bmatrix}
\D_{E,-,p} \D_{E,+,p} & 0\\
0 & \D_{E,+,p}\D_{E,-,p}
\end{bmatrix}$. Using \cite[Theorem 3.2.20]{Ege15}, we deduce that 
$$
|\D_{E,p}|
=
(\D_{E,p}^2)^{\frac12}
=
\begin{bmatrix}
|\D_{E,p}|_+ & 0\\
0 & |\D_{E,p}|_-
\end{bmatrix},
$$ 
where $|\D_{E,p}|_+ \ov{\mathrm{def}}{=} (\D_{E,-,p}\D_{E,+,p})^{\frac12}$ and $|\D_{E,p}|_- \ov{\mathrm{def}}{=} (\D_{E,+,p}\D_{E,-,p})^{\frac12}$. The identity $\D_{E,p} = \sgn(\D_{E,p})|\D_{E,p}|$ which follows from \eqref{eq:sign-Dp-factorization} translates in block form as
\[
\begin{bmatrix}
0 & \D_{E,-,p}\\
\D_{E,+,p} & 0
\end{bmatrix}
=
\begin{bmatrix}
0 & (\sgn \D_{E,p})_-\\
(\sgn \D_{E,p})_+ & 0
\end{bmatrix}
\begin{bmatrix}
|\D_{E,p}|_+ & 0\\
0 & |\D_{E,p}|_-
\end{bmatrix}.
\]
On the even component, we obtain in particular
\begin{equation}
\label{eq:Dp-plus-factorization}
\D_{E,+,p} 
= (\sgn \D_{E,p})_+\,|\D_{E,p}|_+.
\end{equation} 
The kernel of $\D_{E,p}$ coincides with the finite-dimensional space $\cal{H}^{0,\bullet}(M,E)$ of smooth harmonic forms. The Bergman  projection $P \co \L^p(\Omega^{0,\bullet}(M,E)) \to \L^p(\Omega^{0,\bullet}(M,E))$ is a finite-rank operator. It preserves the $\Z_2$-grading, so with respect to the decomposition \eqref{decompo-Lp-even-odd} it has the block form $
P =
\begin{bmatrix}
P_+ & 0\\
0 & P_-
\end{bmatrix}$.
Since $|\D_{E,p}|$ is even and $P$ preserves the grading, the operator $\scr{D} \ov{\eqref{def-de-scrD}}{=} |\D_{E,p}| + P_p$ is also even and can be written as $
\scr{D} =
\begin{bmatrix}
\scr{D}_+ & 0\\
0 & \scr{D}_-
\end{bmatrix}$ 
for appropriate operators $\scr{D}_\pm$. By Proposition \ref{prop-isomorphism}, the map $\scr{D} \co \W^{1,p}_{\bar\partial_E}(\Omega^{0,\bullet}(M,E)) \to \L^p(\Omega^{0,\bullet}(M,E))$  is an isomorphism of Banach spaces. Moreover, the operator $\sgn \D_{E,p}$ vanishes on the subspace $\ker \D_{E,p}$, hence $\sgn(\D_{E,p})\,P = 0$. This gives
\[
\sgn(\D_{E,p})\,\scr{D}
\ov{\eqref{def-de-scrD}}{=} \sgn(\D_{E,p})\,(|\D_{E,p}| + P)
= \sgn(\D_{E,p})\,|\D_{E,p}|
= \D_{E,p}.
\]
Restricting \eqref{eq:S-iso} to the even part, we obtain an isomorphism 
\[
\scr{D}_+ 
\ov{\mathrm{def}}{=} \scr{D}\big|_{\W^{1,p}(\Omega^{0,\even}_{\bar\partial_E}(M,E))}
\co \W^{1,p}_{\bar\partial_E}(\Omega^{0,\even}(M,E)) \xrightarrow{\ \cong\ } X_+.
\]
With respect to the splitting $X_+ \oplus X_-$, the factorization $\D_{E,p} = \sgn(\D_{E,p}) \scr{D}$ reads
\[
\begin{bmatrix}
0 & \D_{E,-,p}\\
\D_{E,+,p} & 0
\end{bmatrix}
=
\begin{bmatrix}
0 & (\sgn \D_{E,p})_-\\
(\sgn \D_{E,p})_+ & 0
\end{bmatrix}
\begin{bmatrix}
\scr{D}_+ & 0\\
0 & \scr{D}_-
\end{bmatrix},
\]
so, on the even component, we obtain
\begin{equation}
\label{eq:Dp-plus-factorization-with-S}
\D_{E,+,p} 
= (\sgn \D_{E,p})_+ \scr{D}_+
\co \W^{1,p}_{\bar\partial_E}(\Omega^{0,\even}(M,E)) \to X_-.
\end{equation}
Recall that the operator $(\sgn \D_{E,p})_+ \co X_+ \to X_-$ is Fredholm. Since the map $\scr{D}_+$ is an isomorphism by \eqref{eq:S-iso}, it is a Fredholm operator with index 0. So the composition \eqref{eq:Dp-plus-factorization-with-S} is Fredholm as well and with the additivity of the index \cite[Theorem 4.43 p.~158]{AbA02}, we deduce that
\[
\Index \D_{E,+,p}
= \Index (\sgn \D_{E,p})_+ + \Index \scr{D}_+ 
= \Index (\sgn \D_{E,p})_+.
\]
\end{proof}

\subsection{Identification of the index}

Now, we describe the cokernel of the operator $\D_{E,+,p} \co \W^{1,p}_{\bar\partial_E}(\Omega^{0,\even}(M,E)) \to \L^p(\Omega^{0,\odd}(M,E))$.

\begin{prop}
\label{prop-coker-HD}
Let $E$ be a Hermitian holomorphic vector bundle of finite rank over a compact K\"ahler manifold $M$.  Suppose that $1 < p < \infty$. We have an isomorphism
\begin{equation}
\label{description-coker}
\coker \D_{E,+,p} 
\cong \cal{H}^{0,\odd}(M,E).
\end{equation}
\end{prop}

\begin{proof}
If $u \in \W^{1,p}_{\bar\partial_E}(\Omega^{0,\bullet}(M,E))$ is an even-degree form, we can write 
$u = \sum_{m \text{ even}} u_m$, where $u_m \in \W^{1,p}_{\bar\partial_E}(\Omega^{0,m}(M,E))$. Since $\bar{\partial}_{E,p}$ raises the degree by one and $\bar{\partial}_{E,p}^*$ lowers the degree by one, we have with Corollary \ref{prop-DEp-sum-dbar-dbarstar}
\[
\D_{E,p}u 
\ov{\eqref{Def-Dirac-DE}}{=} \sqrt{2} (\bar{\partial}_{E,p} + \bar{\partial}_{E,p}^*)u
= \sum_{m \text{ even}} \sqrt{2}(\bar{\partial}_{E,p} u_m + \bar{\partial}_{E,p}^* u_m),
\]
and the term $\bar{\partial}_{E,p} u_m + \bar{\partial}_{E,p}^* u_m$ has degrees $m+1$ and $m-1$, respectively.

Fix an odd integer $q$. The degree-$q$ component of $\D_{E,p}u$ can only come from those $m$ such that $m+1 = q$ or $m-1 = q$, that is $m = q-1$ or $m = q+1$. Therefore, the degree-$q$ component of $\D_{E,p}u$ is 
$$
(\D_{E,p}u)_q
= \sqrt{2}\bar{\partial}_{E,p} u_{q-1} + \sqrt{2}\bar{\partial}_{E,p}^* u_{q+1},
$$ 
with the convention that $u_{q-1}$ or $u_{q+1}$ is zero if the corresponding index is outside $\{0,\dots,n\}$. In particular, we have
\begin{equation}
\label{eq:range-degree-k}
\Ran \D_{E,+,p} \cap \L^p(\Omega^{0,q}(M,E))
=
\bar{\partial}_{E,p} \W^{1,p}_{\bar\partial_{E}}(\Omega^{0,q-1}(M,E))
+
\bar{\partial}_{E,p}^* \W^{1,p}_{\bar\partial_{E}}(\Omega^{0,q+1}(M,E)).
\end{equation}
By the $\L^p$-Hodge decomposition \eqref{eq:Lp-Hodge-decomp}, we deduce that
\begin{equation}
\label{eq:quotient-degree-k}
\L^p(\Omega^{0,q}(M,E)) / (\Ran \D_{E,+,p} \cap \L^p(\Omega^{0,q}(M,E)))
\cong
\cal{H}^{0,q}(M,E).
\end{equation}
Now, we have
\[
\Ran \D_{E,+,p}
=\Ran \D_{E,+,p} \cap \bigg(\bigoplus_{k \text{ odd}} \L^p(\Omega^{0,q}(M,E))\bigg)
= \bigoplus_{q \text{ odd}} \bigl(\Ran \D_{E,+,p} \cap \L^p(\Omega^{0,q}(M,E))\bigr).
\]
Using \eqref{eq:quotient-degree-k} for each odd integer $q$, we obtain
\begin{align*}
\coker \D_{E,+,p}
&=
\L^p(\Omega^{0,\odd}(M,E)) / \Ran \D_{E,+,p} \\
&\cong
\bigoplus_{q \text{ odd}}
\L^p(\Omega^{0,q}(M,E)) / (\Ran \D_{E,+,p} \cap \L^p(\Omega^{0,q}(M,E))) \\
&\ov{\eqref{eq:quotient-degree-k}}{\cong}
\bigoplus_{q \text{ odd}} \cal{H}^{0,q}(M,E)
= \cal{H}^{0,\odd}(M,E).
\end{align*}
\end{proof}

When $M$ is compact of complex dimension $n$, the Euler number $\chi(M,E)$ of the holomorphic vector bundle $E$ is defined in \cite[(1.4.25) p.~33]{MaM07} by $
\chi(M,E)
\ov{\mathrm{def}}{=}
\sum_{q=0}^n (-1)^q \dim \mathrm{H}^{0,q}(M,E)$. 
Now, we show the second main result of this section.

\begin{cor}
\label{cor-Euler operator}
Let $E$ be a Hermitian holomorphic vector bundle of finite rank over a compact K\"ahler manifold $M$. Suppose that $1 < p < \infty$. The index of the Fredholm operator $\D_{E,+,p} \co \W^{1,p}_{\bar\partial_E}(\Omega^{0,\even}(M,E)) \to \L^p(\Omega^{0,\odd}(M,E))$ is 
\begin{equation}
\label{index-Hodge-Dirac-Lp}
\Index \D_{E,+,p}
= \chi(M,E).
\end{equation}
\end{cor}

\begin{proof}
According to Theorem \ref{thm:Euler-characteristic-Lp}, $\D_{E,+,p} \co \W^{1,p}_{\bar\partial_E}(\Omega^{0,\even}(M,E)) \to \L^p(\Omega^{0,\odd}(M,E))$ is a Fredholm operator. By Proposition \ref{prop-ker-HD} and Proposition \ref{prop-coker-HD}, we have
\begin{align*}
\MoveEqLeft
\Index \D_{E,+,p}
\ov{\eqref{Fredholm-Index}}{=}\dim \ker \D_{E,+,p}-\dim \coker \D_{E,+,p} \\
&\ov{\eqref{description-kernel} \eqref{description-coker}}{=} \dim \cal{H}^{0,\even}(M,E) - \dim \cal{H}^{0,\odd}(M,E) \\
&=\sum_{q=0}^{n} (-1)^q \dim \cal{H}^{0,q}(M,E) 
\ov{\eqref{harmonic-space-and-Dolbeault}}{=} \sum_{q=0}^{n} (-1)^q \dim \H^{0,q}(M,E)
= \chi(M,E). 
\end{align*}
\end{proof}

\section{Illustrations and discussions in other contexts}
\label{sec-examples}
\subsection{Riemannian manifolds}
\label{sec-Riemann-manifolds}

Let $M$ be a Riemannian manifold of dimension $d$. We denote by $\Omega^\bullet(M)$ the space of smooth differential forms on $M$ with \textit{complex} coefficients. We consider the Hodge--de Rham-Laplacian $\Delta_{\HdR} \co \Omega^\bullet(M) \to \Omega^\bullet(M)$ defined by $
\Delta_{\HdR} 
\ov{\mathrm{def}}{=} \d\d^* + \d^*\d$, 
see e.g.~\cite[p.~167]{Wel08} or \cite[Definition 9.2.2 p.~425]{GVF01}. 
We can apply the abstract Hodge--Dirac machinery developed in Section~\ref{sec-Hodge-Dirac} with $
X \ov{\mathrm{def}}{=} \L^p(\Omega^\even(M))$, $Y \ov{\mathrm{def}}{=} \L^p(\Omega^\odd(M))$ and the restrictions $A \ov{\mathrm{def}}{=} \Delta_{\HdR,p}^\even$ and $\widetilde{A} \ov{\mathrm{def}}{=}\Delta_{\HdR,p}^\odd$ of the closed Hodge--de Rham Laplacian $\Delta_{\HdR,p}$ to even and odd forms. For any $t >0$, we let $
T_t \ov{\mathrm{def}}{=} \e^{-tA}$ and $\widetilde T_t \ov{\mathrm{def}}{=} \e^{-t\widetilde A}$. With respect to the decomposition $
\L^p(\Omega^\bullet(M))
=
\L^p(\Omega^\even(M)) \oplus \L^p(\Omega^\odd(M))$, 
the closed Hodge--Dirac operator has the off-diagonal form 
$$
D_p
=
\begin{bmatrix}
0 & D_{p,-} \\
D_{p,+} & 0
\end{bmatrix},
$$
where the operators $
D_{p,+} \co \dom D_{p,+} \subset \L^p(\Omega^\even(M)) \to \L^p(\Omega^\odd(M))$ and $
D_{p,-} \co \dom D_{p,-} \subset \L^p(\Omega^\odd(M)) \to \L^p(\Omega^\even(M))$ 
are the even-to-odd and odd-to-even parts of the closure of $D=\d+\d^*$. We apply the abstract theory with $\partial \ov{\mathrm{def}}{=} D_{p,+}$ and $\partial^\dagger \ov{\mathrm{def}}{=} D_{p,-}$. As observed in \cite[Theorem 3.12]{NeV18}, under a suitable assumption of \textit{positive curvature}, the operators $A$ and $\widetilde{A}$ admit bounded $\H^\infty(\Sigma_\theta)$ functional calculi for some angle $\theta \in (0,\frac{\pi}{2})$. Similarly Riesz estimates are available with the same assumption. We also refer to \cite{Arh26b} for the compact case without curvature assumption.

Now, we verify the curvature-intertwining condition. On smooth even forms, the Hodge--de Rham Laplacian commutes with both $\d$ and $\d^*$. It is easy to check that the condition $\Curv_{\partial,\H^\infty}(0)$ is satisfied. The regularization property of Definition~\ref{def-regularization-coreC} is proved in the same way as Corollary \ref{cor-Dolbeault-Dirac-functional-calculus}. Let $P_p^\odd \co \L^p(\Omega^\odd(M)) \to \cal H^\odd(M)$ be the bounded finite-rank projection onto the space of odd harmonic forms, and set $Q \ov{\mathrm{def}}{=} \Id_{\L^p(\Omega^\odd(M))} - P_p^\odd$. We can prove that $\Ran Q = \Ran D_{p,+}$, $\ker Q \subset \dom D_{p,-}$ and $D_{p,-}|_{\ker Q}=0$. 
We conclude that with Theorem \ref{thm-full-bisectorial}, we can recover the unweighted case of \cite[Theorem 1.1 p.~3110]{NeV18}.

\subsection{$q$-Ornstein--Uhlenbeck semigroups}
\label{sec-q-Ornstein}

We consider in this section a classical noncommutative deformation of the Ornstein--Uhlenbeck semigroup called $q$-Ornstein--Uhlenbeck semigroup. Here $q \in [-1,1)$ is a parameter. We refer to 
\cite{BS91}, \cite{BS94}, \cite{BKS97} and \cite{Lus99}  
for more information on this setting. 

We start by recalling several facts about $q$-Gaussian algebras. For any integer $n \geq 1$, we denote by $\mathrm{S}_n$ the symmetric group. If $\sigma$ is a permutation in $\mathrm{S}_n$ we denote by $\Inv(\sigma) \ov{\mathrm{def}}{=} \card \big\{ (i,j) : 1 \leq i < j \leq n, \sigma(i) > \sigma(j) \big\}$ the number of inversions of $\sigma$. Let $H$ be a separable real Hilbert space with complexification $H_{\mathbb{C}} \ov{\mathrm{def}}{=} H+\i H$. We consider the algebraic Fock space
 $\cal{F}_{\mathrm{alg}}(H_{\mathbb{C}}) 
\ov{\mathrm{def}}{=} \bigoplus_{n=0}^{\infty} H_{\mathbb{C}}^{\ot n}$, 
where $H_{\mathbb{C}}^{\otimes 0} \ov{\mathrm{def}}{=} \mathbb{C} \Omega$ for a unit vector $\Omega$. The $q$-Fock space 
$\cal{F}_{q}(H_{\mathbb{C}})$ is the completion of this space for the scalar product 
$$
\la h_1 \ot \cdots \ot h_n ,k_1 \ot \cdots \ot k_m \ra_{q}
\ov{\mathrm{def}}{=} \delta_{m,n} \sum_{\sigma \in \mathrm{S}_n} q^{\Inv(\sigma)} \la h_1,k_{\sigma(1)}\ra_{H_{\mathbb{C}}} \cdots \la h_n,k_{\sigma(n)}\ra_{H_{\mathbb{C}}},
$$
where $\delta_{m,n}$ is the Kronecker symbol. If $q=-1$, we must first divide out by the null space, and we obtain the usual antisymmetric Fock space. For any vector $e \in H$, we can introduce as in \cite[Definition 1.1 p.~133]{BKS97} the creation operator $\ell(e) \co \cal{F}_{q}(H_{\mathbb{C}}) \to \cal{F}_{q}(H_{\mathbb{C}})$, $h_1 \ot \cdots \ot h_n \mapsto e \ot h_1 \ot  \cdots \ot h_n$. 
These operators satisfy the $q$-relation $\ell(f)^*\ell(e)-q\ell(e)\ell(f)^* 
= \langle f,e\rangle_{H} \Id_{\cal{F}_{q}(H_{\mathbb{C}})}$ of \cite[p.~133]{BKS97}. These relations interpolate between the bosonic and fermionic commutation relations. Following \cite[Definition 2.1 p.~136]{BKS97}, we consider for any vector $e \in H$ the selfadjoint operator
\begin{equation}
\label{def-sq}
s_q(e) 
\ov{\mathrm{def}}{=} \ell(e)+\ell(e)^* \co \cal{F}_{q}(H_{\mathbb{C}}) \to \cal{F}_{q}(H_{\mathbb{C}}).
\end{equation}
The $q$-Gaussian von Neumann algebra $\Gamma_q(H)$ is the von Neumann algebra over $\cal{F}_q(H_{\mathbb{C}})$ generated by the operators $s_q(e)$ where $e \in H$. By \cite[Proposition 2.3 p.~136]{BKS97}, the von Neumann algebra $\Gamma_q(H)$ is finite and admits the normal finite faithful trace $\tau$ defined by $\tau(x) \ov{\mathrm{def}}{=} \la \Omega,x(\Omega) \ra_{\mathcal{F}_{q}(H_{\mathbb{C}})}$ where $x \in \Gamma_q(H)$. We also define the $*$-algebra $\cal{A}_q(H) \ov{\mathrm{def}}{=} \ast\mathrm{-alg} \{  s_q(e) : e \in H  \}$.


Let $H$ and $K$ be real Hilbert spaces and $T \co H \to K$ be a contraction with complexification $T_{\mathbb{C}} \co H_{\mathbb{C}} \to K_{\mathbb{C}}$. We can consider the second quantization map $
 \cal{F}_q(T) \co \cal{F}_{q}(H_{\mathbb{C}}) \to \cal{F}_{q}(K_{\mathbb{C}})$, $h_1 \ot \cdots \ot h_n  \mapsto T_{\mathbb{C} }(h_1) \ot \cdots \ot T_{\mathbb{C} }(h_n)$. According to \cite[Theorem 2.11 p.~140]{BKS97}, there exists a unique map $\Gamma_q(T) \co \Gamma_q(H) \to \Gamma_q(K)$ such that 
$$
(\Gamma_q(T)(x))\Omega
=\mathcal{F}_q(T)(x\Omega)
$$ 
for any $x \in \Gamma_q(H)$ and this map is weak* continuous, unital, completely positive and trace preserving. 

Let $(a_t)_{t \geq 0}$ be a strongly continuous semigroup of selfadjoint contractions on $H$. For any $t \geq 0$, set $T_t \ov{\mathrm{def}}{=} \Gamma_q(a_t)$. Then by \cite[Lemma 9.3 p.~97]{JMX06} $(T_t)_{t \geq 0}$ is a weak* continuous semigroup of normal unital completely positive maps, which induces a strongly continuous semigroup of contractions on the noncommutative $\L^p$-space $\L^p(\Gamma_q(H))$ for any $1 \leq p < \infty$. In the special case where $a_t=\mathrm{e}^{-t}\Id_H$, the semigroup $(T_t)_{t \geq 0}$ is called the $q$-Ornstein--Uhlenbeck semigroup, with generator $-A_p$. According to \cite[Proposition 2.3 p.~136]{BKS97}, $\Omega$ is a generating and cyclic vector for $\Gamma_q(H)$. Consequently, each element $x \in \Gamma_q(H)$ is uniquely determined by the vector $\xi=x\cdot\Omega$ of $\cal{F}_q(H_{\mathbb{C}})$ and we write $x=\W_q(\xi)$. For any $h_1, \cdots, h_m \in H_{\mathbb{C}}$, it is known that each element $\W_q(h_1 \ot \cdots \ot h_m)$ belongs to the $*$-algebra $\cal{A}_q(H)$.

In this setting, following \cite[Lemma 5.1]{BGJ23}, we can introduce the linear map $\partial \co \cal{A}_q(H) \to \L^p(\Gamma_q(H \oplus H))$ defined by
\begin{equation*}
\partial(\W_q(h_1 \ot \cdots \ot h_m))
\ov{\mathrm{def}}{=}\sum_{j=1}^m \W_q(h_1 \ot \cdots  \ot (0 \oplus h_j) \ot \cdots \ot h_m).
\end{equation*}
In this framework, the Riesz equivalence \eqref{estim-Riesz-bis} was obtained by Lust-Piquard in \cite[Theorem 3.3 p.~546]{Lus99} (see also \cite{Lus98}). Indeed, we have $
\bnorm{A^{\frac{1}{2}}(f)}_{\L^p(\Gamma_q(H))}
\approx_p \norm{\partial(f)}_{\L^p(\Gamma_q(H \oplus H))}$ for any $f \in \cal{A}_q(H)$ if $1 < p < \infty$. By Proposition \ref{Prop-derivation-closable-sgrp-bis}, we deduce that $\partial$ is closable with closure denoted by $\partial_p$. Now, consider the semigroup $(O_t)_{t \geq 0}$ of contractions acting on the real Hilbert space $H \oplus H$ defined by 
$$
O_t(h_1\oplus h_2)
\ov{\mathrm{def}}{=}\e^{-t} h_1 \oplus h_2, \quad t \geq 0 ,h_1,h_2 \in H.
$$
For any $t \geq 0$, we introduce the operator $\tilde{T}_t \ov{\mathrm{def}}{=} \Gamma_q(O_t)$ acting on the von Neumann algebra $\Gamma_q(H \oplus H)$. It is essentially proved in \cite[Theorem 5.2]{BGJ23} (see also \cite[(5) p.~774]{WiZ21}) that for any $f \in \dom \partial_p$ we have $T_t(f) \in \dom \partial_p$ and that 
$$
\partial_p \circ T_t(f)
=\e^{-t}\tilde{T}_t \circ \partial_p(f)
$$ 
for any $t \geq 0$. Moreover, the negative of the generators of the semigroups $(T_t)_{t \geq 0}$ and $(\tilde{T}_t)_{t \geq 0}$ of contractions acting on the noncommutative $\L^p$-spaces $\L^p(\Gamma_q(H))$ and $\L^p(\Gamma_q(H \oplus H))$ admit bounded $\H^\infty(\Sigma_\theta)$ functional calculi for some angle $\theta \in (0,\frac{\pi}{2})$ by \cite[Theorem 9.4 p.~98, Theorem 9.5 p.~100]{JMX06}. We conclude that we obtain the property $\Curv_{\partial_p,\H^\infty}(1)$ of Definition \ref{curvature-H-infty}. The study of the regularization property in relation to approximation properties of the algebra is beyond the scope of the present paper and will be investigated in a subsequent work.

\subsection{Semigroups of Schur multipliers}
\label{sec-Schur}

Let $I$ be a non-empty index set. Let $(T_t)_{t \geq 0}$ be a symmetric Markovian semigroup of Schur multipliers acting on the von Neumann algebra $\B(\ell^2_I)$ of bounded operators acting on the complex Hilbert space $\ell^2_I$. In this situation, by \cite[Proposition 5.4 p.~415]{Arh13} there exists a \textit{real} Hilbert space $H$ and a family 
$
(\alpha_i)_{i \in I}
$
of vectors of $H$ such that for any $t \geq 0$, the Schur multiplier $T_t \co \B(\ell^2_I) \to \B(\ell^2_I)$ is associated to the matrix $
\big[\e^{-t\|\alpha_i-\alpha_j\|_{H}^2}\big]_{i,j \in I}$. 
This semigroup induces a strongly continuous semigroup on the Schatten space $S^p_I \ov{\mathrm{def}}{=} S^p(\ell^2_I)$ with infinitesimal generator $-A_p$. We denote by $\M_{I,\fin}$ the dense subspace of the Schatten space $S^p_I$ of matrices with a finite number of nonzero  entries. Suppose that $-1 \leq q < 1$. We equip the von Neumann algebra $\Gamma_q(H)$ of Section \ref{sec-q-Ornstein} with its canonical normal finite faithful trace. Following \cite[(2.95) p.~62]{ArK22}, we can consider the unbounded operator $\partial_{\alpha,q} \co \M_{I,\fin} \subset S^p_I \to \L^p(\Gamma_q(H) \otvn \B(\ell^2_I))$, $e_{ij} \mapsto s_q(\alpha_i-\alpha_j) \ot e_{ij}$, where the $q$-Gaussian $s_q(\alpha_i-\alpha_j)$ is defined in \eqref{def-sq}. 

Moreover, if $1 < p < \infty$, the operator $\partial_{\alpha,q}$ is closable by \cite[Proposition 3.11 p.~121]{ArK22}. If we denote by $\partial_{\alpha,q,p}$ its closure, the same result says that $\dom \partial_{\alpha,q,p} =\dom A_p$ and provides the $\partial_{\alpha,q,p}$-Riesz equivalence \eqref{estim-Riesz-bis}, i.e.~we have
$$
\bnorm{A_p^{\frac{1}{2}}(x)}_{S^p_I}
\approx \norm{\partial_{\alpha,q,p}(x)}_{\L^p(\Gamma_q(H),S^p_I)}, \quad x \in \dom \partial_{\alpha,q,p}.
$$
Furthermore, it is implicitly proved in \cite[Section 4.5]{ArK22} that the semigroup $(T_t)_{t \geq 0}$ satisfies the property $\Curv_{\partial_{\alpha,q,p},\H^\infty}(0)$, defined in Definition \ref{curvature-H-infty}. In this context,  regularizing nets, in the sense of Definition \ref{Def-regul}, are provided by truncations of matrices and are described in \cite[Definition 2.3 p.~33]{ArK22}. Finally, the projection $Q \co \L^p(\Gamma_q(H),S^p_I) \to \L^p(\Gamma_q(H),S^p_I)$ defined in \eqref{eq-def-projection-P} satisfies condition \eqref{eq-hodge-splitting-Y} of Proposition \ref{prop-kernelQ-kernelpartialdagger} by \cite[Lemma 4.16 p.~171]{ArK22}.

Consequently, with Theorem \ref{Th-functional-calculus-bisector-Fourier} and Theorem \ref{thm-full-bisectorial}, we recover the result \cite[Theorem 4.9 p.~172]{ArK22} on the functional calculus of the Hodge--Dirac operator $\slashed{D}_p \ov{\mathrm{def}}{=}
\begin{bmatrix} 
0 & (\partial_{p^*})^* \\ 
\partial_p & 0 
\end{bmatrix}$, 
where $p^*$ is the conjugate exponent of $p$.

\section{Appendix: a variant of a result of Duong, Robinson and Haller-Dintelmann}
\label{Appendix}

We prove the following statement, which extends to operators acting on $\L^p$-spaces of vector bundles a result of Duong and Robinson \cite[Theorem 3.4 p.~108]{DuR96} and the vector-valued version established by Haller-Dintelmann \cite{Hal05}. The proof proceeds by reducing the statement to the vector-valued version of Haller-Dintelmann's extrapolation theorem. Finally, the authors of \cite{ALLR25} obtained a closely related finite-dimensional vector-valued version of the Duong--Robinson strategy on $\mathbb{R}^d$, under Gaussian kernel estimates for matrix-valued semigroups.

\begin{prop}
\label{prop:Duong-Robinson-vector-bundle}
Let $E$ be a Hermitian complex vector bundle of finite rank over a compact Riemannian manifold $M$.
Let $A$ be a sectorial operator on $\L^2(M,E)$ such that $-A$ generates a bounded holomorphic semigroup $(S_z)_{z \in \Sigma_\theta}$ on the Hilbert space $\L^2(M,E)$ for some angle $\theta \in (0,\frac{\pi}{2})$. Assume that, for every $t>0$, the operator $S_t$ admits an integral kernel $
K_t(x,y) \in \Hom(E_y,E_x)$ in the sense that
\[
(S_t\omega)(x)
=
\int_M K_t(x,y)\omega(y) \d\mu_g(y),
\qquad
\omega \in \L^2(M,E),
\]
and that there exist constants $c,C>0$ and $m>0$ such that
\begin{equation}
\label{eq:kernel-domination-vector-bundle}
\norm{K_t(x,y)}_{E_y \to E_x}
\leq
\frac{c}{\vol(x,t^{\frac{1}{m}})}
\exp\left(-C\frac{\dist(x,y)^m}{t}\right),
\quad x,y \in M, t>0.
\end{equation}
Suppose that $A$ admits a bounded $\H^\infty(\Sigma_\mu)$ functional calculus on the Banach space $\L^2(M,E)$ for some angle $\mu \in \left(\frac{\pi}{2}-\theta,\frac{\pi}{2}\right)$. Then, for every $p \in (1,\infty)$, the semigroup $(S_t)_{t>0}$ extends consistently to a bounded holomorphic semigroup on $\L^p(M,E)$. If $-A_p$ denotes its generator on $\L^p(M,E)$, then $A_p$ admits a bounded $\H^\infty(\Sigma_{\theta'})$ functional calculus for some angle $\theta' \in (0,\frac{\pi}{2})$.
\end{prop}


\begin{proof}
Let $N \ov{\mathrm{def}}{=} \rank(E)$. By \cite[Remark I.17 p.~13]{Gun17}, there exists a global orthonormal Borel frame $\sigma_1,\ldots,\sigma_N \co M \to E$. Consequently, for every $p \in [1,\infty]$, the map $\Phi_p \co \L^p(M,E) \to \L^p(M,\mathbb C^N)$, $\omega \mapsto (\omega_1,\ldots,\omega_N)$, 
where $
\omega(x)
=\sum_{j=1}^N \omega_j(x)\sigma_j(x)$, is an isometric isomorphism. In particular, the space $\L^p(M,E)$ may be identified isometrically with the Bochner space $\L^p(M,\mathbb{C}^N)$. For any $t > 0$, we define the operator $
\widetilde S_t
\ov{\mathrm{def}}{=}
\Phi_2 S_t \Phi_2^{-1}$ 
acting on the space $\L^2(M,\mathbb C^N)$, which admits a matrix-valued kernel $
\widetilde{K}_t(x,y) \co \mathbb{C}^N \to \mathbb{C}^N$ given by $
\widetilde{K}_t(x,y)
\ov{\mathrm{def}}{=}
\Phi_x K_t(x,y) \Phi_y^{-1}$, where $\Phi_x \co E_x \to \mathbb C^N$ is the fibrewise unitary map induced by the frame. If $f\in \L^1(M,\mathbb C^N)$, the formula
\begin{equation}
\label{kernel-tilde}
(\widetilde S_t f)(x)
=
\int_M \widetilde{K}_t(x,y)f(y) \d\mu_g(y)
\end{equation}
is well-defined for almost every $x \in M$. Since $\Phi_x$ and $\Phi_y$ are unitary, we have
\begin{equation}
\label{eq:matrix-kernel-gaussian}
\bnorm{\widetilde K_t(x,y)}_{\mathbb{C}^N \to \mathbb{C}^N}
=
\norm{K_t(x,y)}_{E_y \to E_x}
\ov{\eqref{eq:kernel-domination-vector-bundle}}{\leq}
\frac{c}{\vol(x,t^{\frac{1}{m}})}
\exp\left(-C\frac{\dist(x,y)^m}{t}\right).
\end{equation}
Since $M$ is compact, as observed in \cite[p.~264]{Hal05} the metric measure space $(M,\dist,\mu_g)$ satisfies the doubling property of \cite[p.~76]{HKST15} and the uniform volume property. More precisely, there exists $C_M \geq 1$ such that
\[
\esssup_{x \in M} \vol(x,r)
\leq
C_M\essinf_{x \in M} \vol(x,r),
\quad r > 0.
\]
Equivalently, we have $
\vol(y,r) \leq C_M \vol(x,r)$ for $x,y \in M$ and $r >0$. Taking $r=t^{\frac{1}{m}}$, we obtain
\begin{equation}
\label{final-34}
\frac{1}{\vol(x,t^{\frac{1}{m}})}
\leq
C_M\frac{1}{\vol(y,t^{\frac{1}{m}})}.
\end{equation}
Thus the Gaussian bound \eqref{eq:matrix-kernel-gaussian} also implies
\[
\bnorm{\widetilde K_t(x,y)}_{\mathbb{C}^N \to \mathbb{C}^N}
\ov{\eqref{eq:matrix-kernel-gaussian}\eqref{final-34}}{\leq} 
\frac{cC_M}{\vol(y,t^{\frac{1}{m}})}
\exp\left(-C\frac{\dist(x,y)^m}{t}\right).
\]
After changing the constants, if we set
\begin{equation}
\label{def-Gt}
G_t(x,y)
\ov{\mathrm{def}}{=}
\frac{c_1}{\vol(x,t^{\frac{1}{m}})}
\exp\left(-c_2\frac{\dist(x,y)^m}{t}\right),
\end{equation}
then we have both $
\bnorm{\widetilde K_t(x,y)}_{\mathbb{C}^N \to \mathbb{C}^N}
\leq
G_t(x,y)$ and $
\bnorm{\widetilde K_t(x,y)}_{\mathbb{C}^N \to \mathbb{C}^N}
\leq
G_t(y,x)$. Therefore
\begin{equation}
\label{estimation-max}
\bnorm{\widetilde{K}_t(x,y)}_{\mathbb{C}^N \to \mathbb{C}^N}
\leq \min\{G_t(x,y), G_t(y,x)\}.
\end{equation}
Equivalently, we can write
$$
G_t(x,y)
=
\frac{1}{\vol(x,t^{\frac{1}{m}})}
g\left(\frac{\dist(x,y)^m}{t}\right)
$$
with $g(s) \ov{\mathrm{def}}{=} c_1 \e^{-c_2s}$. This function is bounded, continuous, decreasing, strictly positive and satisfies the decay condition \cite[(8) p.~93]{DuR96} required in the Poisson bounds of Duong--Robinson, as noticed in \cite[p.~94]{DuR96}. By the scalar estimate \cite[Proposition 2.1 p.~94]{DuR96} applied to the majorant $G_t$, there exists a bounded function $b \co [0,\infty) \to [0,\infty)$, with $b(s) \to 0$ as $s \to \infty$, such that
\[
\esssup_{x \in M} \int_{\dist(x,y)\geq r} G_t(x,y)\d\mu_g(y)
\leq b(\tfrac{r^m}{t}),
\quad r\geq 0, t>0.
\]
Using \eqref{estimation-max}, we obtain, after changing $b$,
\begin{equation}
\label{eq:brownian-vector-kernel}
\esssup_{x \in M} \int_{\dist(x,y) \geq r} 
\Bigl( \bnorm{\widetilde K_t(x,y)} + \bnorm{\widetilde K_t(y,x)} \Bigr)
\d\mu_g(y)
\leq
b(\tfrac{r^m}{t})
\end{equation}
for all $r \geq 0$ and $t > 0$. Taking $r=0$ in \eqref{eq:brownian-vector-kernel}, we obtain
\begin{equation}
\label{final-45}
\esssup_{x \in M} \int_M \bnorm{\widetilde{K}_t(x,y)} \d\mu_g(y)
\leq B
\quad \text{and} \quad
\esssup_{y \in M} \int_M \bnorm{\widetilde K_t(x,y)} \d\mu_g(x)
\leq B
\end{equation}
for some constant $B$ independent of $t > 0$. Now, we prove that the semigroup $(\widetilde{S}_t)_{t \geq 0}$ extends to the whole scale of spaces $\L^p(M,E)$ following the argument of \cite[Proposition 2.3 p.~95]{DuR96}. For any $f \in \L^1(M,\mathbb{C}^N)$, we have
\[
\bnorm{\widetilde{S}_t f(x)}_{\mathbb{C}^N}
\ov{\eqref{kernel-tilde}}{=} \norm{\int_M \widetilde{K}_t(x,y)f(y) \d\mu_g(y)}_{\mathbb{C}^N}
\leq \int_M \bnorm{\widetilde K_t(x,y)} \norm{f(y)}_{\mathbb{C}^N} \d\mu_g(y).
\]
By Fubini's theorem, we obtain
\[
\begin{aligned}
\bnorm{\widetilde{S}_t f}_{\L^1(M,\mathbb{C}^N)}
&\leq
\int_M
\int_M
\bnorm{\widetilde K_t(x,y)}
\norm{f(y)}_{\mathbb{C}^N}
\d\mu_g(y) \d\mu_g(x) \\
&= \int_M \left(\int_M \bnorm{\widetilde K_t(x,y)} \d\mu_g(x)\right) \norm{f(y)}_{\mathbb{C}^N} \d\mu_g(y) 
\ov{\eqref{final-45}}{\leq}
B\norm{f}_{\L^1(M,\mathbb{C}^N)}.
\end{aligned}
\]
Similarly, if $f \in \L^\infty(M,\mathbb{C}^N)$, then
\begin{align*}
\MoveEqLeft
\bnorm{\widetilde{S}_t f}_{\L^\infty(M,\mathbb{C}^N)}
\ov{\eqref{kernel-tilde}}{=}\esssup_{x \in M} \left|\int_M \widetilde{K}_t(x,y)f(y) \d\mu_g(y)\right| \\
&\leq
\esssup_{x \in M}
\int_M \bnorm{\widetilde{K}_t(x,y)} \norm{f(y)}_{\mathbb C^N} \d\mu_g(y) 
\ov{\eqref{final-45}}{\leq}
B\norm{f}_{\L^\infty(M,\mathbb{C}^N)}.
\end{align*}
By interpolation \cite[Theorem 2.2.1 p.~83]{HvNVW16}, for any $p \in [1,\infty]$ we deduce that
\begin{equation}
\label{eq:Lp-uniform-bound-St}
\bnorm{\widetilde S_t}_{\L^p(M,\mathbb{C}^N) \to \L^p(M,\mathbb{C}^N)}
\leq B,
\qquad t > 0.
\end{equation}
Since $\L^p(M,\mathbb{C}^N)\cap \L^2(M,\mathbb{C}^N)$ is dense in the space $\L^p(M,\mathbb C^N)$ for $1 \leq p < \infty$, the operators $\widetilde S_t$, initially defined on $\L^2(M,\mathbb C^N)$, extend uniquely to bounded operators on $\L^p(M,\mathbb C^N)$. The extensions are consistent: if $f$ belongs to two spaces $\L^p$ and $\L^q$, the two definitions agree by density, because they agree on the subspace $\L^p(M,\mathbb C^N)\cap \L^q(M,\mathbb{C}^N) \cap \L^2(M,\mathbb{C}^N)$. Moreover, the semigroup law $\widetilde S_t\widetilde S_s=\widetilde S_{t+s}$ holds first on $\L^2(M,\mathbb C^N)$ and then, by density and boundedness, on every $\L^p(M,\mathbb C^N)$. It remains to check strong continuity for $p \in [1,\infty)$. Let $\cal{D} \ov{\mathrm{def}}{=}\L^\infty(M,\mathbb C^N)\cap \L^2(M,\mathbb{C}^N)$. This space is dense in $\L^p(M,\mathbb{C}^N)$ for every $p<\infty$. If $f \in \cal{D}$, then $\widetilde{S}_t f \to f$ in $\L^2(M,\mathbb{C}^N)$ as $t \to 0^+$, since $(\widetilde {S}_t)_{t \geq 0}$ is the original strongly continuous semigroup on the Hilbert space $\L^2(M,\mathbb{C}^N)$. If $1 \leq p \leq2$, the compactness of $M$ gives
\[
\bnorm{\widetilde{S}_t f-f}_{\L^p(M,\mathbb{C}^N)}
\leq
\mu_g(M)^{\frac1p-\frac12}
\bnorm{\widetilde S_t f-f}_{\L^2(M,\mathbb{C}^N)},
\]
and hence $\widetilde S_t f \to f$ in $\L^p(M,\mathbb{C}^N)$. If $p > 2$, then
\[
\begin{aligned}
\bnorm{\widetilde S_t f-f}_{\L^p}^p
&\leq
\bnorm{\widetilde S_t f-f}_{\L^\infty}^{p-2}
\bnorm{\widetilde S_t f-f}_{\L^2}^2 
\leq
(B+1)^{p-2}\norm{f}_{\L^\infty(M,\mathbb{C}^N)}^{p-2}
\bnorm{\widetilde S_t f-f}_{\L^2(M,\mathbb{C}^N)}^2,
\end{aligned}
\]
which again tends to $0$ as $t \to 0^+$. Thus $\widetilde{S}_t f \to f$ in $\L^p(M,\mathbb{C}^N)$ for every $f \in \cal{D}$. By the uniform bound \eqref{eq:Lp-uniform-bound-St} and the density of $\cal{D}$ in $\L^p(M,\mathbb{C}^N)$, it follows from \cite[Proposition 5.3 p.~38]{EnN00} that the semigroup $(\widetilde{S}_t)_{t \geq 0}$ is strongly continuous on the space $\L^p(M,\mathbb{C}^N)$ for every $p \in [1,\infty)$.

Conjugating by the isometric isomorphism $\Phi_p$, we obtain, by \cite[5.10 p.~43]{EnN00}, bounded operators $
S_t^{(p)}
\ov{\mathrm{def}}{=}
\Phi_p^{-1}\widetilde S_t\Phi_p$ on $\L^p(M,E)$, forming a consistent strongly continuous semigroup for every $p \in [1,\infty)$. The bound $
\bnorm{S_t^{(p)}}_{\L^p(M,E)\to\L^p(M,E)}
\leq B$ is uniform in $t > 0$.

Since $(S_t)_{t \geq 0}$ is bounded and holomorphic on $\L^2(M,E)$ at angle $\theta$, the conjugated semigroup $
\widetilde{S}_t
=
\Phi_2 S_t \Phi_2^{-1}$ 
is bounded and holomorphic on $\L^2(M,\mathbb{C}^N)$ in the same sector. So there exists, by \cite[Definition 3.7.3 p.~150]{ABHN11}, a constant $M_\psi$ such that
\[
\sup_{z \in \Sigma_\psi}
\bnorm{\widetilde{S}_z}_{\L^2(M,\mathbb{C}^N) \to \L^2(M,\mathbb{C}^N)}
\leq M_\psi.
\]
Let $z=t+\i s$ be an element of the sector $\Sigma_\psi$, with $t = \Re z > 0$. As in \cite[p.~107]{DuR96}, choose $\delta = \delta_\psi \in (0,1)$ such that $\delta t+\i s \in \Sigma_\theta$ for all $z=t+\i s \in \Sigma_\psi$. 
Then
\begin{equation}
\label{decomposition}
\widetilde{S}_z
=\widetilde{S}_{t+\i s}
=\widetilde S_{\frac{1-\delta}{2}t}
\widetilde S_{\delta t+\i s}
\widetilde S_{\frac{1-\delta}{2}t}.
\end{equation}
We first derive the following real-time estimates.

\begin{lemma}
We have
\begin{equation}
\label{eq:L1L2-L2Linf-real}
\bnorm{\widetilde S_a}_{\L^1(M,\mathbb{C}^N) \to \L^2(M,\mathbb{C}^N)}
+
\bnorm{\widetilde S_a}_{\L^2(M,\mathbb{C}^N) \to \L^\infty(M,\mathbb{C}^N)}
\lesssim
\sup_{x \in M} \vol(x,a^{\frac1m})^{-\frac12},
\quad a > 0.
\end{equation}
\end{lemma}

\begin{proof}
For any function $f \in \L^1(M,\mathbb{C}^N)$, Minkowski's integral inequality gives
\[
\begin{aligned}
\bnorm{\widetilde{S}_a f}_{\L^2(M,\mathbb{C}^N)}
&\ov{\eqref{kernel-tilde}}{\leq}
\int_M
\left(
\int_M
\bnorm{\widetilde K_a(x,y)f(y)}_{\mathbb C^N}^2
\d\mu_g(x)
\right)^{\frac12}
\d\mu_g(y)     \\
&\leq
\left(
\esssup_{y \in M}
\int_M
\bnorm{\widetilde K_a(x,y)}_{\mathbb C^N\to\mathbb C^N}^2
\d\mu_g(x)
\right)^{\frac12}
\norm{f}_{\L^1(M,\mathbb{C}^N)}.
\end{aligned}
\]
Let $r \ov{\mathrm{def}}{=} a^{\frac1m}$. By definition, we have
\[
G_a(y,x)^2
\ov{\eqref{def-Gt}}{=}
\frac{c_1^2}{\vol(y,r)^2}
\exp\left(-2c_2\frac{\dist(x,y)^m}{r^m}\right).
\]
We decompose $M$ into the ball $B(y,r)$ and the annuli $A_k(y,r) \ov{\mathrm{def}}{=} B(y,2^{k+1}r)\setminus B(y,2^k r)$ for $k \geq 0$. If $x \in A_k(y,r)$, we have $\dist(x,y) \geq 2^k r$. Hence
\[
\exp\left(-2c_2\frac{\dist(x,y)^m}{r^m}\right)
\leq
\exp(-2c_2 2^{km}).
\]
Moreover, by the doubling property, there exist, by \cite[(3.4.9) p.~76]{HKST15} (or \cite[(2.1) p.~265]{Hal05}), some constants $C_D \geq 1$ and $d_0 > 0$ such that
\begin{equation}
\label{estimation-volumes}
\vol(y,2^{k+1}r)
\leq
C_D 2^{(k+1)d_0}\vol(y,r),
\quad k \geq 0.
\end{equation}
Therefore
\begin{align}
\label{fin-fin-32}
\int_M G_a(y,x)^2 \d\mu_g(x)
&\leq
\frac{c_1^2}{\vol(y,r)^2}
\left[
\vol(y,r)
+
\sum_{k=0}^\infty
\e^{-2c_2 2^{km}}
\vol(y,2^{k+1}r)
\right]  \\
&\ov{\eqref{estimation-volumes}}{\lesssim}
\frac{1}{\vol(y,r)^2}
\left[
\vol(y,r)
+
\sum_{k=0}^\infty
\e^{-2c_2 2^{km}}
2^{(k+1)d_0}
\vol(y,r)
\right]  
\lesssim
\frac{1}{\vol(y,r)}. \nonumber
\end{align}
The last series is finite because the exponential decay $\e^{-2c_2 2^{km}}$ dominates the polynomial growth $2^{(k+1)d_0}$. Consequently, using $
\bnorm{\widetilde{K}_a(x,y)}_{\mathbb{C}^N \to \mathbb{C}^N}
\ov{\eqref{estimation-max}}{\leq} G_a(y,x)$, we obtain
\[
\esssup_{y \in M} \int_M \bnorm{\widetilde K_a(x,y)}_{\mathbb{C}^N \to \mathbb{C}^N}^2
\d\mu_g(x)
\leq \esssup_{y \in M} \int_M G_a(y,x)^2 \d\mu_g(x)
\ov{\eqref{fin-fin-32}}{\lesssim}
\esssup_{y \in M} \vol(y,a^{\frac1m})^{-1}.
\]
This proves the $\L^1 \to \L^2$ estimate. Similarly, for $f \in \L^2(M,\mathbb{C}^N)$, Cauchy's inequality gives
\begin{align*}
\MoveEqLeft
\bnorm{\widetilde{S}_a f(x)}_{\mathbb{C}^N}
\ov{\eqref{kernel-tilde}}{=} \Bnorm{\int_M \widetilde{K}_a(x,y)f(y) \d\mu_g(y)}_{\mathbb{C}^N} 
\leq \int_M \bnorm{\widetilde{K}_a(x,y)}_{\mathbb{C}^N \to \mathbb{C}^N}\norm{f(y)}_{\mathbb{C}^N} \d\mu_g(y) \\
&\leq
\left(
\int_M \bnorm{\widetilde{K}_a(x,y)}_{\mathbb{C}^N \to \mathbb{C}^N}^2 \d\mu_g(y) \right)^{\frac12}
\norm{f}_{\L^2(M,\mathbb{C}^N)}.
\end{align*}
Using $
\bnorm{\widetilde{K}_a(x,y)}_{\mathbb{C}^N \to \mathbb{C}^N}
\ov{\eqref{estimation-max}}{\leq} G_a(x,y)$, 
we get
\[
\esssup_{x \in M}
\int_M
\bnorm{\widetilde{K}_a(x,y)}_{\mathbb{C}^N \to \mathbb{C}^N}^2 \d\mu_g(y)
\leq \esssup_{x \in M}
\int_M
G_a(x,y)^2 \d\mu_g(y)
\ov{\eqref{fin-fin-32}}{\lesssim}
\esssup_{x \in M} \vol(x,a^{\frac1m})^{-1}.
\]
This proves \eqref{eq:L1L2-L2Linf-real}.
\end{proof}

Applying \eqref{eq:L1L2-L2Linf-real} with $a=\frac{1-\delta}{2}t$, and using the boundedness of the operator $\widetilde{S}_{\delta t+\i s}$ on the space $\L^2(M,\mathbb{C}^N)$, we obtain
\[
\begin{aligned}
\bnorm{\widetilde S_z}_{\L^1(M,\mathbb{C}^N) \to \L^\infty(M,\mathbb{C}^N)}
&\ov{\eqref{decomposition}}{\leq}
\bnorm{\widetilde S_{\frac{1-\delta}{2}t}}_{\L^2 \to \L^\infty}
\bnorm{\widetilde S_{\delta t+\i s}}_{\L^2 \to \L^2}
\bnorm{\widetilde S_{\frac{1-\delta}{2}t}}_{\L^1 \to \L^2}        \\
&\ov{\eqref{eq:L1L2-L2Linf-real}}{\lesssim_\psi}
\esssup_{x \in M}\vol\left(x,\left(\tfrac{1-\delta}{2}t\right)^{\frac1m}\right)^{-1}.
\end{aligned}
\]
By the doubling property and the uniform volume comparability on the compact manifold $M$, this yields
\begin{equation}
\label{eq:L1Linf-complex}
\bnorm{\widetilde S_z}_{\L^1(M,\mathbb{C}^N)\to\L^\infty(M,\mathbb{C}^N)}
\lesssim_\psi
\esssup_{x \in M}\vol(x,t^{\frac1m})^{-1},
\quad z=t+\i s \in \Sigma_\psi.
\end{equation}
The bound \eqref{eq:L1Linf-complex} implies that, for every $z \in \Sigma_\psi$, the operator $\widetilde{S}_z$ is represented by a matrix-valued kernel $\widetilde{K}_z(x,y)\co \mathbb{C}^N \to \mathbb{C}^N$ such that
\[
(\widetilde{S}_z f)(x)
=
\int_M \widetilde K_z(x,y)f(y)\d\mu_g(y),
\qquad f\in \L^1(M,\mathbb{C}^N),
\]
and
\begin{equation}
\label{eq:complex-rough-kernel}
\bnorm{\widetilde K_z(x,y)}_{\mathbb{C}^N \to\mathbb C^N}
\lesssim_\psi
\vol\big(x,(\Re z)^{\frac1m}\big)^{-1}
\end{equation}
for almost every $(x,y) \in M \times M$. Indeed, let $(e_1,\ldots,e_N)$ be the canonical basis of $\mathbb{C}^N$. For any integers $1 \leq i,j \leq N$, define
\[
T_{ij}(z)f
\ov{\mathrm{def}}{=}
\pi_i\bigl(\widetilde S_z(f e_j)\bigr),
\quad f \in \L^2(M),
\]
where $\pi_i$ denotes the $i$-th coordinate map. It suffices to apply Dunford-Pettis theorem  \cite[p.~528]{Rob91} \cite[Proposition 3.1]{GrH14} to $T_{ij}(z)$ using the estimate \eqref{eq:L1Linf-complex}.

Let $0 < \nu < \psi$. It remains to upgrade the rough complex estimate \eqref{eq:complex-rough-kernel} to a Poisson estimate, in the sense of \cite[Definition 1 p.~266]{Hal05}, in the smaller sector $\Sigma_\nu$. For this, we will apply \cite[Proposition 3.3 p.~105]{DuR96} to the scalar matrix coefficients. Note that $(T_{ij}(z))_{z \in \Sigma_\theta}$ is a holomorphic family of bounded operators on the Hilbert space $\L^2(M)$. For any $z \in \Sigma_\psi$, its scalar kernel is
\begin{equation}
\label{def-kijz}
k_{ij,z}(x,y)
\ov{\mathrm{def}}{=}
\big\la \widetilde{K}_z(x,y)e_j,e_i \big\ra_{\mathbb{C}^N}.
\end{equation}
The rough complex estimate \eqref{eq:complex-rough-kernel} gives
\[
|k_{ij,z}(x,y)|
\ov{\eqref{def-kijz}}{=} \big|\big\la \widetilde{K}_z(x,y) e_j,e_i \big\ra_{\mathbb{C}^N}\big|
\ov{\eqref{eq:complex-rough-kernel}}{\lesssim_\varphi}
\vol\big(x,(\Re z)^{\frac1m}\big)^{-1},
\quad z \in \Sigma_\psi.
\]
On the other hand, the real-time Gaussian estimate gives
\[
|k_{ij,t}(x,y)|
\ov{\eqref{def-kijz}}{\leq}
\bnorm{\widetilde{K}_t(x,y)}_{\mathbb{C}^N \to \mathbb{C}^N}
\ov{\eqref{estimation-max}}{\leq}
\frac{C}{\vol(x,t^{\frac1m})}
g\left(\frac{\dist(x,y)^m}{t}\right),
\quad t > 0.
\]
Thus the assumptions of \cite[Proposition 3.3 p.~105]{DuR96} are satisfied by the scalar holomorphic family $(T_{ij}(z))_{z\in\Sigma_\theta}$ in the sector $\Sigma_\psi$. Consequently, in the smaller sector $\Sigma_\nu$, there exist a constant $C_\nu>0$ and a bounded decreasing function $g_\nu$ satisfying the Poisson decay condition such that
\[
|k_{ij,z}(x,y)|
\leq
\frac{C_\nu}{\vol(x,(\Re z)^{\frac1m})}
g_\nu\left(
\frac{\dist(x,y)^m}{\Re z}
\right),
\quad z \in \Sigma_\nu.
\]
Since $N<\infty$, the operator norm of a matrix is controlled by finitely many scalar matrix coefficients. Therefore, after changing $C_\nu$, we obtain
\[
\bnorm{\widetilde K_z(x,y)}_{\mathbb{C}^N \to \mathbb{C}^N}
\leq
\frac{C_\nu}{\vol(x,(\Re z)^{\frac1m})}
g_\nu\left(
\frac{\dist(x,y)^m}{\Re z}
\right),
\quad z \in \Sigma_\nu.
\]
Using the uniform comparability of volumes on the compact manifold $M$, the same estimate holds with $x$ and $y$ interchanged, after another change of constants. Hence
\begin{equation}
\label{equa-655-final}
\bnorm{\widetilde K_z(x,y)}_{\mathbb{C}^N \to \mathbb{C}^N}
\leq \min \big\{
G_{\Re z}^{(\nu)}(x,y),
G_{\Re z}^{(\nu)}(y,x) \big\},
\quad z \in \Sigma_\nu,
\end{equation}
where
\[
G_t^{(\nu)}(x,y)
\ov{\mathrm{def}}{=}
\frac{C_\nu}{\vol(x,t^{\frac1m})}
g_\nu\left(\frac{\dist(x,y)^m}{t}\right).
\]
Since $\nu<\theta$ was arbitrary, the Poisson angle of $(\widetilde S_z)$ on the space $\L^2(M,\mathbb{C}^N)$ is at least $\theta$.

By \cite[Proposition~2.1 p.~94]{DuR96}, applied to the majorant \(G_t^{(\nu)}\), and then with the variables interchanged, we have
\[
\esssup_{x \in M} \int_M G_{\Re z}^{(\nu)}(x,y) \d\mu_g(y)
\lesssim 1
\quad \text{and} \quad
\esssup_{y \in M} \int_M G_{\Re z}^{(\nu)}(y,x) \d\mu_g(x)
\lesssim 1.
\]
Using \eqref{equa-655-final}, Schur's test for matrix-valued kernels (as for the proof of \eqref{eq:L1Linf-complex}) yields
\[
\sup_{z \in \Sigma_\nu}
\bnorm{\widetilde S_z}_{\L^p(M,\mathbb{C}^N) \to \L^p(M,\mathbb{C}^N)}
<\infty,
\quad 1 \leq p \leq \infty.
\]

Moreover, this family is holomorphic on $\Sigma_\nu$ as a $\B(\L^p(M,\mathbb C^N))$-valued map. Indeed, consider some compact subset  \(K\) of \(\Sigma_\nu\). Let $f \in \L^p(M,\mathbb C^N)$ and \(g\in \L^{p^*}(M,\mathbb C^N)\). Choose sequences $(f_n)$ in $\L^p(M,\mathbb C^N)\cap \L^2(M,\mathbb{C}^N)$ and $(g_n)$ in $\L^{p^*}(M,\mathbb{C}^N) \cap \L^2(M,\mathbb{C}^N)$ such that $f_n\to f$ in $\L^p(M,\mathbb{C}^N)$, $g_n \to g$ in $\L^{p^*}(M,\mathbb{C}^N)$. For every \(n\), the scalar function
\[
F_n(z)
\ov{\mathrm{def}}{=}
\int_M
\big\la \widetilde S_z f_n(x),g_n(x)\big\ra_{\mathbb C^N}
\d\mu_g(x)
\]
is holomorphic on $\Sigma_\nu$, since it is the \(\L^2\)-pairing of the \(\L^2\)-holomorphic map \(z\mapsto \widetilde S_z f_n\) with \(g_n\). Moreover, if
\[
F(z)
\ov{\mathrm{def}}{=}
\int_M
\big\la \widetilde S_z f(x),g(x)\big\ra_{\mathbb{C}^N}
\d\mu_g(x),
\]
then, uniformly for $z \in K$,
\[
\begin{aligned}
|F_n(z)-F(z)|
&\leq
\left|
\int_M
\big\la \widetilde S_z(f_n-f)(x),g_n(x)\big\ra_{\mathbb{C}^N}
\d\mu_g(x)
\right| 
+
\left|
\int_M
\big\la \widetilde S_z f(x),g_n(x)-g(x)\big\ra_{\mathbb{C}^N}
\d\mu_g(x)
\right|  \\
&\lesssim
\norm{f_n-f}_{\L^p(M,\mathbb{C}^N)}
\norm{g_n}_{\L^{p^*}(M,\mathbb{C}^N)}
+
\norm{f}_{\L^p(M,\mathbb{C}^N)}
\norm{g_n-g}_{\L^{p^*}(M,\mathbb{C}^N)}.
\end{aligned}
\]
Since $(g_n)$ is bounded in $\L^{p^*}(M,\mathbb{C}^N)$, the right-hand side tends to \(0\). Thus \(F_n \to F\) uniformly on compact subsets of \(\Sigma_\nu\). Therefore \(F\) is holomorphic on $\Sigma_\nu$. Consequently, for every $f \in \L^p(M,\mathbb C^N)$ and every $g \in \L^{p^*}(M,\mathbb C^N)$, the scalar function $
z \mapsto \la \widetilde S_z f,g \ra_{\L^p,\L^{p^*}}$ is holomorphic on $\Sigma_\nu$. By \cite[Corollary~B.3.3 p.~530]{HvNVW16}, the map $z \mapsto \widetilde{S}_z$ is holomorphic as a $\B(\L^p(M,\mathbb{C}^N))$-valued map on the open sector $\Sigma_\nu$. Consequently, \((\widetilde S_z)_{z\in\Sigma_\nu}\) is a bounded holomorphic semigroup on \(\L^p(M,\mathbb C^N)\). By \cite[Theorem~G.5.2 p.~537]{HvNVW18}, its negative generator \(-\widetilde A_p\) satisfies $\omega_{\sec}(\widetilde A_p)\leq \frac{\pi}{2}-\nu$ for every $0 < \nu < \theta$. In particular, we have $\omega_{\sec}(\widetilde A_p)<\frac{\pi}{2}$. 

Now, we apply the vector-valued extrapolation theorem of Haller-Dintelmann \cite[Theorem 3 p.~267]{Hal05} with $p_0=2$ and value space $X = \mathbb{C}^N$. The space $\mathbb{C}^N$ is finite-dimensional, hence reflexive and has the Radon--Nikodym property by \cite[Theorem 1.3.21 p.~53]{HvNVW16}. Moreover, the compactness of $M$ implies the metric-measure assumptions (A1), (A2) and (A3) of \cite[Section 2]{Hal05}, as observed in \cite[p.~264]{Hal05}.

We now make explicit the consistency of the resolvents. Let $-\widetilde A_r$ be the generator of the strongly continuous semigroup $(\widetilde S_t^{(r)})_{t \geq 0}$ on $\L^r(M,\mathbb{C}^N)$. Let $p,q \in (1,\infty)$. Since the semigroups are consistent, the Laplace formula for the resolvent gives, for $\Re \lambda <0$ and $f \in \L^p(M,\mathbb{C}^N) \cap \L^q(M,\mathbb{C}^N)$,
\[
R(\lambda,\widetilde A_p)f
\ov{\eqref{Resolvent-Laplace}}{=}
-\int_0^\infty \e^{\lambda t}\widetilde S_t^{(p)}f \d t
=
-\int_0^\infty \e^{\lambda t}\widetilde S_t^{(q)}f \d t
\ov{\eqref{Resolvent-Laplace}}{=}
R(\lambda,\widetilde A_q)f.
\]
Here the equality may be read in $\L^1(M,\mathbb C^N)$, since $M$ has finite measure and the spaces $\L^p(M,\mathbb{C}^N)$ and $\L^q(M,\mathbb{C}^N)$ are continuously embedded in $\L^1(M,\mathbb C^N)$.

Let $\Omega_{p,q}$ be the connected component of $\rho(\widetilde A_p)\cap \rho(\widetilde A_q)$ which contains the half-plane $\{\Re \lambda<0\}$. We claim that
\[
R(\lambda,\widetilde A_p)f
=
R(\lambda,\widetilde A_q)f,
\qquad
\lambda \in \Omega_{p,q},\quad f \in \L^p(M,\mathbb{C}^N) \cap \L^q(M,\mathbb{C}^N).
\]
Indeed, denote by $\mathcal C$ the set of all $\lambda \in \Omega_{p,q}$ for which this
identity holds for every $f \in \L^p(M,\mathbb{C}^N) \cap \L^q(M,\mathbb{C}^N)$. The preceding paragraph shows that
$\mathcal C$ is non-empty.

We prove first that $\mathcal C$ is open in $\Omega_{p,q}$. Let $\lambda_0 \in \mathcal C$. Choose $\epsi > 0$ such that
\[
\epsi
<
\min\left\{
\frac{1}{\|R(\lambda_0,\widetilde A_p)\|_{\B(\L^p(M,\mathbb{C}^N))}},
\frac{1}{\|R(\lambda_0,\widetilde A_q)\|_{\B(\L^q(M,\mathbb{C}^N))}}
\right\}
\]
and such that $B(\lambda_0,\varepsilon) \subset \Omega_{p,q}$. By the Neumann series argument for the resolvent, for $|\lambda-\lambda_0| < \epsi$ and $r \in \{p,q\}$ we have
\[
R(\lambda,\widetilde A_r)
=
\sum_{n=0}^{\infty}
(\lambda-\lambda_0)^n R(\lambda_0,\widetilde A_r)^{n+1}
\]
with convergence in $\B(\L^r(M,\mathbb{C}^N))$. Since $\lambda_0 \in \mathcal C$, the operator $R(\lambda_0,\widetilde A_p)$ and the operator $R(\lambda_0,\widetilde A_q)$ agree on $\L^p(M,\mathbb{C}^N) \cap \L^q(M,\mathbb{C}^N)$. In particular, this common operator leaves $\L^p(M,\mathbb{C}^N) \cap \L^q(M,\mathbb{C}^N)$ invariant, and therefore all its powers agree on $\L^p(M,\mathbb{C}^N) \cap \L^q(M,\mathbb{C}^N)$. Hence the two Neumann series above coincide on $\L^p(M,\mathbb{C}^N) \cap \L^q(M,\mathbb{C}^N)$. Thus $\lambda \in \mathcal C$, proving that $\mathcal C$ is open.

We prove next that $\mathcal C$ is closed in $\Omega_{p,q}$. Let $(\lambda_n)$ be a sequence in
$\mathcal C$ converging to some $\lambda \in \Omega_{p,q}$. By continuity of the resolvent on the
resolvent set, for every $f \in \L^p(M,\mathbb{C}^N) \cap \L^q(M,\mathbb{C}^N)$ we have
\[
R(\lambda_n,\widetilde A_p)f \longrightarrow R(\lambda,\widetilde A_p)f
\quad\text{in } \L^p(M,\mathbb{C}^N)
\]
and
\[
R(\lambda_n,\widetilde A_q)f \longrightarrow R(\lambda,\widetilde A_q)f
\quad\text{in } \L^q(M,\mathbb{C}^N).
\]
Since $\L^p(M,\mathbb{C}^N)$ and $\L^q(M,\mathbb{C}^N)$ are continuously embedded in $\L^1(M,\mathbb C^N)$, both convergences hold in $\L^1(M,\mathbb C^N)$. But for every integer $n$ we have
\[
R(\lambda_n,\widetilde A_p)f
=
R(\lambda_n,\widetilde A_q)f.
\]
Passing to the limit in $\L^1(M,\mathbb C^N)$ gives $R(\lambda,\widetilde A_p)f
=
R(\lambda,\widetilde A_q)f$. Thus $\lambda \in \mathcal C$, and $\mathcal C$ is closed.

Since $\Omega_{p,q}$ is connected, we conclude that $\mathcal C=\Omega_{p,q}$. In particular, the resolvents of $\widetilde A_p$ and $\widetilde A_2$ are consistent on the connected component of $\rho(\widetilde A_p) \cap \rho(\widetilde A_2)$ which contains the negative real axis.

On the Hilbert space $\L^2(M,\mathbb{C}^N)$, the operator $\widetilde{A}_2=\Phi_2 A\Phi_2^{-1}$ admits a bounded $\H^\infty(\Sigma_\mu)$ functional calculus by \cite[Proposition 3.2.10 p.~127]{Ege15}, since $A$ does. Furthermore, the preceding Poisson estimate shows that the Poisson angle, defined in \cite[Definition 1 p.~266]{Hal05}, of the semigroup generated by $-\widetilde{A}_2$ is at least $\theta$.  

The proof of Haller-Dintelmann's Theorem~\cite[Theorem 3 p.~267]{Hal05} gives, for every $p \in (1,\infty)$, a bounded $\H^\infty(\Sigma_{\theta'})$ functional calculus for $\widetilde{A}_p$ on the Bochner space $\L^p(M,\mathbb{C}^N)$ for some angle $\theta' \in (0, \frac{\pi}{2})$. Strictly speaking, \cite[Theorem 3 p.~267]{Hal05} assumes $
\omega_{\sec}(\widetilde A_p)\leq \omega_{\sec}(\widetilde A_{p_0})$ in order to obtain the angle estimate $ \omega_{\H^\infty}(\widetilde A_p)\leq \theta_0$. We do not need this sharp estimate. We only use the proof of \cite[Theorem~3]{Hal05}. In the present situation the Poisson angle of \((\widetilde S_z)\) is at least \(\theta\), and \(\widetilde A_2\) admits a bounded \(\H^\infty(\Sigma_\mu)\) functional calculus. Hence the number \(\theta_0\) appearing in \cite[Theorem~3]{Hal05} satisfies
\[
\theta_0
\leq
\max\left\{\mu,\frac{\pi}{2}-\theta\right\}
=
\mu
<
\frac{\pi}{2}.
\]
We have seen that $\omega_{\sec}(\widetilde{A}_p)<\frac{\pi}{2}$. Choose
\[
\theta' \in
\left(\max\{\mu,\omega_{\sec}(\widetilde A_p)\},\frac{\pi}{2}\right).
\]
Then $\theta'>\theta_0$, so \cite[Proposition~12 p.~272]{Hal05} gives the weak type \((1,1)\) estimate for $f(\widetilde A_2)$ for every function $f \in \H^\infty_0(\Sigma_{\theta'})$, which extends by the Marcinkiewicz interpolation theorem to a bounded operator $\big(f(\widetilde{A}_2)\big)_p$ on the Banach space \(\L^p(M,\mathbb{C}^N)\), with 
$$
\bnorm{(f(\widetilde A_2))_p}_{\L^p(M,\mathbb{C}^N) \to \L^p(M,\mathbb{C}^N)}
\lesssim
\norm{f}_{\H^\infty(\Sigma_{\theta'})}.
$$ 
Now, choose an angle $\varphi \in
\big(
\max\{\omega_{\sec}(\widetilde{A}_2),\omega_{\sec}(\widetilde{A}_p)\},\theta'\big)$. For any function $g \in \L^2(M,\mathbb{C}^N) \cap \L^p(M,\mathbb{C}^N)$, the consistency of the resolvents gives
$$
f(\widetilde A_p)g
\ov{\eqref{2CauchySec}}{=}
\frac{1}{2\pi \i}\int_{\partial \Sigma_\varphi}
f(z)R(z,\widetilde{A}_p)g \d z  
=
\frac{1}{2\pi \i}\int_{\partial\Sigma_\varphi}
f(z)R(z,\widetilde A_2)g \d z  
\ov{\eqref{2CauchySec}}{=}
f(\widetilde{A}_2)g
=
(f(\widetilde{A}_2))_p g.
$$
By density, \(f(\widetilde A_p)=\big(f(\widetilde A_2)\big)_p\) on the space \(\L^p(M,\mathbb{C}^N)\). Therefore \(\widetilde{A}_p\) admits a bounded \(\H^\infty(\Sigma_{\theta'})\) functional calculus. Since \(\theta' <\frac{\pi}{2}\), this gives the desired conclusion. We should note that \cite[Theorem~3]{Hal05} is stated for sectorial operators with dense range. In the present situation, this causes no difficulty for our use of the result by classical arguments.
 
Finally, the operator $-A_p \ov{\mathrm{def}}{=}
-\Phi_p^{-1}\widetilde A_p\Phi_p$ is precisely the generator of the extended semigroup $(S_t)_{t \geq 0}$ on the space $\L^p(M,E)$, and the bounded $\H^\infty$ functional calculus transfers back through the isometric isomorphism $\Phi_p$. Hence $A_p$ admits a bounded $\H^\infty(\Sigma_{\theta'})$ functional calculus on the Banach space $\L^p(M,E)$ for some angle $\theta' \in (0,\frac{\pi}{2})$.
\end{proof}

\paragraph{Declaration of interest} None.

\paragraph{Competing interests} The author declares that he has no competing interests.



{\footnotesize


\noindent C\'edric Arhancet\\ 
\noindent 6 rue Didier Daurat, 81000 Albi, France\\
URL: \href{http://sites.google.com/site/cedricarhancet}{https://sites.google.com/site/cedricarhancet}\\
cedric.arhancet@protonmail.com\\
ORCID: 0000-0002-5179-6972 
}

\end{document}